\documentclass[a4paper,12pt,leqno, twoside]{article}
\usepackage{amssymb}
\usepackage{amsthm}
\usepackage{amsmath}
\usepackage[all]{xy} 
\usepackage[mathscr]{eucal}

\usepackage{color}
\usepackage{comment}

\usepackage{mathrsfs} % for \mathscr (script letters)

\usepackage[nonewpage]{imakeidx}
\makeindex[intoc,name=notation,title=Selected notation]
\makeindex[intoc,name=terminology,title=Selected terminology]

\def\NM{{\mathbb{N}}}

\def\RM{{\mathbb{R}}}

\def\QM{{\mathbb{Q}}}
\def\FM{{\mathbb{F}}}
\def\ZM{{\mathbb{Z}}}

\def\CM{{\mathbb{C}}}

\def\mG{{\mathfrak m}}

\def\gG{{\mathfrak g}}
\def\kG{{\mathfrak k}}

\def\hG{{\mathfrak h}}
\def\jG{{\mathfrak j}}
\def\nG{{\mathfrak n}}
\def\sG{{\mathfrak s}}
\def\lG{{\mathfrak l}}
\def\fl{{\mathfrak l}}

\def\uG{{\mathfrak u}}

\def\ZG{{\mathfrak Z}}

\newcommand{\bC}{\mathbb{C}}
\newcommand{\bF}{\mathbb{F}}
\newcommand{\bQ}{\mathbb{Q}}
\newcommand{\bR}{\mathbb{R}}
\newcommand{\bZ}{\mathbb{Z}}

\newcommand{\cC}{\mathcal{C}}
\newcommand{\cF}{\mathcal{F}}
\newcommand{\cH}{\mathcal{H}}
\newcommand{\cI}{\mathcal{I}}
\newcommand{\cL}{\mathcal{L}}

\newcommand{\cW}{\mathcal{W}}
\newcommand{\cV}{\mathcal{V}}

\newcommand{\fg}{\mathfrak{g}}
\newcommand{\fu}{\mathfrak{u}}
\newcommand{\fm}{\mathfrak{m}}

\newcommand{\sS}{\mathscr{S}}

\newcommand{\bfG}{\mathbf{G}}
\newcommand{\bfM}{\mathbf{M}}

\def\AC{{\mathcal A}}

\def\SC{{\mathcal S}}
\def\CC{{\mathcal C}}
\def\HC{{\mathcal H}}
\def\NC{{\mathcal N}}

\def\RC{{\mathcal R}}
\def\IC{{\mathcal I}}
\def\OC{{\mathcal O}}
\def\MC{{\mathcal M}}

\def\LC{{\mathcal L}}
\def\EC{{\mathcal E}}
\def\FC{{\mathcal F}}
\def\VC{{\mathcal V}}
\def\WC{{\mathcal W}}
\def\GC{{\mathcal G}}

\def\BC{{\mathcal B}}
\def\TC{{\mathcal T}}

\def\subsetneq{\varsubsetneq}

\def\simto{\buildrel\hbox{\tiny{$\sim$}}\over\longrightarrow}
\def\simfrom{\buildrel\hbox{\tiny{$\sim$}}\over\longleftarrow}

\def\leq{\leqslant}
\def\geq{\geqslant}
\def\injo{\hookrightarrow}
\def\id{\mathop{\mathrm{Id}}\nolimits}

\def\wt{\widetilde}
\def\wh{\widehat}
\def\o#1{\overline{#1}}
\def\eps{\varepsilon}

\def\application#1#2#3#4#5{\begin{array}{rcl}
                            #1 \;\;\; #2 & \to &  #3 \\
                              #4 & \mapsto & #5 
                            \end{array}}

 \newcommand{\ra}{\rightarrow}
 \newcommand{\xra}{\xrightarrow}
 \newcommand{\ov}{\overline}

\def\To#1{\buildrel\hbox{\tiny{$#1$}}\over\longrightarrow}
\def\to{\rightarrow}

\def\Mod#1{\mathop{\hbox {\sl #1-Mod}}}     %cat des R-modules {?} gauche
\def\ker{{\rm ker}}

\def\im{{\rm im}}

\def\Hom{\mathop{\hbox{\rm Hom}}\nolimits}
\def\Aut{\mathop{\hbox{\rm Aut}}\nolimits}
\def\Ext{\mathop{\hbox{\rm Ext}}\nolimits}

\def\End{\mathop{\hbox{\rm End}}\nolimits}
\def\cind{\mathop{\hbox{\rm ind}}\nolimits}

\def\Mod{{\rm Mod}}
\def\Rep{{\rm {R}ep}} 
\def\Ind#1#2{\hbox {\rm Ind}_{#1}^{#2}}
\def\ind#1#2#3{\hbox {\rm Ind}_{#1}^{#2}\>\!\left(#3\right)}  %induction
\def\cInd#1#2{\hbox {\rm ind}_{#1}^{#2}}
\def\dim{\mathop{\mbox{\rm dim}}\nolimits}
\def\dim{{\rm dim}}

\def\Ad{{\rm Ad}}

\def\BT{{\BC}}

\def\colim{{\rm colim}}

\def\ini{\setcounter{equation}{\value{subsubsection}}\stepcounter{subsubsection}}
\def\inisub{\setcounter{equation}{\value{subsubsection}}\addtocounter{equation}{-1}}

\makeatletter
\renewcommand{\subsubsection}{\@startsection{subsubsection}{3}{\parindent}{-\baselineskip}{-0.01\baselineskip}{\bf}}
\renewcommand*{\@seccntformat}[1]{%
  \csname the#1\endcsname\
}
\makeatother  

\def\ali{\subsubsection{}\setcounter{equation}{0}}
\def\alin#1{\setcounter{equation}{0}\subsubsection{\it  #1}. --- }

\newtheoremstyle{th}
  {\baselineskip}{.5\baselineskip}{\itshape}
  {\parindent}{\bf}
  { ---}{.5em}{}

\newtheoremstyle{def}
  {\baselineskip}{\baselineskip}{}
  {\parindent}{\bf}
  {--}{.5em}{}

\newtheoremstyle{th*}
  {.5\baselineskip}{.5\baselineskip}{\itshape}
  {\parindent}{\bf}
  { ---}{.5em}{}

\newtheoremstyle{remark*}
  {.5\baselineskip}{.5\baselineskip}{}
  {\parindent}{\bf}
  {--}{.5em}{}

\newtheoremstyle{remark}
  {.5\baselineskip}{.5\baselineskip}{}
  {\parindent}{\bf}
  { ---}{.5em}{}

\swapnumbers
\theoremstyle{th}
\newtheorem{theo}[subsubsection]{\it Theorem.\bf}
\newtheorem{lemme}[subsubsection]{\it Lemma.\bf}
\newtheorem{prop}[subsubsection]{\it Proposition.\bf}
\newtheorem{coro}[subsubsection]{\it Corollary.\bf}

\newtheorem{assumption}[subsubsection]{\it Assumption.\bf}

\theoremstyle{def}

\newtheorem{DEf}[subsubsection]{\it D{e}finition.\bf}
\newtheorem{DEflem}[subsubsection]{\it D{e}finition / Lemma.\bf}

\theoremstyle{remark}
\newtheorem{rema}[subsubsection]{\it Remark.\bf}  %\renewcommand{\therema}{}
\newtheorem{nota}[subsubsection]{\it Notation.\bf}  %\renewcommand{\thenota}{}

\theoremstyle{th*}
\newtheorem*{thm}{\it Theorem.}

\newtheorem*{lem}{\it Lemma.}
\newtheorem*{pro}{\it Proposition.}
\newtheorem*{cor}{\it Corollary.}

\newtheorem*{main}{\bf Main result.}

\theoremstyle{remark*}
\newtheorem*{defn}{\it Definition.}
\newtheorem*{rak}{\it Remark.}

\newtheorem*{no}{\it Notation.}

\usepackage{hyperref}

\newcommand{\findem}{\hfill$\Box$\par\medskip}
\newcommand{\dem}{\indent {\it Preuve :} \rm }

\title{Parametrization and reduction to depth zero of $\overline{\mathbb{Z}}[\frac{1}{p}]$-blocks of tame $p$-adic groups}
\author{Jean-Fran\c{c}ois Dat and Jessica Fintzen}
\date{}

\AtEndDocument{
  {\footnotesize
         \vspace{.5cm}
   \textsc{Jean-Fran\c{c}ois Dat,   Institut de Math\'ematiques de
   Jussieu-PRG, Sorbonne Universit\'e, Universit\'e de Paris Cit\'e, CNRS}, 4
     place Jussieu, 75005 Paris, France.

     \textit{E-mail address :} \texttt{jean-francois.dat@imj-prg.fr}

     \vspace{.5cm}

\textsc{Jessica Fintzen, Universität Bonn, Mathematisches Institut,}
Endenicher Allee 60, 53115 Bonn, Germany.

\textit{E-mail address}: \texttt{fintzen@math.uni-bonn.de}
}}

\begin{document}
\maketitle
\bibliographystyle{alpha}
\def\knr{{\wh{K^{nr}}}}
\def\ka{\wh{K^{ca}}}

\def\SL{{\rm SL}}
\def\Sp{{\rm Sp}}
\def\O{{\rm O}}
\def\GL{{\rm GL}}
\def\PGL{{\rm PGL}}
\def\G{{\rm G}}

\def\Ql{\QM_{\ell}}
\def\Zl{\ZM_{\ell}}
\def\Fl{\FM_{\ell}}
\def\oQl{\o\QM_{\ell}}
\def\oZl{\o\ZM_{\ell}}
\def\bZl{\o\ZM_{\ell}}
\def\bQl{\o\QM_{\ell}}
\def\oFl{\o\FM_{\ell}}

\def\sp{{\rm sp}}
\def\Ens{{\SC ets}}
\def\Coef{{\rm Coef}}
\def\Fr{{\rm Fr}}

\def\vG{{K}}
\def\Kp{K^{+}}
\def\Kpp{K^{++}}
\def\Ku{K^{\dagger}}
\def\Ko{K^{\circ}}
\def\K{K^{}}

\def\KpM{K^{+}}

\def\Up{U^{+}}
\def\Uu{U^{\dagger}}
\def\bUp{\bar{U}^{+}}
\def\bUu{\bar{U}^{\dagger}}

\def\o#1{\overline{#1}}

\newcommand{\Weil}{\omega} 
\newcommand{\HeisWeil}{\kappa}
\newcommand{\HeisWeilQ}{\kappa^{\overline\bQ}}
\newcommand{\HeisWeilMQ}{\kappa^{M,\overline\bQ}}

\def\Heis{\eta}
\def\neweta{\nu}
\def\Rmin{\ZM[\mu_{p^\infty},\frac 1 p]}
 \def\Rmintwo{\ZM[\mu_{4p^\infty},\frac{1}{\sqrt{p}}]}
 \def\Rminthree{\ZM[\mu_{p^\infty},\frac{1}{\sqrt{p}}]}
 \def\R{\o\ZM[\frac 1p]}
 \def\H{{\rm H}}

 \begin{abstract}
     Let $\mathbf{G}$ be a reductive group over a non-archimedean local
   field $F$ of residue characteristic $p$.
   We consider pairs $(\phi,I)$ consisting of  
   a ``wild  inertia'' Langlands parameter $\phi:\, P_{F} \longrightarrow \hat{\mathbf{G}}$
   whose centralizer $C_{\hat{\mathbf{G}}}(\phi)$ is a  Levi subgroup of $\hat{\mathbf{G}}$,
   and a cohomological   invariant $I$ whose definition is inspired by  the theory of
   endoscopy.  Assuming 
   that $p$ is odd and   not a torsion prime of $\mathbf{G}$ nor of $\hat{\mathbf{G}}$,
   we associate to each such pair $(\phi,I)$  a Serre subcategory $\mathrm{Rep}^{\phi,I}(\mathbf{G}(F))$
   of the category of smooth $\overline{\mathbb{Z}}[\frac{1}{p}]$-representations of $\mathbf{G}(F)$. Then we
   construct an equivalence between this Serre subcategory and the category of depth-zero
   $\overline{\mathbb{Z}}[\frac{1}{p}]$-representations of a twisted Levi subgroup $\mathbf{G}_{\phi,I}$ of
   $\mathbf{G}$, which is dual to  
   $C_{\hat{\mathbf{G}}}(\phi)$.  This pattern for
   reduction to depth zero  fits well with the conjectural (categorical) local Langlands correspondence. 
    
   When $\mathbf{G}$ is tamely ramified and $p$ does not
   divide the order of its Weyl group, then the above Serre subcategories provide the block decomposition of the category of all smooth $\overline{\mathbb{Z}}[\frac{1}{p}]$-representations of $\mathbf{G}(F)$.
  	In this case, we thus obtain a reduction-to-depth-zero process
   for smooth representations of $\mathbf{G}(F)$ valued in any
   algebraically closed field of characteristic different from  $p$.
	When that field has characteristic $0$, this recovers some of the recent results of Adler--Fintzen--Mishra--Ohara.
	 When that field is
  $\overline{\mathbb{F}}_{\ell}$, we use our results together with  Zhu's unipotent categorical correspondence
    to produce a fully faithful embedding of
    $D\mathrm{Rep}_{\overline{\mathbb{F}}_{\ell}}(\mathrm{GL}_{n}(F))$ into a suitable category of coherent sheaves on the moduli space of
    $n$-dimensional $\overline{\mathbb{F}}_{\ell}$-representations of the Weil group.  
     \end{abstract}

{
	\renewcommand{\thefootnote}{}  % to delete the footnote number 	
	\footnotetext{The first-named author acknowledges the support
          of ANR through grant ``PPAL'' ANR-25-CE40-4664 }
	\footnotetext{The second-named author was partially supported by 
	the European Research Council (ERC) under the European Union's Horizon 2020 research
	and innovation programme (grant agreement n° 950326).}          
}

\newpage

\tableofcontents

\section{Introduction}

 Let ${\mathbf{G}}$ be a connected reductive group defined over a non-archimedean local field $F$ and set
 $G:={\mathbf{G}}(F)$. For any commutative ring $R$, we denote by $\Rep_{R}(G)$ the category of smooth $R$-representations of $G$.
 Here is a coarse representation-theoretic takeaway of what is done in this paper.
 
 \begin{main}[Take 1]  Assume that $\mathbf{G}$ is tamely ramified and that
   $p$ does not divide the order of the absolute Weyl group of $\mathbf{G}$.
   We construct the block decomposition of $\Rep_{\o\ZM[\frac 1p]}(G)$ and show that each
   block is  equivalent to the depth-$0$ block of an explicitly associated tamely
   ramified reductive group (which turns out to be a twisted Levi subgroup of
   $\mathbf{G}$). Moreover, each
   block is defined over $\ZM[\mu_{p^{\infty}},\frac 1p]$ and the equivalences can be
   defined over $\ZM[\mu_{4p^{\infty}},\frac 1{\sqrt p}]$.
 \end{main}
 
 Note that this implies a reduction to depth $0$ for blocks 
    over
 any  $\ZM[\mu_{4p^{\infty}},\frac 1{\sqrt p}]$-algebra $R$, and in
 particular over any algebraically closed field of
 characteristic different from $p$.

 We like to remind readers mostly acquainted with complex representation theory of $p$-adic groups  that there is in general no Bernstein decomposition nor type theoretic
 construction of blocks of $\Rep_{R}(G)$ whenever the pro-order of $G$ has a divisor that
 is not invertible in $R$.
 Even when such a decomposition exists, e.g., in the settings of
 \cite{HKSS}, \cite{Lanard_blocks}, \cite{Minguez_Secherre}, 
 \cite{Vig_Induced}, there is in general no equivalence between a
 block and a category of modules over a nice Hecke algebra. Therefore, our techniques have to be quite
 different from those used in \cite{AFMO2}, where a similar ``reduction to depth $0$'' result is
 proved for complex representations by comparing suitable Hecke
 algebras.\footnote{Note however that our main result, stated below, allows us to reprove
   that the relevant Hecke algebras of \cite{AFMO2} are isomorphic, albeit in a less
   explicit way, see Corollary \ref{cor-Hecke-alg-isom}.}

 To give an idea of our techniques, we start with the following description of the depth-$0$
 subcategory (the notation for the latter will be explained below)
 \[ \Rep^{1}_{R}(G)=\left\{V\in\Rep_{R}(G),\, V=\sum_{x\in \BT} e_{G_{x,0+}}V\right\}\]
where $\BT$ denotes the Bruhat--Tits building of $\mathbf G$, the group $G_{x,0+}$ is the pro-$p$-radical of the
parahoric subgroup attached to $x$, and $e_{G_{x,0+}}$ is the averaging idempotent
along this pro-$p$-group, which is an element of the  ring $\HC_{R}(G)$ of all compactly
supported, locally constant $R$-valued distributions on $G$, and thus acts on $V$ (as a projector onto the fixed vectors
$V^{G_{x,0+}}$).

Our blocks are constructed in a similar way in Section 2, via families of idempotents indexed by
points in the building. These families have nice consistency properties that allow us to
apply results of \cite{MS1} and \cite{Lanard}, which show that our blocks are equivalent to
certain categories of coefficient systems on the building. This is the replacement for the
missing Hecke algebras bridge in this setting. Our equivalences between blocks are then obtained
in Section 3 by constructing equivalences between these categories of
coefficient systems, following an overall strategy introduced in \cite{Datequiv}.

The definition of these families of idempotents ultimately relies on Yu's
notion of generic characters, which is part of the construction of types in \cite{Yu_tamescusp}
and \cite{Kim_Yu}. The hypothesis on $p$ in the first take on our main result above
ensures that  we obtain the whole category, thanks to the exhaustion result of \cite{JF_exhaust}.
For the construction of the equivalences, we  need more than the generic characters,
namely we need consistent families of twisted Heisenberg--Weil
representations, building on the twists introduced in \cite{FKS}.

We note that in \cite{HKSS}, the block decomposition of $\Rep_{\o\ZM[\frac 1p]}(G)$ is
worked out with minimal hypothesis ($p\neq 2$) when $\mathbf G$ is a
classical group, using the notion of semisimple characters due to Bushnell--Kutzko and Stevens.
It would be interesting to see if our strategy for reduction to
depth-$0$ can be carried out in that setting.

\medskip

This paper could have  been written purely in representation theoretic terms, labeling  our blocks by some
suitable equivalence classes of  open pro-$p$-subgroups of $G$ and characters thereof, that appear in the construction of types in \cite{Kim_Yu}.  However, \emph{this is not the path we follow.}

\medskip

 Instead, we will rather start from the spectral side of the Langlands
correspondence, and our blocks will be indexed by ``extended'' wild inertia parameters, to
be defined just below. A practical reason is that the 
reductive groups  appearing in the
targets of our equivalences then have a transparent interpretation as suitable inner forms
of the dual of the centralizer of the parameter.  
A deeper reason is that the pattern for
reduction to  depth $0$ given by this point of
view (exposed in \cite[\S 4.5]{DatIHES} and anticipated in
\cite{Datfuncto}) should work in greater generality, a hope that is strongly supported by the
categorical local Langlands 
program, as we will explain below. In order to state our main result with more details we first need to
introduce some additional notation. 

\bigskip

\textbf{Parameters.} We denote by $^{L}{\mathbf{G}}=\hat{\mathbf{G}}\rtimes W_{F}$  the Weil
 form of the $L$-group of ${\mathbf{G}}$ over $\CM$ and
 let  $\Phi(\mathbf{G})\subset H^{1}(W'_{F},\hat{\mathbf{G}})$ be the set of relevant Langlands
 parameters for $\mathbf{G}$, where $W_F$ denotes the Weil group of $F$ and $W'_{F}=W_F \times \SL_2(\bC)$.

 Let $P_{F}$ be the wild inertia subgroup of $W_{F}$, and let
$\Phi(P_{F},\mathbf{G})$ denote the image
of $\Phi(\mathbf{G})$ under the restriction map
$H^{1}(W'_{F},\hat{\mathbf{G}})\To{}H^{1}(P_{F},\hat{\mathbf{G}})$.
 We call the elements of
$\Phi(P_{F},\mathbf{G})$  ``wild inertia parameters'' for $\mathbf{G}$.
So, an element $\phi\in\Phi(P_{F},\mathbf{G})$ is a $\hat{\mathbf{G}}$-conjugacy class of $L$-homomorphisms
 $P_{F}\To{}{^{L}\mathbf{G}}$ (still denoted $\phi$) that can be extended to some $\varphi:\,
 W_{F}\To{}{^{L}\mathbf{G}}$ (which is a restriction of a relevant Langlands parameter). In this setting, $\varphi(W_{F})$ normalizes the centralizer
 $C_{\hat{\mathbf{G}}}(\phi)$ of $\phi$ in $\hat{\mathbf{G}}$. 
 The action, induced by conjugation, of $W_{F}$  on the center
 $Z(C_{\hat{\mathbf{G}}}(\phi))$ and the cocenter  $(C_{\hat{\mathbf{G}}}(\phi))_{\rm ab}$, as
 well as the induced outer action $W_{F}\To{}{\rm
   Out}(C_{\hat{\mathbf{G}}}(\phi))$,  are tamely
 ramified, and independent of the choice of $\varphi$.

 When the reductive group $C_{\hat{\mathbf{G}}}(\phi)$
 is connected, which is, for example, the case if $p>2$ when
 $\hat{\mathbf{G}}$ is a classical group,
 then we denote by:
 \begin{itemize}
 \item $\mathbf{G}_{\phi}$ the quasi-split reductive $F$-group that is dual to
   $C_{\hat{\mathbf{G}}}(\phi)$ over a separable closure $\overline F$ of $F$, with $F$-rational structure given by the above
   outer action.  
 \item $X_{\phi}$ the kernel of the map
   $X^{*}(Z(C_{\hat{\mathbf{G}}}(\phi))^{\varphi(W_{F})})_{\rm
     tors}\To{h_{\phi}}X^{*}(Z(\hat{\mathbf{G}})^{W_{F}})_{\rm tors}$ induced by the
   inclusion
   $Z(\hat{\mathbf{G}})^{W_{F}}\subset Z(C_{\hat{\mathbf{G}}}(\phi))^{\varphi(W_{F})}$
   after taking torsion subgroups of character groups.
   For any inner form $\mathbf{G}'_{\phi}$ of $\mathbf{G}_{\phi}$, composition with  Kottwitz'
isomorphisms turns the above  map into a map
$H^{1}(F,{\mathbf{G}}'_{\phi})\To{h_{\phi}} H^{1}(F,{\mathbf{G}})$ that allows to view
$X_{\phi}$ as a subset of  $H^{1}(F,{\mathbf{G}}'_{\phi})$.
 \end{itemize}

When $C_{\hat{\mathbf{G}}}(\phi)$ is additionally a Levi subgroup of
$\hat{\mathbf{G}}$, which is, for example,  the case if $p$ does not
divide the order of the Weyl group of $\hat{\mathbf{G}}$, 
we  denote by:
\begin{itemize}
\item  $\mathbf{S}_{\phi}$  the $F$-torus that is dual to the cocenter of
  $C_{\hat{\mathbf{G}}}(\phi)$ with the action
  of $W_{F}$ described above. 
  In
   \ref{dual_embeddings_phi} we explain how $\mathbf{S}_{\phi}$ comes with a
  $\mathbf{G}(\o F)$-conjugacy class of ``Levi-center-embeddings'' $\iota:\, (\mathbf{S}_{\phi})_{\o F}\injo
  \mathbf{G}_{\o F}$, meaning an embedding whose image is the connected center of its
  centralizer. For each such $\iota$, the centralizer
  $C_{\mathbf{G}_{\o F}}(\iota((\mathbf{S}_{\phi})_{\o F}))$ is a Levi subgroup of
  $\mathbf{G}_{\o F}$ isomorphic to  $(\mathbf{G}_{\phi})_{\o F}$.
    
\item   $\IC_{\phi}$ the set  of $G$-conjugacy classes of \emph{$F$-rational} Levi-center-embeddings
  $\iota :\, \mathbf{S}_{\phi}\hookrightarrow\mathbf{G}$ in the above $\mathbf G(\overline F)$-conjugacy class.
  For each such $\iota$, the centralizer $\mathbf{G}_{\iota}:=C_{\mathbf{G}}(\iota(\mathbf{S}_{\phi}))$ is a
   twisted Levi subgroup of $\mathbf{G}$ and an inner form of $\mathbf{G}_{\phi}$.
 The group  $X_{\phi}$ acts simply transitively on $\IC_{\phi}$, but in
general there is no natural base point.

\item $\mathbf{G}_{\phi,I}$, for $I \in \IC_{\phi}$, a reductive $F$-group in the
  isomorphism class of all centralizers $\mathbf{G}_{\iota}$ for $\iota\in I$.
  For any $\alpha\in
  X_{\phi}$, the group $\mathbf{G}_{\phi,\alpha.I}$ is the pure inner form of
  $\mathbf{G}_{\phi,I}$ associated to $\alpha\in X_{\phi}\subset H^{1}(F,\mathbf{G}_{\phi, I})$.
\end{itemize}

We refer to Sections
  \ref{sec:levi-cent-embedd} and \ref{sec:levi-fact-param} for more details on these
  objects.
  Here is a more faithful account of what is done in this paper.
  \def\II{{I}}
        
\begin{main}[Take 2] Assume that $\mathbf{G}$ is tamely ramified and
  that $p$ is odd, and not a torsion prime of $\mathbf{G}$, nor of $\hat{\mathbf{G}}$.
  Let $R$ be a commutative $\ZM[\mu_{p^{\infty}},\frac 1p]$-algebra.
  \begin{enumerate}
  \item 
    To any $\phi\in\Phi(P_{F},\mathbf{G})$ 
    such that $C_{\hat{\mathbf{G}}}(\phi)$ is a Levi subgroup of $\hat{\mathbf{G}}$,
    we associate a non-zero Serre subcategory $\Rep^{\phi}_{R}(G)$ of $\Rep_{R}(G)$, and
    a decomposition with non-zero factors 
    $$\Rep^{\phi}_{R}(G) = \prod_{\II\in \IC_{\phi}}\Rep^{\phi,\II}_{R}(G).$$
    Moreover, $\Rep^{\phi}_{R}(G)$ and $\Rep^{\phi'}_{R}(G)$ are orthogonal if
    $\phi\neq\phi'$, and the category $\Rep^{1}_{R}(G)$ associated to the trivial
    parameter is the depth-$0$ category (note that $\cI_{1}$ is a
    singleton).
    
  \item \label{ii} 
    If $R$ is a $\ZM[\mu_{4p^{\infty}},\frac 1{\sqrt p}]$-algebra, we construct
    equivalences of categories
    \[\IC_{\phi,I}:\, \Rep^{1}_{R}(G_{\phi,\II})\simto \Rep^{\phi,\II}_{R}(G).\]
     When $R\subseteq \o\ZM[\frac 1p]$, it follows from \cite[Thm 1.0.5]{Dat-Lanard}
     that each $\Rep^{\phi,\II}_{R}(G)$ is a block.
     Moreover, there is a $G$-conjugacy class of embeddings of extended Bruhat--Tits buildings
     $i:\,\BT_{\phi,I}\injo\BT$ of the building $\BT_{\phi,I}$ of $\mathbf{G}_{\phi, I}$ into the building $\BT$ of $\mathbf{G}$ such that for any $x\in \BT_{\phi,I}$ and any
     $R$-representation $\rho$ of the stabilizer $(G_{\phi,I})_{x}$ of $x$ in $G_{\phi,I}$ that is trivial on $(G_{\phi,I})_{x,0+}$, we have
     \[ \IC_{\phi,I}\left(\cind_{(G_{\phi,I})_{x}}^{G_{\phi,I}}(\rho)\right)\simeq
     \cind_{K_{\phi,I,i(x)}}^{G}\left(\kappa_{\phi,I,i(x)}\otimes_{R}\rho\right).\] 
   Here $\cind$ denotes compact induction, $K_{\phi,I,i(x)}$ is an explicit open subgroup
    of $G_{i(x)}$ that maps surjectively to
    $(G_{\phi,I})_{x}/(G_{\phi,I})_{x,0+}$, and $\kappa_{\phi,I,i(x)}$
    is an explicit representation, called a  twisted
     Heisenberg--Weil representation and defined in
     \ref{sec:heis-weil-rep}.
 
   \item \label{iii}  
    If $p$ does not divide the order of the absolute Weyl group of $\mathbf{G}$, then
    $$ \Rep_{R}(G) = \prod_{\phi\in\Phi(P_{F},\mathbf{G})} \Rep^{\phi}_{R}(G).$$
  \item \label{iv}
    Let ${\mathbf{P}} $ be a parabolic
  $F$-subgroup of ${\mathbf{G}}$ with Levi subgroup ${\mathbf{M}}$, and denote by $i_{P}$
  the associated parabolic induction functor, and by $r_{P}$ the corresponding Jacquet functor.
  \begin{enumerate}
  \item  If $\phi$ is the image of $\phi_{M} \in \Phi(P_{F},\mathbf{M})$ under the natural map
    $\Phi(P_{F},\mathbf{M})\To{}\Phi(P_{F},\mathbf{G})$, and $\II$ is the image of
    $\II_{M} \in \IC_{\phi_{M}}$ under the natural  map $\IC_{\phi_{M}}\To{}\IC_{\phi}$, then  we have 
$$ i_{P}\left(\Rep^{\phi_{M},\II_{M}}_{R}(M)\right) \subseteq \Rep^{\phi,\II}_{R}(G).$$ 
\item Assuming that $p$ does not divide the order of the absolute Weyl group of ${\mathbf{G}}$, we have
$$r_{P}(\Rep^{\phi,\II}_{R}(G))\subseteq \prod_{(\phi_{M},\II_{M})\mapsto
      (\phi,\II)} \Rep^{\phi_{M},\II_{M}}_{R}(M).$$
\item Assuming further that $C_{\hat{\mathbf{G}}}(\phi)\subseteq \mathbf{\hat{\mathbf{M}}}$, then $i_{P}$
    induces an equivalence of categories 
$$\Rep_{R}^{\phi_{M},\II_{M}}(M)\simto\Rep_{R}^{\phi,\II}(G)$$
with quasi inverse the composition of $r_{P}$ and
    the projection onto the $(\phi_{M},\II_{M})$-factor.
  \end{enumerate}
  \end{enumerate}

\end{main}

Let us comment on the hypotheses,  the construction, and where to find proofs in
the text.
\begin{itemize}
\item We recall that a torsion prime for $\mathbf{G}$ is a prime that either is bad for
  the root system of $\mathbf{G}$, or divides the order of $\pi_{1}(\mathbf{G}_{\rm
    der})$. We refer to Section \ref{centralizer_Levi} for a discussion of the effect
  of the various hypotheses we impose on the prime $p$ in our results. We also note that the assumption that
  $\mathbf{G}$ is tamely ramified is actually implied by the existence of a $\phi$ whose centralizer
  is a Levi subgroup.

\item  Let us outline the construction of $\Rep^{\phi,I}_{R}(G)$, which is achieved in
  Section \ref{sec:2}.
   The full subcategory $\Rep^{\phi,I}_{R}(G)$ is defined by a system of
  idempotents $(e_{\phi,I,x})_{x\in\BT}$ by letting
  the objects of $\Rep^{\phi,I}_{R}(G)$ be $\left\{V\in\Rep_{R}(G),\, V=\sum_{x\in \BT} e_{\phi,I,x}V\right\}$. 
  For $x\in \BT$, each idempotent $e_{\phi,I,x}$ is a sum of idempotents $e_{\phi,\iota,x}$
  associated to embeddings $\iota\in I$ such that $x$ belongs to the building $\BT_{\iota}\subset
  \BT$ of $\mathbf{G}_{\iota}$. In turn, $e_{\phi,\iota,x}$ is the averaging idempotent along a
  character $\check \phi^+_{\iota, x}$ of a certain open pro-$p$-subgroups of $G_{x,0+}$. This
  character is obtained by applying a construction of Yu \cite{Yu_tamescusp} that takes as input a
  sequence of characters of twisted Levi subgroups of $\mathbf{G}$. In this case, the 
  smallest 
   member
  of the twisted
  Levi subgroup sequence is $\mathbf{G}_{\iota}$, and the product $\check\varphi : G_{\iota}\To{}\mu_{p^{\infty}} \To{} R^\times$
  of the desired characters restricted to ${G}_{\iota}$ is 
  obtained by the Borel--Langlands reciprocity applied to
   a $1$-cocycle 
  $\varphi:\, W_{F}\To{} Z(\hat{\mathbf{G}}_{\iota})=Z(\hat{\mathbf{G}}_{\phi})$ of $p$-power order
  that extends $\phi$ (whose existence is proved in Lemma \ref{splittings_Levi}).
  The intermediate twisted Levi subgroups and (products of their) characters are obtained in a similar way by replacing $\phi$
  by its restriction to suitable higher ramification subgroups of $P_{F}$,  following a strategy of Kaletha in
  \cite{Kaletha}. This is explained in Section \ref{sec:ramif-groups-twist}.
  Yu's construction is recalled in Section ~\ref{sec:char-idemp}.
  Thanks to the strong intertwining
  properties of Yu's characters (Section \ref{sec:intertwining}), it turns out that for two
  $\iota,\iota'$, the idempotents 
  $e_{\phi,\iota,x}$ and $e_{\phi,\iota',x}$ are either equal or orthogonal for $x\in\BT_{\iota}\cap\BT_{\iota'}$.
  Summing over non-equal such idempotents, we get the desired idempotent $e_{\phi,I,x}$
  (which can be zero) for any $x\in\BT$. Their behavior when $x$ is moving is studied in Section
  \ref{sec:systems-idempotents}, and applied in Section \ref{sec:category-repphi-i_rg} to 
  the proof that $\Rep^{\phi,I}_{R}(G)$ is indeed a Serre subcategory (Theorem \ref{thm_Serresubcat}).

\item The main properties of $\Rep^{\phi,I}_{R}(G)$ are proved in Section
  \ref{sec:some-prop-repphi}. The compatibility with parabolic induction, \ref{iv}), is
  Theorem \ref{compat_parab_induc} and the decomposition of \ref{iii}) is Theorem \ref{exhaust}. 
  
\item The construction of the equivalences of categories in \ref{ii}) is the subject of
  Section 3. The strategy is the following. In Section  \ref{sec:main-result-strategy-2},
  we show how the categories $\Rep^{\phi,I}_{R}(G)$ and $\Rep^{1}_{R}(G_{\phi,I})$ are equivalent to 
  certain categories of equivariant coefficient systems on the respective Bruhat--Tits buildings of $\mathbf G$ and
  $\mathbf{G}_{\phi,I}$. Then in Section \ref{sec:main-result-strategy-3} we
  introduce the notion of a ``Heisenberg--Weil'' coefficient system
  on the Bruhat--Tits building of $\mathbf{G}_{\phi,I}$ and state an existence claim,
  Theorem \ref{thm_WH_coef_system-general}.  These gadgets  will serve as
  a bridge  between coefficient  systems on both buildings, pretty much like a bimodule is
  used to produce Morita equivalences. We refer to  the appendix of \cite{Datequiv} for
  more details on the analogy with Morita equivalences. In this paper, our approach will
  be more akin to the familiar ``pull-push with kernel'' functors from sheaf theory, as
  explained in \ref{geom_interp}.  
  The  construction of these ``Morita-type'' or ``pull-push''  equivalences is achieved in
  Section \ref{sec:constr-equiv}, admitting Theorem \ref{thm_WH_coef_system-general}.  The proof
  of this theorem, i.e., the construction of a Heisenberg--Weil coefficient system, is obtained
  in Section \ref{sec:heis-weil-coeff}, after preliminary work on 
  Heisenberg representations in Section \ref{sec:heisenberg-rep}, and on Weil
  representations in Section \ref{sec:heis-weil-rep}, where we adapt
  classical and more recent results from coefficients $\CM$
  to our coefficient ring $R$, which might be of independent interest.
\end{itemize}

Let us now comment on the connections with the existing
literature on reduction to depth $0$.
Since our functor $\IC_{\phi,I}$ is an equivalence, it induces an isomorphism between
the endomorphism algebra of an object and that of its image. Therefore, the explicit description
in \ref{ii}) above of the evaluation of $\IC_{\phi,I}$ on certain compactly induced representations
implies the existence of isomorphisms between certain Hecke algebras, some of them being the key to
previous reduction to depth $0$ results.
\begin{itemize}
\item In the case $R=\bC$ (or more generally when $R$ is an algebraically closed field), we recover in \ref{cor-Hecke-alg-isom} the
  existence of an  isomorphism between the 
  Hecke algebra attached to the types constructed by Kim and Yu and that of its ``twisted depth-$0$ component'', as first proved in \cite{AFMO2}. Note
  that we do not compare explicitly both constructions. In particular, we do not verify here whether
  our isomorphism is ``support-preserving''.  
\item In the case $R=\o\ZM[\frac 1p]$ and $G$ an inner form of $\GL_{n}$, we recover in \ref{example_gln} the existence
  of an isomorphism between two relevant Hecke algebras arising in Bushnell--Kutzko's type theory
  that were previously shown to be isomorphic by explicit computation of generators and relations by 
  Chinello in \cite{Chinello}. Again, we do not compare both isomorphisms. Note however that
  Chinello's results do not need any tameness assumption.
\end{itemize}

\medskip

\textbf{Connections, and expected connections, with the local Langlands correspondence, its
  geometrization, and its categorification.}

\begin{itemize}
\item When the  local Langlands correspondence ${\rm Irr}_{\CM}(G)\To{} \Phi(W'_{F},\mathbf{G})$,
  $\pi\mapsto\varphi_{\pi}$ is
  available for $\mathbf{G}$, we conjecture that for an irreducible complex
  representation $\pi$ of $G$, we have 
  $\pi\in \Rep^{\phi}_{\CM}(G) \Leftrightarrow (\varphi_{\pi})|_{P_{F}}\sim\phi$. By
  design, this is at least true when $\pi$ is a regular supercuspidal representation and
  $\varphi_{\pi}$ is the parameter defined in \cite{Kaletha}.
  
  Assume further that $\mathbf{G}$ is quasi-split and that the extended
  local Langlands correspondence $\pi\mapsto (\varphi_{\pi},\varepsilon_{\pi})$ is
  available for $\mathbf{G}$, where $\varepsilon_{\pi}$ is a finite order character of
  $C_{\hat{\mathbf{G}}}(\varphi_{\pi})$. Then  we conjecture that, for a good choice of base point
  $I\in\IC_{\phi}$, we have, for any irreducible complex
  representations $\pi$ of $G$ :
  $$\pi\in \Rep^{\phi,\alpha\cdot I}_{\CM}(G) \Leftrightarrow
  \left[(\varphi_{\pi})|_{P_{F}}\sim\phi \hbox{ and }
  (\varepsilon_{\pi})|_{Z(C_{\hat{\mathbf{G}}}(\phi))^{\varphi_{\pi}(W_{F})}}= \alpha\right],$$
where the restriction is along the inclusion
$Z(C_{\hat{\mathbf{G}}}(\phi))^{\varphi_{\pi}(W_{F})} \subset C_{\hat{\mathbf{G}}}(\varphi_{\pi})$.
  
\item Without any hypothesis on $\mathbf{G}$ nor on $\phi$, a construction
  of a direct factor  $\Rep^{\phi}_{R}(G)$ and a decomposition as
  in \ref{iii}) can be deduced from the motivic version of the Fargues--Scholze spectral action
  of \cite{Scholze_motiv}. Indeed, a consequence of this spectral action is the existence
  of a ring homomorphism ${\rm FS_{mot}}$
  $$\OC(Z^{1}(W_{F}^{0},\hat{\mathbf{G}}))^{\hat{\mathbf{G}}}  \buildrel\hbox{\tiny $\approx$}\over\longleftarrow
  {\rm \EC xc}(W_{F}^{0},\hat{\mathbf{G}}) \To{{\rm FS_{mot}}}\ZG_{\o\ZM[\frac 1p]}(G),$$
where $Z^{1}(W^{0}_{F},\hat{\mathbf{G}})$ is the space of parameters over $\o\ZM[\frac 1p]$
  introduced in \cite{DHKM1} (so that here $\hat{\mathbf{G}}$ denotes the dual over
  $\o\ZM[\frac 1p]$  and $W_F^0$ is the subgroup of $W_F$ defined in \cite[\S 1.2]{DHKM1}), ${\rm \EC xc}$ denotes an ``excursion   algebra'' and the 
  $\approx$ sign indicates a map of rings that induces a universal homeomorphism on spectra,
  and $\ZG_{\o\ZM[\frac 1p]}(G)$ denotes the center of $\Rep_{\o\ZM[\frac 1{p}]}(G)$.
  According to \cite{DHKM1}, the stack
  $Z^{1}(W^{0}_{F},\hat{\mathbf{G}})/\hat{\mathbf{G}}$
  is the disjoint union over all  $\phi\in\Phi(P_{F},\mathbf{G})$ of 
  open closed substacks
  $Z^{1}(W^{0}_{F},\hat{\mathbf{G}})_{\phi}/C_{\hat{\mathbf{G}}}(\phi)$, where a
  representative $\phi:P_{F}\To{}\hat{\mathbf{G}}$ has been chosen and
  $Z^{1}(W^{0}_{F},\hat{\mathbf{G}})_{\phi}$ is the space of extensions of $\phi$ to $W_{F}^{0}$.
  So each $\phi$ provides an idempotent $e_{\phi}^{\rm spec}$ in
  $\OC(Z^{1}(W_{F}^{0},\hat{\mathbf{G}}))^{\hat{\mathbf{G}}}$, that provides in turn an idempotent
  $e_{\phi}^{\rm geom}$ in $\ZG_{\o\ZM[\frac 1p]}(G)$, and we get a decomposition
  $\Rep_{R}(G)=\prod_{\phi}e_{\phi}^{\rm geom}\Rep_{R}(G)$. 
  In general, not much is known about this decomposition, but we conjecture that it
  coincides with ours when applicable. Proving this is equivalent to proving that the
  Fargues--Scholze semisimple correspondence for the so-called ``regular cuspidal''
  representations of Levi subgroups of $G$ matches the semisimplification of 
  Kaletha's correspondence in \cite{Kaletha}.

\item In our context, and more generally whenever $C_{\hat{\mathbf{G}}}(\phi)$ is
  connected, the scheme $Z^{1}(W^{0}_{F},\hat{\mathbf{G}})_{\phi}$ is connected and the
  idempotent $e_{\phi}^{\rm spec}$ is thus primitive.
  So our decomposition of $\Rep_{R}^{\phi}(G)$ does not seem to be explained by the spectral action.
  However, it 
  would be explained by (a motivic form of) the categorical local Langlands correspondence (CLLC below).
  To explain this, recall that the CLLC predicts  an equivalence of categories between some variant
   ${\rm IndCoh}_{\rm nilp}(Z^{1}(W^{0}_{F},\hat{\mathbf{G}})/\hat{\mathbf{G}})$  of
   the category of coherent sheaves on $Z^{1}(W^{0}_{F},\hat{\mathbf{G}})/\hat{\mathbf{G}}$ and a category of
  sheaves ${\rm Shv}_{G}$ on either the $v$-stack ${\rm Bun}_{G}$ of \cite{FS}, or the perfect stack
  ${\rm Isoc}_{G}$ of \cite{Zhu}. In each case,  $\Rep(G)$ identifies with the subcategory of
  sheaves supported on a certain stratum ``$b=1$'', which is isomorphic to the classifying space of
  $G$.  The connected components of both these stacks are
  labeled by $X^{*}(Z(\hat{\mathbf{G}})^{W_{F}})$, and the resulting decomposition of
  ${\rm Shv}_{G}$  corresponds, on the
  paramers side, to the grading coming from the triviality of the action of
  $Z(\hat{\mathbf{G}})^{W_{F}}$ on $Z^{1}(W^{0}_{F},\hat{\mathbf{G}})$. The stratum
  ``$b=1$'' lies in the component corresponding to the trivial character $1$ of 
  $Z(\hat{\mathbf{G}})^{W_{F}}$. Now recall further that the spectral action is defined
  over ${\rm Shv}_{G}$, yielding a direct summand $e_{\phi}^{\rm geom}{\rm Shv}_{G}$, so
  that the CLLC should restrict to an equivalence between 
  ${\rm IndCoh}_{\rm nilp}(Z^{1}(W^{0}_{F},\hat{\mathbf{G}})_{\phi}/C_{\hat{\mathbf{G}}}(\phi))$
  and $e_{\phi}^{\rm geom}{\rm Shv}_{G}$.
  Since $Z(C_{\hat{\mathbf{G}}}(\phi))^{\varphi(W_{F})}$ acts trivially on
   $Z^{1}(W^{0}_{F},\hat{\mathbf{G}})_{\phi}$, the category
  ${\rm IndCoh}_{\rm nilp}(Z^{1}(W^{0}_{F},\hat{\mathbf{G}})_{\phi}/C_{\hat{\mathbf{G}}}(\phi))$
  is graded over $X^{*}(Z(C_{\hat{\mathbf{G}}}(\phi))^{\varphi(W_{F})})$.
  We then infer a splitting of
  $e_{\phi}^{\rm geom}{\rm Shv}_{G}$ indexed by
  $X^{*}(Z(C_{\hat{\mathbf{G}}}(\phi))^{\varphi(W_{F})})$ that refines the one indexed by
  $X^{*}(Z(\hat{\mathbf{G}})^{W_{F}})$. When we restrict to the stratum $b=1$, only the
  summands labeled by $X_{\phi}$ can contribute, yielding a decomposition of
  $e_{\phi}^{\rm geom}\Rep(G)$ indexed by $X_{\phi}$, as in our results above.

\item Our equivalences also fit in the framework of the CLLC. Namely, in our setting
   there exists $\varphi \in Z^{1}(W^{0}_{F},\hat{\mathbf{G}})_{\phi}(\o\ZM[\frac 1p])$
   such that $\varphi(W_{F})$ normalizes a pinning of
   $C_{\hat{\mathbf{G}}}(\phi)$. Multiplication by $\varphi$ then induces an isomorphism
   $Z^{1}(W^{0}_{F},\hat{\mathbf{G}}_{\phi})_{1}/\hat{\mathbf{G}}_{\phi}
   \simto Z^{1}(W^{0}_{F},\hat{\mathbf{G}})_{\phi}/C_{\hat{\mathbf{G}}}(\phi)$.
   Applying the CLLC for both $\mathbf{G}$ and $\mathbf{G}_{\phi}$, we would get an
   equivalence
   $e_{1}^{\rm geom}{\rm Shv}_{G_{\phi}}\simto e_{\phi}^{\rm geom}{\rm Shv}_{G}$. Of
   course, it would need further work to see whether this equivalence restricts well to
   the $b=1$ strata, and is $t$-exact there, for the standard $t$-structure. Our result suggests
   that it actually does.

 \item Some important particular cases of the categorical Langlands correspondence have been
   recently proved on ${\rm Isoc}_{G}$ by X. Zhu. They are concerned with the tame part over
   $\o\QM_{\ell}$-coefficients and the unipotent part over $\o\FM_{\ell}$-coefficients. In
   particular, \cite[Thm. 5.4]{Zhu} constructs a fully faithful embedding
   $D\Rep_{\o\FM_{l}}(G)_{\rm unip}\injo {\rm
     IndCoh}(Z^{1}(W_{F}^{0},\hat{\mathbf{G}}_{\o\FM_{\ell}})_{\rm
     unip}/\hat{\mathbf{G}}_{\o\FM_{\ell}})$ 
   when $\mathbf{G}$ is unramified, and
   where $\hat{\mathbf{G}}_{\o\FM_{\ell}}$ denotes a Langlands dual group over $\o\FM_{\ell}$,
   $D\Rep_{\o\FM_{l}}(G)$ denotes the derived ($\infty$)-category of $\Rep_{\o\FM_{l}}(G)$,
   the subscript $\rm unip$ 
   denotes the unipotent summand on the left hand side, 
   resp., 
   the unipotent connected component (i.e. the one that contains the trivial parameter) on the right hand side.
 
     In the case $\mathbf{G}=\GL_{n}$, one can combine our ``reduction to depth $0$'' result here with
     the ``reduction to unipotent'' result of \cite{Datequiv}. Recall that, in this case,
     Vignéras   decomposed $\Rep_{\o\FM_{l}}(G)$ in \cite{Vig_Induced}, as a product of blocks indexed by semisimple
     representations $W_{F}\To{}\GL_{n}(\o\FM_{\ell})$ up to inertial equivalence, while the same indexing set
     also parametrizes the connected components of
     $Z^{1}(W_{F}^{0},\hat{\mathbf{G}}_{\o\FM_{\ell}})$, according to \cite[Cor 4.21]{DHKM1}. 
     \begin{cor} Let $\mathbf{G}=\GL_{n}$, and let $p$ and $\ell$ be greater than $n$.
       Then there is a fully faithful embedding
       $D\Rep_{\o\FM_{l}}(G)\injo {\rm
         IndCoh}(Z^{1}(W_{F}^{0},\hat{\mathbf{G}}_{\o\FM_{\ell}})/\hat{\mathbf{G}}_{\o\FM_{\ell}})$ 
       that sends $D\Rep_{\o\FM_{l}}(G)_{[\varphi]}$ into
       ${\rm IndCoh}(Z^{1}(W_{F}^{0},\hat{\mathbf{G}}_{\o\FM_{\ell}})_{[\varphi]}/\hat{\mathbf{G}}_{\o\FM_{\ell}})$
       for each inertia class $[\varphi]$ of semi-simple representations $W_{F}\To{}\GL_{n}(\o\FM_{\ell})$. 
     \end{cor}

     Actually, we construct a whole class of embeddings as in this corollary, and we do not quite
     specify the ``correct'' one, i.e., the one that is compatible with Vigneras' correspondence.
     We refer to \ref{CLLC_GLn} for a discussion of this question. 
     Among the ingredients needed to find the correct one are the compatibility of our equivalences with parabolic
     induction, as well
     as the compatibility with Whittaker/Gelfand--Graev representations.  We  plan to address these
     points in future work.
     
\end{itemize}

\subsection{Notation} \label{sec:notation}

\begin{itemize}
\item  $F$ denotes a non-archimedean local field whose residue field
  characteristic is denoted by $p$.  We fix a separable closure $\o F$ of $F$ and denote by
  $\Gamma_{F}$ Galois group of $\o F / F$. Inside $\Gamma_{F}$ we have the usual subgroups
  $W_{F}\supset I_{F}\supset P_{F}$, respectively, the Weil group, its inertia subgroup,
  and the wild inertia subgroup.

\item Bold letters $\mathbf{G}$, $\mathbf{B}, \mathbf{T}$, etc., denote algebraic groups
  over $F$, unless specified otherwise, 
   and we use the corresponding plain letters to denote their groups of $F$-rational points,
  e.g., $G=\mathbf{G}(F)$.

\item For a connected reductive group $\mathbf{G}$ over some field, we denote by $\mathbf{G}_{\mathrm{ad}}$ the adjoint quotient group of $\mathbf{G}$, and by $\mathbf{G}_{\mathrm{sc}}$ the simply-connected cover of the derived subgroup of $\mathbf{G}$.

\item $R$ will denote the commutative coefficient ring of our representations. 
 We always assume $p\in R^{\times}$, and most often will assume that $R$ is a
  $\ZM[\frac 1p,\mu_{p^{\infty}}]$-algebra.

\item For a locally pro-$p$-group $H$,
  \begin{itemize}
    
  \item $\Rep_{R}(H)$ denotes the category of smooth $R$-representations of $H$.
  \item $R[H]$ denotes the group algebra of $H$.

  \item $RH$ denotes the $R$-algebra of compactly supported
    $R$-valued distributions on $H$.  It contains $R[H]$ as the span of Dirac distributions, as well
    as averaging idempotents $e_{K}$ for any closed pro-$p$-subgroup $K\subseteq H$. The
    action of $R[H]$ on $V\in\Rep_{R}(H)$ extends canonically to $RH$, identifying $\Rep_{R}(H)$ 
    with the category of smooth $RH$-modules.
    
    \item $\HC_{R}(H)$ denotes the $R$-subalgebra of $RH$ consisting of locally constant 
      distributions.  This non-unital ring  contains averaging idempotents $e_{K}$ for all open 
    pro-$p$-subgroups $K\subseteq H$. 
    By restricting the action of $RH$,  $\Rep_{R}(H)$ identifies
    with the category of non-degenerate $\HC_{R}(H)$-modules.

  \end{itemize}

\item For a connected reductive group $\mathbf{G}$ over $F$,
\begin{itemize}
\item $\hat{\mathbf{G}}$ denotes the complex\footnote{For many purposes, including the
    relation with the geometrization and the categorification of the LLC, it
    would be  more natural to consider the dual group over the coefficient ring
    $R$. However,  the complex setting
    will be sufficient in this paper, and saves us additional technical difficulties.} 
  dual reductive group of
  $\mathbf{G}$, which comes with a pinning
  $(\hat{\mathbf{B}},\hat{\mathbf{T}},\hat{X})$ and a
  pinning-preserving action of $\Gamma_{F}$.  We denote by
  $^{L}\mathbf{G}:=\mathbf{G}\rtimes W_{F}$ the Weil form of
  Langlands' dual group.

\item $\phi$ will denote a \emph{wild inertia parameter}\index[terminology]{wild inertia parameter} for $\mathbf{G}$, i.e., a continuous
  $L$-homomorphism $P_{F}\To{}{^{L}\mathbf{G}}$, that admits an extension
  to a relevant Langlands parameter $\varphi : W'_{F}\To{}{^{L}\mathbf{G}}$, where
  $W'_{F}$ denotes any form the Weil--Deligne group of $F$.

\item $\Phi(P_{F},\mathbf{G})$ denotes the set of $\hat{\mathbf{G}}$-conjugacy classes of
  wild inertia parameters for $\mathbf{G}$.   
\end{itemize}  
  
\end{itemize}

\subsection*{Acknowledgments} The authors thank the Collège de France and the Hausdorff Research Institute for Mathematics in Bonn for their hospitality and supporting the stay of the second-named author in Paris as part of the visit to deliver the Course Peccot and the stay of the first-named author in Bonn during the Hausdorff trimester program ``The Arithmetic of the Langlands Program''. 
The second-named author also thanks Ana Caraiani for discussions related to this paper.

\section{From parameters to subcategories}\label{sec:2}
Throughout the paper, ${\mathbf{G}}$ denotes a connected reductive
group over $F$.

\subsection{Levi-center-embeddings and duality}\label{sec:levi-cent-embedd}
 From the construction of the dual group $\hat{\mathbf{G}}$ we have a bijection between
${\mathbf{G}}(\ov F)$-conjugacy classes of \emph{maximal} $\o F$-tori embeddings  
$\mathbf{S}\injo {\mathbf{G}}_{\ov F}$ and 
$\hat{\mathbf{G}}$-conjugacy classes of \emph{maximal} tori embeddings  
$\hat{\mathbf{S}}\injo \hat{\mathbf{G}}$. We seek a generalization of this for embeddings of
tori as connected center of a Levi subgroup.
  
\begin{lemme}\label{levicenter}
  Let $\mathbf{S}$ be a $\o F$-torus contained in ${\mathbf{G}}_{\ov F}$. The
  following are equivalent
  \begin{enumerate}
  \item $\mathbf{S}$ is the connected center of a Levi subgroup of ${\mathbf{G}}_{\ov F}$,
  \item $\mathbf{S}= C_{\mathbf{G}_{\ov F}}(C_{\mathbf{G}_{\ov F}}(\mathbf{S}))^{\circ}$,
  \item there is a maximal $\o F$-torus ${\mathbf{T}}$ of ${\mathbf{G}}_{\ov F}$ containing
    $\mathbf{S}$ and a Levi subroot system $\Sigma'\subseteq
    \Sigma({\mathbf{T}},{\mathbf{G}}_{\ov F})$ of the root system of $\mathbf{G}_{\ov F}$ with respect to $\mathbf{T}$ such that $\mathbf{S}=\left(\bigcap_{\alpha\in\Phi'}\ker(\alpha)\right)^{\circ}$.
  \end{enumerate}
\end{lemme}
\begin{proof}
  Standard, cf., for example, \cite[14.18]{BorelLAG}.
\end{proof}

\begin{DEf}
A \emph{Levi-center-embedding}\index[terminology]{Levi-center-embedding} in ${\mathbf{G}}_{\ov F}$ is a pair $(\mathbf{S},\iota)$ with
$\mathbf{S}$ a $\o F$-torus
and  $\iota : \mathbf{S}\injo
 {\mathbf{G}}_{\ov F}$  an embedding  such that $\iota(\mathbf{S})$ satisfies the properties of Lemma \ref{levicenter}.
\end{DEf}

\alin{Duality} \label{dual_embeddings}
Let $\hat{\mathbf{S}}$\index[notation]{Shat@$\hat{\mathbf{S}}$} be a complex algebraic torus. 
The algebraic group $\hat{\mathbf{G}}$ acts by conjugation on Levi-center-embeddings 
$\hat\iota :\, \hat{\mathbf{S}}\injo\hat{\mathbf{G}}$. Our aim is to attach to a conjugacy class
$\{\hat\iota\}$ of such embeddings, a ``dual'' conjugacy class $\{\iota\}$ of
Levi-center-embeddings in the $\o F$-algebraic group ${\mathbf{G}}_{\ov F}$.

The stabilizer $\hat{\mathbf{G}}_{\hat\iota}$\index[notation]{Ghatiotahat@$\hat{\mathbf{G}}_{\hat\iota}$} of $\hat\iota$ for the action induced by the conjugation action of
$\hat{\mathbf{G}}$ on itself is the
centralizer $C_{\hat{\mathbf{G}}}(\hat\iota(\hat{\mathbf{S}}))$ of the torus  $\hat\iota(\hat{\mathbf{S}})$, which is a
Levi subgroup of $\hat{\mathbf{G}}$. Its cocenter
$\hat{\mathbf{G}}_{\hat\iota, \rm ab}$ is a torus which \emph{only
 depends  on the conjugacy class $\{\hat\iota\}$ of $\hat\iota$}, in the sense that for
$\hat\iota' \in \{ \hat\iota\}$ we have an isomorphism $\hat{\mathbf{G}}_{\hat\iota,\rm ab}\simto
\hat{\mathbf{G}}_{\hat\iota',\rm ab}$ given by conjugation under any $\hat g\in \hat{\mathbf{G}}$ that conjugates
$\hat\iota$ to $\hat\iota'$. We denote by $\hat{\mathbf{S}}_{\{\hat\iota\}}:={\rm lim}_{\hat\iota'\in \{ \hat \iota\}}
\hat{\mathbf{G}}_{\hat\iota',\rm ab}$\index[notation]{Shatiota@$\hat{\mathbf{S}}_{\{\hat\iota\}}$} this common torus.
Since $\hat\iota(\hat{\mathbf{S}})$ is the
connected center of $\hat{\mathbf{G}}_{\hat\iota}$, the embedding $\hat\iota$ induces an isogeny 
$\hat{\mathbf{S}}\To{} \hat{\mathbf{G}}_{\hat\iota,\rm ab}$ which is compatible with the analogous
isogeny $\hat{\mathbf{S}}\To{} \hat{\mathbf{G}}_{\hat\iota',\rm ab}$ through the
isomorphism
$\hat{\mathbf{G}}_{\hat\iota,\rm ab}\simto \hat{\mathbf{G}}_{\hat\iota',\rm ab}$ for
any $\hat\iota'\in\{\hat\iota\}$. Therefore, this
defines an isogeny $\hat{\mathbf{S}}\To{}\hat{\mathbf{S}}_{\{\hat\iota\}}$
which \emph{only depends
  on  $\{\hat\iota\}$.} Concretely, the (finite) kernel $H_{\{\hat\iota\}}:=\ker(\hat{\mathbf{S}}\To{}
\hat{\mathbf{G}}_{\hat\iota,\rm ab})$ is independent of $\hat\iota$ in $\{\hat\iota\}$ and we have $\hat{\mathbf{S}}_{\{\hat\iota\}}=\hat{\mathbf{S}}/H_{\{\hat\iota\}}$.

Now let us choose a maximal torus $\hat{\mathbf{T}}$ in $\hat{\mathbf{G}}_{\hat\iota}$.
The isogeny $\hat{\mathbf{S}}\To{}\hat{\mathbf{S}}_{\{\hat\iota\}}$ factors as
$\hat{\mathbf{S}}\To{\hat\iota}\hat{\mathbf{T}}\To{\hat\pi}\hat{\mathbf{S}}_{\{\hat\iota\}}$ 
and identifies 
$$ \hat\iota:\,\hat{\mathbf{S}}\simto \left(\bigcap_{\alpha\in\Sigma_{\hat\iota}}\ker(\alpha)\right)^{\circ}\subset\hat{\mathbf{T}} 
\,\hbox{  and  }\, 
\hat\pi:\, \hat{\mathbf{T}}/\left(\sum_{\alpha\in\Sigma_{\hat\iota}}\im(\alpha^{\vee})\right)\simto \hat{\mathbf{S}}_{\{\hat\iota\}},$$
where  $\Sigma_{\hat\iota}$ is the root system $\Sigma(\hat{\mathbf{T}},\hat{\mathbf{G}}_{\hat\iota})$.
Hence the dual isogeny $\mathbf{S}_{\{\hat\iota\}}\To{} \mathbf{S}$\index[notation]{Siota@${\mathbf{S}}_{\{\hat\iota\}}$}\index[notation]{S@$\mathbf{S}$} of tori over $\ov F$ factors through
the dual $\ov F$-torus ${\mathbf{T}}$ of $\hat{\mathbf{T}}$ giving isomorphisms
$$ \pi:\, \mathbf{S}_{\{\hat\iota\}}\simto \left(\bigcap_{\alpha\in\Sigma_{\hat\iota}}\ker({\alpha^{\vee}})\right)^{\circ}
\subset  {\mathbf{T}}
\,\hbox{  and  }\, 
\hat{\hat\iota}:\,{\mathbf{T}}/\left(\sum_{\alpha\in\Sigma_{\hat\iota}}\im(\alpha)\right)\simto \mathbf{S},$$
where ${\alpha^{\vee}}$, resp., $\alpha$, is seen as a character, resp.,
a cocharacter, of ${\mathbf{T}}$. Now recall that the embedding $\hat{\mathbf{T}}\injo
\hat{\mathbf{G}}$ gives rise to a canonical ${\mathbf{G}}(\ov F)$-conjugacy class of embeddings
${\mathbf{T}}\injo {\mathbf{G}}_{\ov F}$. Choose such a ``dual embedding'' $j$. By construction it identifies 
 $\Sigma(\hat{\mathbf{T}},\hat{\mathbf{G}})$ with   $\Sigma({\mathbf{T}},{\mathbf{G}}_{\ov F})^{\vee}$. In
 particular, we see from point iii) in Lemma \ref{levicenter} that the composition
 $\iota:\,\mathbf{S}_{\{\hat\iota\}}\To{\pi} {\mathbf{T}}\injo  {\mathbf{G}}_{\ov F}$ is a Levi-center
 embedding such that, by construction, the Levi subgroup ${\mathbf{G}}_{\iota}=C_{\mathbf{G}_{\ov F}}(\iota(\mathbf{S_{\{\hat\iota\}}}))$\index[notation]{Giota@${\mathbf{G}}_{\iota}$} is dual
 to $\hat{\mathbf{G}}_{\hat\iota}$.

 \begin{lem}
   The ${\mathbf{G}}(\ov F)$-conjugacy class $\{\iota\}$ of $\iota$ 
 only depends on the $\hat{\mathbf{G}}$-conjugacy class $\{\hat\iota\}$.
 \end{lem}
 \begin{proof} Let $\hat\iota'$ be conjugate to $\hat\iota$, let $\hat{\mathbf{T'}}$ be a maximal
   torus in $\hat{\mathbf{G}}_{\hat\iota'}$ and let $j':\, {\mathbf{T}}'\injo {\mathbf{G}}_{\ov F}$ be a
   choice of dual embedding. Since all maximal tori of
   $\hat{\mathbf{G}}_{\hat\iota}$ are conjugate, 
there is an element $\hat g\in \hat{\mathbf{G}}$ which conjugates $\hat\iota$ to $\hat\iota'$ and
$\hat{\mathbf{T}}$ to $\hat{\mathbf{T'}}$. Then we have a commutative diagram
$$ \xymatrix{ 
\hat{\mathbf{S}} \ar[r]^{\hat\iota} \ar[d]_{\hat\iota'} & \hat{\mathbf{T}} \ar[d]^{\hat\pi}
\ar[ld]^{\sim}_{{\rm Ad_{\hat g}}}\\
\hat{\mathbf{T'}} \ar[r]_{\hat\pi'} & \hat{\mathbf{S}}_{\{\hat\iota\}}
}.$$ It follows that on the dual side we get $\pi= \hat{\rm Ad}_{\hat g} \circ\pi'$, where
$\hat{\rm Ad}_{\hat g}$ is the isomorphism  ${\mathbf{T'}}\simto {\mathbf{T}}$ dual to ${\rm Ad}_{\hat g}$.
On the other hand, $j\circ \hat{\rm Ad}_{\hat g}$ is a dual embedding of ${\mathbf{T'}}$ into
${\mathbf{G}}_{\ov F}$, hence there is $ g\in {\mathbf{G}}(\ov F)$ such that  
${\rm Ad}_{  g}\circ j'= j\circ \hat{\rm Ad}_{\hat g}$.
It follows that $g$ conjugates the  embedding ${\iota'}= j'\circ\pi'$ to the
embedding  $\iota=j\circ\pi$.
 \end{proof}

We refer to $\{ \iota\}$ as the ${\mathbf{G}}(\ov F)$-conjugacy class of Levi-center-embeddings that is dual to $\{ \hat \iota\}$. 
Although we will not need it in this paper, note that we can play the game in the other
direction and get the following result.

\begin{pro}
  The above construction sets up a bijection between ${\mathbf{G}}(\ov F)$-conjugacy classes of
  Levi-center-embeddings in ${\mathbf{G}}_{\ov F}$ and $\hat{\mathbf{G}}$-conjugacy classes of
  Levi-center-embeddings in $\hat{\mathbf{G}}$.
\end{pro}

\alin{Rationality} \label{rationality}
We now assume that the complex torus
$\hat{\mathbf{S}}$\index[notation]{Shat@$\hat{\mathbf{S}}$}  is  endowed with a finite action of $W_{F}$, thus
 corresponding to an $F$-rational structure on the torus $\mathbf{S}$\index[notation]{S@$\mathbf{S}$}.
 Then $W_{F}$ acts on the 
set of Levi-center-embeddings $\hat\iota:\, \hat{\mathbf{S}}\To{}\hat{\mathbf{G}}$ by the
formula
$^{\gamma}\hat\iota:=\gamma_{\hat{\mathbf{G}}}\circ\hat\iota\circ\gamma_{\hat{\mathbf{S}}}^{-1}$. 
\emph{We further assume that the
conjugacy class $\{\hat\iota\}$ is $W_{F}$-stable.} 

In this case, the finite subgroup $\hat H_{\{\hat\iota\}}$ of
$\hat{\mathbf{S}}$ is $W_{F}$-stable, its quotient torus $\hat{\mathbf{S}}_{\{\hat\iota\}}$\index[notation]{Shatiota@$\hat{\mathbf{S}}_{\{\hat\iota\}}$} is therefore also
equipped with a finite action of $W_{F}$, allowing to define an $F$-structure on the dual
torus $\mathbf{S}_{\{\hat\iota\}}$.\index[notation]{Siota@${\mathbf{S}}_{\{\hat\iota\}}$}
We also define a quasi-split $F$-group ${\mathbf{G}}_{\{\hat\iota\}}$ as follows.
First note that we have an action of
 the $L$-group $^{L}{\mathbf{G}}=\hat{\mathbf{G}}\rtimes W_{F}$ on Levi-center-embeddings in $\hat{\mathbf{G}}$ given  by 
$^{(\hat g,\gamma)}\hat\iota:={\rm Ad}_{\hat g}\circ {^{\gamma}\hat\iota}$, and that the
$\hat{\mathbf{G}}$-conjugacy class $\{\hat\iota\}$ is $W_{F}$-stable if and only if the
  stabilizer $(^{L}{\mathbf{G}})_{\hat\iota}$ surjects onto $W_{F}$ through the projection
  $^{L}{\mathbf{G}}\To{}W_{F}$. 
Since we assumed that the $\hat{\mathbf{G}}$-conjugacy class $\{\hat\iota\}$ is $W_{F}$-stable, we therefore obtain a short exact sequence
$\hat{\mathbf{G}}_{\hat\iota}\injo (^{L}{\mathbf{G}})_{\hat\iota}\twoheadrightarrow
W_{F}$. It follows that the conjugation action  $(^{L}{\mathbf{G}})_{\hat\iota}\To{}
{\rm Aut}(\hat{\mathbf{G}}_{\hat\iota})$ induces an outer action
$$W_{F}\To{}{\rm Out}(\hat{\mathbf{G}}_{\hat\iota}) =
\Aut(\psi_{0}(\hat{\mathbf{G}}_{\hat\iota})),$$
where $\psi_{0}$ denotes the based root datum associated to a reductive group.
For any conjugate $\hat\iota'$, this outer action is compatible with the canonical isomorphism
$\psi_{0}(\hat{\mathbf{G}}_{\hat\iota})\simto \psi_{0}(\hat{\mathbf{G}}_{\hat\iota'})$
induced by conjugation under any $\hat g$ such that $\hat\iota'={\rm Ad}_{\hat g}\circ
\hat\iota$. Further, this outer action is \emph{finite} since it induces the given action on the connected center $\hat{\mathbf{S}}$ of
 $\hat{\mathbf{G}}_{\hat\iota}$. Therefore, there is a quasi-split $F$-group
 ${\mathbf{G}}_{\{\hat\iota\}}$ endowed with a $W_{F}$-equivariant isomorphism
$\alpha:\,\psi_{0}({\mathbf{G}}_{\{\hat\iota\}})\simto
\psi_{0}(\hat{\mathbf{G}}_{\hat\iota})^{\vee}$. This pair is unique up to isomorphism and
its automorphism group is ${\mathbf{G}}_{\{\hat\iota\},\rm ad}(F)$.
 Thanks to $\alpha$,  we have $F$-rational isomorphisms 
$$ {\mathbf{G}}_{\{\hat\iota\},\rm ab}\simto \mathbf{S}
\,\,\hbox{ and }\,\,
\mathbf{S}_{\{\hat\iota\}}\simto Z({\mathbf{G}}_{\{\hat\iota\}})^{\circ}.$$
Moreover, and again thanks to $\alpha$,  we also have a map
$$H^{1}(F,{\mathbf{G}}_{\{\hat\iota\}})\To{}H^{1}(F,{\mathbf{G}}),$$ defined through
``Kottwitz duality'' \cite[Prop.~6.4]{KottCusp}
by the inclusion 
$Z(\hat{\mathbf{G}})^{W_{F}}\subset Z(\hat{\mathbf{G}}_{\hat\iota})^{W_{F}}$ where the
 action of $W_{F}$ on $Z(\hat{\mathbf{G}}_{\hat\iota})$ is induced by the conjugation action
of $(^{L}{\mathbf{G}})_{\hat\iota}$. \footnote{We could also write $Z(^{L}{\mathbf{G}})$ for
$Z(\hat{\mathbf{G}})^{W_{F}}$ and  $Z((^{L}{\mathbf{G}})_{\hat\iota})$ for
$Z(\hat{\mathbf{G}}_{\hat\iota})^{W_{F}}$ since the center of $W_{F}$ is trivial.}

\begin{prop} \label{dual-embedding-qs}
  Assume that ${\mathbf{G}}$ is quasi-split over $F$. Then the ${\mathbf{G}}(\ov F)$-conjugacy class of Levi-center-embeddings
    $\{\mathbf{S}_{\{\hat\iota\}}\injo{\mathbf{G}}_{\ov F}\}$ dual to $\{\hat\iota\}$
        contains an $F$-rational embedding $\iota$ whose stabilizer ${\mathbf{G}}_{\iota}:=C_{\mathbf{G}}(\iota(\mathbf{S_{\{\hat\iota\}}}))$\index[notation]{Giota@${\mathbf{G}}_{\iota}$} 
is naturally isomorphic  to ${\mathbf{G}}_{\{\hat\iota\}}$.  Moreover, the map 
$H^{1}(F,{\mathbf{G}}_{\iota})\To{}H^{1}(F,{\mathbf{G}})$ induced by the inclusion
${\mathbf{G}}_{\iota}\subseteq {\mathbf{G}}$ coincides with the map defined above.
\end{prop}
Here, ``naturally'' isomorphic means that there is an 
isomorphism unique up to inner automorphism, or equivalently that there is a 
$W_{F}$-equivariant isomorphism between the associated based root data.
\begin{proof}
Let us choose a maximal torus $\hat{\mathbf{T}}\subset \hat{\mathbf{G}}_{\hat\iota}$
and a Borel subgroup $\hat{\mathbf{B}}$ of $\hat{\mathbf{G}}_{\hat\iota}$ that contains
$\hat{\mathbf{T}}$. Let $\TC_{\hat{\mathbf{B}}}$ be the normalizer of the Borel pair
$(\hat{\mathbf{T}},\hat{\mathbf{B}})$ in $(^{L}{\mathbf{G}})_{\hat\iota}$. Since
$\hat{\mathbf{G}}_{\hat\iota}$ acts transitively on the set of its Borel pairs, we see
that  the map $\TC_{\hat{\mathbf{B}}}\To{}W_{F}$ is surjective, and we have a short exact sequence
$\hat{\mathbf{T}}\injo\TC_{\hat{\mathbf{B}}}\twoheadrightarrow W_{F}$. In
particular, the conjugation action of 
$\TC_{\hat{\mathbf{B}}}$  on $\hat{\mathbf{T}}$ factors through an action of $W_{F}$ on $\hat{\mathbf{T}}$.
By construction, this action preserves the based root datum
$\psi_{(\hat{\mathbf{T}},\hat{\mathbf{B}})}=(X^{*}(\hat{\mathbf{T}}),\Sigma(\hat{\mathbf{T}},\hat{\mathbf{G}}_{\hat\iota}),
\Delta(\hat{\mathbf{T}},\hat{\mathbf{B}}), X_{*}(\hat{\mathbf{T}}),\Sigma(\hat{\mathbf{T}},\hat{\mathbf{G}}_{\hat\iota})^\vee,
\Delta(\hat{\mathbf{T}},\hat{\mathbf{B}})^{\vee})$
of $\hat{\mathbf{G}}_{\hat\iota}$
associated to the Borel pair $(\hat{\mathbf{T}},\hat{\mathbf{B}})$ and the induced action
$W_{F}\To{}\Aut(\psi_{(\hat{\mathbf{T}},\hat{\mathbf{B}})})$ coincides with the outer action 
$W_{F}\To{}{\rm Aut}(\psi_{0}(\hat{\mathbf{G}}_{\hat\iota}))$ defined before the proposition,
through the canonical isomorphism 
$\psi_{(\hat{\mathbf{T}},\hat{\mathbf{B}})}=\psi_{0}(\hat{\mathbf{G}}_{\hat\iota})$.

In particular the action of $W_{F}$ on $\hat{\mathbf{T}}$ is finite and induces the given
action of $W_{F}$ on $\hat{\mathbf{S}}$, so that the whole factorization
$\hat{\mathbf{S}}\To{\hat\iota}\hat{\mathbf{T}}\To{\hat\pi}\hat{\mathbf{S}}_{\{\hat\iota\}}$ is
$W_{F}$-equivariant.
Therefore,
 endowing the dual $\o F$-torus
${\mathbf{T}}$ with the $F$-structure associated with this $W_{F}$-action,
the dual morphism
$\mathbf{S}_{\{\hat\iota\}}\To{\pi}{\mathbf{T}}$ is defined over $F$.
But since $\TC_{\hat{\mathbf{B}}}$ surjects onto $W_{F}$, any $W_{F}$-conjugate of
the embedding $\hat{\mathbf{T}}\subset\hat{\mathbf{G}}$ is also $\hat{\mathbf{G}}$-conjugate to
it. In other words, the $\hat{\mathbf{G}}$-conjugacy class of this embedding is
$W_{F}$-stable. It follows that the dual ${\mathbf{G}}(\ov F)$-conjugacy class of embeddings
${\mathbf{T}}_{\ov F}\injo{\mathbf{G}}_{\ov F}$ is also Galois stable. Since ${\mathbf{G}}$
is quasisplit, we know by
\cite{Raghunathan} that  there is a dual embedding $j:\,{\mathbf{T}}\injo{\mathbf{G}}$ defined over
$F$. Then the composite $\iota= j\circ\pi$ is also defined over $F$.

Now, the stabililizer
${\mathbf{G}}_{\iota}=C_{\mathbf{G}}(\iota(\mathbf{S}_{\{\hat\iota\}}))$\index[notation]{Giota@${\mathbf{G}}_{\iota}$} is 
 an 
$F$-subgroup of ${\mathbf{G}}$, with $j({\mathbf{T}})$ a maximal $F$-torus. By construction, the $W_{F}$ action on $\hat{\mathbf{T}}$ preserves the basis
$\Delta(\hat{\mathbf{T}},\hat{\mathbf{B}})$ of
the root system $\Sigma(\hat{\mathbf{T}},\hat{\mathbf{G}}_{\hat\iota})$, therefore the
$W_F$-action on ${\mathbf{T}}_{\ov F}$ preserves a basis of the absolute root system
$\Sigma(j({\mathbf{T}}),{\mathbf{G}}_{\iota})$, which determines a Borel subgroup $\mathbf{B}$ of
${\mathbf{G}}_{\iota}$ defined over $F$ and containing $j({\mathbf{T}})$. The associated
based root datum $\psi_{(j({\mathbf{T}}),\mathbf{B})}$ of ${\mathbf{G}}_{\iota}$ is then
$W_{F}$-equivariantly dual to $\psi_{(\hat{\mathbf{T}},\hat{\mathbf{B}})}$, and this provides
a $W_{F}$-equivariant isomorphism $\psi_{0}({\mathbf{G}}_{\iota})\simto
\psi_{0}(\hat{\mathbf{G}}_{\hat\iota})$, hence a whole class of $F$-rational isomorphisms 
${\mathbf{G}}_{\iota}\simto {\mathbf{G}}_{\{\hat\iota\}}$ modulo inner automorphisms.

It remains to check that the map 
$H^{1}(F,{\mathbf{G}}_{\iota})\To{}H^{1}(F,{\mathbf{G}})$ induced by 
the inclusion $\mathbf{G}_{\iota} \subseteq \mathbf{G}$
 coincides with the map
$H^{1}(F,{\mathbf{G}}_{\{\hat\iota\}})\To{}H^{1}(F,{\mathbf{G}})$  defined before the
proposition through any such isomorphism. 
We have Kottwitz' isomorphisms 
$\xi_{\mathbf{G}}:\,H^{1}(F,{\mathbf{G}})\simto\pi_{0}(Z(\hat{\mathbf{G}})^{W_{F}})^{*}$ and
$\xi_{{\mathbf{G}}_{\iota}}:\,H^{1}(F,{\mathbf{G}}_{\iota})\simto\pi_{0}(Z(\widehat{{\mathbf{G}}_{\iota}})^{W_{F}})^{*}$
and a canonical $W_{F}$-equivariant isomorphism
$Z(\widehat{{\mathbf{G}}_{\iota}})=Z(\hat{\mathbf{G}}_{\hat\iota})$, so the question is a
matter of compatibility of Kottwitz' isomorphisms with the inclusion maps
${\mathbf{G}}_{\iota}\subseteq {\mathbf{G}}$ on one side, and $Z(\hat{\mathbf{G}})\subseteq
Z(\hat{\mathbf{G}}_{\hat\iota})$ on the other side. This compatibiliy easily follows from
Kottwitz' argument in \cite[Prop 6.4]{KottCusp}. 
Indeed, assume first that ${\mathbf{G}}_{\rm der}$ is simply connected, so that
also ${\mathbf{G}}_{\iota,\rm der}$ is simply connected. Then 
$\xi_{\mathbf{G}}$ and $\xi_{{\mathbf{G}}_{\iota}}$ factor as follows
$$\xymatrix{\xi_{{\mathbf{G}}_{\iota}}:\, H^{1}(F,{\mathbf{G}}_{\iota}) \ar[r]^{\sim} \ar[d] &
H^{1}(F,{\mathbf{G}}_{\iota,\rm ab}) \ar[r]^{\sim} \ar[d] &
\pi_{0}(Z(\widehat{{\mathbf{G}}_{\iota}})^{W_{F}})^{*} \ar[d] \\
\xi_{\mathbf{G}}:\, H^{1}(F,{\mathbf{G}}) \ar[r]^{\sim} &
H^{1}(F,{\mathbf{G}}_{\rm ab}) \ar[r]^{\sim} &
\pi_{0}(Z(\hat{\mathbf{G}})^{W_{F}})^{*}
}$$
where the first square is obviously commutative (since it is obtained by applying
$H^{1}(F,-)$ to a commutative diagram of algebraic $F$-groups) and the second square is
also commutative since it can be reduced to 
 local duality for tori, which is functorial.
Now, to tackle the general case, Kotwittz considers a central extension $\mathbf{H}$ of
${\mathbf{G}}$ by an anistropic torus ${\mathbf{Z}}$ such that $\mathbf{H}_{\rm der}$ is simply connected.
Then the fiber product
$\mathbf{H}_{\iota}={\mathbf{G}}_{\iota}\times_{\mathbf{H}}{\mathbf{G}}$ is a central
extension of ${\mathbf{G}}_{\iota}$ by ${\mathbf{Z}}$ with simply connected derived subgroup.
We have just seen that the diagram
$$\xymatrix{ H^{1}(F,\mathbf{H}_{\iota}) \ar[r]^-{\sim}_-{\xi_{\mathbf{H}_{\iota}}} \ar[d]  &
\pi_{0}(Z(\widehat{\mathbf{H}_{\iota}})^{W_{F}})^{*} \ar[d] \\
H^{1}(F,\mathbf{H}) \ar[r]^-{\sim}_-{\xi_{\mathbf{H}}} &
\pi_{0}(Z(\hat{\mathbf{H}})^{W_{F}})^{*}
}$$
is commutative. It is moreover equivariant for the action of $H^{1}(F,{\mathbf{Z}})$ given
as usual on the first column and through $\pi_{0}(\hat{\mathbf{Z}}^{W_{F}})^{*}$ on the right
column. But Kottwitz shows that the diagram we are interested in (with ${\mathbf{G}}$'s
instead of $\mathbf{H}$'s) is obtained form this one by modding out by
this action. Therefore this diagram is commutative too.
\end{proof}

 When ${\mathbf{G}}$ is not quasi-split, there may be no $F$-rational
 Levi-center-embedding $\iota:\, \mathbf{S}_{\{\hat\iota\}}\injo {\mathbf{G}}$  dual to $\{\hat\iota\}$.
We call $\{\hat\iota\}$ \emph{relevant to ${\mathbf{G}}$}\index[terminology]{relevant $F$-rational Levi-center-embedding} if there
exists such an $F$-rational embedding $\iota$. 
We will make a connection with the notion of relevance of \cite[3]{BorelCorvallis}. To
this aim, consider the centralizer
$\MC_{\hat\iota}$ in $^{L}{\mathbf{G}}$ 
of the torus $Z((^{L}{\mathbf{G}})_{\hat\iota})^{\circ}$. 
It contains $(^{L}{\mathbf{G}})_{\hat\iota}$, hence it surjects to $W_{F}$ and by
\cite[Lemma~3.5]{BorelCorvallis}, it is a Levi subgroup of $^{L}{\mathbf{G}}$ in the sense of
\emph{loc.\ cit}.

\begin{prop} \label{dual_embedding_nqs} 
The conjugacy class $\{\hat\iota\}$ is relevant to ${\mathbf{G}}$
if and only if 
 $\MC_{\hat\iota}$ is relevant to ${\mathbf{G}}$ in the sense of  \cite[3.4]{BorelCorvallis}.
Moreover, in this case, the
  centralizer ${\mathbf{G}}_{\iota}$ of any
  $F$-rational  Levi-center-embedding $\iota:\, \mathbf{S}_{\{\hat\iota\}}\injo
  {\mathbf{G}}$  dual to $\{\hat\iota\}$  is an inner form of ${\mathbf{G}}_{\{\hat\iota\}}$. 
\end{prop}
\begin{proof}
Assume first that $\{\hat\iota\}$ is relevant and let $\iota:\, \mathbf{S}_{\{\hat\iota\}}\injo
{\mathbf{G}}$ be an $F$-rational dual embedding.
 Choose a maximal $F$-torus ${\mathbf{T}}$ of ${\mathbf{G}}_{\iota}$, and a dual embedding
$\hat\jmath:\,\hat{\mathbf{T}}\injo \hat{\mathbf{G}}$ that extends $\hat\iota$. 
Its stabilizer $({^{L}{\mathbf{G}}})_{\hat\jmath}$ in
$^{L}{\mathbf{G}}$ is contained in $({^{L}{\mathbf{G}}})_{\hat \iota}$, hence 
 $Z((^{L}{\mathbf{G}})_{\hat\jmath})^{\circ}$ contains
 $Z((^{L}{\mathbf{G}})_{\hat\iota})^{\circ}$ and therefore the Levi subgroup
 $\MC_{\hat\jmath}:=C_{^{L}{\mathbf{G}}}(Z((^{L}{\mathbf{G}})_{\hat\jmath})^{\circ})$ of
 $^{L}{\mathbf{G}}$ is contained in $\MC_{\hat\iota}$. Since
any Levi subgroup of $^{L}{\mathbf{G}}$ that contains a relevant Levi subgroup is relevant,
it suffices to show that $\MC_{\hat\jmath}$ is relevant.
Now observe that, by definition, $(^{L}{\mathbf{G}})_{\hat\jmath}$ is an extension of $W_{F}$ by
$\hat{\mathbf{T}}$ such that the action of $W_{F}$ on $\hat{\mathbf{T}}$ induced by conjugation is
the one inherited from the $F$-structure on ${\mathbf{T}}$. In particular we have
 $Z((^{L}{\mathbf{G}})_{\hat\jmath})^{\circ}
=\hat\jmath(\hat{\mathbf{T}}^{W_{F},\circ})$, and we see that
$$\Sigma(\hat{\mathbf{T}}, \MC_{\hat\jmath}^{\circ})=\left\{\alpha^{\vee} \in
\Sigma(\hat{\mathbf{T}},\hat{\mathbf{G}}), 
\langle\alpha^{\vee},X_{*}(\hat{\mathbf{T}})^{W_{F}}\rangle=0\right\}.$$
We claim that  for $\alpha\in\Sigma({\mathbf{T}},{\mathbf{G}})$ we have 
$\langle\alpha^{\vee},X_{*}(\hat{\mathbf{T}})^{W_{F}}\rangle=0 \Leftrightarrow
\langle\alpha,X_{*}({\mathbf{T}})^{W_{F}}\rangle=0$. Indeed, let $W_{F,\alpha}$ be the
finite subgroup of $\Aut_{\QM}(X_{*}(\hat{\mathbf{T}})_{\QM})$ generated by the image of
$W_{F}$ and the reflection $s_{\alpha}$. Then
$\langle\alpha^{\vee},X_{*}(\hat{\mathbf{T}})^{W_{F}}\rangle=0 \Leftrightarrow
\dim_{\QM}(X_{*}(\hat{\mathbf{T}})_{\QM}^{W_{F}})=
\dim_{\QM}(X_{*}(\hat{\mathbf{T}})_{\QM}^{W_{F,\alpha}})$, which by duality is equivalent to 
$\dim_{\QM}(X_{*}({\mathbf{T}})_{\QM}^{W_{F}})= \dim_{\QM}(X_{*}({\mathbf{T}})_{\QM}^{W_{F,\alpha}})$
hence to $\langle\alpha,X_{*}({\mathbf{T}})^{W_{F}}\rangle=0$.
Now, denoting by ${\mathbf{T}}^{\rm split}$
the maximal split subtorus of ${\mathbf{T}}$, we obtain 
$\Sigma(\hat{\mathbf{T}},\MC_{\hat\jmath}^{\circ})=\{\alpha\in
\Sigma({\mathbf{T}},{\mathbf{G}}),\,\alpha|_{{\mathbf{T}}^{\rm split}}\equiv 1\}^{\vee}$. 
It follows that $\MC_{\hat\jmath}$ is
dual to the $F$-Levi subgroup $C_{\mathbf{G}}({\mathbf{T}}^{\rm split})$ of ${\mathbf{G}}$
and is therefore relevant.

Conversely, assume now that $\MC_{\hat\iota}$ is relevant.
After replacing $\hat\iota$ by a conjugate, we may assume that $\MC_{\hat\iota}$ is
a standard Levi subgroup of $^{L}{\mathbf{G}}$, and in particular of the
form $\hat{\mathbf{M}}_{\hat\iota}\rtimes W_{F}$ for some $W_{F}$-stable Levi subgroup
$\hat{\mathbf{M}}_{\hat\iota}$ of $\hat{\mathbf{G}}$. Since $\MC_{\hat\iota}$ is relevant to ${\mathbf{G}}$,
$\hat{\mathbf{M}}_{\hat\iota}\rtimes W_{F}$ is the $L$-group of some 
$F$-Levi subgroup ${\mathbf{M}}_{\hat\iota}$ of ${\mathbf{G}}$.
On the other hand, $\hat\iota$ factors through
$\hat{\mathbf{M}}_{\hat\iota}$  and provides a Levi-center-embedding for this group. Since
$(^{L}{\mathbf{G}})_{\hat\iota}$ is contained in $\MC_{\hat\iota}$, the stabilizer
$(^{L}{\mathbf{M}}_{\hat\iota})_{\hat\iota}=(^{L}{\mathbf{G}})_{\hat\iota}$ surjects to
$W_{F}$ so that the
$\hat{\mathbf{M}}_{\hat\iota}$-conjugacy class of $\hat\iota$ is $W_{F}$-stable. So we are
now left to show that $\hat\iota$ is relevant for ${\mathbf{M}}_{\hat\iota}$.
Equivalently, we may and will restrict to the case where $\MC_{\hat\iota}={^{L}{\mathbf{G}}}$, that is
$Z(({^{L}{\mathbf{G}}})_{\hat\iota})^{\circ}=Z({^{L}{\mathbf{G}}})^{\circ}$.

  We will now  reduce further to the case where ${\mathbf{G}}$ is an adjoint group. To this
  aim, denote by $\pi: {\mathbf{G}}\To{}{\mathbf{G}}_{\rm ad}$ the adjoint quotient map
  (defined over $F$) and by
  $\hat\pi: \widehat{{\mathbf{G}}_{\rm ad}}=\hat{\mathbf{G}}_{\rm sc}\To{}\hat{\mathbf{G}}$ its
  dual ($W_{F}$-equivariant) map. Consider the connected fiber product $\hat{\mathbf{S}}_{\rm ad}:=
  (\hat{\mathbf{S}}\times_{\hat{\mathbf{G}}}\widehat{{\mathbf{G}}_{\rm ad}})^{\circ}$. This is a
  torus with finite $W_{F}$-action and the second projection  
  $\hat\iota_{\rm ad}:\, \hat{\mathbf{S}}_{\rm ad} \To{}\widehat{{\mathbf{G}}_{\rm ad}}$ is a
  Levi-center-embedding whose stabilizer $(\widehat{\mathbf{G}_{\rm ad}})_{\hat\iota_{\rm ad}}$
  is the inverse image $\hat\pi^{-1}(\hat{\mathbf{G}}_{\hat\iota})$ of that of $\hat\iota$. 
Moreover, if we  write an element $\hat g\in \hat{\mathbf{G}}$ in the form 
$\hat g= \hat z \hat\pi(\hat h)$ according to the decomposition
 $\hat{\mathbf{G}}=Z(\hat{\mathbf{G}})\hat\pi(\hat{\mathbf{G}}_{\rm sc})$, then we see that
 $({^{\hat g}\hat\iota})_{\rm ad}={^{\hat h}(\hat\iota_{\rm ad})}$. It follows  that $\{\hat\iota\}$
 determines a $\widehat{{\mathbf{G}}_{\rm ad}}$-conjugacy class $\{\hat\iota_{\rm ad}\}$. 
Since $\hat\pi$ is $W_{F}$-equivariant, $\{\hat\iota_{\rm ad}\}$ is $W_{F}$-stable, and
its stabilizer $(^{L}{\mathbf{G}}_{\rm ad})_{\hat\iota_{\rm ad}}$ is the preimage of
$(^{L}{\mathbf{G}})_{\hat\iota}$ along $\hat\pi\rtimes\id_{W_{F}}$. Also $\hat\pi$ induces a
$W_{F}$-equivariant morphism $\hat{\mathbf{S}}_{\{\hat\iota_{\rm ad}\}}\To{}
\hat{\mathbf{S}}_{\{\hat\iota\}}$ which, dually, induces an $F$-morphism
$\mathbf{S}_{\{\hat\iota\}}\To{}\mathbf{S}_{\{\hat\iota_{\rm ad}\}}$.
Now we claim that 
\begin{center}
  \emph{$\{\hat\iota\}$ is relevant to ${\mathbf{G}}$ if and only if $\{\hat\iota_{\rm
      ad}\}$ is relevant to ${\mathbf{G}}_{\rm ad}$.}
\end{center}
Indeed, suppose there is an $F$-rational Levi-center-embedding $\iota_{\rm ad}:\,
\mathbf{S}_{\{\hat\iota_{\rm ad}\}}\injo {\mathbf{G}}_{\rm ad}$ in
${\mathbf{G}}_{\rm ad}$ dual to $\{\hat\iota_{\rm ad}\}$.  
Then consider the torus
$\mathbf{S}:=(\mathbf{S}_{\{\hat\iota_{\rm ad}\}}\times_{{\mathbf{G}}_{\rm ad}}{\mathbf{G}})^{\circ}$. The
second projection provides an $F$-rational Levi-center  embedding $\iota:\,
\mathbf{S}\injo{\mathbf{G}}$ and we need to prove it is dual to $\{\hat\iota\}$. 
This is a problem over $\bar F$ and we need to go through the duality procedure of \ref{dual_embeddings}.  
So let us choose a maximal torus $\hat{\mathbf{T}}$ in
$\hat{\mathbf{G}}_{\hat\iota}$ with dual $\textbf{T}$ over $\bar F$. 
It provides  a maximal torus  $\widehat{{\mathbf{T}}_{\rm ad}}=\hat\pi^{-1}(\hat{\mathbf{T}})$
in $(\widehat{{\mathbf{G}}_{\rm ad}})_{\hat\iota_{\rm ad}}$ whose dual over $\overline F$ we denote by
$\textbf{T}_{\rm ad}$. Also $\hat\pi$ provides a dual morphism
${\mathbf{T}}\To{\pi}{\mathbf{T}}_{\rm ad}$.
Now choose an embedding 
$j:\,{\mathbf{T}}\injo {\mathbf{G}}_{\ov F}$ dual to $\hat{\mathbf{T}}\subset
\hat{\mathbf{G}}$ that factors through $({\mathbf{G}}_{\iota})_{\ov F}$. Then $\pi\circ
j$ factors over an embedding
$j_{\rm ad}:\,{\mathbf{T}}_{\rm ad}\injo ({\mathbf{G}}_{\rm ad})_{\ov F}$ dual to 
$\widehat{{\mathbf{T}}_{\rm ad}}\subset\widehat{{\mathbf{G}}_{\rm ad}}$ and
that factors through
$(({\mathbf{G}}_{\rm ad})_{\iota_{\rm ad}})_{\ov F}$.
As in \ref{dual_embeddings}, the embedding $\iota_{\rm ad}$ identifies 
$(\mathbf{S}_{\{\hat\iota_{\rm ad}\}})_{\ov F}$ with the subtorus 
$\left(\bigcap_{\alpha\in\Sigma_{\hat\iota_{\rm ad}}}\ker({\alpha^{\vee}})\right)^{\circ}$
of ${\mathbf{T}}_{\rm ad}$  involving the subroot system $\Sigma_{\hat\iota_{\rm ad}}$ of 
$\Sigma(\widehat{{\mathbf{T}}_{\rm ad}},\widehat{{\mathbf{G}}_{\rm ad}})$. 
This subroot system coincides wit $\Sigma_{\hat\iota}$ through the canonical identification
$\Sigma(\widehat{{\mathbf{T}}_{\rm ad}},\widehat{{\mathbf{G}}_{\rm
    ad}})=\Sigma(\hat{\mathbf{T}},\hat{\mathbf{G}})$. Now our definition of $\mathbf{S}$ and
$\iota$ show that $\iota$ identifies
$\mathbf{S}_{\ov F}$ with the subtorus 
$\left(\bigcap_{\alpha\in\Sigma_{\hat\iota}}\ker({\alpha^{\vee}})\right)^{\circ}$ of
${\mathbf{T}}$, hence $\iota$ is dual to $\{\hat\iota\}$ as desired.
The other implication is seen in a similar way but we omit the proof
since we do not need it here.

So we are now left to prove that if ${\mathbf{G}}$ is an adjoint
group and $Z(({^{L}{\mathbf{G}})_{\hat\iota}})^{\circ}=\{1\}$, 
then $\{\hat\iota\}$ is relevant.
Since ${\mathbf{G}}$ is adjoint, there is $\eta\in H^{1}(\Gamma_{F},{\mathbf{G}})$ such that
the associated pure inner form
${\mathbf{G}}_{\eta}$ over $F$ is quasi-split. Then $\eta^{-1}\in
H^{1}(\Gamma_{F},{\mathbf{G}}_{\Heis})$ and we have
$({\mathbf{G}}_{\eta})_{\eta^{-1}}= {\mathbf{G}}$.
 Through Kottwitz' duality
we can view $\eta^{-1}$ as a character of the finite group $Z({^{L}{\mathbf{G}}})$. Since 
$Z(({^{L}{\mathbf{G}})_{\hat\iota}})^{\circ}=\{1\}$ we may extend $\eta^{-1}$ to a character
of the finite group  $Z(({^{L}{\mathbf{G}})_{\hat\iota}})$ that we denote by $\zeta^{-1}$. 
Going through Kottwitz' duality again, we get a cohomology class $\zeta^{-1}\in
H^{1}(F,{\mathbf{G}}_{\{\hat\iota\}})$. Now by Proposition \ref{dual-embedding-qs}
there is an $F$-rational Levi-center-embedding $\iota: \mathbf{S}_{\{\hat\iota\}}\injo
{\mathbf{G}}_{\eta}$ with a natural $F$-rational isomorphism ${\mathbf{G}}_{\eta,\iota}\simeq
{\mathbf{G}}_{\{\hat\iota\}}$. Let us choose a $1$-cocycle $\zeta^{-1} :
\Gamma_{F}\To{}{\mathbf{G}}_{\eta,\iota}$ that represents the cohomology class
$\zeta^{-1}$. Then $\iota$ is still  $F$-rational for the $F$-structure of
${\mathbf{G}}_{\eta}$ twisted by $\zeta^{-1}$, \emph{i.e.}, $\iota$ is an $F$-rational
Levi-center-embedding $\mathbf{S}_{\{\hat\iota\}}\injo
({\mathbf{G}}_{\eta})_{\zeta^{-1}}$. However, we know by 
Proposition \ref{dual-embedding-qs} that the map 
$H^{1}(F,{\mathbf{G}}_{\eta,\iota})\To{} H^{1}(F,{\mathbf{G}}_{\eta})$ is induced by the
inclusion $Z({^{L}{\mathbf{G}}})\subset Z(({^{L}{\mathbf{G}}})_{\hat\iota})$
through Kottwitz' duality. Therefore we have $\zeta^{-1}=\eta^{-1}$ in
$H^{1}(F,{\mathbf{G}}_{\eta})$, so that $({\mathbf{G}}_{\eta})_{\zeta^{-1}}\simeq {\mathbf{G}}$
 and $\iota$ finally provides the desired $F$-rational
Levi-center-embedding into ${\mathbf{G}}$.

We now turn to the second assertion of the proposition. Our argument has provided one $\iota$
with centralizer ${\mathbf{G}}_{\iota}$ an inner form of ${\mathbf{G}}_{\{\hat\iota\}}$. The
fact that this property remains true for all $F$-rational embeddings dual to
$\{\hat\iota\}$ follows from the discussion above Lemma~\ref{dual-embedding-classif} below. 
\end{proof}

Now that we have studied the existence of $F$-rational dual Levi-center-embeddings, we may
try to classify all of them. Obviously ${\mathbf{G}}(F)$ acts by conjugation on these
$F$-rational embeddings.
So, let us fix one of them, $\iota$, and 
let  $\iota'$ be another one. Then pick some $g\in {\mathbf{G}}(\ov F)$ such that $\iota' = {\rm
  Ad}_{g}\circ\iota$. Then for any $\gamma\in \Gamma_{F}$ we also have 
$\iota'={^{\gamma}\iota'} = {\rm Ad}_{\gamma(g)}\circ{^{\gamma}\iota}= {\rm
  Ad}_{\gamma(g)}\circ{\iota}$, so that $g^{-1}\gamma(g)\in {\mathbf{G}}_{\iota}(\ov F)$. 
We then see that 
\begin{itemize}
\item $(\gamma\mapsto g^{-1}\gamma(g)) \in Z^{1}(F,{\mathbf{G}}_{\iota})$ and its image
$\eta_{\iota,\iota'}$  in $H^{1}(F,{\mathbf{G}}_{\iota})$ is independent of the choice of $g$.
\item ${\rm Ad}_{g}$ is an inner twisting $({\mathbf{G}}_{\iota})_{\ov F}\simto
  ({\mathbf{G}}_{\iota'})_{\ov F}$ with associated inner cocycle  $\gamma\mapsto g^{-1}\gamma(g)$. 
\end{itemize}

\begin{lemme} \label{dual-embedding-classif}
The map $\iota'\mapsto \eta_{\iota,\iota'}$ induces a bijection 
between the set of ${\mathbf{G}}(F)$-conjugacy classes of $F$-rational embeddings in
$\{\iota\}$ and  $\ker(H^{1}(F,{\mathbf{G}}_{\iota})\To{}H^{1}(F,{\mathbf{G}})).$
\end{lemme}
\begin{proof}
  Indeed, it is easily seen that $\eta_{\iota,\iota'}$ only depends on the ${\mathbf{G}}(F)$-conjugacy class of $\iota'$, and by construction it lies in the above kernel.
  Conversely, let $\eta$ belong to this kernel. Then it can be represented by a
  $1$-cocycle of the form $\gamma\mapsto g^{-1}\gamma(g)$ for some $g\in {\mathbf{G}}(\o
  F)$, and the embedding $\iota'={\rm Ad}_{g}\circ\iota$ is thus $F$-rational. This
  element $g$ is not unique, but any other one is of the form $hgk$ with $h\in
  {\mathbf{G}}(F)$ and $k\in \mathbf{{\mathbf{G}}_{\iota}}(\ov F)$ and thus leads to a
  ${\mathbf{G}}(F)$-conjugate rational embedding. We thus have constructed the inverse map.
\end{proof}

\subsection{Levi factorization of a parameter}\label{sec:levi-fact-param}

We start with a continuous $L$-homorphism $\phi:\, P_{F}\To{}{^{L}{\mathbf{G}}}$ that admits an extension to $W_{F}$.

\alin{The group $\mathbf{L}_{\phi}$}  The centralizer
$\hat{\mathbf{L}}_{\phi}:=C_{\hat{\mathbf{G}}}(Z(C_{\hat{\mathbf{G}}}(\phi))^{\circ})$\index[notation]{Lphihat@$\hat{\mathbf{L}}_{\phi}$} of the connected center
$Z(C_{\hat{\mathbf{G}}}(\phi))^{\circ}$ of $C_{\hat{\mathbf{G}}}(\phi)$ is a Levi subgroup
of $\hat{\mathbf{G}}$ which contains $C_{\hat{\mathbf{G}}}(\phi)$. 
If $\varphi:\,W_{F}\To{}{^{L}{\mathbf{G}}}$ extends $\phi$, then the conjugation
action  ${\rm Ad}_{\varphi}$ of $W_{F}$ on $C_{\hat{\mathbf{G}}}(\phi)$ preserves its
connected center and therefore also $\hat{\mathbf{L}}_{\phi}$. Since for any other extensions
$\varphi'$ the ratio $\varphi^{-1}\varphi'$ takes values in
$C_{\hat{\mathbf{G}}}(\phi)\rtimes\{1\}$,
the outer action $W_{F}\To{{\rm Ad}_{\varphi}}{\rm Out}(\hat{\mathbf{L}}_{\phi})$
is independent of the choice of $\varphi$. 
We know from \cite[Lemma 2.1.1]{Datfuncto} that this action is finite. Hence we may  denote by
$\mathbf{L}_{\phi}$\index[notation]{Lphi@${\mathbf{L}}_{\phi}$} a quasi-split  group over $F$ endowed with a
$W_{F}$-equivariant isomorphism $\psi_{0}(\mathbf{L}_{\phi})\simto \psi_{0}(\hat{\mathbf{L}}_{\phi})^{\vee}$.
Note that $\psi_{0}(\hat{\mathbf{L}}_{\phi})$ only depends on the $\hat{\mathbf{G}}$-conjugacy class
of $\phi$
in the sense that if $\phi'$ is conjugate to $\phi$,  there is a canonical isomorphism
$\psi_{0}(\hat{\mathbf{L}}_{\phi})\simto\psi_{0}(\hat{\mathbf{L}}_{\phi'})$ given by any $\hat
g$ that conjugates $\phi$ to $\phi'$. Note also that the inclusion
$Z(\hat{\mathbf{L}}_{\phi})^{W_{F}}\subset Z(\hat{\mathbf{G}})^{W_{F}}$ induces by Kottwitz'
duality a map $H^{1}(F,\mathbf{L}_{\phi})\To{}H^{1}(F,{\mathbf{G}})$.
We put 
$$H^{1}(F,\mathbf{L}_{\phi},{\mathbf{G}}):=\ker(H^{1}(F,\mathbf{L}_{\phi})\To{}H^{1}(F,{\mathbf{G}})).$$

\alin{The group $\LC_{\phi}$ and the $L$-group of $\mathbf{L}_{\phi}$}
Consider the subgroup $\LC_{\phi}:=\hat{\mathbf{L}}_{\phi}\cdot\varphi(W_{F})$\index[notation]{Lphical@$\cL_{\phi}$} of 
${^{L}{\mathbf{G}}}$. As the notation suggests, it is independent of the
choice of a parameter
$\varphi$ extending $\phi$. It sits in a split exact sequence $\hat{\mathbf{L}}_{\phi}\injo
\LC_{\phi}\twoheadrightarrow W_{F}$ and we may ask whether it is isomorphic to $^{L}\mathbf{L}_{\phi}$.
To this aim, fix a pinning $\varepsilon_\phi$ of $\hat{\mathbf{L}}_{\phi}$ and consider the
stabilizer $\LC_{\phi,\varepsilon_\phi}$\index[notation]{Lphicaleps@$\cL_{\phi, \varepsilon_{\phi}}$} of $\varepsilon_\phi$ in $\LC_{\phi}$. It sits in an exact
sequence $Z(\hat{\mathbf{L}}_{\phi})\injo\LC_{\phi,\varepsilon_\phi}\twoheadrightarrow W_{F}$.

\begin{lem} \label{splittings}
  The extension  $Z(\hat{\mathbf{L}}_{\phi})\injo\LC_{\phi,\varepsilon_\phi}\twoheadrightarrow
  W_{F}$ splits continuously, and the set of its splittings
  $W_{F}\To{}\LC_{\phi,\varepsilon_{\phi}}$ is principal homogeneous under $Z^{1}(W_{F},Z(\hat{\mathbf{L}}_{\phi}))$.
\end{lem}
\begin{proof}
Only the existence of a splitting requires a proof, the second assertion being easy.
Recall first that, by \cite[Lemma 2.1.1]{Datfuncto},  the extension under consideration comes
from a finite quotient of $W_{F}$. By Langland's Lemma 4 in \cite{Langlands_stable}, the image of 
$H^{2}_{cts}(\Gamma_{F},Z(\hat{\mathbf{L}}_{\phi})^{\circ})\To{}
H^{2}(W_{F},Z(\hat{\mathbf{L}}_{\phi})^{\circ})$ is $\{1\}$. 
This reduces the problem to showing that the extension
$$\pi_{0}(Z(\hat{\mathbf{L}}_{\phi}))\injo\LC_{\phi,\varepsilon_\phi}/Z(\hat{\mathbf{L}}_{\phi})^{\circ}\twoheadrightarrow
W_{F}$$
splits. This in turn follows from the argument in Kaletha's Lemma 5.2.6 in \cite{Kaletha}. In order
to explain this,
observe that the short
 exact sequence of the lemma is a pull-back of the short exact sequence
$$
Z(\hat{\mathbf{L}}_{\phi})\injo\NC_{^{L}{\mathbf{G}}}(\hat{\mathbf{L}}_{\phi})_{\varepsilon_\phi}\twoheadrightarrow
\NC_{^{L}{\mathbf{G}}}(\hat{\mathbf{L}}_{\phi})_{\varepsilon_\phi}/Z(\hat{\mathbf{L}}_{\phi}),$$
where the index
$\varepsilon_{\phi}$ indicates the stabilizer of the pinning $\varepsilon_{\phi}$.
To make the latter  more explicit,
we may assume that the pinning
$\varepsilon_\phi=(\hat{\mathbf{T}},\hat{\mathbf{B}}_{\phi},\{X_{\hat\alpha}\}_{\hat\alpha\in\Delta(\hat{\mathbf{T}},\hat{\mathbf{B}}_{\phi})})$
 is the restriction of a $W_{F}$-stable pinning
$\varepsilon=(\hat{\mathbf{T}},\hat{\mathbf{B}},\{X_{\hat\alpha}\}_{\hat\alpha\in\Delta(\hat{\mathbf{T}},\hat{\mathbf{B}})})$
of $\hat{\mathbf{G}}$ (after conjugating $(\phi,\varepsilon_\phi)$ by some appropriate $\hat g\in\hat{\mathbf{G}}$).
Then the inclusion
$\NC_{^{L}{\mathbf{G}}}(\hat{\mathbf{L}}_{\phi})_{\varepsilon_\phi}
\subset
\NC_{^{L}{\mathbf{G}}}(\hat{\mathbf{L}}_{\phi},\hat{\mathbf{T}},\hat{\mathbf{B}}_{\phi})$
induces an isomorphism
$$\NC_{^{L}{\mathbf{G}}}(\hat{\mathbf{L}}_{\phi})_{\varepsilon_\phi}/Z(\hat{\mathbf{L}}_{\phi})
\simto
\NC_{^{L}{\mathbf{G}}}(\hat{\mathbf{L}}_{\phi},\hat{\mathbf{T}},\hat{\mathbf{B}}_{\phi})/\hat{\mathbf{T}}
=(\Omega(\hat{\mathbf{T}},\hat{\mathbf{G}})\rtimes W_{F})_{\Delta(\hat{\mathbf{T}},\hat{\mathbf{B}}_{\phi})}
$$
where the index $\Delta(\hat{\mathbf{T}},\hat{\mathbf{B}}_{\phi})$
denotes the stabilizer of this set of characters of $\hat{\mathbf{T}}$. 
Now, consider the set-theoretic section  $\Omega(\hat{\mathbf{T}},\hat{\mathbf{G}})\rtimes
W_{F}\To{} \NC_{^{L}{\mathbf{G}}}(\hat{\mathbf{T}})$
given by Tit's liftings with respect to the pinning $\varepsilon$.
By \cite[Prop 9.3.5]{SpringerLAG},  it restricts to a map 
$(\Omega(\hat{\mathbf{T}},\hat{\mathbf{G}})\rtimes W_{F})_{\Delta(\hat{\mathbf{T}},\hat{\mathbf{B}}_{\phi})}\To{}
\NC_{^{L}{\mathbf{G}}}(\hat{\mathbf{L}}_{\phi})_{\varepsilon_{\phi}}$. 
The latter may not be a homomorphism of groups, but the content of Kaletha's study of the Tits liftings in the proof of \cite[Lemma 5.2.6]{Kaletha} is
that the composed map
$$(\Omega(\hat{\mathbf{T}},\hat{\mathbf{G}})\rtimes
W_{F})_{\Delta(\hat{\mathbf{T}},\hat{\mathbf{B}}_{\phi})}\To{}
\NC_{^{L}{\mathbf{G}}}(\hat{\mathbf{L}}_{\phi})_{\varepsilon_{\phi}}/Z(\hat{\mathbf{L}}_{\phi})^{\circ}$$ 
is a homomorphism. This provides a splitting for the first displayed
exact sequence of this proof, as desired.
\end{proof}

Let $\psi :W_{F}\To{}\LC_{\phi,\varepsilon_{\phi}}$ be a
continuous splitting as in the lemma. We get an isomorphism of extensions  $\id\times\psi:\,
^{L}\mathbf{L}_{\phi}\simto \LC_{\phi}$, where the $L$-group is formed by using the
section ${\rm Out}(\hat{\mathbf{L}}_{\phi})\injo \Aut(\hat{\mathbf{L}}_{\phi})$ associated to
$\varepsilon_{\phi}$. Then $\varphi_{L}:=(\id\times\psi)^{-1}\circ \varphi$ is a Langlands
parameter for $\mathbf{L}_{\phi}$, whose restriction to $P_{F}$ we denote by
$\phi_{L}\in\Phi(P_{F},\mathbf{L}_{\phi})$. We thus get a factorization of $\phi$
$$ \phi\, : P_{F}\To{\phi_{L}} {^{L}\mathbf{L}_{\phi}}\To{\xi_{\psi}} {^{L}{\mathbf{G}}}$$
with $\xi_{\psi}$ the composition of $\id\times\psi$ and the inclusion
$\LC_{\phi}\subset{^{L}{\mathbf{G}}}$.
Then we see that  $\xi_{\psi}$ induces an isomorphism
$C_{\hat{\mathbf{L}}_{\phi}}(\phi_{L})\simto C_{\hat{\mathbf{G}}}(\phi)$, which makes it fall
into the framework of \cite[Expectation 1.3.2]{Datfuncto},
which predicts (at least when ${\mathbf{G}}$ is quasi-split) the existence of an equivalence of categories $\prod_{\eta\in
  H^{1}(F,\mathbf{L}_{\phi},{\mathbf{G}})} \Rep^{\phi_{L}}(L_{\phi,\eta})\simto \Rep^{\phi}(G)$
where
  $\mathbf{L}_{\phi,\eta}$ is the pure inner form of $\mathbf{L}_{\phi}$ associated to
$\eta$. 
Interestingly, this set $H^{1}(F,\mathbf{L}_{\phi},{\mathbf{G}})$ and the associated pure
inner forms of $\hat{\mathbf{L}}_{\phi}$ also appear when we try to go from
$\hat{\mathbf{L}}_{\phi}$ to  twisted Levi
subgroups  of ${\mathbf{G}}$.

\alin{Twisted Levi subgroups of ${\mathbf{G}}$}  \label{dual_embeddings_phi}
With the outer action map, also the action maps $W_{F}\To{{\rm Ad}_{\varphi}}{\rm
  Aut}(Z(\hat{\mathbf{L}}_{\phi})^{\circ})$ and $W_{F}\To{{\rm Ad}_{\varphi}}{\rm
  Aut}(\hat{\mathbf{L}}_{\phi,\rm ab})$ 
are independent of the choice of
$\varphi$.   Moreover, the existence of $\varphi$ tells us that the
$\hat{\mathbf{G}}$-conjugacy class of the embedding
$Z(\hat{\mathbf{L}}_{\phi})^{\circ}\subset\hat{\mathbf{G}}$ is $W_{F}$-stable. 
\begin{no}
  We denote by $\mathbf{S}_{\phi}$\index[notation]{Sphi@${\mathbf{S}}_{\phi}$} the $F$-torus dual to the complex torus
  $\hat{\mathbf{L}}_{\phi,\rm ab}$ with its $W_{F}$-action, and by
  $I_{\phi}$\index[notation]{Iphi@$I_\phi$} the
${\mathbf{G}}(\ov F)$-conjugacy class of Levi-center-embeddings $(\mathbf{S}_{\phi})_{\ov F}\injo {\mathbf{G}}_{\ov F}$
which is dual to the Levi-center-embedding $Z(\hat{\mathbf{L}}_{\phi})^{\circ}\subseteq\hat{\mathbf{G}}$
in the sense of \ref{dual_embeddings}.
\end{no}

Recall from \ref{sec:notation} that we say that $\phi$ is \emph{a wild inertia parameter} of $\mathbf{G}$
if it admits an extension $\varphi':W'_{F}\To{}{^{L}\mathbf{G}}$ that is \emph{relevant} to $\mathbf{G}$.

\begin{pro} With the above notation: 
\begin{enumerate}
\item If $\phi$ is a wild inertia parameter of $\mathbf{G}$, then  $I_{\phi}$ contains an
  $F$-rational embedding  $\iota$.  Moreover, the following holds:
  \begin{enumerate}
  \item The centralizer $\mathbf{G}_{\iota}$ of $\iota$ is an inner form of
    $\mathbf{L}_{\phi}$, and  if $\mathbf{G}$ is quasi-split, one can choose $\iota$ such that
    $\mathbf{G}_{\iota}$ is isomorphic to $\mathbf{L}_{\phi}$. 
  \item For any 
    $F$-rational $\iota'\in I_{\phi}$ there is $\eta_{\iota,\iota'}\in H^{1}(F,\mathbf{G}_{\iota}) $
    such that ${\mathbf{G}}_{\iota'}$ is isomorphic to the pure inner form
    $\mathbf{G}_{\iota,\eta_{\iota,\iota'}}$.
  \item The map $\iota'\mapsto \eta_{\iota,\iota'}$ induces a bijection between the set of
    $G$-conjugacy classes of $F$-rational embeddings in $I_{\phi}$ and the set
    $H^{1}(F,\mathbf{G}_{\iota},{\mathbf{G}})$. 
  \end{enumerate}
\item  Assume that $C_{\hat{\mathbf{G}}}(\phi)$ is connected and that
  there exists an extension $\varphi$ of $\phi$ that preserves a pinning of
  $C_{\hat{\mathbf{G}}}(\phi)$. If $I_{\phi}$ contains an
  $F$-rational embedding  $\iota$, then $\phi$ is a wild inertia parameter for $\mathbf{G}$.
\end{enumerate}
\end{pro}
\begin{proof}
Denote by $\hat\iota$ the Levi-center-embedding
$\hat\iota:\,Z(\hat{\mathbf{L}}_{\phi})^{\circ}\subseteq \hat{\mathbf{G}}$, and recall the notation
$\MC_{\hat\iota}=C_{^{L}\mathbf{G}}(Z(({^{L}\mathbf{G}})_{\hat\iota})^{\circ})$ from Proposition 
\ref{dual_embedding_nqs}.
By definition of the $W_{F}$-action on $Z(\hat{\mathbf{L}}_{\phi})$, 
we have $({^{L}\mathbf{G}})_{\hat\iota}=\LC_{\phi}$,  hence
$\MC_{\hat\iota}=\MC_{\phi}:=C_{^{L}\mathbf{G}}(Z(\LC_{\phi})^{\circ})$.
Then Proposition \ref{dual_embedding_nqs} shows that \emph{$I_{\phi}$ contains an $F$-rational
embedding if, and only if, the Levi subgroup $\MC_{\phi}$ of $^{L}\mathbf{G}$ is relevant to $\mathbf{G}$.}

i) 
Fix a relevant Langlands parameter $\varphi':W'_{F}\To{}{^{L}{\mathbf{G}}}$ that extends
$\phi$. We claim that  $\varphi'(W'_{F})\subset \LC_{\phi}$. Indeed, the inclusion
$\varphi'(W_{F})\subset \LC_{\phi}$ holds by definition, and 
the inclusion $\varphi'({\rm SL}_{2})\subset \LC_{\phi}$ holds too since
$\varphi'({\rm SL}_{2})$ is contained in $C_{\hat{\mathbf{G}}}(\phi)$, 
hence  commutes with $Z(C_{\hat{\mathbf{G}}}(\phi))^{\circ}$ and is thus contained in
$\hat{\mathbf{L}}_{\phi}$.
Because $\MC_{\phi}$ contains $\LC_{\phi}$, it follows that $\varphi'$
factors through $\MC_{\phi}$ and, since $\varphi'$ is relevant, $\MC_{\phi}$ is relevant to
${\mathbf{G}}$. The first claims of i) and i)(a) now follow from Proposition
\ref{dual_embedding_nqs},  the remaining claims of i) follow from Proposition \ref{dual-embedding-qs}
and Lemma \ref{dual-embedding-classif}.

ii)
Fix an extension  $\varphi:W_{F}\To{}{^{L}\mathbf{G}}$ of $\phi$  that normalizes a pinning of
$C_{\hat{\mathbf{G}}}(\phi)^{\circ}$, and define $\varphi'=\varphi\times\alpha
:W_{F}\times\SL_{2}\To{}\LC_{\phi}$ with $\alpha$ an isomorphism on a principal $\SL_{2}$ in
$C_{\hat{\mathbf{G}}}(\varphi)^{\circ}$. We then have
$C_{\hat{\mathbf{G}}}(\varphi')^{\circ} =
C_{C_{\hat{\mathbf{G}}}(\varphi)^{\circ}}(\alpha)^{\circ}=Z(C_{\hat{\mathbf{G}}}(\varphi)^{\circ})^{\circ}$.
Writing
$C_{\hat{\mathbf{G}}}(\varphi)^{\circ}=(C_{\hat{\mathbf{G}}}(\phi)^{\circ})^{\varphi(W_{F}),\circ}$,
it follows from  the next lemma that
$Z(C_{\hat{\mathbf{G}}}(\varphi)^{\circ})^{\circ} =
Z(C_{\hat{\mathbf{G}}}(\phi)^{\circ})^{\varphi(W_{F}),\circ}$.
Since we assume $C_{\hat{\mathbf{G}}}(\phi)^{\circ}=C_{\hat{\mathbf{G}}}(\phi)$, this
means that $C_{\hat{\mathbf{G}}}(\varphi')^{\circ}$ is contained in 
(hence equal to) $Z(\LC_{\phi})^{\circ}$.  It follows that for any Levi subgroup $\MC$ of $^{L}\mathbf{G}$
that contains $\varphi'(W_{F})$, we have $Z(\MC)^{\circ}\subseteq Z(\LC_{\phi})^{\circ}\subseteq Z(\MC_{\phi})$
and, therefore, $\MC\supseteq \MC_{\phi}$. By assumption, $\MC_{\phi}$ is relevant, so $\MC$ is
relevant, and so is $\varphi'$. 
\end{proof}

\begin{lem}
  Let $\mathbf{H}$ be a complex reductive group and $\Gamma$ a  group acting on
  $\mathbf{H}$ and preserving a pinning of $\mathbf{H}$. Then
  $Z(\mathbf{H}^{\Gamma,\circ})^{\circ} =Z(\mathbf{H})^{\Gamma,\circ}$. 
\end{lem}
\begin{proof}
  The inclusion $Z(\mathbf{H}^{\Gamma,\circ})^{\circ} \supset
  Z(\mathbf{H})^{\Gamma,\circ}$ is clear. To get the other inclusion it is enough to show
  that $Z(\mathbf{H}^{\Gamma,\circ})\subset Z(\mathbf{H})$. 
Observe that any isogeny $\mathbf{H}'\To{}\mathbf{H}$ is $\Gamma$-equivariant for the
action of $\Gamma$ on $\mathbf{H}'$ obtained by lifting
a $\Gamma$-stable pinning from $\mathbf{H}$ to $\mathbf{H}'$ and identifying
${\rm Out}(\mathbf{H})={\rm Out}(\mathbf{H}')$. In such a situation, the image of
$(\mathbf{H}')^{\Gamma}$ has finite index in $\mathbf{H}^{\Gamma}$ so that the statement
of the lemma is true for $\mathbf{H}$ if and only if it is true for $\mathbf{H}'$.
Since this statement is clear for tori, the isogeny $\mathbf{H}'=\mathbf{H}_{\rm
  sc}\times Z(\mathbf{H})\To{}\mathbf{H}$ allows us to reduce to the case where $\mathbf{H}$
is semi-simple and simply connected.
Then $\Gamma$ permutes the set of simple factors
of $\mathbf{H}$, so we may  restrict to the case with one orbit, and then restrict to a simple factor
with the action of its stabilizer. Hence we may assume that $\mathbf{H}$ is simple
and replace $\Gamma$ by its image in ${\rm
  Out}(\mathbf{H})$ which is either $\ZM/2\ZM$  or $S_{3}$. At this point we could conclude with a case by case
inspection. But we can also invoke Steinberg's Thm 8.1 in \cite{steinberg_endo}, which ensures that
$\mathbf{H}^{\Gamma}=\mathbf{H}^{\Gamma,\circ}$ is a 
reductive group with maximal torus ${\mathbf{T}}^{\Gamma}={\mathbf{T}}^{\Gamma,\circ}$, where
${\mathbf{T}}$ is part of a $\Gamma$-stable pinning. In particular
$Z(\mathbf{H}^{\Gamma})^{\circ}\subset {\mathbf{T}}^{\Gamma}$.
Now let 
  $({\mathbf{T}},\mathbf{B},(X_{\alpha})_{\alpha\in\Delta({\mathbf{T}},\mathbf{B})})$ be a $\Gamma$-stable pinning
  of $\mathbf{H}$, where $X_{\alpha}$ is a non-zero element of the weight $\alpha$
  subspace in the Lie
  algebra $\hG$ of $H$.
Then $Z(\mathbf{H}^{\Gamma})^{\circ}$ must act trivially on the elements $\sum_{\gamma\in
  \Gamma}X_{\gamma\alpha}\in \hG^{\Gamma}$ for $\alpha\in \Delta({\mathbf{T}},\mathbf{B})$. These elements
are non-zero (here, compared to Steinberg's result, we need the fact that $\Gamma$
preserves the pinning and not only the pair $({\mathbf{T}},\mathbf{B})$),  therefore we
have
$Z(\mathbf{H}^{\Gamma})^{\circ}\subset \bigcap_{\alpha\in\Delta}\ker(\alpha)=Z(\mathbf{H})$.
\end{proof}

\begin{rak}
  For any $F$-rational $\iota\in I_{\phi}$, we may identify the group
  $X_{\phi}:=X^{*}(Z(\LC_{\phi})/Z({^{L}\mathbf{G}}))_{\rm tors}$ with $H^{1}(F,\mathbf{G}_{\iota},\mathbf{G})$ as in
  the proof of Proposition \ref{dual-embedding-qs}. Then, denoting by
   $\IC_{\phi}$ the set of $G(F)$-conjugacy classes of $F$-rational elements in $I_{\phi}$, items
   ii) and iii) of the last proposition provide $\IC_{\phi}$ with a structure of a $X_{\phi}$-torsor.
\end{rak}

\begin{lemme} \label{splittings_tr}
  Assume that ${\mathbf{G}}$ is tamely ramified. Then
  $\mathbf{L}_{\phi}$ is tamely ramified, the subgroup $1\times P_{F}$ of $^{L}{\mathbf{G}}$ is contained in
$\LC_{\phi,\varepsilon_{\phi}}$ and there is a splitting $\psi:\, W_{F}\injo \LC_{\phi,\varepsilon_{\phi}}$ which
is \emph{tame} in the sense that $\psi|_{P_{F}}=1\times\id$.
\end{lemme} 
\begin{proof}
We can write $\phi=\hat\phi\times\id$ with $\hat\phi : P_{F}\To{} \hat{\mathbf{G}}$ a homomorphism.
Then $C_{\hat{\mathbf{G}}}(\phi)=C_{\hat{\mathbf{G}}}(\hat\phi)$ so that
$\hat\phi(P_{F})\subset \hat{\mathbf{L}}_{\phi}$. Since $1\times P_{F}$ acts trivially on
$\hat{\mathbf{G}}$, it follows that the action of $\phi(P_{F})$ on $\hat{\mathbf{L}}_{\phi}$
is inner, hence $\mathbf{L}_{\phi}$ is tamely ramified.
Moreover, since $\phi(P_{F})\subset \LC_{\phi}$ by
construction, we get that $1\times P_{F}\subset \LC_{\phi}$, and because this group acts
trivially on $\hat{\mathbf{G}}$, we even have $1\times P_{F}\subset \LC_{\phi,\varepsilon_{\phi}}$.
Now, the extension $\LC_{\phi,\varepsilon_{\phi}}$ considered above is 
 the pullback of the extension 
$Z(\hat{\mathbf{L}}_{\phi})\injo \LC_{\phi,\varepsilon_{\phi}}/(1\times
P_{F})\twoheadrightarrow W_{F}/P_{F}$ by the projection $W_{F}\twoheadrightarrow
W_{F}/P_{F}$ and we need to show that the latter extension splits.
  By \cite[Lemma 3.8]{DHKM1}, we know that for any complex torus $\hat{\mathbf{S}}$ with a finite action of
  $W_{F}/P_{F}$ we have $H^{2}_{cts}(W_{F}/P_{F},\hat{\mathbf{S}})=\{1\}$ (an alternative argument
  relying on Langlands' Lemma 4 in \cite{Langlands_stable} can be found in the proof of
  \cite[Lemma 5.2.8]{Kaletha}). On the other hand, the same
  argument  as in Lemma \ref{splittings}  shows that the
  extension $\pi_{0}(Z(\hat{\mathbf{L}}_{\phi}))\injo \LC_{\phi,\varepsilon_{\phi}}/((1\times
P_{F})Z(\hat{\mathbf{L}}_{\phi})^{\circ})\twoheadrightarrow W_{F}/P_{F}$ splits.
Indeed, it suffices to replace $^{L}{\mathbf{G}}$ by its quotient ${\mathbf{G}}\rtimes (W_{F}/P_{F})$.
\end{proof}

\begin{lemme} \label{splittings_Levi}
  Assume that $C_{\hat{\mathbf{G}}}(\phi)$ is a Levi subgroup of $\hat{\mathbf{G}}$.
Then ${\mathbf{G}}$ is tamely ramified, $C_{\hat{\mathbf{G}}}(\phi)=\hat{\mathbf{L}}_{\phi}$, the subgroup $\phi(P_{F})$ is
contained in $\LC_{\phi,\varepsilon_{\phi}}$, and $\hat\phi(P_{F})\subset
Z(\hat{\mathbf{L}}_{\phi})$, where we write $\phi=\hat \phi \times \id$. Moreover, the following properties hold true :
\begin{enumerate}
\item There is a splitting $\varphi:W_{F}\injo\LC_{\phi,\varepsilon_{\phi}}$ that extends
  $\phi$.
\item There is a $1$-cocycle $\hat\varphi: W_{F}\To{}Z(\hat{\mathbf{L}}_{\phi})$ that extends $\hat\phi$.
\end{enumerate}

\end{lemme}
\begin{proof}
  The equality $C_{\hat{\mathbf{G}}}(\phi)=\hat{\mathbf{L}}_{\phi}$ is clear by definition of $\hat{\mathbf{L}}_{\phi}$. The inclusion
  $\phi(P_{F})\subset \LC_{\phi}$ holds by construction, and since $\phi(P_{F})$
  centralizes $\hat{\mathbf{L}}_{\phi}$, it normalizes $\varepsilon_{\phi}$, whence the
  inclusion $\phi(P_{F})\subset \LC_{\phi,\varepsilon_{\phi}}$. Actually, $\phi(P_{F})$
  centralizes any maximal torus of $\hat{\mathbf{L}}_{\phi}$, so a $\hat{\mathbf{G}}$-conjugate
  of $\phi(P_{F})$ centralizes a reference maximal torus $\hat{\mathbf{T}}$ of
  $\hat{\mathbf{G}}$ (i.e., a part of a $W_{F}$-stable pinning $\varepsilon$ of
  $\hat{\mathbf{G}}$). But since $\Omega(\hat{\mathbf{T}},\hat{\mathbf{G}})\rtimes_{\varepsilon}{\rm Out}(\hat{\mathbf{G}})\injo
  \Aut(\hat{\mathbf{T}})$, where $\Omega(\hat{\mathbf{T}},\hat{\mathbf{G}})$ denotes the Weyl group of $\hat{\mathbf{G}}$ with respect to $\hat{\mathbf{T}}$, the centralizer of $\hat{\mathbf{T}}$ in $^{L}{\mathbf{G}}$ is $\hat{\mathbf{T}}\times
  \ker(W_{F}\To{}{\rm Out}(\hat{\mathbf{G}}))$. It follows that the restriction of the action of $W_F$ on
  $\hat{\mathbf{T}}$ to $P_F$ is trivial, hence that ${\mathbf{G}}$ is tamely ramified. Now, with $1\times P_{F}$ and
  $\phi(P_{F})$, also $\hat\phi(P_{F})$ centralizes $\hat{\mathbf{L}}_{\phi}$, hence
  $\hat\phi(P_{F})\subset Z(\hat{\mathbf{L}}_{\phi})$.

We claim that properties i) and ii) are equivalent. Indeed, since ${\mathbf{G}}$ is tamely ramified, Lemma \ref{splittings_tr} provides us with a splitting $\psi:
W_{F}\injo \LC_{\phi,\varepsilon_{\phi}}$ such that $\psi|_{P_{F}}=1\times\id$. Therefore, if $\varphi$ is as
in item i), we can write it in the form $\varphi=\hat\varphi\cdot\psi$ and
$\hat\varphi$ is as in item ii). Conversely, the same formula shows the equivalence
$i)\Leftrightarrow ii)$.

Let us finally prove i).
Since $\phi(P_{F})\subset \LC_{\phi,\varepsilon_{\phi}}$, the
  extension $Z(\hat{\mathbf{L}}_{\phi})\injo \LC_{\phi,\varepsilon_{\phi}}\twoheadrightarrow
W_{F}$  is a pullback of the extension   
$Z(\hat{\mathbf{L}}_{\phi})\injo \LC_{\phi,\varepsilon_{\phi}}/\phi(P_{F})\twoheadrightarrow
W_{F}/P_{F}$. Therefore, there is a splitting as in i) if and only if the latter
extension splits. Now, recalling that
$\phi=\hat\phi\times\id_{P_{F}}$ and $\psi|_{P_{F}}=
1\times\id_{P_{F}}$, we have a commutative diagram
$$\xymatrix{
  Z(\hat{\mathbf{L}}_{\phi})\ar@{^{(}->}[r] &
  \LC_{\phi,\varepsilon_{\phi}}/\phi(P_{F}) \ar@{->>}[r] & W_{F}/P_{F}
  \\
 P_{F} \ar@{^{(}->}[r] \ar[u]^{\hat\phi} &
  W_{F} \ar@{->>}[r]  \ar[u]^{\psi} & W_{F}/P_{F} \ar@{=}[u]
  }$$
  that shows that the class of the upper line in
  $H^{2}(W_{F}/P_{F},Z(\hat{\mathbf{L}}_{\phi}))$ is the
  pushforward $\hat\phi(\EC)$ of  the canonical extension $\EC=[W_{F}/\o{[P_{F},P_{F}]}]\in
  H^{2}(W_{F}/P_{F},P_{F}^{\rm ab})$. But the latter is known to vanish, e.g., by the main
  result of \cite{Iwasawa} (whose proof works also in equal characteristic).
\end{proof}

\alin{The category $\Rep_{R}^{\phi}(G'_{\phi})$}  \label{subsec:twisted_depth_zero}
 Let us assume that  $C_{\hat{\mathbf{G}}}(\phi)$ is a Levi subgroup. 
 In accordance with our notation in the introduction and in \cite{Datfuncto},  we write
 ${\mathbf{G}}_{\phi}=\mathbf{L}_{\phi}$\index[notation]{Gphi@${\mathbf{G}}_{\phi}$}
 and we denote by ${\mathbf{G}}'_{\phi}$ an inner form of ${\mathbf{G}}_{\phi}$.
 Then we have $C_{\hat{\mathbf{G}}}(\phi)=\widehat{\mathbf{G}_{\phi}}=\widehat{\mathbf{G}'_{\phi}}$,
 and any choice of $\psi$ as in Lemma \ref{splittings_tr} provides a tamely ramified
 embedding $^{L}\mathbf{G}_{\phi}\injo {^{L}\mathbf{G}}$, through which the morphism $\phi$ factors.

\begin{cor}\label{coro_relevant}
  $\phi$ is a wild inertia parameter for $\mathbf{G}'_{\phi}$.
\end{cor}
\begin{proof}
  This follows from Proposition \ref{dual_embeddings_phi} ii) applied to $\mathbf{G}'_{\phi}$.
  Concretely, with
  $\hat\varphi$ as in ii) of Lemma \ref{splittings_Levi}, we define $\varphi'$ by
  $\varphi'_{|W_{F}}=\hat\varphi\rtimes\id$ and $\varphi'_{|\SL_{2}}$ given by a principal $\SL_{2}$
  of $C_{\hat{\mathbf{G}}}(\hat\varphi\rtimes\id)$.  Then $\varphi'$ does not factor through any
  proper Levi $L$-subgroup of $^{L}\mathbf{G}_{\phi}$. 
\end{proof}

Let us now fix a $1$-cocycle $\hat\varphi : W_{F}\To{} Z(C_{\hat{\mathbf{G}}}(\phi))$ as in
 Lemma \ref{splittings_Levi} ii). Multiplication by $\hat\varphi$ induces a bijection
 $Z^{1}(W_{F},\widehat{\mathbf{G}_{\phi}'})_{1} \simto Z^{1}(W_{F},\widehat{\mathbf{G}_{\phi}'})_{\hat\phi}$
 where the index prescribes the restriction of cocycles to $P_{F}$. We are interested here in the
 representation theoretic counterpart of this bijection.
  
 Let $R$ be a commutative $\Rmin$-algebra.
 Recall from the introduction the full subcategory $\Rep_{R}^{1}(G'_{\phi})$  of
 $\Rep_{R}(G'_{\phi})$ consisting of all depth-zero $R$-representations, i.e., whose objects are given by 
 \(\Rep_{R}^{1}(G'_{\phi})=\left\{V\in\Rep_{R}(G'_{\phi}), V=\sum_{x\in \BT({\mathbf{G}}'_{\phi},F)} e_{G'_{\phi,x,0+}}
 V\right\},\) 
where  $\BT({\mathbf{G}}'_{\phi},F)$ denotes the extended Bruhat--Tits building of ${\mathbf{G}}'_{\phi}$
and $e_{G'_{\phi,x,0+}}\in\HC_{R}(G'_{\phi})$ is the averaging idempotent along  the pro-$p$-radical
 $G'_{\phi,x,0+}$ of the parahoric group at the point $x$. 

  Borel's construction in  \cite[10.2]{BorelCorvallis} associates to the $1$-cocycle $\hat\varphi$
  chosen above  a character  
\begin{center}
  $\check\varphi : G'_{\phi}={\mathbf{G}}'_{\phi}(F)\To{}\CM^{\times}.$
\end{center}
Any other choice $\hat\varphi'$ differs from $\hat\varphi$ by a cocycle $\hat\delta \in
Z^{1}(W_{F},Z(\hat{\mathbf{G}}_{\phi}))$ such that $\hat\delta|_{P_{F}}=1\times\id$.
We then have $\check\varphi'=\check\varphi\check\delta$ for a character 
$\check\delta$ of $G'_{\phi}$ that is trivial
on 
$G'_{\phi,x,0+}$
for all $x\in\BT({\mathbf{G}}'_{\phi},F)$ 
(see Lemma \ref{depth_char}).
It follows that for every $x \in \BT({\mathbf{G}}'_{\phi},F)$, the restriction
$(\check\varphi)|_{G'_{\phi,x,0+}}$ is independent of the choice of $\hat\varphi$, 
and therefore the subcategory  $\check\varphi\otimes\Rep_{\CM}^{1}(G'_{\phi})$ of $\Rep_{\bC}(G'_{\phi})$ is
also independent of this choice. 
Now, the expected compatibility between the Langlands correspondence and twisting naturally leads
us to put 
$\Rep_{\CM}^{\phi}(G'_{\phi}):=\check\varphi\otimes\Rep_{\CM}^{1}(G'_{\phi}).$
It is defined over any commutative $\ZM[\mu_{p^{\infty}},\frac 1p]$-algebra $R$ by the following formula for its objects
\[\Rep_{R}^{\phi}(G'_{\phi})=\left\{V\in\Rep_{R}(G'_{\phi}),
  V=\sum_{x\in \BT({\mathbf{G}}'_{\phi},F)} e^{\phi}_{x}  V\right\}\]
where $e_{x}^{\phi}\in RG'_{\phi,x,0+}$ is the idempotent 
associated to the restriction of any $\check\varphi$ to $G'_{\phi,x,0+}$.
Actually, the next lemma together with the fact that Borel's procedure
produces a group homomorphism
$H^{1}(W_{F},Z(\hat{\mathbf{G}}_{\phi})) \To{}
\Hom(\mathbf{G}'_{\phi}(F),\CM^{\times})$, shows that 
we can choose $\hat\varphi$ such that $\check\varphi$ has $p$-power
order,  and in
particular is valued in $\ZM[\mu_{p^{\infty}},\frac 1p]^{\times}$. We then have
$\Rep_{R}^{\phi}(G'_{\phi})=\check\varphi\otimes\Rep_{R}^{1}(G'_{\phi}).$

\begin{lemme} \label{p-power_order_cocycle}
  Let $\mathbf{H}$ be a tamely ramified reductive group over $F$. 
  Given any $\hat\varphi\in Z^{1}(W_{F},Z(\hat{\mathbf{H}}))$, one can
  find $\hat\varphi'\in Z^{1}(W_{F},Z(\hat{\mathbf{H}}))$ such that
  $\hat\varphi'|_{P_{F}}=\hat\varphi|_{P_{F}}$ and
  $[\hat\varphi']$ has $p$-power order in the abelian group
  $H^{1}(W_{F},Z(\hat{\mathbf{H}}))$.
\end{lemme}
\begin{proof}
  We first claim that there exists an unramified $1$-cocycle
  $\hat\delta\in Z^{1}(W_{F}/I_{F},Z(\hat{\mathbf{H}})^{I_{F}})$ such that
  $[\hat\varphi.\hat\delta^{-1}]$ has finite order in $H^{1}(W_{F},Z(\hat{\mathbf{H}}))$.
  Indeed, consider the  exact sequence
  $$ 1\To{}H^{1}(W_{F}/I_{F},Z(\hat{\mathbf{H}})^{I_{F}})
  \To{i} H^{1}(W_{F},Z(\hat{\mathbf{H}}))
  \To{p} H^{1}(I_{F},Z(\hat{\mathbf{H}}))^{W_{F}/I_{F}}.$$
Since $I_{F}$ is profinite and cohomology is continuous for the discrete topology on the
coefficients, the last $H^{1}$ is a torsion group. So there is some integer $N$ and some
unramified cocycle $\hat\delta_{1}$ such that $[\hat\varphi]^{N}=i[\hat\delta_{1}]$. Let
$c$ be the order of the component group $\pi_{0}(Z(\hat{\mathbf{H}})^{I_{F}})$. Then,
there is $\hat\delta_{2}\in Z^{1}(W_{F}/I_{F},(Z(\hat{\mathbf{H}})^{I_{F}})^{\circ})$ such that 
$[\hat\varphi]^{Nc}=i[\hat\delta_{2}]$. But
$Z^{1}(W_{F}/I_{F},(Z(\hat{\mathbf{H}})^{I_{F}})^{\circ})\simeq
(Z(\hat{\mathbf{H}})^{I_{F}})^{\circ}$  is a divisible group, 
so we can find $\hat\delta\in Z^{1}(W_{F}/I_{F},(Z(\hat{\mathbf{H}})^{I_{F}})^{\circ})$ such
that $[\hat\varphi]^{Nc}=i[\hat\delta]^{Nc}$, showing that $[\hat\varphi\hat\delta^{-1}]$
has finite order in $H^{1}(W_{F},Z(\hat{\mathbf{H}}))$.

Now, write the order of $[\hat\varphi\hat\delta^{-1}]$ as $p^{r}\cdot M$ with $M$ prime to $p$,
choose $M'\in \NM$ such that $MM'\equiv 1 [p^{r}]$, and set
$\hat\varphi':=(\hat\varphi\hat\delta^{-1})^{MM'}$.
Then $[\hat\varphi']$ has order $p^{r}$ in $H^{1}(W_{F},Z(\hat{\mathbf{H}}))$
and $\hat\varphi'|_{P_{F}}=(\hat\varphi|_{P_{F}})^{MM'}$.
Since $P_{F}$ acts trivially on $Z(\hat{\mathbf{H}})$, we have
$Z^{1}(P_{F},Z(\hat{\mathbf{H}}))=H^{1}(P_{F},Z(\hat{\mathbf{H}}))$, hence
the order of $\hat\varphi|_{P_{F}}$ in 
$Z^{1}(P_{F},Z(\hat{\mathbf{H}}))$ divides $p^{r}M$, hence it divides $p^{r}$, and the
congruence $MM'\equiv 1 [p^{r}]$ ensures that 
$(\hat\varphi|_{P_{F}})^{MM'}=\hat\varphi|_{P_{F}}$ in $Z^{1}(P_{F},Z(\hat{\mathbf{H}}))$.
\end{proof}

\alin{Levi-center-embeddings and root systems} We still assume that
$C_{\hat{\mathbf{G}}}(\phi)$ is a Levi subgroup. 
Fix an $F$-rational
  Levi-center-embedding $\iota: \mathbf{S}_{\phi}\injo {\mathbf{G}}$ in the set $I_{\phi}$
  of Proposition \ref{dual_embeddings_phi}.
From paragraph \ref{subsec:twisted_depth_zero}, we get  a class  of
  characters $\check\varphi: \,{\mathbf{G}}_{\iota}(F)\To{}\CM^{\times}$ modulo depth-$0$ characters,
  associated to $\phi$.

Let $\mathbf{S}$ be any tamely ramified maximal $F$-torus of ${\mathbf{G}}$ containing
$\iota(\mathbf{S}_{\phi})$, and let $E\supseteq F$ be a tamely ramified Galois extension that
splits $\mathbf{S}$. We then have a norm map  $N_{E|F}: \mathbf{S}(E)\to{}\mathbf{S}(F)$
and an inclusion $\mathbf{S}(F)\subseteq {\mathbf{G}}_{\iota}(F)$.
  \begin{lem} \label{rootsystemC}
For any character $\check\varphi$ of ${\mathbf{G}}_{\iota}(F)$ associated to $\phi$, the (absolute) root system of
${\mathbf{G}}_{\iota}$ with respect to $\mathbf{S}$ is given by 
$$\Sigma(\mathbf{S},{\mathbf{G}}_{\iota})=\left\{\alpha\in\Sigma(\mathbf{S},{\mathbf{G}}),\,
\check\varphi(N_{E|F}(\alpha^{\vee}(E^{\times}_{0+})))=\{1\}\right\}.$$
  \end{lem}
  \begin{proof}
    The inclusion $\mathbf{S}\subseteq {\mathbf{G}}_{\iota}$ gives rise to a
    $W_{F}$-stable conjugacy class of maximal torus embeddings $\hat{\mathbf{S}}\injo
    C_{\hat{\mathbf{G}}}(\phi)$. Fix any such embedding and identify $\hat{\mathbf{S}}$ with a
    maximal torus in $C_{\hat{\mathbf{G}}}(\phi)$ thanks to this choice. Then, through the
    bijection $\alpha\leftrightarrow \alpha^{\vee}, \Sigma(\mathbf{S},{\mathbf{G}}) \leftrightarrow
    \Sigma(\hat{\mathbf{S}},\hat{\mathbf{G}})$, the subset
    $\Sigma(\mathbf{S},{\mathbf{G}}_{\iota})$ corresponds to
    $\Sigma(\hat{\mathbf{S}},C_{\hat{\mathbf{G}}}(\phi))$, by the construction in \ref{dual_embeddings}.

On the other hand,  $\hat{\mathbf{S}}$ contains $Z(C_{\hat{\mathbf{G}}}(\phi))$ and hence $\hat\varphi$ (as in \ref{splittings_Levi}.ii) and \ref{subsec:twisted_depth_zero}) factors through
$\hat{\mathbf{S}}$, giving a Langlands parameter that we still denote by $\hat\varphi\in Z^{1}(W_{F},\hat{\mathbf{S}})$.
This is the Langlands parameter of the character $\check\varphi: \mathbf{S}(F)\injo {\mathbf{G}}_{\iota}(F)\To{}\CM^{\times}$.
Then the Langlands parameter of
    the character $\check\varphi\circ N_{E|F} : \mathbf{S}(E)\To{}\CM^{\times}$ is
    $\hat\varphi|_{W_{E}}$. Accordingly, the character $\check\varphi\circ N_{E|F}\circ\alpha^{\vee}:\,E^{\times}\To{}\CM^{\times}$
corresponds via the local class field reciprocity to the character $\alpha^{\vee}\circ
\hat\varphi$ of $W_{E}$ (where $\alpha^{\vee}$ is first seen as a cocharacter of
$\mathbf{S}$, then as a character of
$\hat{\mathbf{S}}$). 
Its restriction to $E_{0+}^{\times}$ is therefore trivial if and only
if $\alpha^{\vee}\circ\hat\phi$ is a trivial character of  $P_{E}=P_{F}$, which is
equivalent to $\alpha^{\vee}$ being a root of $\hat{\mathbf{S}}$ in the centralizer
$C_{\hat{\mathbf{G}}}(\phi)$, as desired.
  \end{proof}

\alin{On the condition that $C_{\hat{\mathbf{G}}}(\phi)$ is a Levi subgroup} \label{centralizer_Levi}
Starting from Lemma \ref{splittings_Levi} above, our results have been conditional on the
hypothesis that $C_{\hat{\mathbf{G}}}(\phi)$ is a Levi subgroup, and this hypothesis will
be in force in the rest of the paper.
Here we discuss  how strong this hypothesis is.
It is easy to see that it implies that $\hat\phi$ factors through a maximal torus of
$\hat{\mathbf{G}}$, which implies in turn that $\hat\phi(P_{F})$ is an abelian group. In
general the reverse implications may not be true.
For this reason, Steinberg introduced in \cite{Steinberg_torsion} the notion of ``torsion prime'' for
$\hat{\mathbf{G}}$, which is a prime that either is bad for
  the root system of $\hat{\mathbf{G}}$, or divides the order of
  $\pi_{1}(\hat{\mathbf{G}}_{\rm der})$.
The following lemma follows from his results. 

\begin{lem}
  Let $P$ be a finite $p$-subgroup of $\hat{\mathbf{G}}$. Consider the following
  properties of $P$ :
  \begin{itemize}
  \item[(a)]  The centralizer $C_{\hat{\mathbf{G}}}(P)$ is a Levi subgroup of $\hat{\mathbf{G}}$.
  \item[(b)]  The connected centralizer $C_{\hat{\mathbf{G}}}(P)^{\circ}$ is a Levi
    subgroup of $\hat{\mathbf{G}}$.
  \item[(c)] $P$ is contained in a maximal torus of $\hat{\mathbf{G}}$.
  \item[(d)] $P$ is abelian.
  \end{itemize}
  Then the following holds.
  \begin{enumerate}
  \item We always have  $(a)\Rightarrow (b)\Rightarrow (c)\Rightarrow (d)$.
  \item If $p$ is good for the root system of $\hat{\mathbf{G}}$, then $(b)\Leftrightarrow
    (c)$.
  \item If $p$ is not a torsion prime of $\hat{\mathbf{G}}$, then
    $(a)\Leftrightarrow (b)\Leftrightarrow (c)\Leftrightarrow (d)$.
  \item If $p$ does not divide the order of the Weyl group of $\hat{\mathbf{G}}$, then
    $(a)$ 
    holds true for all $P$.
  \end{enumerate}
\end{lem}
Note that each implication $(d)\Rightarrow (c)\Rightarrow
(b)\Rightarrow (a)$ fails to be true in full generality.

\begin{proof}
  i) Only the implication $(b)\Rightarrow (c)$ is not tautological.
  So assume  $C_{\hat{\mathbf{G}}}(P)^{\circ}$ is a Levi subgroup of
  $\hat{\mathbf{G}}$. Then its  centralizer in $\hat{\mathbf{G}}$ is its center, which is
  contained in all its maximal tori. Since $P$ centralizes
  $C_{\hat{\mathbf{G}}}(P)^{\circ}$, we are done.

  ii) Assume $p$ good. When $P$ is cyclic, the fact that $C_{\hat{\mathbf{G}}}(P)^{\circ}$
   is a Levi subgroup is proved e.g. in \cite[Prop. A.7]{AdlerSpice}. For a more general $P$
   contained in a maximal torus $\hat{\mathbf{T}}$, let us split it as a product
   $P=P_{1}\times\cdots\times P_{r}$ of cyclic $p$-groups.
   Each $C_{\hat{\mathbf{G}}}(P_{i})^{\circ}$ is a Levi subgroup
   that contains $\hat{\mathbf{T}}$, hence their intersection
   $\bigcap_{i}C_{\hat{\mathbf{G}}}(P_{i})^{\circ}$ is also a Levi subgroup containing
   $\hat{\mathbf{T}}$. Since $C_{\hat{\mathbf{G}}}(P)=\bigcap_{i}C_{\hat{\mathbf{G}}}(P_{i})$, we
   infer that $C_{\hat{\mathbf{G}}}(P)^{\circ} = \bigcap_{i}C_{\hat{\mathbf{G}}}(P_{i})^{\circ}$
   is a Levi subgroup.

   iii) Assume $p$ is not a torsion prime of $\hat{\mathbf{G}}$.
   We only have to prove $(d)\Rightarrow (a)$, so let us assume that $P$ is abelian.
   When $P$ is cyclic, hence contained in a torus, we already know that
   $C_{\hat{\mathbf{G}}}(P)^{\circ}$ is a Levi subgroup by ii). On top of that,  since $p$
   does not divide $|\pi_{1}(\hat{\mathbf{G}}_{\rm der})|$, we have
   $C_{\hat{\mathbf{G}}}(P)=C_{\hat{\mathbf{G}}}(P)^{\circ}$  by \cite[Cor. 2.16]{Steinberg_torsion}. 
  For a general abelian $P$, let us argue by induction on its rank $r$, i.e. the number of cyclic
  factors of $P$.
  We can write  $P=P'\times P''$ with $P'$ cyclic and $P''$ of rank $r-1$. Then
  $P''$ is contained in $\hat{\mathbf{L}}:=C_{\hat{\mathbf{G}}}(P')$ 
  and we have
  $C_{\hat{\mathbf{G}}}(P)=C_{\hat{\mathbf{L}}}(P'')$. Since $\hat{\mathbf{L}}$ is a Levi subgroup of
  $\hat{\mathbf{G}}$, the prime  $p$ is not a torsion prime of $\hat{\mathbf{L}}$. So, by
  induction $C_{\hat{\mathbf{L}}}(P'')$ is a Levi subgroup of $\hat{\mathbf{L}}$, hence
  also of $\hat{\mathbf{G}}$.

  iv) Let us assume that $p$ does not divide the order of the Weyl group of $\hat{\mathbf{G}}$, and let $P$ be a
  finite $p$-subgroup of $\hat{\mathbf{G}}$. By \cite[Cor. 2.8]{Steinberg_torsion},
  $p$ is not a   torsion prime of $\hat{\mathbf{G}}$, hence, by iii),  the centralizer of any abelian
  subgroup $A$ of $P$ is a Levi subgroup $\hat{\mathbf{M}}$. If $A$ is normal in $P$, then
  $P$ normalizes $\hat{\mathbf{M}}$. Since
  $N_{\hat{\mathbf{G}}}(\hat{\mathbf{M}})/\hat{\mathbf{M}}$ is a subquotient of the
  Weyl   group of $\hat{\mathbf{G}}$, and thus has order prime to $p$, it follows that $P$ is actually contained in
  $\hat{\mathbf{M}}$, and thus centralizes $A$. However, a non-abelian $p$-group $Q$ always
  contains an abelian, normal, but non-central, subgroup ; for example, the subgroup 
  generated by the center $Z(Q)$ and any non-central element of the second-center (the
  inverse image of the center of $Q/Z(Q)$).  Therefore $P$ must be abelian.
\end{proof}

From the point of view of the representation theory of $G$, the hypothesis that
$p$ is good is quite mild. For a quasi-simple group, it is empty in type $A_{n}$ and it means $p\neq 2$ 
in types $B_{n}, C_{n}$ and $D_{n}$, $p>3$ in type
$G_{2}$, $E_{6}$ and $E_{7}$, and  $p>5$ in type $E_{8}$.
For such groups, not being a torsion prime is equivalent to being good except
in type $A_{n}$. There, it is
at least  satisfied if $p$ does not divide $n+1$, but, for example $\GL_{n}$ has no
torsion prime.
The hypothesis that $p$ does not divide the order of the Weyl group is
obviously a much stronger one, since it excludes $\GL_{n}$ for $n\geq p$.

\subsection{Ramification groups and twisted Levi sequences}\label{sec:ramif-groups-twist}
We denote by $I_{F}^{r}$, $r\in\RM_{+}$, the ramification subgroups of the Galois group $\Gamma_{F}$
in the upper numbering. We also put
$I_{F}^{r+}:=\o{\bigcup_{s>r}I_{F}^{s}}$. So we have $I_{F}^{0}=I_{F}$
and $I_{F}^{0+}=P_{F}$. As seems to be customary,  we  use the notation $\wt\RM:=\RM\sqcup
\{r+,r\in\RM\}$, which is ordered by letting $r<r+<s$ for any $r<s\in\RM$.

\begin{center}
\emph{We fix a wild inertia parameter $\phi:P_{F}\To{}{{}^L{\mathbf{G}}}$ and we assume that
  $C_{\hat{\mathbf{G}}}(\phi)$ is a Levi subgroup.}
\end{center}

By Lemma \ref{splittings_Levi}, this implies that ${\mathbf{G}}$ is tamely ramified and that
$\hat\phi(P_{F})$ is a finite abelian $p$-group contained in the center of
$C_{\hat{\mathbf{G}}}(\phi)=\hat{\mathbf{G}}_{\phi}=\hat{\mathbf{L}}_{\phi}$. 
Recall  that
$\mathbf{S}_{\phi}$\index[notation]{Sphi@${\mathbf{S}}_{\phi}$} denotes the $F$-torus that is dual to
$\hat{\mathbf{S}}_{\phi}=(\hat{\mathbf{G}}_{\phi})_{\rm ab}$\index[notation]{Sphihat@$\hat{\mathbf{S}}_{\phi}$} with its canonical Galois action. We
are going to define a filtration of $\mathbf{S}_{\phi}$ by $F$-subtori. 
For this, we assume from now on that the following hypothesis is satisfied:
\begin{equation}\label{H1}
  \tag{\em H1} \emph{$p$ is not a torsion prime  of $\hat{\mathbf{G}}$.}
\end{equation}

\alin{The groups ${\mathbf{G}}_{\phi,r}$ and $\mathbf{S}_{\phi,r}$} \label{defGphir}
Fix $r\in\widetilde\RM_{>0}$. We put
$\hat{\mathbf{G}}_{\phi,r}:=C_{\hat{\mathbf{G}}}(\phi(I_{F}^{r}))^{}$
and $\hat{\mathbf{S}}_{\phi,r}:=(\hat{\mathbf{G}}_{\phi,r})_{\rm ab}$. By
Hypothesis \eqref{H1} and Lemma \ref{centralizer_Levi}, $\hat{\mathbf{G}}_{\phi,r}$ is a Levi
subgroup of $\hat{\mathbf{G}}$ that contains 
$C_{\hat{\mathbf{G}}}(\phi)$. Therefore, the group
$\GC_{\phi,r}:=\hat{\mathbf{G}}_{\phi,r}\cdot\varphi(W_{F})$ does not
depend on the choice $\varphi$ of an extension
of $\phi$ to $W_{F}$ and sits in an exact sequence $\hat{\mathbf{G}}_{\phi,r}\injo
\GC_{\phi,r}\twoheadrightarrow W_{F}$ which provides a 
 finite outer action
$W_{F}\To{}{\rm Out}(\hat{\mathbf{G}}_{\phi,r})$ and thus defines a 
quasi-split reductive $F$-group ${\mathbf{G}}_{\phi,r}$. Since $\hat\phi(P_{F})$ is
contained in $C_{\hat{\mathbf{G}}}(\phi)$ hence also in $\hat{\mathbf{G}}_{\phi,r}$, the outer
action factors through $W_{F}/P_{F}$ and accordingly ${\mathbf{G}}_{\phi,r}$ is tamely ramified.
Also this outer action descends to
$\hat{\mathbf{S}}_{\phi,r}$, providing a dual tamely ramified $F$-torus $\mathbf{S}_{\phi,r}$ with a canonical
isomorphism $\mathbf{S}_{\phi,r}\simto Z({\mathbf{G}}_{\phi,r})^{\circ}$.

Let  us choose a pinning $\varepsilon_{\phi,r}$ of $\hat{\mathbf{G}}_{\phi,r}$ and consider
the stabilizer $\GC_{\phi,r,\varepsilon_{\phi,r}}$ of this pinning in $\GC_{\phi,r}$, which sits in an exact
sequence $ Z(\hat{\mathbf{G}}_{\phi,r})\injo
\GC_{\phi,r,\varepsilon_{\phi,r}}\twoheadrightarrow W_{F}$. Observe that $1\times P_{F}$
is contained in $\GC_{\phi,r,\varepsilon_{\phi,r}}$. By the same arguments as in the proof of Lemma
\ref{splittings_tr}, we have :

\begin{lem} 
  There exists a splitting $\psi_{r}:\, W_{F}\To{}\GC_{\phi,r,\varepsilon_{\phi,r}}$ of the
  exact sequence $ Z(\hat{\mathbf{G}}_{\phi,r})\injo
  \GC_{\phi,r,\varepsilon_{\phi,r}}\twoheadrightarrow W_{F}$ such that $\psi_{r}|_{P_{F}}=1\times\id$.
\end{lem}

Such a $\psi_{r}$ induces an isomorphism $^{L}\mathbf{G}_{\phi,r}\simto \GC_{\phi,r}$ that is the
identity on $1\times P_{F}$, hence allows us to see $\phi$ as a wild inertia parameter for $\mathbf{G}_{\phi,r}$.

Now, the inclusion $\hat{\mathbf{G}}_{\phi}\subseteq \hat{\mathbf{G}}_{\phi,r}$
induces a $W_{F}$-equivariant epimorphism $\hat{\mathbf{S}}_{\phi}\twoheadrightarrow
\hat{\mathbf{S}}_{\phi,r}$, which on the dual side induces an $F$-rational embedding
$\mathbf{S}_{\phi,r}\injo \mathbf{S}_{\phi}$. Note that the latter embedding only depends
on $\phi$, and on no other choice.
For any $F$-rational Levi-center-embedding
$\iota:\mathbf{S}_{\phi}\injo {\mathbf{G}}$ in the set $I_{\phi}$ of Proposition
\ref{dual_embeddings_phi}, the restriction $\iota_{|\mathbf{S}_{\phi,r}}$ is an $F$-rational
Levi-center-embedding, and the twisted Levi subgroup $C_{\mathbf{G}}(\iota(\mathbf{S}_{\phi,r}))$ of
$\mathbf{G}$ has $L$-group $^{L}\mathbf{G}_{\phi,r}$.
\begin{lemme}\label{phi_admissible}
$\phi$ is a wild inertia parameter of $C_{\mathbf{G}}(\iota(\mathbf{S}_{\phi,r}))$.
\end{lemme}
\begin{proof} 
  This follows from  ii) of Proposition \ref{dual_embeddings_phi} applied to the group
  $\mathbf{G}'_{\phi,r}:=C_{\mathbf{G}}(\iota(\mathbf{S}_{\phi,r}))$. Indeed, since
  $C_{\hat{\mathbf{G}}_{\phi,r}}(\phi)=C_{\hat{\mathbf{G}}}(\phi)$ is a Levi subgroup of $\hat{\mathbf{G}}_{\phi,r}$, it is
  connected, and any pinning is normalized by a suitable extension of $\phi$, according to Lemma
  \ref{splittings_Levi} i). On the other hand,
  $\iota$ induces a $F$-rational Levi-center-embedding $\mathbf{S}_{\phi}\injo
  \mathbf{G}_{\phi,r}'=C_{\mathbf{G}}(\iota(\mathbf{S}_{\phi,r}))$ that is dual to
  the Levi-center-embedding $\hat\iota :\,Z(C_{\hat{\mathbf{G}}}(\phi))^{\circ}\injo \hat{\mathbf{G}}_{\phi,r}$.
  So we may apply Proposition \ref{dual_embeddings_phi} ii).
\end{proof}

\begin{lemme} \label{rootsystemCr}
  Let $\iota:\mathbf{S}_{\phi}\injo {\mathbf{G}}$ be an $F$-rational Levi-center-embedding
in the set $I_{\phi}$ of Proposition \ref{dual_embeddings_phi},  and let $\mathbf{S}$ be 
a maximal $F$-torus of ${\mathbf{G}}_{\iota}$ split
  by some tamely ramified Galois extension $E$ of $F$. Then
for any character $\check\varphi$ of ${\mathbf{G}}_{\iota}(F)$ associated to $\phi$ as in
\ref{subsec:twisted_depth_zero}, we have  
$$\Sigma(\mathbf{S},C_{\mathbf{G}}(\iota(\mathbf{S}_{\phi,r})))=\left\{\alpha\in\Sigma(\mathbf{S},{\mathbf{G}}),\,
\check\varphi(N_{E|F}(\alpha^{\vee}(E^{\times}_{r})))=\{1\}\right\}.$$
\end{lemme}
\begin{proof}
  As in the proof of Lemma \ref{rootsystemC}, fix a dual embedding $\hat{\mathbf{S}}\subset
  C_{\hat{\mathbf{G}}}(\phi)$. Then, by the construction in \ref{dual_embeddings}, the
    bijection $\alpha\leftrightarrow \alpha^{\vee}, \Sigma(\mathbf{S},{\mathbf{G}}) \leftrightarrow
    \Sigma(\hat{\mathbf{S}},\hat{\mathbf{G}})$ takes
    $\Sigma(\mathbf{S},C_{\mathbf{G}}(\iota(\mathbf{S}_{\phi,r})))$  to
    $\Sigma(\hat{\mathbf{S}},C_{\hat{\mathbf{G}}}(\phi(I_{F}^{r}))^{})=\{\alpha^{\vee}\in
    \Sigma(\hat{\mathbf{S}},\hat{\mathbf{G}}), \alpha^{\vee}\circ
    \hat\phi(I_{F}^{r})=\{1\}\}$. It remains to follow the proof of 
\cite[Lemma 3.6.1]{Kaletha}. 
Indeed  $\alpha^{\vee}\circ\hat\varphi|_{W_{E}}$
    corresponds to $\check\varphi\circ N_{E|F}\circ\alpha^{\vee}$ via the local class
    field reciprocity $E^{\times}\simto W_{E}^{\rm ab}$, while  the latter also takes $E_{r}^{\times}$ to
    the image of $I_{E}^{r}=I_{F}^{r}$ in $W_{E}^{\rm ab}$. The lemma follows.
\end{proof}

\alin{The twisted Levi sequence associated to $\phi$ and $\iota$} \label{deftwistedLeviseq}
We denote by $0<r_{0}<\hdots <r_{d-1}$ the jumps of the decreasing
filtration $(\mathbf{S}_{\phi,r})_{r>0}$ of
$\mathbf{S}_{\phi}$. Namely we have
$$ \{r_{0},\hdots, r_{d-1}\}=\{r>0,\, \mathbf{S}_{\phi,r+}\subsetneqq \mathbf{S}_{\phi,r}\}=
\{r>0, C_{\hat{\mathbf{G}}}(\phi(I_{F}^{r+}))^{}\supsetneqq C_{\hat{\mathbf{G}}}(\phi(I_{F}^{r}))^{}\}.$$
Note that $\mathbf{S}_{\phi,r}=\mathbf{S}_{\phi}$ for $r\leq r_{0}$ while
$\mathbf{S}_{\phi,r}=Z({\mathbf{G}})^{\circ}$ for $r>r_{d-1}$.
We also put  $r_{-1}:=0$ and $r_{d}:={\rm depth}(\phi):={\rm inf}\{r>0,\phi(I_{F}^{r})=\{1\}\}$, which
satisfies $r_{d}\geq r_{d-1}$.

Now fix an $F$-rational Levi-center-embedding $\iota :\, \mathbf{S}_{\phi}\injo
{\mathbf{G}}$ in $I_{\phi}$. In order to simplify the notation a bit, we set 
$${\mathbf{G}}_{\iota}^{i}:=C_{\mathbf{G}}(\iota(\mathbf{S}_{\phi,r_{i}}))
=C_{\mathbf{G}}(\iota(\mathbf{S}_{\phi,r_{i-1}+}))
 \hbox{ for }
i=0,\hdots, d-1 \hbox{ and } {\mathbf{G}}_{\iota}^{d}:={\mathbf{G}}$$
We thus obtain a tamely ramified twisted Levi sequence in  ${\mathbf{G}}$
$$\vec{\mathbf{G}}_{\iota}:=\left({\mathbf{G}}_{\iota}={\mathbf{G}}_{\iota}^{0}\subset\hdots\subset
  {\mathbf{G}}_{\iota}^{d}={\mathbf{G}}\right).$$ 
  
\subsection{Characters and idempotents}\label{sec:char-idemp}

In this section, we will use the Langlands correspondence for characters
described by Borel in
\cite[10.2]{BorelCorvallis}
to construct certain characters of
${\mathbf{G}}_{\iota}^{i}(F)$ that are suitable to apply Yu's procedure in
\cite{Yu_tamescusp} and obtain characters of certain open pro-$p$ subgroups of $G$.
Yu's work 
involves the group side analogue of the ramification filtration, namely
 the Moy--Prasad filtrations \cite{MP1, MP2}. For each point $x$ in the
(enlarged) Bruhat--Tits building $\BT({\mathbf{G}},F)$ of $\mathbf{G}$ we thus have a filtration
$(G_{x,r}={\mathbf{G}}(F)_{x,r})_{r\geq 0}$ of the 
stabilizer $G_{x}={\mathbf{G}}(F)_{x}$ of $x$ by compact, open, normal subgroups.
If we set ${G}_{x,r+}:=\bigcup_{s>r}{G}_{x,s}$, then $G_{x,0+}$ is
known to be the pro-$p$-radical of the parahoric group $G_{x,0}$.
We will need the following relation between both filtrations, which follows easily from
Yu's \cite[Theorem 7.10]{Yu_LLCtori}.

\begin{lemme} \label{depth_char}
  Let ${\mathbf{G}}$ be a tamely ramified reductive group over $F$ and let $\check\varphi :
  {\mathbf{G}}(F)\To{}\CM^{\times}$ be the character associated to some
$\hat\varphi\in H^{1}(W_{F},Z(\hat{\mathbf{G}}))$. Then $\check\varphi$ is trivial on
${\mathbf{G}}_{\rm sc}(F)$ and for every $x\in \BT({\mathbf{G}},F)$ and every
$r\in\wt\RM_{\geq 0}$ we have
$\check\varphi|_{G_{x,r}}\equiv 1 \Leftrightarrow \hat\varphi|_{I_{F}^{r}}\equiv 1$.
\end{lemme}
\begin{proof}
  We need to go through Borel's procedure in \cite[10.2]{BorelCorvallis}. So let
  $\tilde{\mathbf{G}}\To{}{\mathbf{G}}$ be a $z$-extension, i.e., a central extension $\tilde{\mathbf{G}}$ of ${\mathbf{G}}$  by an induced torus ${\mathbf{Z}}$ such that the derived subgroup of $\tilde{\mathbf{G}}$ is  simply connected. 
On the dual side we get a $W_{F}$-equivariant embedding of
  $Z(\hat{\mathbf{G}})$ into the torus $Z(\hat{\tilde{\mathbf{G}}})$. Pushing $\hat\varphi$ by
  this embedding we get a Langlands parameter for the tamely ramified torus 
$\tilde{\mathbf{G}}_{\rm ab}$, whence a character 
$\tilde\theta$ of $\tilde{\mathbf{G}}_{\rm ab}(F)$. 
By \cite[Thm 7.10]{Yu_LLCtori} we have $\tilde\theta|_{\tilde{\mathbf{G}}_{\rm
    ab}(F)_{r}}\equiv 1 \Leftrightarrow \hat\varphi|_{I_{F}^{r}}\equiv 1$. 
Now, $\check\varphi$ is defined as follows.
The map $\tilde{\mathbf{G}}(F)\To{}{\mathbf{G}}(F)$ is surjective and 
 the character
$\tilde\theta:\tilde{\mathbf{G}}(F)\To{}\tilde{\mathbf{G}}_{\rm ab}(F)\To{}\CM^{\times}$ is
trivial on the kernel ${\mathbf{Z}}(F)$ of this map and on $\tilde{\mathbf{G}}_{\rm
  der}(F)={\mathbf{G}}_{\rm sc}(F)$. Therefore $\tilde\theta$ descends to 
the desired character $\check\varphi$ of ${\mathbf{G}}(F)$, which is trivial on (the image of)
${\mathbf{G}}_{\rm sc}(F)$. 
Now, for every $x\in \BT({\mathbf{G}},F)$ and every $\tilde x\in \BT(\tilde{\mathbf{G}},F)$ that projects onto $x$,
Lemma 3.5.3 of \cite{Kaletha} tells us 
that the maps $\tilde{\mathbf{G}}(F)_{\tilde x,r}\To{} \tilde{\mathbf{G}}_{\rm ab}(F)_{r}$ and 
$\tilde{\mathbf{G}}(F)_{\tilde x,r}\To{}{\mathbf{G}}(F)_{x,r}$ are both surjective. This implies the
equivalence claimed in the lemma.
\end{proof}

\begin{rema} \label{Borel_proc}
  Conversely, \emph{any character $\check\varphi :  {\mathbf{G}}(F)\To{}\CM^{\times}$ that is trivial on
${\mathbf{G}}_{\rm sc}(F)$ comes from some $\hat\varphi\in H^{1}(W_{F},Z(\hat{\mathbf{G}}))$
via Borel's procedure.} Indeed, with the notation of the above proof, the surjectivity of $\tilde{\mathbf{G}}(F)\To{}{\mathbf{G}}(F)$
allows one to inflate $\check\varphi$ to a character $\wt\theta$ of $\tilde{\mathbf{G}}(F)$ that is
trivial on ${\mathbf{Z}}(F)\tilde{\mathbf{G}}_{\rm der}(F)$. In particular $\tilde\theta$
factors through the
surjective map $\tilde{\mathbf{G}}(F)\To{}\tilde{\mathbf{G}}_{\rm ab}(F)$, giving  a character of $\tilde{\mathbf{G}}_{\rm ab}(F)$ which, by Langlands'
correspondence for tori, comes from some $\hat\varphi\in
H^{1}(W_{F},Z(\hat{\tilde{\mathbf{G}}}))$. But the pushforward of $\hat\varphi$ into
$H^{1}(W_{F},\hat{\mathbf{Z}})$ has to be trivial, hence $\hat\varphi$ comes from
$H^{1}(W_{F},Z(\hat{\mathbf{G}}))$. 
\end{rema}

Recall the definitions of $\hat{\mathbf{G}}_{\phi,r}$ and $\GC_{\phi,r}$ from \ref{defGphir}, as
well as that of $\GC_{\phi,r,\varepsilon_{\phi,r}}$ for a pinning
$\varepsilon_{\phi,r}$ of $\hat{\mathbf{G}}_{\phi,r}$.
Observe that 
 $\phi(I_{F}^{r})$ is contained in $\GC_{\phi,r,\varepsilon_{\phi,r}}$.

\begin{lemme} \label{splittings_ram}
  The following hypotheses are
  equivalent.
  \begin{enumerate}
  \item There is a splitting $\varphi_{r}:\, W_{F}\To{}\GC_{\phi,r,\varepsilon_{\phi,r}}$ such
    that $\varphi_{r}|_{I_{F}^{r}}=\phi|_{I_{F}^{r}}$.
  \item There is $\hat\varphi_{r}\in Z^{1}(W_{F},Z(\hat{\mathbf{G}}_{\phi,r}))$ such that $\hat\varphi_{r}|_{I_{F}^{r}}=\hat\phi|_{I_{F}^{r}}$.
  \item The image $\hat\phi(\EC_{r})\in H^{2}(W_{F}/I_{F}^{r},Z(\hat{\mathbf{G}}_{\phi,r}))$ 
of the canonical extension $\EC_{r}=[W_{F}/\o{[I_{F}^{r},I_{F}^{r}]}]\in
  H^{2}(W_{F}/I_{F}^{r},I_{F}^{r,\rm ab})$  vanishes.
  \end{enumerate}
Further, these hypotheses are satisfied if $p$ does not divide
$|\pi_{0}(Z(\hat{\mathbf{G}}))|=|\pi_{1}({\mathbf{G}}_{\rm der})|$.
\end{lemme}
\begin{proof}
Thanks to Lemma \ref{defGphir},
  the equivalence between the three
hypotheses is proved as in Lemma \ref{splittings_Levi}.
For the last assertion, observe that if $p$ does not divide the order
$|\pi_{0}(Z(\hat{\mathbf{G}}))|$ of $\pi_{0}(Z(\hat{\mathbf{G}}))$, then it does not divide
$|\pi_{0}(Z(\hat{\mathbf{G}}_{\phi,r}))|$ either, since
 $Z(\hat{\mathbf{G}}_{\phi,r})=Z(\hat{\mathbf{G}}_{\phi,r})^{\circ}Z(\hat{\mathbf{G}})$.
 Therefore, in this case, we have
 $\hat\phi(P_{F})\subset Z(\hat{\mathbf{G}}_{\phi,r})^{\circ}$,
 hence $\hat\phi(\EC_{r})$ lies in the image of
 the map
 $H^{2}(W_{F}/I_{F}^{r},Z(\hat{\mathbf{G}}_{\phi,r})^{\circ})\To{}H^{2}(W_{F}/I_{F}^{r},Z(\hat{\mathbf{G}}_{\phi,r}))$.
By Lemma \ref{h2_ramif} below, it follows that $\hat\phi(\EC_{r})$ vanishes. 
\end{proof}

\begin{lemme} \label{h2_ramif}
  Let $\mathbf{S}$ be a tamely ramified torus and $r\in\wt\RM_{>0}$. Then the image of
  $H^{2}_{cts}(\Gamma_{F}/I_{F}^{r},\hat{\mathbf{S}})$ in
  $H^{2}(W_{F}/I_{F}^{r},\hat{\mathbf{S}})$ is trivial.
\end{lemme}
\begin{proof}
  Start with $\eta\in H^{2}_{cts}(\Gamma_{F}/I_{F}^{r},\hat{\mathbf{S}})$ and let $\bar\eta$
  denote its image in $H^{2}(W_{F}/I_{F}^{r},\hat{\mathbf{S}})$. By definition of
  continuous cohomology, $\eta$ comes
  from an element $\eta\in H^{2}(\Gamma_{E/F},\hat{\mathbf{S}})$ with $E$ a finite extension
  that splits $\mathbf{S}$ and that is $r$-ramified in the sense that
  $I_{F}^{r}$ maps to $\{1\}$ in $\Gamma_{E/F}$. We may and will assume that  $\mathbf{S}$ is also split by the maximal tamely
  ramified subextension $E^{tr}$ of $E$ over $F$.
 As in the proof of \cite[Lemma 4]{Langlands_stable}, we can choose an exact sequence
  $\hat{\mathbf{S}}\injo \hat{\mathbf{S}}_{1}\twoheadrightarrow \hat{\mathbf{S}}_{2}$ with
  $\mathbf{S}_{1}$ an induced torus for $\Gamma_{E^{tr}/F}$. Note that each $\mathbf{S}_{i}$
  is then tamely ramified.  Let us look at the exact sequence
$$
H^{1}(W_{F}/I_{F}^{r},\hat{\mathbf{S}}_{1})\To{}H^{1}(W_{F}/I_{F}^{r},\hat{\mathbf{S}}_{2})
\To{}
H^{2}(W_{F}/I_{F}^{r},\hat{\mathbf{S}}) \To{}H^{2}(W_{F}/I_{F}^{r},\hat{\mathbf{S}}_{1})
$$
Since
  $H^{2}(\Gamma_{E/F},\hat{\mathbf{S}}_{1})=\{1\}$ (because $\mathbf{S}_{1}$ is an induced
  torus also for $E/F$), the image of $\bar\eta$ in
  $H^{2}(W_{F}/I_{F}^{r},\hat{\mathbf{S}}_{1})$ is trivial. So if we can prove that the
  first map is surjective, we infer that $\bar\eta$ itself is trivial.
But by \cite[Theorem 7.10]{Yu_LLCtori} the local Langlands correspondence identifies
$H^{1}(W_{F}/I_{F}^{r},\hat{\mathbf{S}}_{i})$ with the group of characters of the group
$\mathbf{S}_{i}(F)/\mathbf{S}_{i}(F)_{r}$. Moreover,  by \cite[Lemma 3.1.3]{Kaletha} the dual embedding
$\mathbf{S}_{2}\injo \mathbf{S}_{1}$ satisfies
$\mathbf{S}_{2}(F)_{r}=\mathbf{S}_{2}(F)\cap\mathbf{S}_{1}(F)_{r}$. So this dual embedding
induces an injective map
$\mathbf{S}_{2}(F)/\mathbf{S}_{2}(F)_{r}\injo\mathbf{S}_{1}(F)/\mathbf{S}_{1}(F)_{r}$ which
shows the  surjectivity of the map $H^{1}(W_{F}/I_{F}^{r},\hat{\mathbf{S}}_{1})\To{}H^{1}(W_{F}/I_{F}^{r},\hat{\mathbf{S}}_{2})$.
\end{proof}

\alin{Characters} \label{def_characters}
We now fix an $F$-rational embedding $\iota\in I_{\phi}$ and we take up the notation
$\vec{\mathbf{G}}_{\iota}$ of \ref{deftwistedLeviseq}.
From now on, we will make the following additional hypothesis:
\begin{equation}\label{H2}
\tag{\em H2}  \emph{ $p$ does not divide
    $|\pi_{0}(Z(\hat{\mathbf{G}}))|=|\pi_{1}({\mathbf{G}}_{\rm der})|$.} 
\end{equation}
Note that \eqref{H1} and \eqref{H2} together mean that $p$ is neither a torsion prime
of $\mathbf{G}$, nor of $\hat{\mathbf{G}}$.
Thanks to this hypothesis, Lemma \ref{splittings_ram} ensures the existence of a
$1$-cocycle $\hat\varphi_{i}:\, W_{F}\To{}  Z(\hat{\mathbf{G}}_{\phi,r_{i-1}+})$ that extends
$\hat\phi|_{I_{F}^{r_{i-1}+}}$  for each $i=0,\hdots, d$.
Using Lemma \ref{p-power_order_cocycle}, we may assume that
$\hat\varphi_{i}$ has $p$-power order in $H^{1}(W_{F},Z(\hat{\mathbf{G}}_{\phi,r_{i-1}+}))$.
Then, since ${\mathbf{G}}_{\iota}^{i}$ is an inner form of ${\mathbf{G}}_{\phi,r_{i-1}+}$, 
the Langlands correspondence for characters  \cite[10.2]{BorelCorvallis} associates to
$\hat\varphi_{i}$ a character\index[notation]{phicheckvar@$\check\varphi_{i}$}
\[\check\varphi_{i}:\,{\mathbf{G}}_{\iota}^{i}(F)\To{}\mu_{p^{\infty}}\subset \CM^{\times},\]
which we may view as an $R$-valued character for any commutative
$\ZM[\mu_{p^{\infty}},\frac 1p]$-algebra $R$. 
  Lemma \ref{depth_char} has the following consequences, for every
$x\in \BT({\mathbf{G}}_{\iota}^{i},F)$ :
\begin{enumerate}
\item the restriction 
$(\check\varphi_{i})|_{{\mathbf{G}}^{i}_{\iota}(F)_{x,r_{i-1}+}}$
only depends on $\hat\phi|_{I_{F}^{r_{i-1}+}}$, and not on the choice of $\hat\varphi_{i}$,
\item for all $j\geq i$ we have
  $(\check\varphi_{i})|_{{\mathbf{G}}^{i}_{\iota}(F)_{x,r_{j-1}+}}=(\check\varphi_{j})|_{{\mathbf{G}}^{i}_{\iota}(F)_{x,r_{j-1}+}}$,
\item the character\index[notation]{psii@$\psi_{i}$} $\psi_{i}:=\check\varphi_{i}\check\varphi_{i+1}^{-1}$ of ${\mathbf{G}}_{\iota}^{i}(F)$
  is trivial on ${\mathbf{G}}^{i}_{\iota}(F)_{x,r_{i}+}$ (where we set $\check\varphi_{d+1}=1$). 
\end{enumerate}

\alin{The subset $\BT_{\iota}$ of the building} \label{BTiota}
We write $\BT$\index[notation]{B@$\BT$} for the (extended) Bruhat--Tits building $\BT({\mathbf{G}},F)$. If $\mathbf{S}$
is a  maximal $F$-torus of ${\mathbf{G}}$ that splits over some tamely ramified finite field
extension $E$ of $F$, we set 
$\BC(\mathbf{S},F):= \AC({\mathbf{G}},\mathbf{S},E)\cap \BT({\mathbf{G}},F)$,
where $\AC({\mathbf{G}},\mathbf{S},E)$ is the apartment of  $\BT({\mathbf{G}},E)$
associated with $\mathbf{S}$ and the intersection is taken in $\BT({\mathbf{G}},E)$. 
As the notation suggests, this does not depend on the choice of $E$.
Note that it need not be  an apartment of
$\BT$, unless  $\mathbf{S}$ has maximal $F$-split rank.
Now we  associate to $\iota$ the following  subset of $\BT$:\index[notation]{Biota@$\BT_{\iota}$}
$$ \BT_{\iota} := \bigcup_{\mathbf{S}\subset {\mathbf{G}}_{\iota}}
\BC(\mathbf{S},F)$$
where $\mathbf{S}$ runs over tamely ramified maximal $F$-tori of ${\mathbf{G}}_{\iota}$.
The set $\BT_{\iota}$ is also the common image of all
the  embeddings $\BT({\mathbf{G}}_{\iota},F)\injo \BT$ obtained as restriction of a
${\rm Gal}(E/F)$-equivariant admissible embedding (in the sense of \cite[\S14.2]{KP-BTbook})
$\BT({\mathbf{G}}_{\iota},E)\injo \BT(\mathbf{G},E)$ for a Galois, tamely ramified field extension $E$ such
that  $\mathbf{G}_{\iota,E}$ is the Levi component of an $E$-rational parabolic subgroup of $\mathbf{G}_{E}$. From now on when writing ``admissible embedding'' we mean an embedding of the kind just described.
 The set of such embeddings is a torsor under
$X_{*}(\mathbf{S}_{\iota})^{W_{F}}_{\RM}$. 
We could have restricted the above union to maximally $F$-split (tamely ramified maximal)
 $F$-tori of ${\mathbf{G}}_{\iota}$, thanks to \cite[Lemma 2.1]{Yu_tamescusp}. For such a
 maximally split torus, the subset $\BC(\mathbf{S},F)$ is an apartment of
 $\BT({\mathbf{G}}_{\iota}, F)$,  but it is not an apartment of $\BT$ unless
 ${\mathbf{G}}_{\iota}$ is an $F$-Levi subgroup.

\alin{A construction of Yu} \label{Yu_construction}
 Let us fix $x\in \BT_{\iota}$. For each $i=0,\hdots, d$, the intersection
 $$G_{\iota,x,r}^{i}:={\mathbf{G}}_{\iota}^{i}(F)_{x,r}:=G_{x,r}\cap
 {\mathbf{G}}_{\iota}^{i}(F)$$ is the Moy--Prasad 
 group associated to $r$ and the preimage of $x$ by any admissible embedding 
 $\BT({\mathbf{G}}_{\iota}^{i},F)\injo \BT$. Note that $G^{i}_{\iota,x,r}$ normalizes
 $G^{j}_{\iota,x,s}$ whenever $i\leq j$, so that we can define an open subgroup of
 $G_{x,0+}$ by\index[notation]{Kplusplus@$\Kpp_{\iota,x}$} 
\[ \Kpp_{\iota,x} := G^{0}_{\iota,x,0+}G^{1}_{\iota,x, r_{0}+}\cdots G^{d}_{\iota,x,r_{d-1}+}.\] 
By property ii) of \ref{def_characters}, 
there exists a character
$\check\phi^{++}_{\iota,x}$ of $\Kpp_{\iota,x}$ whose restriction to $G^i_{\iota, x, r_{i-1}+}$ agrees with $\check\varphi_{i}|_{G^i_{\iota, x, r_{i-1}+}}$. By property i) of
\ref{def_characters},  $\check\phi^{++}_{\iota,x}$ \emph{only depends on  $\phi$} and not on the
choice of the extensions $\hat\varphi_{i}$. 
Now let us consider the following bigger open subgroup of $G_{x,0+}$: \index[notation]{Kplus@$\Kp_{\iota,x}$} 
\[ \Kp_{\iota,x} := G^{0}_{\iota,x,0+}G^{1}_{\iota,x, (r_{0}/2)+}\cdots G^{d}_{\iota,x,(r_{d-1}/2)+}.\] 
In \cite[\S 4]{Yu_tamescusp}, Yu describes a construction of a
character $\check\phi^{+}_{\iota,x}$ of $\Kp_{\iota,x}$ that extends $\check\phi^{++}_{\iota,x}$, starting from the characters $\psi_{i}$ of \ref{def_characters}. 
To explain this, we first observe that  $\check\phi^{++}_{\iota,x}$ is also the product
$\prod_{i=0}^{d} (\psi^{++}_{i,x})|_{\Kpp_{\iota,x}}$ where
 $\psi^{++}_{i,x}$ denotes the unique character
of the group $G^{i}_{\iota,x,0+}G_{x,r_{i}+}$ that extends both $\psi_{i}|_{G^{i}_{\iota,x,0+}}$
and the trivial character of $G_{x,r_{i}+}$.
Similarly, Yu defines $\check\phi^{+}_{\iota,x}$ as a product \index[notation]{phihatplus@$\check\phi^{+}_{\iota,x}$} 
\[\check\phi^{+}_{\iota,x}:=\prod_{i=0}^{d}(\psi^{+}_{i,x})|_{\Kp_{\iota,x}}\]
where $\psi_{i,x}^{+}$  is a certain character of
$G^{i}_{\iota,x,0+}G_{x,(r_{i}/2)+}$ that extends $\psi^{++}_{i,x}$ as in
\cite[\S4]{Yu_tamescusp} and \cite[Lemma~4.2.1]{FS25} and as we recall below.

Yu's construction of $\psi_{i,x}^{+}$ uses the 
Moy--Prasad filtrations $(\gG_{x,r})_{r\in\tilde\RM}$ on the Lie algebra
$\gG:={\rm Lie}({\mathbf{G}})(F)$. We adopt Yu's notation $G_{x,(r/2)+:\,r+}$ for the quotient
group $G_{x,(r/2)+}/G_{x,r+}$. This group is abelian and Moy and
Prasad have defined an isomorphism 
$\gG_{x,(r/2)+:\,r+}\simto G_{x,(r/2)+:\,r+}$. Now the Lie subalgebra
$\gG^{i}_{\iota}={\rm Lie}({\mathbf{G}}_{\iota}^{i})(F)$ of 
$\gG$ has as complement the sum $\nG^{i}_{\iota}$ of non-zero weight spaces of
$\iota(\mathbf{S}_{\phi, r_{i}})$ acting on $\gG$ through the adjoint representation.  
This induces a decomposition  
$$\gG_{x,(r/2)+:\,r+}=\gG^{i}_{\iota,x,(r/2)+:\,r+}\oplus \nG^{i}_{\iota,x,(r/2)+:\,r+}.$$
Thanks to this decomposition, any character $\psi$ of $G^{i}_{\iota,x,(r/2)+:\,r+}$ can be extended to a
character $\wt\psi$ of $G_{x,(r/2)+:\,r+}$
 by letting it be trivial on
$\nG^{i}_{\iota,x,(r/2)+:\,r+}$. In particular we obtain from
$\psi_{i|G^{i}_{\iota,x,(r_{i}/2)+}}$ 
 a character $\wt\psi_{i}$ of  $G_{x,(r_{i}/2)+}$ which,
 in turn,  can be glued with $\psi_{i|G^{i}_{\iota,x,0+}}$ to
 yield the desired character $\psi_{i,x}^{+}$ of $G^{i}_{\iota,x,0+}G_{x,(r_{i}/2)+}$.
Note that this character depends on the choices of $\hat\varphi_{i}$ and  $\hat\varphi_{i+1}$ and,
 a priori, also the restriction $(\psi_{i,x}^{+})|_{\Kp_{\iota, x}}$ depends on these choices.
However we have the following independence result.

\begin{lemme}\label{lemme_indep}
  The character $\check\phi^{+}_{\iota,x}$ only depends on $\phi$,  $\iota$, $x$, and not on
  the choice of $\hat\varphi_{i}$.
\end{lemme}
\begin{proof}
We first note that the map $\psi\mapsto \wt\psi$ described above is 
obviously multiplicative in $\psi$, and has the following property:
If $\xi$ is a character of $G$ of depth
$\leq r$, then $\widetilde{\xi|_{G^{i}_{\iota,x,(r/2)+}}}=\xi|_{G_{x,(r/2)+}}$. Indeed, $\xi$ is trivial on
root subgroups of $G$, hence $\xi|_{G_{x,(r/2)+}}$ has
to be trivial on $\nG^{i}_{\iota,x,(r/2)+:\,r+}$. More generally, if $\xi_{j}$ is a character of
$G^{j}$ for some $j\geq i$, then $(\wt{\xi_{j|G^{i}_{\iota,x,(r/2)+}}})|_{G^{j}_{\iota,x,(r/2)+}}=
\xi_{j|G^{j}_{\iota,x,(r/2)+}}$ hence also $\wt{\xi_{j|G^{i}_{\iota,x,(r/2)+}}}=\wt{\xi_{j|G^{j}_{\iota,x,(r/2)+}}}$.
So we may unambiguously denote this character by $\wt\xi_{j}$.

Let us now check that the product $\prod_{i=1}^{d} (\psi_{i,x}^{+})|_{\Kp_{\iota,x}}$
is independent of the choices of cocycles $\hat\varphi_{i}$. 
So let $(\hat\varphi'_{i})_{i=0,\hdots,d}$ be another choice of cocycles leading to 
characters $\psi_{i,x}^{'+}$, and write 
$\check\varphi'_{i}=\check\varphi_{i}\xi_{i}$. Then $\xi_{i}$ is a character of $G_{i}$ of depth $\leq
r_{i-1}$, hence $\xi_{i|G^{i}_{\iota,x,(r_{i-1}/2)+}}$ extends to a character $\wt\xi_{i}$ of
$G_{x,(r_{i-1}/2)+}$ according to the procedure described before the lemma.
Then we see that for all $i,j\leq d$ we have
$$ (\psi_{i,x}^{'+})|_{G^{j}_{\iota,x,(r_{j-1}/2)+}} = 
\left\{
  \begin{array}{ll}
(\psi_{i,x}^{+})|_{G^{j}_{\iota,x,(r_{j-1}/2)+}} \cdot 
(\xi_{i}\xi_{i+1}^{-1})|_{G^{j}_{\iota,x,(r_{j-1}/2)+}} & \hbox{if } j\leq i    \\
(\psi_{i,x}^{+})|_{G^{j}_{\iota,x,(r_{j-1}/2)+}} \cdot
(\wt\xi_{i}\wt\xi_{i+1}^{-1})|_{G^{j}_{\iota,x,(r_{j-1}/2)+}} & \hbox{if } j> i       
  \end{array}
\right.
$$
where we agree that $\xi_{d+1}=1$. Taking products we obtain for all $j=0,\hdots, d$
$$ \prod_{i=0}^{d} (\psi_{i,x}^{'+})|_{G^{j}_{x,(r_{j-1}/2)+}}
= \prod_{i=0}^{d} (\psi_{i,x}^{+})|_{G^{j}_{x,(r_{j-1}/2)+}}  \cdot
(\wt\xi_{j})^{-1}|_{G^{j}_{x,(r_{j-1}/2)+}} (\xi_{j})|_{G^{j}_{x,(r_{j-1}/2)+}} = 
\prod_{i=0}^{d} (\psi_{i,x}^{+})|_{G^{j}_{x,(r_{j-1}/2)+}} $$
as desired.
\end{proof}

Since $\Kp_{\iota,x}$ is a pro-$p$-group, the smooth character
$\check\phi^{+}_{\iota,x}$ takes values in the ring
$\ZM[\mu_{p^{\infty}},\frac 1p]$.

\alin{Iwahori decomposition} \label{Iwahori}
The pair $(\Kp_{\iota,x},\check\phi^{+}_{\iota,x})$
admits Iwahori decompositions with respect to certain pairs of opposite parabolic subgroups
adapted to $\iota$ and $x$. More precisely, let $({\mathbf{P}} ,\bar{\mathbf{P}} )$ be a pair of
opposite $F$-rational parabolic subgroups with common Levi component ${\mathbf{M}}$ that satisfies the
following conditions:
\begin{enumerate}
\item $\iota(\mathbf{S}_{\phi})\subseteq {\mathbf{M}}$ (equivalently,
  $Z({\mathbf{M}})^{\circ}\subseteq {\mathbf{G}}_{\iota}$),
\item $x$ belongs to $\BT({\mathbf{M}},F)$  (the union of all apartments of $\BT$
  corresponding  to tori of ${\mathbf{M}}$). 
\end{enumerate}
We denote by ${\mathbf{U}}$, resp., $\bar{\mathbf{U}}$, the unipotent radical of ${\mathbf{P}}$,
resp., $\bar{\mathbf{P}}$. As usual,
we also denote the groups of $F$-points of these algebraic groups by $U$, $\bar U$, etc.

\begin{lem}
Under the above assumptions, the groups  $\Kp_{\iota,x}$ and $\Ku_{\iota,x}$ have the Iwahori decomposition property with respect to $P,\bar
  P$, i.e., the multiplication map
$ (U\cap K^{\bullet}_{\iota,x})\times(M\cap K^{\bullet}_{\iota,x})\times(\bar
U\cap K^{\bullet}_{\iota,x}) \To{} K^{\bullet}_{\iota,x}$ for $K^{\bullet}_{\iota,x} \in \{\Kp_{\iota,x}, \Ku_{\iota,x}\}$
is a bijection. Moreover, we have 
$(\check\phi^{+}_{\iota,x})|_{U\cap\Kp_{\iota,x}}\equiv 1$ and $(\check\phi^{+}_{\iota,x})|_{\bar U\cap\Kp_{\iota,x}} \equiv 1$.
\end{lem}

\begin{proof}
	Let $K^{\bullet}_{\iota,x} \in \{\Kp_{\iota,x}, \Ku_{\iota,x}\}$.
    By the properties of the big cell ${\mathbf{U}}{\mathbf{M}}\bar{\mathbf{U}}$, the first
    claim is equivalent to the equality 
    $K^{\bullet}_{\iota,x}=
    (U\cap K^{\bullet}_{\iota,x})(M\cap K^{\bullet}_{\iota,x})(\bar U\cap K^{\bullet}_{\iota,x})$. 
But for
each $i=0,\hdots, d$, the intersections ${\mathbf{P}} \cap {\mathbf{G}}_{\iota}^{i}$ and
$\bar{\mathbf{P}} \cap {\mathbf{G}}_{\iota}^{i}$ are a pair of opposite parabolic subgroups of
${\mathbf{G}}_{\iota}^{i}$ with intersection ${\mathbf{M}}\cap{\mathbf{G}}_{\iota}^{i}$, 
since $Z({\mathbf{M}})^{\circ}\subseteq {\mathbf{G}}_{\iota}^{i}$.
Then,
because $x$ belongs to $\BT({\mathbf{M}}\cap{\mathbf{G}}_{\iota}^{i}, F)\subseteq
\BT({\mathbf{G}}_{\iota}^{i}, F)$, we know that the Moy--Prasad group
$G^{i}_{\iota,x,r_{i-1}/2+}$ has the Iwahori decomposition
$G^{i}_{\iota,x,(r_{i-1}/2)+}=(U\cap G^{i}_{\iota,x,(r_{i-1}/2)+})(M\cap
G^{i}_{\iota,x,(r_{i-1}/2)+})(\bar U\cap G^{i}_{\iota,x,(r_{i-1}/2)+})$ and similarly for $G^{i}_{\iota,x,r_{i-1}/2}$.
Now, using the fact that for $i<j$, the group $\bar U\cap G^{i}_{\iota,x,(r_{i-1}/2)+}$
normalizes $G^{j}_{\iota,x,(r_{j-1}/2)+}$ and the group $\bar M\cap
G^{i}_{\iota,x,(r_{i-1}/2)+}$ normalizes $U\cap G^{j}_{\iota,x,(r_{j-1}/2)+}$, the desired
decomposition follows inductively for $\Kp_{\iota,x}$, and one proceeds similarly for $\Ku_{\iota,x}$.  

Let us now prove that 
$\check\phi^{+}_{\iota,x}$ is trivial on $U\cap \Kp_{\iota,x}$ and on $\bar U\cap \Kp_{\iota,x}$. Going back to the construction of $\check\phi^{+}_{\iota,x}$ it suffices to show that 
for each $i=0,\hdots, d$, the character $\psi_{i,x}^{+}$ of Paragraph
\ref{Yu_construction} is trivial on the group $U\cap (G_{\iota,x,0+}^{i}G_{x,(r_{i}/2)+})$
(and similarly with $\o U$). By the same argument as above, this group is $(U\cap
G_{\iota,x,0+}^{i})(U\cap G_{x,(r_{i}/2)+})$. On the one hand, the restriction of $\psi_{i,x}^{+}$ to
$U\cap G_{\iota,x,0+}^{i}$ is trivial since it is also the restriction of the character
$\psi_{i}$ of $G_{\iota}^{i}$ and $U\cap G_{\iota}^{i}$ is contained in the derived
subgroup of $G_{\iota}^{i}$. On the other hand, the restriction of $\psi_{i,x}^{+}$ to
$G_{x,(r_{i}/2)+}$ is the character denoted by $\wt\psi_{i}$ in \ref{Yu_construction},
which extends $(\psi_{i})|_{G_{\iota,x,(r_{i}/2)+}^{i}}$ according to the
decomposition $\gG_{x,(r_{i}/2)+:r_{i}+}=\gG^{i}_{\iota, x,(r_{i}/2)+:r_{i}+}\oplus 
\nG^{i}_{\iota,x,(r_{i}/2)+:r_{i}+}$ and via the Moy--Prasad isomorphism.
Recall that the latter decomposition is induced by $\gG=\gG_{\iota}^{i}\oplus\nG_{\iota}^{i}$ where
$\gG^{i}_{\iota}$, resp., $\nG^{i}_{\iota}$, is the trivial eigenspace, resp., 
the sum of all non-trivial eigenspaces, of 
$\iota(\mathbf{S}_{\phi,r_{i}})$ acting on $\gG$. 
On the other hand, the image  $\uG_{x,(r_{i}/2)+:r_{i}+}$ of $U\cap G_{x,(r_{i}/2)+}$ in
$\gG_{x,(r_{i}/2)+:r_{i}+}$ is induced by the Lie algebra $\uG$ of ${\mathbf{U}}$ which is a sum of
(non-trivial) eigenspaces for $Z({\mathbf{M}})^{\circ}$ acting on $\gG$. Now $\iota(\mathbf{S}_{\phi,r_{i}})$ and $Z({\mathbf{M}})^{\circ}$ commute since ${\mathbf{M}}$ contains $\iota(\mathbf{S}_{\phi,r_{i}})$. Therefore, $\uG$ is
stable under $\iota(\mathbf{S}_{\phi,r_{i}})$, and we have $\uG=(\uG\cap\gG^{i}_{\iota})\oplus (\uG\cap\nG^{i}_{\iota})$.
Correspondingly,
$\uG_{x,(r_{i}/2)+:r_{i}+}$ decomposes as the direct sum of 
$\uG_{x,(r_{i}/2)+:r_{i}+}\cap\nG^{i}_{\iota,x,(r_{i}/2)+:r_{i}+}$ and
$\uG_{x,(r_{i}/2)+:r_{i}+}\cap\gG^{i}_{\iota,x,(r_{i}/2)+:r_{i}+}$. We have already seen that
$\psi_{i}$ is trivial on the latter intersection, and by definition $\wt\psi_{i}$ is
trivial on the former one, which finishes the proof.
\end{proof}

\alin{A source of pairs of opposite  parabolic subgroups} \label{opposite_parabolic_sgps}
For later reference, we describe here a useful source of pairs $({\mathbf{P}} ,\bar{\mathbf{P}} )$ that
satisfy conditions i) and ii) of \ref{Iwahori}.

Start with two points $x\neq x'$ in $\BC_{\iota}$, and pick a tamely ramified maximal torus
$\mathbf{S}\subseteq {\mathbf{G}}_{\iota}$ such that $x,x'\in 
  \BC(\mathbf{S},F)$. Let $E$ be a tame Galois extension of $F$ that splits
  $\mathbf{S}$. The affine space $\AC({\mathbf{G}},\mathbf{S},E)$ is principal homogeneous under
  the vector space $X_{*}(\mathbf{S})_{\RM}$, so we can write $x'=x+\lambda$ for a unique
  $\lambda\in  X_{*}(\mathbf{S})_{\RM}$. Moreover, the action of
  $X_{*}(\mathbf{S})_{\RM}$ on $\AC({\mathbf{G}},\mathbf{S},E)$ is compatible with the
  respective Galois actions of $\Gamma_{E/F}$. Since both
  $x$ and $x'$ are $\Gamma_{E/F}$-fixed, we have $\lambda\in
  (X_{*}(\mathbf{S})_{\RM})^{\Gamma_{E/F}}$.
To $\lambda$ is associated a pair of $F$-rational parabolic
subgroups  $({\mathbf{P}} ,\bar{\mathbf{P}} ):=({\mathbf{P}} _{\lambda}, {\mathbf{P}} _{-\lambda})$
with common Levi component 
${\mathbf{M}}:={\mathbf{M}}_{\lambda}={\mathbf{M}}_{-\lambda}$ and respective unipotent radicals
denoted by 
${\mathbf{U}}:={\mathbf{U}}_{\lambda}$ and $\bar{\mathbf{U}}:={\mathbf{U}}_{-\lambda}$
(see also   \cite[\S 1.27]{Vigsheaves}).
By construction $\mathbf{M}$ contains $\mathbf{S}$, hence also the connected center
$\iota(\mathbf{S}_{\phi})$ of $\mathbf{G}_{\iota}$, and both $x$ and $x'$ belong to $\BC(\mathbf{M},F)$.

\begin{DEf}
We recall that $\HC_{R}(G)$  is the $R$-algebra 
of compactly supported, locally constant $R$-valued distributions on $G$. We denote by
$e_{\iota,x}$\index[notation]{eiotax@$e_{\iota,x}$} the distribution that averages along the character
$\check\phi^{+}_{\iota,x}$ of the pro-$p$, open subgroup $\Kp_{\iota,x}$ of $G$. 
This is an idempotent of $\HC_{R}(G)$ supported on $\Kp_{\iota,x}$.
\end{DEf}

From the
construction of $\Kp_{\iota,x}$ and $e_{\iota,x}$, and in particular the fact that the
latter does not depend on any further choice than $\iota$, $x$ and $\phi$, we see that 
\ini
\begin{equation}
\forall g\in G,\,\,  {^{g}\Kp_{\iota,x}}=\Kp_{{g}\iota,gx}
\hbox{ and } {^{g}e_{\iota,x}}=e_{{g}\iota,gx},\label{Gaction}
\end{equation}
where $g\iota$ is the $g$-conjugate of $\iota$, i.e., we have
$(g\iota)(s):=g\iota(s)g^{-1}$ for all $s\in\mathbf{S}_{\phi}$.

Since the pair
$(\Kp_{\iota,x},e_{\iota,x})$  depends on $\phi$, we may write
$\Kp_{\phi,\iota,x}$\index[notation]{Kplusphiiotax@$\Kp_{\phi,\iota,x}$} and $e_{\phi,\iota,x}$\index[notation]{ephiiotax@$e_{\phi,\iota,x}$} whenever we want to emphasize this dependence.

\subsection{Intertwining}\label{sec:intertwining}
\ali \label{sec:K-def}
We keep the data $\phi$, $\iota \in I_{\phi}$, and $x\in \BT_{\iota}$ of
the previous subsection and we now introduce the following compact, open subgroups:
\index[notation]{Kdagger@$\Ku_{\iota,x}$}\index[notation]{Kcirc@$\Ko_{\iota,x}$}\index[notation]{Kaiotax@$\K_{\iota,x}$} 
\begin{eqnarray*}
 & \Ku_{\iota,x} &:= G^{0}_{\iota,x,0+}G^{1}_{\iota,x, (r_{0}/2)}\cdots G^{d}_{\iota,x,(r_{d-1}/2)}, \\
 & \Ko_{\iota,x} &:= G^{0}_{\iota,x,0}G^{1}_{\iota,x, (r_{0}/2)}\cdots G^{d}_{\iota,x,(r_{d-1}/2)}, \\
 &  \K_{\iota,x} &:= G^{0}_{\iota,x}G^{1}_{\iota,x, (r_{0}/2)}\cdots G^{d}_{\iota,x,(r_{d-1}/2)}.
\end{eqnarray*}
So we have inclusions $\Kp_{\iota,x}\subseteq\Ku_{\iota,x}\subseteq\Ko_{\iota,x}\subseteq\K_{\iota,x}$ and
all subgroups are normal in $\K_{\iota,x}$. Moreover, $\Ku_{\iota,x}$ is the
pro-$p$-radical of $\Ko_{\iota,x}$. 
We will write $K^{?}_{\phi,\iota,x}$ whenever we want to emphasize
the dependence on $\phi$. 
The aim of this subsection is to prove the following result.

\begin{prop} \label{prop_intertwining}
i) The group $\K_{\phi,\iota,x}$ centralizes $e_{\phi,\iota, x}$.
 
ii) If $(\phi',\iota',x')$ is another triple of the same nature,
    then
\ini
\begin{equation}
 e_{\phi,\iota,x}e_{\phi',\iota',x'}\neq 0 \Rightarrow
 \left(\phi\simeq\phi' \hbox{ and } \Ku_{\phi,\iota,x}\cdot \iota\cap
   \Ku_{\phi',\iota',x'}\cdot \iota'\neq\emptyset\right).\label{intertw}
\end{equation}
\end{prop}
On the left hand side of (\ref{intertw}),  the product
$e_{\phi,\iota,x}e_{\phi',\iota',x'}$ takes place in the $R$-algebra 
$\HC_{R}(G)$ of compactly supported, locally constant $R$-valued distributions on $G$. Therefore, 
\ini\begin{equation}
  \label{eq:intertw_in_terms_of_idemp}
  e_{\phi,\iota,x}e_{\phi',\iota',x'}\neq 0 \Leftrightarrow
  (\check\phi^{+}_{\iota,x})|_{\Kp_{\phi,\iota,x}\cap \Kp_{\phi',\iota', x'}} 
  =(\check{\phi}^{'+}_{\iota',x'})|_{\Kp_{\phi,\iota,x}\cap \Kp_{\phi',\iota',x'}}.
\end{equation}
On the right hand side of (\ref{intertw}), $\Ku_{\phi,\iota,x}\cdot \iota$ denotes the
$\Ku_{\phi,\iota,x}$-orbit 
of the embedding $\iota$ inside $I_{\phi}$, which is also the $\K_{\phi,\iota,x}$-orbit since
$G_{\iota}$ centralizes $\iota$.
If we are in the situation that
 $\phi\simeq\phi'$, then we have  $\mathbf{S}_{\phi}=\mathbf{S}_{\phi'}$ and
 $I_{\phi}=I_{\phi'}$, and the intersection $\Ku_{\phi,\iota,x}\cdot \iota\cap
 \Ku_{\phi',\iota',x'}\cdot \iota'$ is taken as subsets of $I_{\phi}$.

The outline of the proof is the following : thanks to a lemma of Kaletha we prove that the characters
$\psi_{i}$ of \ref{def_characters} iii) are generic of depth $r_{i}$, in the sense of Yu in
\cite[\S9]{Yu_tamescusp} and \cite[Definition~3.8]{Fintzen-IHES}. This genericity condition is precisely what allows to control the
intertwining  as in ii).

\alin{Strata and intertwining}\label{paragraph-strata} We denote by $\gG^{*}$ the dual of the Lie
algebra $\gG$. In order to simplify the notation we merely write
$\gG$ for $\gG(F)$ and $\gG^{*}$ for $\gG^{*}(F)$ if there is no
ambiguity.  If $\LC$ is any lattice in $\gG$, we put
$\LC^{\bullet}=\{f\in \gG^{*}, \langle f,\LC\rangle\subset (F)_{0+}\}$.
Then, following Moy and Prasad, we write $\gG^{*}_{x,-r}:=(\gG_{x,r+})^{\bullet}$.

Let us fix a  character $\Psi : F\To{} R^{\times}$ of depth $0$, and
recall Adler's version of the Moy--Prasad isomorphism
$\varphi_{x,r+}:\,\gG_{x,(r/2)+}/\gG_{x,r+}\simto G_{x,(r/2)+}/G_{x,r+}$ from \cite[1.6.6]{Adler} for $r\in \bR_{\geq 0}$.
 Any group $J$ between $G_{x,r+}$ and $G_{x,(r/2)+}$
 corresponds to a lattice $\jG$ between $\gG_{x,r+}$ and $\gG_{x,(r/2)+}$.
A character
 $\psi$ of $J$ is said to be \emph{realized} by an element $X\in
 \gG^{*}_{x,-r}$ if we have $\psi(h) = \Psi(\langle X,
 \varphi_{x,r+}^{-1}(h)\rangle)$ for all $h\in J$. Such an $X$ is not
 uniquely determined by $\psi$, but the stratum $X+\jG^{\bullet}$
 is. The following result is certainly well known to the specialists, but we could not
 find a reference in this generality.

 \begin{lem}  \label{fact_intertwining}
   Let $x,r,J,\psi$ be as above, and let $x',r',J',\psi'$ be another
   tuple of the same nature. Suppose that $\psi$ is realized by some
   $X\in\gG^{*}_{x,-r}$ and that $\psi'$ is realized by some $X'\in\gG^{*}_{x',-r'}$.
Then we have $\psi|_{J\cap J'}=\psi'|_{J\cap J'}$ if and only if
$(X+\jG^{\bullet})\cap (X'+\jG^{'\bullet})\neq \emptyset$.
 \end{lem}
 \begin{proof}
We first claim that there exist two ``mock exponential maps'' $\varphi_x:\fg_{x,0+} \To{} G_{x,0+}$
and $\varphi_{x'}: \fg_{x',0+} \To{} G_{x',0+}$ in the sense of \cite[\S1.5]{Adler} such that their
restriction to $\fg_{x,0+}  \cap \fg_{x',0+}$ agree.  

This can be achieved as follows. 
We choose an
 apartment that contains both $x$ and $x'$ and denote by $\mathbf{S}$ the corresponding
 maximal $F$-split torus. As in \ref{opposite_parabolic_sgps}, the points $x$ and $x'$ provide us with a Levi subgroup $\mathbf{M} \subset \mathbf{G}$ and opposite parabolic subgroups $\mathbf{P}=\mathbf{M}\mathbf{U}$ and $\bar{\mathbf{P}}=\mathbf{M}\bar{\mathbf{U}}$ such that 
$G_{x,0+}=(G_{x,0+}\cap \bar U)(G_{x,0+}\cap M)(G_{x,0+}\cap U)$, $G_{x',0+}=(G_{x',0+}\cap \bar U)(G_{x',0+}\cap M)(G_{x',0+}\cap U)$, $G_{x,0+}\cap M=M_{x,0+}=M_{x',0+}=G_{x',0+}\cap M$, and the analogous equations for the Lie algebra. Further we have $U=\prod_{a \in \Sigma_{x,x'}(\mathbf{S})} U_a$ and $\bar U=\prod_{a \in \Sigma_{x,x'}(\mathbf{S})} U_{-a}$ as topological spaces and $\fu=\oplus_{a \in \Sigma_{x,x'}(\mathbf{S})} \fu_a$ and  and $\bar \fu=\oplus_{a \in \Sigma_{x,x'}(\mathbf{S})} \fu_{-a}$, where $\Sigma_{x,x'}(\mathbf{S})$ denotes the subset of the  non-multipliable relative roots with respect to $\mathbf{S}$ that occur as characters of $\mathbf{S}$ acting on $\fu$, and where $\fu_a$ denotes the sum of the $a$ and $2a$-eigenspaces, and $U_a$ denotes the corresponding root subgroup with Lie algebra $\fu_a$. 

Let $a \in \pm \Sigma_{x,x'}(\mathbf{S})$.
Then the Moy--Prasad filtration submodules of $\fu_a$ at $x$ and $x'$ agree up to a shift in the depth-parameterization, and likewise the sets of subgroups $\{U_{a,x,r} \cap U_{a,x', 0+}\}_{r \in \bR_{r>0}}$ and $\{U_{a,x',r} \cap U_{a,x, 0+}\}_{r \in \bR_{r>0}}$ agree. Following \cite[\S1.3]{Adler} and using these filtrations, we can now define homeomorphisms $\varphi_{x,a}: \fu_{a,x,0+} \To{} U_{a,x,0+}$ and $\varphi_{x',a}: \fu_{a,x',0+} \To{} U_{a,x',0+}$ that agree with the restrictions of $\varphi_{x,r+}$ and $\varphi_{x',r+}$, respectively, for all $r \in \bR_{\geq 0}$. By using the same coset representatives in the construction of \textit{loc.\ cit.} for $\varphi_{x,a}$ and $\varphi_{x',a}$ for all the cosets contained in $\fu_{a,x,0+} \cap \fu_{a,x',0+}$ and their images, we can ensure that $\varphi_{x,a}$ and $\varphi_{x',a}$ agree on $\fu_{a,x,0+} \cap \fu_{a,x',0+}$. Let $\varphi_{M, x}: \fm_{x,0+}=\fm_{x', 0+} \to{} M_{x,0+}=M_{x',0+}$ be a mock exponential map as provided by \cite[\S1.5]{Adler}. Then using the decompositions 
$\fg_{x,0+}=(\oplus_{a \in \Sigma_{x,x'}(\mathbf{S})} \fu_{-a, x, 0+}) \oplus \fm_{x,0+}\oplus (\oplus_{a \in \Sigma_{x,x'}(\mathbf{S})} \fu_{a, x, 0+})$ and	
$G_{x,0+}=(\prod_{a \in \Sigma_{x,x'}(\mathbf{S})} U_{-a, x, 0+})M_{x,0+}(\prod_{a \in \Sigma_{x,x'}(\mathbf{S})} U_{a, x, 0+})$ and the product of the above maps on each respective summand, we obtain a mock exponential map $\varphi_x: \fg_{x,0+} \to{} G_{x,0+}$  as in \cite[Remark~1.3.3~and~\S1.5]{Adler}. Similarly we obtain the mock exponential map $\varphi_{x'}: \fg_{x',0+} \to{} G_{x',0+}$.
By construction these two mock exponential maps agree on $\fg_{x,0+}\cap \fg_{x',0+}$.

Now, by definition,
 $\varphi_{x,r+}$ and $\varphi_{x',r'+}$ are induced by restriction from
 $\varphi_{x}$ and $\varphi_{x'}$, respectively. 
It follows that $\varphi_{x}^{-1}(J)=\jG$ and,
since $\Psi(\langle X,\gG_{x,r+}\rangle)={1}$,  that 
$\psi(j)=\Psi(\langle X,\varphi_{x}^{-1}(j)\rangle$ for all $j\in J$. Similarly
$\varphi_{x'}^{-1}(J')=\jG'$ and 
$\psi(j')=\Psi(\langle X',\varphi_{x'}^{-1}(j')\rangle$ for all $j'\in J'$, 
and $\psi(j')=\Psi(\langle X',\varphi_{x}^{-1}(j')\rangle$ for all $j'\in J' \cap J$.
Therefore, we have $\psi|_{J\cap J'}=\psi'|_{J\cap J'}$ if and only if 
$\Psi(\langle X,Y\rangle)=\Psi(\langle X',Y\rangle)$ for all $Y\in \jG\cap \jG'$. This is
equivalent to $X-X'\in (\jG\cap\jG')^{\bullet}=\jG^{\bullet}+\jG^{'\bullet}$, which in turn is 
equivalent to $(X+\jG^{\bullet})\cap(X'+\jG^{'\bullet})\neq\emptyset$.
 \end{proof}

\alin{Generic elements}
	\inisub
\begin{subequations}
 Let $\mathbf{L}$ be a tamely ramified  twisted Levi subgroup of
${\mathbf{G}}$. We may identify $\lG^{*}={\rm Lie}^{*}(\mathbf{L})$ with the weight-$0$ subspace
of $\gG^{*}$ for the coadjoint action of the connected center of
$\mathbf{L}$, 
and we write $(\fl^*)^{\mathbf{L}}$ for the invariants under the coadjoint action of $\mathbf{L}$. We recall from \cite[Lemma~2.3]{Fi-mod-ell} that $(\fl^*)^{\mathbf{L}}\cap (\fl^*)_{x,r}=(\fl^*)^{\mathbf{L}}\cap (\fl^*)_{x',r}$ for $x, x' \in \BT(\mathbf{L},F)$ and $r \in \wt{\mathbb{R}}$. The proof of \cite[Lemma~2.3]{Fi-mod-ell} also shows that
\begin{equation}\label{eqn-MP-intersection}
(\fl^*)^{\mathbf{L}}\cap ((\fl^*)_{x,r} + (\fl^*)_{x',r})=(\fl^*)^{\mathbf{L}}\cap (\fl^*)_{x,r} .
\end{equation}

An element $X\in (\fl^*)^{\mathbf{L}} \subseteq \fg^*$ is called
${\mathbf{G}}$-\emph{generic\footnote{Here we only recall condition GE1
  of \cite[\S 8]{Yu_tamescusp} in the form stated in \cite[Definition~3.8]{Fintzen-IHES} since under our hypotheses \eqref{H1} and \eqref{H2},
  \cite[Lemma 8.1]{Yu_tamescusp} shows that condition GE2 is implied by GE1.} 
 of depth $-r$}
if $X\in (\fl^*)_{x,-r} \smallsetminus (\fl^*)_{x,-r+}$ for some (equivalently, every) point $x \in \BT(\mathbf{H}, F)$ and if for some (equivalently, every)  maximal torus $\mathbf{S}\subseteq
\mathbf{L}$ and every
$\alpha\in\Sigma(\mathbf{S},{\mathbf{G}})\setminus\Sigma(\mathbf{S},\mathbf{L})$ we have 
$v(\langle X, H_{\alpha}\rangle)=-r$, where $E$ is a 
splitting field of $\mathbf{S}$, the valuation $v$ extends that of $F$, and
$H_{\alpha}:=\mathrm{Lie}(\alpha^\vee)(1)\in \sG(E)$. Note that we do not exclude
the case $\mathbf{L}={\mathbf{G}}$, where only the first condition,
$X\in (\fl^*)_{x,-r} \smallsetminus (\fl^*)_{x,-r+}$, is non-empty.

For $x\in \BT(\mathbf{L},F)$, we define $\jG_{x,r}:=\lG_{x,r}\oplus\nG_{x,r/2}$ and 
$\jG_{x,r}^{+}:=\lG_{x,r}\oplus\nG_{x,(r/2)+}$. Here $\nG$ denotes the sum of the  non-invariant
eigenspaces in $\gG$ under the adjoint action of the center of $\mathbf{L}$.
These lattices correspond to subgroups
$J_{x,r}$ and $J_{x,r}^{+}$ between $G_{x,r}$ and $G_{x,r/2}$. With Yu's notation in
\cite{Yu_tamescusp} we would write $J_{x,r}= (L,G)_{x,(r,r/2)}$ and $J_{x,r}^{+}=(L,G)_{x,(r,r/2+)}$.
We let $X \in (\fl^*)^{\mathbf{L}} \subseteq \fg^*$ be $\mathbf{G}$-generic of depth $-r$.
As in \ref{paragraph-strata}, the element $X$ defines a character $\psi_{x}$ of $J_{x,r}^{+}$ that is trivial
on $G_{x,r+}$. Moreover this character
is centralized by $J_{x,r}$ since $[J_{x,r},J_{x,r}^{+}]\subset G_{x,r+}$.
The following lemma follows from  an adaptation to our setting of Yu's arguments in \cite[\S 8]{Yu_tamescusp}.

\begin{lem}\label{lemme_generic_intertwining}
  Let $(\mathbf{L},X,r,x)$ be as above and let $(\mathbf{L}',X',r',x')$ be another tuple of
  the same nature (so in particular $X'$ is $\mathbf{G}$-generic of depth $-r'$).
 If $(X+(\jG_{x,r}^{+})^{\bullet})\cap
  (X'+(\jG_{x',r'}^{+})^{\bullet})\neq \emptyset$, then $r=r'$ and there are $g\in J_{x,r}$ and
  $g'\in J_{x',r'}$ such that $^{g}\mathbf{L}={^{g'}\mathbf{L}'}$ and $^{g}X-{^{g'}X'}\in
  ({}^g\fl)^*_{x'',-r+}\cap ({}^g\fl^*)^{{}^g\mathbf{L}}$ for every $x'' \in \BT({{}^g\mathbf{L}}, F)$. 
\end{lem}
\begin{proof}
Let us first show that $r=r'$. Indeed, since
$(\jG_{x,r}^{+})^{\bullet}\subset \gG^{*}_{x,-r+}$, the strata
$X+\gG^{*}_{x,-r+}$ and $X'+\gG^{*}_{x',-r'+}$ have a non-empty intersection. But $-r$ is the
\emph{depth} $d(X)$ of the non-nilpotent element $X$ in the sense of 
\cite[\S3.3]{Adler_DeBacker}, and by
\cite[Lemma~3.3.7]{Adler_DeBacker}, this depth function is constant on the coset $X+\gG^{*}_{x,-r+}$.
Similarly,  the depth function is constant equal to  $-r'$ on the coset
$X'+\gG^{*}_{x',-r'+}$. Hence we conclude that  $r=r'$.

Suppose first that $\mathbf{L}=\mathbf{L'}={\mathbf{G}}$. Then using \eqref{eqn-MP-intersection} (and \cite[Lemma~2.3]{Fi-mod-ell}) we have $X-X'\in (\gG^{*}_{x,-r+} + \gG^{*}_{x',-r+}) \cap (\fg^{*})^{\mathbf{G}}=\fg^{*}_{x,-r+} \cap (\fg^{*})^{\mathbf{G}}=\fg^{*}_{x'',-r+} \cap (\fg^{*})^{\mathbf{G}}$.

Let us return to the general case.
 The  intersection $(X+(\jG_{x,r}^{+})^{\bullet})\cap  (X'+(\jG_{x',r}^{+})^{\bullet})$
 is open in $\gG^{*}$. Since it is assumed to be non-empty, it contains a regular semi-simple
 element $Y$ so that the centralizer $\mathbf{S}$ of $Y$ is a maximal torus of ${\mathbf{G}}$.
Now Lemma 8.6 of \cite{Yu_tamescusp} provides us with an element $g\in G_{x,r/2}$ such that 
$^{g^{-1}}Y\in X+\lG^{*}_{x,-r+}$. It follows that $\mathbf{S}\subset
{^{g}\mathbf{L}}$ and that
we have $v(\langle Y, H_{\alpha}\rangle)\geq -r$ for $\alpha\in
\Sigma(\mathbf{S},{\mathbf{G}})$, with equality if and only if
$\alpha\notin \Sigma(\mathbf{S},{^{g}\mathbf{L}})$.
Similarly, there is an element $g'\in G_{x',r/2}$ such that 
$^{g^{'-1}}Y\in X'+\lG^{'*}_{x',-r+}$ and  it follows that $\mathbf{S}\subset
{^{g'}\mathbf{L'}}$ and that we have 
$v(\langle Y, H_{\alpha}\rangle)= -r$ if and only if
$\alpha\notin \Sigma(\mathbf{S},{^{g'}\mathbf{L'}})$.

We thus obtain the equality
$\Sigma(\mathbf{S},{^{g}\mathbf{L}})=\Sigma(\mathbf{S},{^{g'}\mathbf{L'}})$. This
implies $^{g}\mathbf{L}={^{g'}\mathbf{L}'}$, and also 
$(^{g}X+{^{g}\lG^{*}}_{x,-r+})\cap (^{g'}X'+{^{g}\lG^{*}}_{x',-r+})\neq \emptyset$.
This situation is similar to the case $\mathbf{L}=\mathbf{L'}={\mathbf{G}}$ treated above, hence we
conclude that $^{g}X-{^{g'}X'}\in
({}^g\fl)^*_{x'',-r+}\cap ({}^g\fl^*)^{{}^g\mathbf{L}}$.
\end{proof}

We resume the setting associated to $(\phi,\iota)$. Taking invariants under the coadjoint action of
 $\iota(\mathbf{S}_{\phi, r_{i}})$ we obtain
an increasing sequence of subspaces $\gG^{0*}_{\iota}\subset\hdots \subset\gG^{d*}_{\iota}=\gG^{*}$ whose
Moy--Prasad filtrations 
are induced from those of $\gG^{*}$.
Finally, recall the character $\psi_{i}$ of ${\mathbf{G}}_{\iota}^{i}(F)$ defined in
\ref{def_characters} for each $i=0,\hdots, d-1$.
\end{subequations}

\begin{lemme}\label{generic}
There is a ${\mathbf{G}}_{\iota}^{i+1}$-generic element $X_{i}\in (\fg_\iota^{i*})^{\mathbf{G_\iota^i}}$ of depth $-r_i$ that represents
$\psi_{i}|_{G^{i}_{\iota,x,r_{i}}}$ for all $x\in \BT({\mathbf{G}}_{\iota}^{i},F)$. (In
particular $\psi_{i}$ has depth $r_{i}$).
\end{lemme}
\begin{proof}
Let $\mathbf{S}$ be a maximal $F$-torus of ${\mathbf{G}}_{\iota}^i$ split by some tamely
ramified Galois extension $E$. By Lemma \ref{rootsystemCr}, for all roots
$\alpha$ in
$\Sigma(\mathbf{S},{\mathbf{G}}_{\iota}^{i})\setminus\Sigma(\mathbf{S},{\mathbf{G}}_{\iota}^{i})$
we have
$\check\varphi_{i}(N_{E|F}(\alpha^{\vee}(E^{\times}_{r_{i}})))\neq\{1\}$.
On the other hand, we also have
$\check\varphi_{i+1}(N_{E|F}(\alpha^{\vee}(E^{\times}_{r_{i}})))=\{1\}$ for such roots $\alpha$, since
$N_{E|F}(\alpha^{\vee}(E^{\times}))\subset
({\mathbf{G}}_{\iota}^{i+1})_{\rm sc}(F)$. Since every maximally split torus in ${\mathbf{G}}_{\iota}^i$ is contained in a tame maximal torus, we deduce that $\psi_i=\check\varphi_i\check\varphi_{i+1}^{-1}$ has depth $\geq r_i$.
Combined with \ref{def_characters}iii) it follows that the character
$\psi_{i}$ of ${\mathbf{G}}_{\iota}^{i}(F)$ has depth exactly $r_{i}$  (at every point $x \in \BT(\mathbf{G}_\iota^{i}, F)$)
and the hypothesis of \cite[Lemma 3.6.8]{Kaletha} are satisfied.
This lemma asserts  that 
for every $x \in \BT(\mathbf{G}_\iota^{i}, F)$, there exists some
${\mathbf{G}}^{i+1}_{\iota}$-generic $X_{i}(x)$ of depth $-r_i$  that represents $\psi_{i}|_{G^{i}_{\iota, x, r_{i}}}$. 
However, it follows, for example, 
by the proof of \cite[Lemma~3.3.1]{AFMO2} (cf.\ also  \cite[Lemma 2.51]{Hakim_Murnaghan}),
that we may choose $X_{i}$ uniformly for all $x$.
 \end{proof}

\alin{Some auxiliary groups}
 \label{auxiliary_groups}
For $i=1,\hdots, d$ and $x\in \BT_{\iota}$,
let us introduce the open, compact subgroup $J_{\iota,x}^{i}$ of $G_{\iota}^{i}$ 
that lies in between $G^{i}_{\iota,x,r_{i-1}}$ and
  $G^{i}_{\iota,x,r_{i-1}/2}$ and corresponds to the lattice
  $$\jG^{i}_{\iota,x}=\gG^{i-1}_{\iota,x,r_{i-1}} \oplus (\nG^{i-1}_{\iota}\cap \gG^{i}_{\iota})_{x,r_{i-1}/2}$$ of
  $\gG^{i}_{\iota}$ through the   Moy--Prasad isomorphism. 
This is the group 
$J^{i}_{\iota,x}=(G^{i-1}_{\iota},G^{i}_{\iota})_{x,r_{i-1},r_{i-1}/2}$ in Yu's notation.
Similarly we define $J_{\iota,x}^{i+}$ by  replacing $r_{i-1}/2$ by $r_{i-1}/2+$.

Recall the character $\psi_{i,x}^{+}$ of the group $G^{i}_{\iota,x,0+}G_{x,r_{i}/2+}$ defined in
\ref{Yu_construction}. 
By construction, the restriction $(\psi_{i-1,x}^{+})|_{J^{i+}_{\iota,x}}$ is
represented by the element $X_{i-1}$ provided by  Lemma \ref{generic}.

\alin{Proof of Proposition \ref{prop_intertwining}}
 i) We have $\K_{\iota, x}= G_{\iota,x}^{0} \prod_{i}J^{i}_{\iota,x}$, and since $G_{\iota,x}^{0}$ centralizes $e_{\phi,\iota,x}$ by  \eqref{Gaction},
it is enough to prove that for each $i$
  and $j$, the group $J^{i}_{\iota,x}$ centralizes the character $\psi_{j,x}^{+}$. 
When $j\neq i-1$, this is immediate since $J^{i}_{\iota,x}$ is
  contained in the group $G^{j}_{\iota,x,0+}G_{\iota,x,r_{j}/2+}$ on which $\psi_{j,x}^{+}$
  is defined. When $j=i-1$, this follows from 
  $[J^{i}_{\iota,x},G^{i-1}_{\iota,x,0+}G_{\iota,x,r_{i-1}/2+}]\subset (G_\iota^{i-1}, G_{\iota})_{x,r_{i-1}+,r_{i-1}/2+} \subset \ker
  (\psi_{i-1,x}^{+})$, compare also \cite[Lemma~4.2.]{Yu_tamescusp}.

ii) We will prove (\ref{intertw}) by an inductive argument. 

We start with two triples $(\phi,\iota,x)$ and $(\phi',\iota',x')$ such that
\ini
\begin{equation}\label{hyp}
    (\check\phi_{\iota,x})|_{\Kp_{\phi,\iota,x}\cap \Kp_{\phi',\iota',x'}}
=(\check\phi'_{\iota',x'})|_{\Kp_{\phi,\iota,x}\cap \Kp_{\phi',\iota',x'}}.
\end{equation}
Here we lighten the notation by omitting the exponent $+$ of $\check\phi^{+}_{\iota,x}$. We will decorate with the symbol $'$ all objects pertaining to the triple
$(\phi',\iota',x')$. In particular the jumps of the filtration $\mathbf{S}_{\phi',r}$ are
denoted by $r'_{0},\hdots, r'_{d'-1}$ and $r'_{d'}$ is the depth of the character $\phi'$.

We first reduce to the case where $r_{d-1}=r_{d}$. Indeed, if $r_{d-1}<r_d$, then it
suffices to prove the result after replacing $\phi$ by $\phi\cdot(\varphi_{d}^{-1})|_{P_{F}}$  and 
$\phi'$ by $\phi'\cdot(\varphi_{d}^{-1})|_{P_{F}}$ because this operation does not affect
(\ref{hyp}) and does not change
$\Ku_{\phi,\iota,x}$ nor $\Ku_{\phi',\iota',x'}$.
We are thus left to prove the conclusion of (\ref{intertw}) 
for these new $\phi$ and $\phi'$.
 We now have
$r_{d}=r_{d-1}$ as desired, but $r'_{d'-1}$ and $r'_{d'}$ might be
distinct a priori.

Since $r_{d}=r_{d-1}$, the character  $\check\phi_{\iota,x}$ is trivial on $G_{x,r_{d-1}+}$ and
$(\check\phi_{\iota,x})|_{J_{\phi,\iota,x}^{d+}}=(\psi_{d-1,x}^{+})|_{J^{d+}_{\phi,\iota,x}} $ is
represented by the generic element $X_{d-1}$ of Proposition
\ref{generic}. On the other hand, we have a priori two possibilities
for  $\check\phi'_{\iota',x'}$ :
\begin{itemize}
\item either $r'_{d'}>r'_{d'-1}$ and
  $(\check\phi'_{\iota',x'})|_{G_{x',r'_{d'}}}$ is represented by some
  generic  element $X'_{d'}\in (\fg^*)^{\mathbf{G}}$ of depth $-r'_{d'}$,
\item or $r'_{d'}=r'_{d'-1}$ and
  $(\check\phi'_{\iota',x'})|_{J_{\phi',\iota',x'}^{d'+}}$ is represented by
 $X'_{d'-1}$  as provided by Proposition \ref{generic}. 
\end{itemize}
The first case is actually impossible. Indeed by (\ref{hyp}) the characters
$\check\phi_{\iota,x}$ and $\check\phi'_{\iota',x'}$ coincide on $J^{d+}_{\phi,\iota,x}\cap
G_{x',r'_{d'}}$. So in the setting of the first case, Lemma \ref{fact_intertwining}
and  Lemma \ref{lemme_generic_intertwining} imply that
${\mathbf{G}}_{\iota}^{d-1}\subsetneqq \mathbf{G}$ is conjugate to
${\mathbf{G}}_{\iota'}^{d'}={\mathbf{G}}$, which is absurd.

So we are in the second case, and by (\ref{hyp}) the characters
$\check\phi_{\iota,x}$ and $\check\phi'_{\iota',x'}$ coincide on
$J^{d+}_{\phi,\iota,x}\cap J^{d'+}_{\phi',\iota',x'}$. Then,  Lemma \ref{fact_intertwining}
and  Lemma \ref{lemme_generic_intertwining} tell us that $r_{d-1}=r'_{d'-1}$ and provide
elements  $j\in J^{d}_{\phi,\iota,x}$ and $j'\in J^{d'}_{\phi',\iota',x'}$ such that 
$^{j}{\mathbf{G}}_{\iota}^{d-1}= {^{j'}{\mathbf{G}}_{\iota'}^{d'-1}}$. 
Note that $^{j}{\mathbf{G}}_{\iota}^{d-1}={\mathbf{G}}_{j\iota}^{d-1}$. Since $j\in\Ku_{\phi,\iota,x}$,
statement i) of the proposition and (\ref{Gaction}) show that
$e_{\phi,\iota,x}=e_{\phi,j\iota,x}$. Therefore it is sufficient to prove the
conclusion of (\ref{intertw}) for the triples $(\phi, j\iota,x)$ and $(\phi',
j'\iota',x')$. In other words, we may and will assume that
${\mathbf{G}}^{d-1}_{\iota}={\mathbf{G}}^{d'-1}_{\iota'}$.

\def\vH{{{^H}K}}

Let us put $\mathbf{H}:={\mathbf{G}}^{d-1}_{\iota}={\mathbf{G}}^{d'-1}_{\iota'}$. The wild inertia
parameter $\hat\phi: P_{F}\To{}\hat{\mathbf{G}}$ factors through $\hat{\mathbf{H}}$, giving,
according to Lemma \ref{phi_admissible}, a
wild inertia parameter of $\mathbf{H}$ denoted by $\phi^{|H}$. 
Its associated $F$-torus
$\mathbf{S}_{\phi^{|H}}$ is equal to $\mathbf{S}_{\phi}$ and  the Levi-center 
embedding $\iota: \mathbf{S}_{\phi}\injo
{\mathbf{G}}$ factors through $\mathbf{H}$. Moreover, the associated set
$\BT_{\iota}\subset\BT(\mathbf{H},F)$ is the same as the one considered so far. We thus get
a triple $(\phi^{|H},\iota,x)$ pertaining to $\mathbf{H}$, whence groups $\vH_{\phi,\iota,x}$ and
$\vH^{+}_{\phi,\iota,x}$ and a character $\check\phi_{\iota,x}^{|H}$.
Actually we simply have 
$$ \vH^{+}_{\phi,\iota,x}= G_{\iota,x,0+}^{0}\cdots G_{\iota,x,(r_{d-2}/2)+}^{d-1} \subset 
\Kp_{\phi,\iota,x}
\hbox{ and }  \check\phi_{\iota,x}^{|H} = (\check\phi_{\iota,x})|_{\vH^{+}_{\iota,x}}.$$
Similarly we have a triple $(\phi^{'|H},\iota',x')$ pertaining to $\mathbf{H}$, groups 
$\vH_{\phi',\iota',x'}\subset \K_{\phi',\iota',x'}$ and $\vH^{+}_{\phi',\iota',x'}\subset
\Kp_{\phi',\iota',x'}$, as well as a character $\check\phi^{'|H}_{\iota',x'}$ of
$\vH^{+}_{\phi',\iota',x'}$ that coincides with the restriction of $\check\phi'_{\iota',x'}$.
In particular (\ref{hyp}) implies that the characters $\check\phi^{|H}_{\phi,\iota,x}$ and
$\check\phi^{'|H}_{\phi',\iota',x'}$ coincide on the intersection
$\vH^{+}_{\phi,\iota,x}\cap \vH^{+}_{\phi',\iota',x'}$. 
Suppose now that the conclusion of (\ref{intertw}) is known for the triples
$(\phi^{|H},\iota,x)$ and $(\phi^{'|H},\iota',x')$. It then implies that the same conclusion
holds for the triples $(\phi,\iota,x)$ and $(\phi',\iota',x')$.

It follows that  in order to finish the proof of
(\ref{intertw}), we may  argue by induction, for example on the number
$n(\phi,{\mathbf{G}})=\dim(\mathbf{S}_{\phi})-\dim(Z({\mathbf{G}}))$. It remains however  to
initiate the induction process by considering the case $n(\phi,{\mathbf{G}})=0$. In this
case we have $d=0$, $\Ku_{\phi,\iota,x} = G_{x,0+}$ and $\check\phi_{\iota,x}$ is the restriction
of a character $\check\varphi_{0}$ of $G$. As we have done above, we
may multiply both $\phi$ and $\phi'$ by $(\check\varphi_{0}^{-1})|_{P_{F}}$ so that we may assume now that 
$\phi$ is trivial. Then we need to show that $\phi'$ is trivial too, or equivalently that it
has depth $r'_{d'}=0$. However if $\phi'$ had depth $r'_{d'}>0$, then
$(\check\phi'_{\iota',x'})|_{G_{x',r'_{d'}}}$ would be represented by a generic element
$X'_{d'}$ of depth $-r'_{d'}$. Since the depth function of \cite{Adler_DeBacker} is
constant on the stratum $X'_{d'}+\gG^{*}_{x',-r'_{d'}+}$, the latter cannot intersect the
stratum $\gG^{*}_{x,0+}$,  hence by Lemma \ref{fact_intertwining} we would get a contradiction
with \eqref{hyp}.
\findem

\alin{A variant}  \label{rem_intertw}
 Fix $\phi$ and consider the open subgroup
 $K'_{\iota,x}:=\prod_{i=1}^{d}J^{i+}_{\iota,x}$ of $\Kp_{\iota,x}$.
The following  slight strengthening of \ref{prop_intertwining} ii)
will be useful in Section 3.
 \begin{pro}
If   $(\check\phi_{\iota,x})|_{K'_{\iota,x}\cap K'_{\iota',x'}}
=(\check\phi_{\iota',x'})|_{K'_{\iota,x}\cap K'_{\iota',x'}},$
then  $\Ku_{\iota,x}\iota\cap\Ku_{\iota',x'}\iota'\neq \emptyset$
 \end{pro}
 \begin{proof}
   This is the actual output of the foregoing proof when $\phi=\phi'$.
 \end{proof}

\alin{A converse to Proposition \ref{prop_intertwining} ii)} \label{intertw_converse}
Using \eqref{Gaction}, the following lemma shows that the implication (\ref{intertw}) is actually an
equivalence.

\begin{lem}
  Fix $\phi$, $\iota\in I_{\phi}$ and two points  $x,x'\in \BT_{\iota}$. Then the two
  characters   $\check\phi^{+}_{\iota,x}$ and $\check\phi^{+}_{\iota,x'}$ agree on the
  intersection $\Kp_{\iota,x}\cap \Kp_{\iota,x'}$.
\end{lem}

\begin{proof}
  Only the case $x\neq x'$ is nontrivial. Let $(\mathbf{P},\bar{\mathbf{P}})$ be the pair of opposite
  parabolic subgroups of $\mathbf{G}$ with common Levi subgroup $\mathbf{M}$ as defined in \ref{opposite_parabolic_sgps}, and write
  $\mathbf{P}=\mathbf{M}\mathbf{U}$ 
and $\bar{\mathbf{P}}=\mathbf{M}\bar{\mathbf{U}}$ for their respective Levi decompositions.
It follows from  Lemma
  \ref{Iwahori} that we have an Iwahori  decomposition 
  $$ \Kp_{\iota,x}\cap \Kp_{\iota,x'}
  =(U\cap \Kp_{\iota,x}\cap \Kp_{\iota,x'})
  (M\cap \Kp_{\iota,x}\cap \Kp_{\iota,x'})
  (\bar U\cap \Kp_{\iota,x}\cap \Kp_{\iota,x'})$$
  and that both $\check\phi^{+}_{\iota,x}$ and $\check\phi^{+}_{\iota,x'}$ are trivial on
  $(U\cap \Kp_{\iota,x}\cap \Kp_{\iota,x'})$ and
  $(\bar U\cap \Kp_{\iota,x}\cap \Kp_{\iota,x'})$.  From the proof of that
  lemma, setting ${\mathbf{M}}^{i}:={\mathbf{M}}\cap{\mathbf{G}}_{\iota}^{i}$, we also have
  $$M\cap \Kp_{\iota,x}= M_{x,0+}^{0} M^{1}_{x, (r_{0}/2)+}\cdots M^{d}_{\iota,x,(r_{d-1}/2)+}$$ 
  and the same applies to $x'$. But since $x'$ is a translate of $x$ under
  $X_{*}(Z({\mathbf{M}}))_{\RM}^{\Gamma_{E/F}}$, we have
  $M^{i}_{x,(r_{i-1}/2)+}=M^{i}_{x',(r_{i-1}/2)+}$ for all $i$ and, therefore,
  $M\cap \Kp_{\iota,x}=M\cap \Kp_{\iota,x'}=M\cap \Kp_{\iota,x}\cap
  \Kp_{\iota,x'}$.
  It remains to see that the restrictions of $\check\phi^{+}_{\iota,x}$ and
  $\check\phi^{+}_{\iota,x'}$  to this group are equal. 
  For this, it suffices to check that for all $i=0,\hdots, d$, the restrictions of 
  the characters $\psi_{i,x}^{+}$ and $\psi_{i,x'}^{+}$ of \ref{Yu_construction} to this
  group are equal. The character $\psi_{i,x}^{+}$ is defined on
  $G^{i}_{\iota,x,0+}G_{x,(r_{i}/2)+}$.  
  By the same argument as above, we have
  $M\cap (G_{\iota,x,0+}^{i}G_{x,(r_{i}/2)+})=
  M^{i}_{x,0+}M_{x,(r_{i}/2)+}=M^{i}_{x',0+}M_{x',(r_{i}/2)+}$, and it suffices to see that
  $\psi_{i,x}^{+}$ and $\psi_{i,x'}^{+}$ agree on this group.
On one hand, we have by definition
$(\psi_{i,x}^{+})|_{M^{i}_{x,0+}}=(\psi_{i})|_{M^{i}_{x,0+}}=(\psi_{i,x'}^{+})|_{M^{i}_{x,0+}}$.
   On the other hand, the restriction of $\psi_{i,x}^{+}$ to
$G_{x,(r_{i}/2)+}$ is the character denoted by $\wt\psi_{i}$ in \ref{Yu_construction},
which extends $(\psi_{i})|_{G_{\iota,x,(r_{i}/2)+}^{i}}$ trivially  according to the
decomposition $\gG_{x,(r_{i}/2)+:r_{i}+}=\gG^{i}_{\iota, x,(r_{i}/2)+:r_{i}+}\oplus 
\nG^{i}_{\iota,x,(r_{i}/2)+:r_{i}+}$ and via the Moy--Prasad isomorphism.
Recall that the latter decomposition is induced by
$\gG=\gG_{\iota}^{i}\oplus\nG_{\iota}^{i}$ where 
$\gG^{i}_{\iota}$, resp., $\nG^{i}_{\iota}$, is the trivial eigenspace, resp., 
the sum of all non-trivial eigenspaces, of 
$\iota(\mathbf{S}_{\phi,r_{i}})$ acting on $\gG$.
Since  $\iota(\mathbf{S}_{\phi,r_{i}})$ and $Z({\mathbf{M}})^{\circ}$ commute with each other, taking the weight-$0$ part of the $Z({\mathbf{M}})^{\circ}$-action on the above decomposition provides us with the 
decomposition $\mG=\mG^{i}\oplus(\mG\cap\nG_{\iota}^{i})$, and we see that
$(\psi_{i,x}^{+})|_{M_{x,(r_{i}/2)+}}$ is the character that
extends $(\psi_{i})|_{M_{x,(r_{i}/2)+}^{i}}$ trivially according to the
decomposition $\mG_{x,(r_{i}/2)+:r_{i}+}=\mG^{i}_{x,(r_{i}/2)+:r_{i}+}\oplus 
(\mG\cap \nG^{i})_{\iota,x,(r_{i}/2)+:r_{i}+}$ through the Moy--Prasad isomorphism. The same
description applies to $(\psi_{i,x'}^{+})|_{M_{x,(r_{i}/2)+}}$, finishing the proof.
\end{proof}

\alin{The Heisenberg property} \label{Heisenberg}
 By \cite[Lemma 1.3]{Yu_tamescusp}, we know that the quotient
group $\Ku_{\iota,x}/\Kp_{\iota,x}$ is abelian, hence the derived subgroup
$[\Ku_{\iota,x},\Ku_{\iota,x}]$ is contained in $\Kp_{\iota,x}$ and we have a map
$$  \Ku_{\iota,x}\times \Ku_{\iota,x} \To{}
\mu_{p^{\infty}},\,
(g,h)\mapsto \check\phi^{+}_{\iota,x}(ghg^{-1}h^{-1}) . $$ 
Since $\Ku_{\iota,x}$ centralizes the
character $\check\phi^{+}_{\iota,x}$, this map descends to a
map
$$\theta:\, \Ku_{\iota,x}/\Kp_{\iota,x}\times \Ku_{\iota,x}/\Kp_{\iota,x} \To{}
\mu_{p^{\infty}}.$$

\begin{pro}Assume that $p$ is odd, on top of our running assumptions
  \eqref{H1} and \eqref{H2}.
  \begin{enumerate}
  \item The group $\Ku_{\iota,x}/\Kp_{\iota,x}$ has exponent $p$ and the map
    $\theta$ defines a perfect alternating pairing on this group, taking values in
    $\mu_{p}$.
  \item If $({\mathbf{P}} ,\bar{\mathbf{P}} )$ are opposite parabolic subgroups of
    ${\mathbf{G}}$  satisfying conditions i) and ii) of \ref{Iwahori}, then
    $\theta$  induces a perfect pairing  between  
    $(U\cap \Ku_{\iota,x})/(U\cap \Kp_{\iota,x})$ and
    $(\bar U\cap \Ku_{\iota,x})/(\bar U\cap \Kp_{\iota,x})$,
    and a perfect alternating pairing on
    $(M\cap \Ku_{\iota,x})/(M\cap \Kp_{\iota,x})$.
  \end{enumerate}
\end{pro}
\begin{proof}  i) By \cite[Lemma 1.3]{Yu_tamescusp}, we know that
  $\Ku_{\iota,x}/\Kp_{\iota,x}$ is isomorphic to its Lie algebra counterpart
  $\kG_{\iota,x}/\kG^{+}_{\iota,x}$. 
By construction, the latter decomposes as 
$$ \kG_{\iota,x}/\kG^{+}_{\iota,x} = \bigoplus_{i=1}^{d}
(\gG_{\iota}^{i}\cap\nG_{\iota}^{i-1})_{x, (r_{i-1}/2):(r_{i-1}/2)+} ,$$
which is a direct sum of vector spaces over the residue field of $F$, hence has exponent $p$. 
  The equality $[g,hh']=[g,h]\cdot{^{h}[g,h']}$ shows that the map $\theta$ is
  $\ZM$-bilinear, hence the image of this map is contained in the only subgroup $\mu_{p}$
  of $\mu_{p^{\infty}}$ of exponent $p$. Note that in the above decomposition of
  $\kG_{\iota,x}/\kG^{+}_{\iota,x}$ the summand
  $(\gG_{\iota}^{i}\cap\nG_{\iota}^{i-1})_{x, (r_{i-1}/2):(r_{i-1}/2)+} =
  \jG^{i}_{\iota,x}/\jG^{i+}_{\iota,x}$ identifies with the image of the subgroup $J^{i}_{\iota,x}$ in
  $\Ku_{\iota,x}/\Kp_{\iota,x}$, so that another way
  to write this   decomposition is as 
\ini\begin{equation}
	\label{decomposition-of-Vx}
 \Ku_{\iota,x}/\Kp_{\iota,x} = \prod_{i=1}^{d} \left( J_{\iota,x}^{i}/J_{\iota,x}^{i+}\right).
\end{equation}  
Now, recall the factorization
$\check\phi^{+}_{\iota,x}=\prod_{k=0}^{d}(\psi_{k,x}^{+})|_{\Kp_{\iota,x}}$ of
\ref{Yu_construction}
and let $i,j\in \{1,\hdots, d\}$ and  $k \in \{0,\hdots, d\}$.  
We claim that
$[J^{i}_{\iota,x},J^{j}_{\iota,x}]\subset \ker\, \psi_{k,x}^{+}$ unless $i=j=k+1$. 
When $k+1 \neq i, j$, this
follows from the fact that both $J_{\iota,x}^{i}$ and $J_{\iota,x}^{j}$ are contained in
the group $G_{\iota,x,0+}^{k}G_{x,r_{k}/2+}$ on which $\psi_{k,x}^{+}$ is defined. When $k+1=i$
and $i\neq j$,
this follows from the inclusion
$[J^{i}_{\iota,x},J^{j}_{\iota,x}]\subset [J^{k+1}_{\iota,x},
G_{\iota,x,0+}^{k}G_{x,r_{k}/2+}] \subset (G_\iota^k,G)_{x,r_{k}+, r_k/2+} \subset \ker(\psi_{k,x}^{+})$, and
similarily for $k+1=j\neq i$.

As a consequence, the last displayed decomposition is orthogonal for the bilinear
form $\theta$, and the restriction of $\theta$ to the summand
$J_{\iota,x}^{i}/J_{\iota,x}^{i+} $ is given  by
$\theta(\bar j,\bar h)= \psi_{i-1,x}^{+}(jhj^{-1}h^{-1})$. By \cite[Lemma 11.1]{Yu_tamescusp} the latter
bilinear form on $J_{\iota,x}^{i}/J_{\iota,x}^{i+}$ is non-degenerate, hence so is $\theta$. 

ii) We have $[(U\cap \Ku_{\iota,x}),(P\cap \Ku_{\iota,x})]\subset (U\cap
\Kp_{\iota,x})$. Since $\check\phi^{+}_{\iota,x}$ is trivial on $U\cap
\Kp_{\iota,x}$, we see that $P\cap \Ku_{\iota,x}$ is orthogonal to $U\cap
\Ku_{\iota,x}$ for 
the bilinear form $\theta$. 
From the Iwahori decomposition of $\Ku_{\iota,x}$ and $\Kp_{\iota,x}$ with respect to 
$({\mathbf{P}} ,\bar{\mathbf{P}} )$, see Lemma \ref{Iwahori}, we obtain a decomposition of $\FM_{p}$-vector
spaces
$$ \Ku_{\iota,x}/\Kp_{\iota,x} =
(U\cap \Ku_{\iota,x})/(U\cap \Kp_{\iota,x}) \oplus
(M\cap \Ku_{\iota,x})/(M\cap \Kp_{\iota,x}) \oplus
(U\cap \Ku_{\iota,x})/(U\cap \Kp_{\iota,x})$$
and ii) now follows from the non-degeneracy of $\theta$.
\end{proof}

\subsection{Systems of idempotents}\label{sec:systems-idempotents}

In this section we fix a wild inertia parameter $\phi$ 
and  a ${\mathbf{G}}(F)$-conjugacy class \index[notation]{I@$I$}$I\subseteq I_{\phi}$ of $F$-rational embeddings
$\mathbf{S}_{\phi}\injo {\mathbf{G}}$.  Recall that we assume that $C_{\hat{\mathbf{G}}}(\phi)$ is a Levi subgroup and that $p$ satisfies \eqref{H1} and \eqref{H2}.

\alin{An orthogonality property}
For $x\in \BT=\BT({\mathbf{G}},F)$, we put \index[notation]{Ix@$I_x$}$I_{x}:=\{\iota\in I, x\in \BT_{\iota}\}$. 
If
$\iota,\iota'\in I_{x}$, we declare that $\iota'\sim_{x}\iota$ if $\iota'\in \Ku_{\iota,x}\cdot\iota$. 
Using (\ref{Gaction}) we see that if $\iota'\in \Ku_{\iota,x}\cdot\iota$, then we have
$\Ku_{\iota',x}=\Ku_{\iota,x}$,  whence the transitivity and symmetry of the relation $\sim_{x}$, which is
thus an equivalence relation on $I_{x}$.  
\begin{lem} \label{ortho_idemp}
 For $\iota',\iota\in I_{x}$ we have 
$e_{\iota,x}e_{\iota',x} =\left\{
\begin{array}{ll}
e_{\iota,x} & \hbox{if } \iota\sim_{x}\iota' \\
0 & \hbox{else}.
\end{array}
\right.$
\end{lem}
\begin{proof}
If $e_{\iota,x}e_{\iota',x}\neq 0$, then by ii) of Proposition \ref{prop_intertwining} we have that $\Ku_{\iota',x}\cdot
\iota'\cap \Ku_{\iota,x}\cdot\iota\neq\emptyset$, and hence
$\iota\sim_{x}\iota'$. Thus $e_{\iota,x}e_{\iota',x}= 0$ if we are not in the case $\iota\sim_{x}\iota'$.
If $\iota\sim_{x}\iota'$, then, by (\ref{Gaction}) and  Proposition \ref{prop_intertwining} i), 
we have $e_{\iota',x}=e_{\iota,x}$, hence  $e_{\iota,x}e_{\iota',x}=e_{\iota,x}$.
\end{proof}

\alin{A variant} \label{rem_ortho_idemp} The following strengthening of the above lemma is not
needed in Section 2 but will occasionally be useful in Section 3. Recall the open subgroup
$K'_{\iota',x}$ of $\Kp_{\iota',x}$ introduced in 
\ref{rem_intertw}, and denote by
$e'_{\iota',x}$ the idempotent associated to  $\check\phi^{+}_{\iota',x}|_{K'_{\iota',x}}$. It is
coarser than $e_{\iota',x}$, in the sense that $e_{\iota',x}e'_{\iota',x}=e_{\iota,x'}$, and the
analog of (\ref{Gaction}) still holds, namely $^{g}e'_{\iota,x}=e'_{g\iota,gx}$ for
all $g\in G$.

\begin{lem}
 For $\iota',\iota\in I_{x}$ we have  $e_{\iota,x}e'_{\iota',x} =\left\{
\begin{array}{ll}
e_{\iota,x} & \hbox{if } \iota\sim_{x}\iota' \\
0 & \hbox{else}.
\end{array}
\right.$
\end{lem}
\begin{proof}
  Same proof as the last lemma, with Proposition \ref{rem_intertw} instead of
  Proposition \ref{prop_intertwining}.
\end{proof}

\alin{The idempotents associated to $\phi$ and $I$}  \label{def_idemp_I}
Since there is  a finite number of idempotents $e \in\HC_{R}(G)$ that are supported
on $G_{x}$ and satisfy $e\cdot e_{G_{x,r_{\phi}+}}=e$,
Lemma \ref{ortho_idemp} shows that the
sum\index[notation]{ex@$e_x$}\index[notation]{eIx@$e_{I,x}$}\index[notation]{ephiIx@$e_{\phi,I,x}$} 
$$ e_{x} := e_{I,x} :=e_{\phi,I,x} := \sum_{\iota\in I_{x}/\sim_{x}} e_{\iota,x}$$
is finite and defines an idempotent of $\HC_{R}(G)$ supported on $G_{x,0+}$. 
It is non-zero if and only if $I_{x}$ is
non empty. By construction we have the following equivariance property:
\ini
\begin{equation}
  \label{Gequivariant}
\forall x\in \BT, \, \forall g\in G, \, e_{gx}= {^{g}e_{x}}.  
\end{equation}
In particular, $e_{x}$ is a central idempotent in $\HC_{R}(G_{x})$. 

\begin{lem}
  Fix $x,x'\in \BT$,  put $I_{x,x'}:=I_{x}\cap I_{x'}$ and endow this set with the
  equivalence relation $\iota\sim_{x,x'}\iota'\Leftrightarrow (\iota\sim_{x}\iota' 
\hbox{ and } \iota\sim_{x'}\iota')$. 
Then we have 
\ini\begin{equation}
e_{x}e_{x'}=\sum_{\iota\in I_{x,x'}/\sim_{x,x'}} e_{\iota,x}e_{\iota,x'}.\label{exex'}
\end{equation}
\end{lem}
\begin{proof}
Denote by $\bar\iota:= \Ku_{\iota,x}\cdot\iota$  
the $\sim_{x}$-equivalence class of $\iota\in I_{x}$, and similarly for $\iota'\in I_{x'}$.
 By Proposition \ref{prop_intertwining}.ii)
 we have 
$\bar\iota\cap\bar\iota'\neq\emptyset$ whenever $e_{\iota,x}e_{\iota',x'}\neq 0$.
Hence 
$e_{x}e_{x'}=\sum_{\bar\iota\cap\bar\iota'\neq\emptyset}e_{\bar\iota,x}e_{\bar\iota',x'}$.
Now,  the intersection $\bar\iota\cap\bar\iota'$ in $I$ is contained in
$I_{x,x'}$ and, if $\bar\iota\cap\bar\iota'$ is non-empty, it is actually a $\sim_{x,x'}$-equivalence class. We thus have a map
$(\bar\iota,\bar\iota')\mapsto \bar\iota\cap\bar\iota'$,
$$ \left\{(\bar\iota,\bar\iota')\in {I_{x}}_{/\sim_{x}}\times {I_{x'}}_{/\sim_{x'}},
\bar\iota\cap\bar\iota'\neq\emptyset\right\} \To{} {I_{x,x'}}_{/\sim_{x,x'}} ,$$
which is easily seen to be a bijection. 
\end{proof}

\alin{A telescopic identity} \label{telescopic_iota}
We will provide a telescopic identity when moving along a geodesic. In the following
lemma, we fix $\iota\in I$ and two points $x,x'\in \BT_{\iota}$. Since $\BT_{\iota}$ is
convex, it contains the segment $[x,x']$.
\begin{lem}
  Suppose $x''\in [x,x']$. Then
  $e_{\iota,x}e_{\iota,x'}=e_{\iota,x}e_{\iota,x''}e_{\iota,x'}$ (product in $\HC_{R}(G)$).
\end{lem}
\begin{proof}
  As explained in 
	\ref{opposite_parabolic_sgps}, 
  the segment $(x,x')$ determines a pair $({\mathbf{P}} ,\bar{\mathbf{P}} )$ of $F$-rational
  opposite parabolic subgroups of ${\mathbf{G}}$ with common Levi subgroup $\mathbf{M}$ satisfying the conditions i) and ii) of
  \ref{Iwahori} for each of the points $x,x',x''$. The following equalities and inclusions
  \begin{itemize}
  \item $M\cap \Kp_{\iota,x}=M\cap \Kp_{\iota,x''}=M\cap\Kp_{\iota,x'}$
  \item $U\cap \Kp_{\iota,x}\supset U\cap \Kp_{\iota,x''}\supset U\cap\Kp_{\iota,x'}$
  \item $\bar U\cap \Kp_{\iota,x}\subset\bar U\cap \Kp_{\iota,x''}\subset\bar
    U\cap\Kp_{\iota,x'}$ 
  \end{itemize}
  follow from similar equalities and inclusions for each $G_{\iota,x,(r_{i-1}/2)+}^{i}$ in
  place of $\Kp_{\iota,x}$.
  Using the Iwahori decomposition of Lemma \ref{Iwahori}, we then deduce that
  $$\Kp_{\iota,x''} = (\Kp_{\iota,x}\cap\Kp_{\iota,x''})
  (\Kp_{\iota,x''}\cap \Kp_{\iota,x'}) .$$
Now denote by $e_{\iota,x,x''}$ and $e_{\iota,x'',x'}$ the idempotents associated
to $(\check\phi^{+}_{\iota,x''})|_{\Kp_{\iota,x}\cap\Kp_{\iota,x''}}$ and 
$(\check\phi^{+}_{\iota,x''})|_{\Kp_{\iota,x''}\cap\Kp_{\iota,x'}}$, respectively. We obtain a factorization
$e_{\iota,x''}=e_{\iota,x,x''}e_{\iota,x'',x'}$ in $\HC_{R}(G_{x''})$. By Lemma \ref{intertw_converse}, it follows that
$e_{\iota,x}e_{\iota,x''}e_{\iota,x'}=e_{\iota,x}e_{\iota,x,x''}e_{\iota,x'',x'}e_{\iota,x'}=e_{\iota,x}e_{\iota,x'}$.
 \end{proof}

\begin{pro} \label{telescopic2}
  Let $x,x'\in\BT$ and $x''\in [x,x']$. Suppose that there exists a facet of $\BT$ whose closure contains both $x$ and $x''$. Then $e_{x}e_{x''}e_{x'}=e_{x}e_{x'}$.
\end{pro}
\begin{proof}
Since there exists some facet whose closure contains $x$ and $x''$, 
the group
$G_{x''}$ contains $G_{x,0+}$,  and it follows
  that $e_{x''}$ commutes with all $e_{\iota,x}$ since the former is central in
  $\HC_{R}(G_{x''})$ and the latter are supported on $G_{x,0+}$. 
This commutation property provides the first and fourth equality in the following computation.
  \begin{eqnarray*}
 e_{x}e_{x''}e_{x'}=e_{x''}e_{x}e_{x'}
& = &e_{x''}\left(\sum_{\iota\in I_{x,x'}/\sim_{x,x'}}e_{\iota,x}e_{\iota,x'}\right)
=e_{x''}\left(\sum_{\iota\in I_{x,x'}/\sim_{x,x'}}e_{\iota,x}e_{\iota,x''}e_{\iota,x'}\right) \\ 
& =& \sum_{\iota\in I_{x,x'}/\sim_{x,x'}} e_{\iota,x}e_{x''}e_{\iota,x''}e_{\iota,x'}
= \sum_{\iota\in I_{x,x'}/\sim_{x,x'}} e_{\iota,x}e_{\iota,x''}e_{\iota,x'} \\
&= & \sum_{\iota\in I_{x,x'}/\sim_{x,x'}} e_{\iota,x}e_{\iota,x'}=e_{x}e_{x'}.
\end{eqnarray*}
In the second and the last equality, we used (\ref{exex'}).
In the third and sixth equality we used the last lemma, 
and in the fifth one we used Lemma \ref{def_idemp_I}. 
\end{proof}

\alin{$E$-facets} \label{facets}
From now on we also assume that $p$ is odd.
We aim at finding a polysimplicial $G$-equivariant structure on $\BT$
such that $e_{x}$ only depends on the facet it belongs
to. Simple examples show that the usual Bruhat--Tits structure on $\BT$
will not work. 
Instead, we will consider the intersection with $\BT$ of the Bruhat--Tits structure on
$\BT({\mathbf{G}},E)$ 
for a suitable tamely ramified extension $E$ over $F$.

\begin{lem} Assume $p$ is odd.
  There is a tamely ramified Galois field extension $E$ of $F$ that splits a maximal $F$-torus of
  ${\mathbf{G}}_{\iota}$ (for any $\iota\in I$) and such that $\{r_{0}/2,\hdots, r_{d-1}/2\}\subset v(E^{\times})$,
  where $v$ is the unique valuation on $E$ that extends the normalized valuation of $F$. 
\end{lem}
\begin{proof}
  Let $E$ be a tamely ramified splitting field of some maximal $F$-torus $\mathbf{S}$ in ${\mathbf{G}}_{\iota}$. Lemma
  \ref{rootsystemCr} shows that for $0\leq i\leq d-1$, the real number $r_{i}$ is a jump
  of the filtration on $E^{\times}$, hence belongs to $v(E^{\times})=\frac 1e\ZM$ with $e$ the
  ramification index of $E/F$. After replacing $E$ by a quadratic ramified 
  extension if  necessary, we obtain that $r_{i}/2\in v(E^{\times})$ for each $i$, as desired. Then the Galois closure of
  this $E$ meets the requirements of the lemma.
\end{proof}

Let $E$ be as in the above lemma. Recall that the reduced Bruhat--Tits building $\BT^{\rm red}({\mathbf{G}},E)$ of the reductive group $\mathbf{G}_E$ over $E$
carries a polysimplicial structure. The inverse
image in $\BT({\mathbf{G}},E)$ of a polysimplex of $\BT^{\rm red}({\mathbf{G}},E)$ will be
called a facet of $\BT({\mathbf{G}},E)$. Recall also, e.g.,  from \cite[Thm 2.1.1]{Landvogt_functorial},
that there is a canonical embedding 
$\BT({\mathbf{G}},F)\injo \BT({\mathbf{G}},E)$, so we may and will identify $\BT$ with a subset
of $\BT({\mathbf{G}},E)$.  Since $\mathbf{G}$ is defined over $F$, the building $\BT({\mathbf{G}},E)$ is equipped with an action of ${\rm Gal}(E/F)$, and, by a result of Rousseau, see \cite{Prasad_Galois}, we have 
$\BT=\BT({\mathbf{G}},E)^{{\rm Gal}(E/F)}$, 
because $E/F$ is tamely ramified.
The  intersection of a facet of  $\BT({\mathbf{G}},E)$
with $\BT$
will be called
an ``$E$-facet of $\BT$''. We thus obtain a partition of $\BT$ into ``$E$-facets'', such that
the closure of an $E$-facet is a union of $E$-facets. We will
denote by $\FC_{E}(x)$ the $E$-facet of $\BT$ that contains $x \in \BT$.

\begin{prop}   \label{indep_facet}
  Let $E$ be as in the previous lemma. Then we have:
  \begin{enumerate}
  \item $\forall\iota\in I, \forall x\in\BT,\,\, x\in\BT_{\iota}\Rightarrow
    \FC_{E}(x)\subseteq\BT_{\iota}$
  \item $\forall\iota\in I,\forall x,x'\in\BT_{\iota},\;\;  \FC_{E}(x)=\FC_{E}(x')\Rightarrow
    (e_{\iota,x}=e_{\iota,x'} \hbox{ and } \Ku_{\iota,x}=\Ku_{\iota,x'})$
  \item $\forall x,x'\in \BT, \,\,\, \FC_{E}(x)=\FC_{E}(x')\Rightarrow e_{x}=e_{x'}.$
  \end{enumerate}
\end{prop}
\begin{proof}
i) Denote by $\tilde\FC_{E}(x)$ the facet of $\BT({\mathbf{G}},E)$ that contains $x$. Since
  $({\mathbf{G}}_{\iota})_E$ is Levi subgroup of $\mathbf{G}_E$, the facet $\tilde\FC_{E}(x)$ is
  contained in the image $\BT_{\iota}(E)$ of any admissible embedding
  $\BT({\mathbf{G}}_{\iota},E)\injo\BT({\mathbf{G}},E)$. 
Therefore $\FC_{E}(x)=\tilde\FC_{E}(x)\cap
  \BT\subset \BT_{\iota}(E)^{{\rm Gal}(E/F)}=\BT_{\iota}$, as desired.

  ii) By \cite[Prop 1.1]{Vigsheaves}, the group ${\mathbf{G}}(E)_{x,r+}$ only depends on the
facet containing $x$ provided that $r\in v(E^{\times})$ (in \emph{loc.cit.} the valuation
is normalized by $v(E^{\times})=\ZM$). Recall also from \cite[\S 2]{Yu_tamescusp} that
$G_{x,r+}=G\cap {\mathbf{G}}(E)_{x,r+}$.
It follows that for each $i=0,\hdots, d$ the group
$G^{i}_{\iota,x,(r_{i-1}/2)+}=G^{i}_{\iota}\cap {\mathbf{G}}(E)_{x,(r_{i-1}/2)+}$ only
depends on the $E$-facet $\FC_{E}(x)$. Therefore we obtain $\Kp_{\iota,x}=\Kp_{\iota,x'}$.
A similar argument shows that $\Ku_{\iota,x}=\Ku_{\iota,x'}$.
Now Lemma \ref{intertw_converse} implies that
$\check\phi^{+}_{\iota,x}=\check\phi^{+}_{\iota,x'}$ hence also
$e_{\iota,x}=e_{\iota,x'}$.

iii)
 By ii), it suffices to show that 
$I_{x}=I_{x'}$, which follows from i).
\end{proof}

\begin{prop} \label{idemp_commut}
Let $E$ be an extension of $F$ as in Lemma \ref{facets} and let $x,x'\in\BT$. Suppose there exists an $E$-facet  of $\BT$ whose closure contains both $x$ and $x'$. Then for any
  $x''\in ]x,x'[$ we have $e_{x}e_{x'}=e_{x''}=e_{x'}e_{x}$.
\end{prop}
\begin{proof} 
Note first that in this situation we have $x,x'\in \o{\FC_{E}(x'')}$.

We claim that $I_{x,x'}=I_{x''}$.
Indeed, the inclusion $I_{x,x'}=I_x \cap I_{x'} \subseteq I_{x''}$ follows from the convexity of
  $\BT_{\iota}$ in $\BT$, while the other inclusion follows from $\BT_{\iota}$ being a closed subset of $\BT$ together with
    $\FC_{E}(x'')\subseteq\BT_{\iota}$ for any $\iota\in I_{x''}$ by Proposition \ref{indep_facet}.i).

Now, let us fix $\iota\in I_{x,x'}$.
 By \cite[Prop 1.1]{Vigsheaves}, the
groups ${\mathbf{G}}^{i}_{\iota}(E)_{x,r+}$ and ${\mathbf{G}}^{i}_{\iota}(E)_{x',r+}$ are contained in
${\mathbf{G}}^{i}_{\iota}(E)_{x'',r+}$ for each $i$, provided that $r\in v(E^{\times})$. By taking Galois-fixed elements, it
follows that $G^{i}_{\iota,x,(r_{i-1}/2)+}$ and $G^{i}_{\iota,x',(r_{i-1}/2)+}$ are
contained in $G^{i}_{\iota,x'',(r_{i-1}/2)+}$. 
We infer that 
$\Kp_{\iota,x}\subseteq \Kp_{\iota,x''}$ and
$\Kp_{\iota,x'}\subseteq \Kp_{\iota,x''}$. 
Lemma \ref{intertw_converse} then implies that $e_{\iota,x}e_{\iota,x''}=e_{\iota,x''}$ and
$e_{\iota,x''}e_{\iota,x'}=e_{\iota,x''}$. We conclude thanks to Lemma
\ref{telescopic_iota} that $e_{\iota,x}e_{\iota,x'}=e_{\iota,x}e_{\iota,x''}e_{\iota,x'}=e_{\iota,x''}$.
Using (\ref{exex'}) 
we thus obtain the formula
$$ e_{x}e_{x'}=\sum_{\iota\in I_{x''}/\sim_{x,x'}} e_{\iota,x''}, 
 \hbox{ to  compare with } e_{x''}=\sum_{\iota\in I_{x''}/\sim_{x''}} e_{\iota,x''}.$$
For $\iota_{1},\iota_{2}\in I_{x''}$, Lemma \ref{ortho_idemp} tells us that
$$
\begin{array}{l}
\iota_{1}\sim_{x''}\iota_{2}\Leftrightarrow e_{\iota_{1},x''}=e_{\iota_{2},x''} \\
\iota_{1}\sim_{x,x'}\iota_{2}\Leftrightarrow (e_{\iota_{1},x}=e_{\iota_{2},x} \hbox{  and }
e_{\iota_{1},x'}=e_{\iota_{2},x'})
\end{array}
$$
The equality $e_{\iota,x}e_{\iota,x'}=e_{\iota,x''}$ proved just above 
shows that
$\iota_{1}\sim_{x,x'}\iota_{2}\Rightarrow \iota_{1}\sim_{x''}\iota_{2}$. On the other
hand, Lemma \ref{ortho_idemp} also shows that 
$$
\begin{array}{l}
\iota_{1}\sim_{x''}\iota_{2}\Leftrightarrow e_{\iota_{1},x''}e_{\iota_{2},x''}\neq 0 \\
\iota_{1}\sim_{x,x'}\iota_{2}\Leftrightarrow (e_{\iota_{1},x}e_{\iota_{2},x}\neq 0 \hbox{  and }
e_{\iota_{1},x'}e_{\iota_{2},x'}\neq 0)
\end{array}
$$
This time, the equalities $e_{\iota,x}e_{\iota,x'}=e_{\iota,x''}=e_{\iota,x'}e_{\iota,x}$ 
show that
$\iota_{1}\sim_{x''}\iota_{2}\Rightarrow \iota_{1}\sim_{x,x'}\iota_{2}$.
\end{proof}

\subsection{The category $\Rep^{\phi,I}_{R}(G)$}\label{sec:category-repphi-i_rg}

We now construct the category attached to a wild inertia parameter
$\phi:\,P_{F}\To{}{^{L}{\mathbf{G}}}$ and a $G$-conjugacy class $I\subset I_{\phi}$
of $F$-rational embeddings $\iota : \mathbf{S}_{\phi}\injo {\mathbf{G}}$.
If $V$ is any smooth
$RG$-module, it has an action of the Hecke algebra $\HC_{R}(G)$ and in particular the
idempotents $e_{x}=e_{\phi,I,x}$ act on it.   This subsection is
mainly devoted to the proof of the following theorem.

\begin{theo} \label{thm_Serresubcat}
  The subcategory $\Rep^{\phi,I}_{R}(G)$ of $\Rep_{R}(G)$ defined by\index[notation]{RepphiIG@$\Rep^{\phi,I}_{R}(G)$} 
$$\Rep^{\phi,I}_{R}(G):=\left\{V\in\Rep_{R}(G),\, V=\sum_{x\in \BT} e_{\phi,I,x}V\right\}$$
is a Serre subcategory of $\Rep_{R}(G)$, stable under arbitrary colimits, and  generated
 by the following compact projective object of $\Rep_{R}(G)$
$$ P^{\phi,I}:=\bigoplus_{x\in \Delta_{0}} \bigoplus_{\iota\in I_{x}/\sim_{x}}
\cInd{\Kp_{\iota,x}}{G}(\check\phi^{+}_{\iota,x}),$$ 
where $\Delta_{0}$ denotes the set of $e$-vertices of a $1$-chamber $\Delta$ of $\BT$ (a notion that will be introduced in \ref{subsection-efacets} below). 
Moreover, any object $V\in\Rep^{\phi,I}_{R}(G)$ is functorially 
an extension 
\ini\begin{equation}\label{exactseq}
 V^{\phi,I}\injo V \twoheadrightarrow V_{\phi,I}
\end{equation}
where $V^{\phi,I} \in \Rep^{\phi,I}_{R}(G)$ 
and $V_{\phi,I}$ has no subquotient that belongs to $\Rep^{\phi,I}_{R}(G)$.
\end{theo}

The strategy is to put ourselves in a position where we can apply \cite[Thm
3.1]{MS1}, or at least closely follow its proof. This reference is concerned with systems of idempotents associated to
vertices  (more generally to polysimplices) in the reduced building
$\BT':=\BT({\mathbf{G}}_{\rm ad},F)$,\index[notation]{Bprime@$\BT'$} while we have constructed idempotents associated to points
of $\BT$. However, our idempotent $e_{x}$ only depends on the  image of $x$  in $\BT'$, so that we
actually have idempotents associated to points of $\BT'$. Unfortunately, these idempotents
are not constant on $F$-facets, but only on $E$-facets for some Galois extension $E$ of
$F$ as in \ref{facets}.

\alin{The $e$-subdivision of $\BT'$}\label{subsection-efacets} 
 Fix an integer $e\geq 1$. We define a subdivision of the  polysimplicial structure on
 $\BT'$ in the following way. 

Start with an apartment $A$ of $\BT'$ and define an $e$-wall to be
 an affine hyperplane of the form $\varphi^{-1}(\frac ke t_{1}+ \frac{e-k}e t_{2})$ where
 $\varphi$ is an affine root on $A$,  $t_{1},t_{2}\in \RM$ are such that $\varphi^{-1}(t_{1})$ and
 $\varphi^{-1}(t_{2})$ are walls of $A$, and $k$ is an integer between $0$ and $e$. In
 particular, $1$-walls are the usual walls and are also $e$-walls for any $e\geq 1$, and
 moreover any $e$-wall is parallel to some $1$-wall. We thus get an enlarged collection of
 hyperplanes, leading to a refined partition of $A$ into facets, that we call $e$-facets.
We note that if $o$ is a special point of $A$, then the $e$-walls of $A$ are the
 images of the walls by  the
 homothety of  ratio $1/e$ centered at $o$. Indeed, for any affine
 root $\varphi$ on $A$, it follows from \cite[(6.2.16)]{BT1} that 
 the set of all $t\in\RM$  such that
 $\varphi^{-1}(t+\varphi(o))$ is a wall is a discrete subgroup of $\RM$. 
 As a consequence, the $e$-facets are the images of
 the usual facets by the same homothety.  

If $A'$ is another apartment, we define $e$-walls and $e$-facets in the same way. 
Then for any $g\in G$, the action of $g$ on $\BT'$ takes an $e$-wall of $A$ to an $e$-wall
of $gA$. In particular, if $F$ is an $e$-facet of $A$ which intersects $A'$,
then $F\subset A'$ and $F$ is an $e$-facet of $A'$. Indeed, there is some $g\in G$ with
$A'=gA$ and such that $g$ fixes the $1$-facet $F_{1}\in \BT'$ that contains $F$.
This allows to define unambiguously the $e$-facets of  $\BT'$, and we get a
polysimplicial structure on $\BT'$ which is preserved by the action of $G$. We will call it the
$e$-subdivision of $\BT'$.

\begin{lem}\label{efacets}
Let $E$ be a tamely ramified Galois field extension of $F$ with ramification index $e=e(E/F)$.
Suppose that ${\mathbf{G}}$ is split over $E$ and quasi-split over the maximal unramified subextension $E_{0}$ of $E$. Then the
$e$-subdivision of $\BT'$ refines the partition of $\BT'$ into $E$-facets,
i.e., any $E$-facet is a union of $e$-facets.
\end{lem}
  \begin{proof}
Given a maximal $F$-split torus
${\mathbf{T}}$ of ${\mathbf{G}}$, \cite[Cor. 5.1.12]{BT2} ensures that we can find a maximal $E_{0}$-split $F$-torus
${\mathbf{T}}_{0}$ that contains ${\mathbf{T}}$. Then the centralizer $\mathbf{S}=C_{\mathbf{G}}(\mathbf{T_{0})}$ of
${\mathbf{T}}_{0}$ is a maximal $F$-torus of ${\mathbf{G}}$, and is split over $E$.
In this situation there are inclusions of apartments  $A=A({\mathbf{G}},{\mathbf{T}},F)\subset
A_{0}=A({\mathbf{G}},{\mathbf{T}}_{0},E_{0}) \subset A_{S}=A({\mathbf{G}},\mathbf{S},E)$ and
each subspace is obtained by taking suitable Galois invariants.
By \cite[Thm 5.1.20 iii)]{BT2}, the walls of the apartment $A$ are exactly the
non-trivial intersections of $A$ with the walls of the apartment $A_{0}$.
Moreover, by \cite[4.2.4]{BT2} each wall of $A_{0}$ is the intersection of $A_{0}$ with a wall of $A_{S}$.
   Conversely, the intersection of a wall of $A_{S}$ with $A_{0}$, when non-trivial, may not be a wall
of $A_{0}$ but, at least, is parallel to a wall of $A$. More precisely, fix an origin $o$ which is
a special point in
$A_{0}$ (e.g.\ that comes from a Chevalley--Steinberg system as in \cite[4.2.3]{BT2}) and
let $a$ be a non-divisible root of
${\mathbf{T}}_{0}$ in ${\mathbf{G}}$, and let $E_{0}\subset E_{a}\subset E$ be the associated
extension (denoted by $L_{a}$ in \emph{loc.\ cit.}).
 Denote by $\Gamma_{a}\subset \RM$ the set of real numbers $v$ such
that $\{x\in A_{0}, a(x)= v\}$ is a wall of $A_{0}$.
Then by 
\cite[4.2.21]{BT2} we have 
  $\Gamma_{a}=v(E_{a}^{\times})$ (the
valuation lattice of $E_{a}$) if $2a$ is not a root,  and  $\Gamma_{a}= \frac 12 v(E_{a}^{\times})$ if $2a$ is
a root. 
On the other hand, let  $\Gamma_{a,E}\subset \RM$ be the set of real numbers $v$ such
that $\{x\in A_{0}, a(x)= v\}$ is the intersection of $A_{0}$ with a wall of $A_{S}$.
If $2a$ is not a root and $v\in \Gamma_{a,E}$, then there is a root $\alpha$ of
$\mathbf{S}$ in ${\mathbf{G}}$ that restricts to $a$ and such that $\{x\in A_{S},
\alpha(x)=v\}$ is a wall of $A_{S}$, hence $v\in v(E^{\times})$. If $2a$ is a root, then
either there is $\alpha$ as above and then $v\in v(E^{\times})$, or there are
$\alpha,\alpha'$ as above with $\alpha+\alpha'$ a root, and $\{x\in A_{S},
(\alpha+\alpha')(x)=2v\}$ is a wall of $A_{S}$, in which case $v\in \frac 12 v(E)^{\times}$.
It follows that we have $\Gamma_{a,E}=v(E^{\times})$ if $2a$ is not a root, and
$\Gamma_{a,E}=\frac 12 v(E^{\times})$ if $2a$ is a root.  
In any case,  for all
non-divisible roots $a$ of ${\mathbf{T}}_{0}$ we have $\Gamma_{a,E}= \frac 1{e(E/E_{a})}\Gamma_{a}$.
 
Since $e=e(E/E_{0})$ is a common multiple of all $e(E/E_{\alpha})$, 
the above discussion
shows that the $e$-subdivision of the polysimplicial structure on $A_{0}$
refines the one that comes from $A_{S}$. Since the polysimplicial structure on $A_{0}$
induces the one on $A$ (again by Thm 5.1.20 iii) of \cite{BT2}),
it follows that the  $e$-subdivision
of the polysimplicial structure on $A$ refines its partition into $E$-facets.
\end{proof}

\ali \label{choice_e}
From now on, we pick an integer $e$ that is divisible by the ramification index of a
tamely ramified field extension $E$ of $F$ that fulfills the requirements of both
Lemmas \ref{facets} and \ref{efacets}. Then, for each
$\iota\in I$, the image $\BT'_{\iota}$\index[notation]{Bprimeiota@$\BT'_{\iota}$} of $\BT_{\iota}$ in $\BT'$ is a union of $e$-facets.

 For $x\in \BT'$ we
  denote by $\FC_{e}(x)$ the unique $e$-facet of $\BT'$ that contains $x$. 
  Further, we denote by $\BT'_{\bullet/e}$ the partially ordered set of all $e$-facets,
  with the order given by   $\FC'\preceq\FC\Leftrightarrow \FC'\subseteq \o\FC$.
 We will also write $\BT'_{d/e}$ for the set of $d$-dimensional $e$-facets . For $d=0$ we also speak of
  ``$e$-vertices''. A family $x_{1},\hdots, x_{r}$ of $e$-vertices are called ``adjacent''
  if there exits an $e$-facet whose closure contains all these $e$-vertices. Then there is a unique $e$-facet
  $\FC_{e}(x_{1},\hdots, x_{r})$ with this property and which is minimal for the order
  defined above.

If $x,x'$ are two points in $\BT$, they are contained in a common apartment $A$. The
intersection of all the half spaces associated to walls of $A$ that
contain $x$ and $x'$ is known to be independent of the choice of $A$. It is called the
``combinatorial convex hull'' of $x$ an $x'$ and we will denote it by
$\HC_{1}(x,x')$. It is a union of facets. Similarly we denote by $\HC_{e}(x,x')$ the
intersection of all $e$-half spaces (corresponding to $e$-walls) of $A$ that contain $x$
and $x'$. This is again independent of $A$ and a union of $e$-facets. Obviously
$[x,x']\subset \HC_{e}(x,x')\subset \HC_{1}(x,x')$.

\begin{lemme}\label{consistent_system}
  Let the integer $e$ be as in \ref{choice_e}.
  
  Then the idempotents $(e_{x})_{x\in \BT'}$ have the following properties.
  \begin{enumerate}
  \item for all $x,x'\in \BT'$ we have $\FC_{e}(x)=\FC_{e}(x')\Rightarrow e_{x}=e_{x'}$.
  \item If $x,x'$ are adjacent $e$-vertices and $x'' \in ]x, x'[$, then
    $e_{x}e_{x'}=e_{x'}e_{x}=e_{x''}$.
  \item If $x,x',x''$ are three $e$-vertices with $x'\in \HC_{e}(x,x'')$ and $x'$
    adjacent to $x$, then $e_{x}e_{x'}e_{x''}=e_{x}e_{x''}$.
\end{enumerate}
In particular the system $(e_{x})_{x\in \BT'_{0/e}}$ is consistent in the sense of
\cite[Def. 2.1]{MS1}.
\end{lemme}
\begin{proof}
  Thanks to Lemma \ref{efacets}, Part i) follows from  Proposition \ref{indep_facet} iii),  and ii) follows from
  Proposition \ref{idemp_commut}. To prove statement iii),  let $y\in ]x,x''[$ be sufficiently closed to $x$ so
  that $x\in\o{\FC_{e}(y)}$. Then, 
  $\FC_{e}(y)$ is the unique $e$-facet that is maximal among those $e$-facets $\FC$ with $\o\FC\subset
  \HC_{e}(x,x'')$ and $x\in \o\FC$. Indeed, this is proved as Lemma 2.9 of \cite{MS1}, since
  the geometric properties of the polysimplicial structure of $\BT'$ used in the proof of that lemma
  are  satisfied by its $e$-subdivision. 
  In particular, we have $x'\in \o{\FC_{e}(y)}$, and we may  choose  $y' \in \FC_{e}(y)$ such that $y \in ]x', y'[$. By Part
  i) we have $e_y=e_{y'}$. 
  Using $e_y=e_{y'}$ and Propositions \ref{telescopic2}
  and  \ref{idemp_commut} repeatedly, we obtain $e_{x}e_{x''}=e_{x}e_{y}e_{x''}=e_{x}e_{x'}e_{y'}e_{x''}=e_{x}e_{x'}e_{y}e_{x''}=e_{x'}e_{x}e_{y}e_{x''}=e_{x'}e_{x}e_{x''}=e_{x}e_{x'}e_{x''}$.
\end{proof}

Based on Part i) of the above lemma, we can use the following notation.

\begin{nota}
	If $\cF$ is an $e$-facet of $\BT'$, we set $e_{\cF}:=e_x$ for any $x \in \cF$.
\end{nota}

By ii) of the last lemma, we have $e_{\FC}e_{\FC'}=e_{\FC}$ for any two $e$-facets $\FC,\FC'$ such that
$\FC'\subset\o\FC$.

We now check that the proof of Theorem 2.4 of \cite{MS1} can be adapted to
our setting.   Let $V\in\Rep_{R}(G)$ be a smooth $RG$-module. It defines a coefficient
system $\FC\mapsto \VC(\FC):=e_{\FC}V$ over $\BT'_{\bullet/e}$, in which the transition maps
$e_{\FC}V\To{}e_{\FC'}V$ for $\FC'\subset\o\FC$ are inclusions, thanks to the
above identities $e_{\FC}e_{\FC'}=e_{\FC}$. After choosing an orientation of
$\BT'_{\bullet/e}$ we may form the cellular chain complex $\CC(\BT'_{\bullet/e},\VC)$, whose
homology we denote by $H_{*}(\BT'_{\bullet/e},\VC)$. More generally, for any 
polysimplicial subcomplex $\Sigma$ of $\BT'_{\bullet/e}$ we have a chain complex and its
homology $H_{*}(\Sigma,\VC)$. We refer e.g. to \cite[\S 1.1.2]{MS1} for a brief introduction to these cellular chain complexes.

\begin{lemme}[{cf.\ \cite[Theorem~2.4]{MS1}}]\label{lemma_MS_e_subdivided}
For any convex polysimplicial subcomplex $\Sigma$ of $\BT'_{\bullet/e}$, we have 
$H_{0}(\Sigma,\VC)=\sum_{x\in \Sigma_{0}}e_{x}V$ and $H_{n}(\Sigma,\VC)=0$ for $n>0$.
\end{lemme}
\begin{proof}
  We review the different steps of Meyer and Solleveld's proof of Theorem~2.4 in \cite{MS1}.

\emph{Step 1.} Prove it when $\Sigma$ is a polysimplex. The argument below Lemma 2.18
of \emph{loc.\ cit.}
relies directly on Properties i), ii) and iii) of Lemma \ref{consistent_system} 
and works without any change.

\emph{Step. 2.} Divide and conquer method: Suppose $\Sigma$ is \emph{finite} and is the union of two convex
subcomplexes $\Sigma_{+}$ and $\Sigma_{-}$ with convex intersection $\Sigma_0$. Then if the
statement holds for $\Sigma_{+}$, $\Sigma_{-}$ and $\Sigma_{0}$, it holds for $\Sigma$.
This reduction step follows from Theorem 2.12 of \emph{loc.\ cit.}, which asserts that the
distribution $e_{\Sigma}:=\sum_{\FC\subset \Sigma}(-1)^{\dim(\FC)}e_{\FC}$ is an
idempotent such that $e_{\Sigma}e_{\FC}=e_{\FC}e_{\Sigma}=e_{\FC}$ for all $\FC\subset\Sigma$. This
theorem in turn follows from Lemmas 2.8 and 2.9 and Proposition 2.2 of \emph{loc.\ cit.}
But, provided Properties i), ii) and iii) of Lemma \ref{consistent_system}, all these statements are concerned with the geometry of combinatorial convex hulls in
an apartment, hence they still hold for any  subdivision as in our case.

\emph{Step. 3.} Prove that if $\Sigma$ is \emph{finite} and not a polysimplex, then it can be split as in
Step 2. Here the argument has to be complemented a bit. Suppose first that $\Sigma$ is
contained in an apartment. Then there is an $e$-wall whose two associated open half-spaces
intersect $\Sigma$ non trivially. Simply take $\Sigma_{\pm}$ to be the intersection with
the closed half spaces, and $\Sigma_{0}$ the intersection with the wall. Now suppose that
$\Sigma$ is not contained in a single apartment. Then  we can find a $1$-chamber $\Delta$
whose closure intersects $\Sigma$ non-trivially but does not contain it. Pick an apartment $A$ that
contains $\Delta$ and a wall of $A$ that supports a face of $\Delta$ and intersects $\Sigma$
non-trivially.  It corresponds to some affine
root $a$ and we can use the retraction on $A$ centered at $\Delta$ exactly as on p.140 of
\emph{loc.cit.} 

\emph{Step. 4.} The statement is now known when $\Sigma$ is finite. It follows in the case that $\Sigma$ is
infinite by writing $\Sigma$ as the union of an increasing sequence of finite convex
subcomplexes $\Sigma_{n}$, which can always be done. Indeed, the chain complex $\Sigma$ is
the direct limit of the chain complexes of the $\Sigma_{n}$.
\end{proof}

\alin{Proof of Theorem \ref{thm_Serresubcat}}
If $x\in\BT'$, there is an $e$-vertex $y$ that lies in the closure of the $e$-facet
$\FC_{e}(x)$. Then it follows from Proposition~\ref{idemp_commut} that $e_{x}=e_{y}e_{x}$.
Thus we see
that $V\in\Rep_{R}^{\phi,I}(G)$ if and only if $V=\sum_{x\in\BT_{0/e}} e_{\phi,I,x}V$. 
Therefore,
thanks to the case $\Sigma=\BT'_{\bullet/e}$ of the previous lemma, the proof of Theorem 3.1 of
\cite{MS1} adapts verbatim to show that the category $\Rep_{R}^{\phi,I}(G)$ is a Serre
category that is stable under arbitrary colimits and generated as claimed in the theorem. 
Further let $V$ be any smooth $RG$-module and put $V^{\phi,I}:=
\sum_{x\in\BT}e_{x}V$. We certainly have $V^{\phi,I}\in\Rep^{\phi,I}_{R}(G)$, and
we see that the quotient $V_{\phi,I}:=V/V^{\phi,I}$ is killed by all $e_{x}$ so that no
non-zero subquotient of $V_{\phi,I}$ belongs to $\Rep^{\phi,I}_{R}(G)$.

\begin{rema}\label{esigma}
  Let $(\Sigma_{n})_{n\in\NM}$ be an increasing sequence of convex polysimplicial
  subcomplexes of $\BT'$ such that $\BT'=\bigcup_{n}\Sigma_{n}$. We have already recalled
  that $e_{\Sigma_{n}} := \sum_{\FC\subset\Sigma_{n}}(-1)^{\dim(\FC)}e_{\FC}$ is an idempotent such that
  $e_{\Sigma_{n}}e_{\FC}=e_{\FC}e_{\Sigma_{n}}=e_{\FC}$ for all
  $\FC\subset\Sigma_{n}$. It follows that $e_{\Sigma_{n-1}}e_{\Sigma_{n}}=e_{\Sigma_{n}}$
  and that for all $V\in\Rep_{R}(G)$ we have
$$ \sum_{x\in\BT}e_{x} V= \bigcup_{n}e_{\Sigma_{n}}V.$$
\end{rema}

\subsection{Some properties of $\Rep^{\phi,I}_{R}(G)$}\label{sec:some-prop-repphi}

\alin{The direct factor problem}
We strongly believe that the category $\Rep^{\phi,I}_{R}(G)$ is
actually a direct factor of $\Rep_{R}(G)$, but this does not follow
formally from the work of Meyer--Solleveld, nor from the one of Yu.
The problem is to show that the extension  (\ref{exactseq}) splits, and more
precisely that the 
subspace $\bigcap_{x}\ker(e_{x}|V)$ of $V$ maps onto $V_{\phi,I}$.

What is missing is the existence of injective cogenerators in $\Rep^{\phi,I}_{R}(G)$, or
equivalently, of sufficiently many projective representations killed by all $e_{\phi,I,x}$.
It seems that in order to bypass these problems, one needs some exhaustion result, e.g., as
the ones proved by Kim \cite{Kim} and the second-named author \cite{JF_exhaust}, or the one provided by the
Bushnell--Kutzko--Stevens type theory. However, if one restricts attention to admissible
objects over a complete local ring, then a duality trick implies the desired splitting.

\begin{pro}
  Let $\RC$ be a complete local commutative $R$-algebra, and let $V$ be an  admissible smooth
  $\RC G$-module (meaning that for any open, compact subgroup
$H$ of $G$, the $\RC$-module $V^{H}$ is noetherian). Then the extension (\ref{exactseq}) splits. In other words, the admissible
  category ${\rm Adm}_{\RC}^{\phi,I}(G)$ is a direct factor of   ${\rm Adm}_{\RC}(G)$.
\end{pro}
\begin{proof} 
Let $\EC$ be a Matlis module over $\RC$ (i.e., an injective hull of the residue field), and
extend the  Matlis duality functor to smooth $\RC G$-modules
  by putting  $V^{*}:=\Hom_{\RC}(V,\EC)^{\infty}$. Since $p$ is invertible in $\RC$,
  we have $(V^{*})^{H}=\Hom_{\RC}(V^{H},\EC)$ for any open
  pro-$p$-subgroup. Therefore, the usual Matlis duality theorem  for
  noetherian $\RC$-modules implies that for an  admissible $\RC G$-module,
 the canonical map $V\To{}V^{**}$ is  an isomorphism, and induces isomorphisms $W\mapsto
 (W^{\perp})^{\perp}$ for each $\RC G$-submodule $W$ of $V$. 

Now, let $e_{x}^{*}$ be the image of $e_{x}$ by the anti-involution $g\mapsto g^{-1}$ on
$\HC_{R}(G)$.  Note that the system of idempotents $(e_{x}^{*})_{x\in\BT_{0/e}}$ is the
one attached to the pair $(\bar\phi,I)$ where $\bar{ }$ denotes the automorphism of
$^{L}{\mathbf{G}}$ induced by complex conjugation. Let us put $eV:=\sum_{x\in\BT}e_{x}V$ and
$e^{*}V^{*}:=\sum_{x\in\BT} e_{x}^{*}V^{*}$. Then we see that
$(e^{*}V^{*})^{\perp}=\bigcap_{x\in\BT}\ker(e_{x}|V)$ and
$(eV)^{\perp}=\bigcap_{x\in\BT}\ker(e_{x}^{*}|V^*)$. By biduality we have
$(\bigcap_{x\in\BT}\ker(e_{x}|V))^{\perp}=e^{*}V^{*}$ and it follows that $V=eV\oplus
\bigcap_{x\in\BT}\ker(e_{x}|V)$ as desired.
\end{proof}

\begin{prop}[Disjonction] \label{disjonction}
 Let $(\phi',I')$ and $(\phi,I)$ be two distinct pairs as in
 \ref{thm_Serresubcat}.
Then the categories $\Rep^{\phi',I'}_{R}(G)$ and
$\Rep^{\phi,I}_{R}(G)$ are orthogonal in the sense that for all
objects $V\in \Rep^{\phi,I}_{R}(G)$ and $V'\in
\Rep^{\phi',I'}_{R}(G')$ we have 
$ \Ext^{*}_{RG}(V,V')=\Ext^{*}_{RG}(V',V)=\{0\}.$
\end{prop}
\begin{proof}
 By (\ref{intertw}) we have
$e_{\phi',I',x'}e_{\phi,I,x}=0$  for all $x',x\in\BT$. It follows that
$e_{\phi',I',x'}V=0$ for all $V\in \Rep^{\phi,I}_{R}(G)$, hence also $\Hom_{RG}(V',V)=0$
for all $V'\in\Rep^{\phi',I'}_{R}(G')$. Using projective resolutions inside
$\Rep^{\phi',I'}_{R}(G')$, we also obtain that $\Ext^{*}_{RG}(V',V)=0$.
\end{proof}

\alin{The ``essentially tame'' subcategory} 
We introduce the full subcategory 
$$ \Rep_{R}^{\rm et}(G) := \left\{ V\in \Rep_{R}(G), V
  =\sum_{\phi,\iota,x} e_{\phi,\iota,x}V\right\}.$$
Here ``et'' stands for ``essentially tame'' in order to stick to
Bushnell and Henniart's terminology, although we fear that this is a bit
misleading. Morally (and under the hypotheses \eqref{H1} and \eqref{H2}), this
subcategory should capture all the representations 
associated to Langlands parameters that are trivial on the derived
subgroup $[P_{F},P_{F}]$ of the wild inertia subgroup $P_F$. By construction, this  is a Serre
subcategory generated by projective objects and closed under arbitrary
colimits, and Proposition \ref{disjonction} tells us that it decomposes as a direct
product
$$ \Rep_{R}^{\rm et}(G) =\prod_{(\phi,I)} \Rep^{\phi,I}_{R}(G).$$

As mentioned above, we don't know in general whether it is a direct factor subcategory, but we
will prove that under a further hypothesis it is the entire category.
Recall that our construction of $\Rep^{\phi,I}_{R}(G)$ applies to any
$\phi$ such that $C_{\hat{\mathbf{G}}}(\phi)$ is a Levi subgroup of $\hat{\mathbf{G}}$,
 under the hypothesis that 
 $p$ is odd, and not a torsion prime of $\mathbf{G}$, nor of $\hat{\mathbf{G}}$.
With a further (and much stronger) hypothesis, one can actually ensure that
 \emph{all} centralizers $C_{\hat{\mathbf{G}}}(\phi)$ for
$\phi\in\Phi(P_{F},{\mathbf{G}})$ are Levi subgroups.  Indeed, by
Lemma \ref{centralizer_Levi} iv), this is the case whenever $p$ does not divide the order of the absolute  Weyl group of
   ${\mathbf{G}}$. 

Under this hypothesis, it is thus natural to expect that all
representations are essentially tame in the sense introduced  above. This is indeed a consequence
of the exhaustion results in \cite{JF_exhaust} and \cite{Fi-mod-ell}.

\begin{theo}\label{exhaust}
  Suppose that $p$ does not divide the order of the absolute  Weyl group of
  ${\mathbf{G}}$. 
Then $\Rep^{\rm et}_{R}(G)=\Rep_{R}(G)$.
\end{theo}
\begin{proof}
	In order to prove the theorem, we have to show that for every $V\in \Rep_{R}(G)$ there exists a wild inertia parameter $\phi$, an embedding $\iota \in I_\phi$ and $y \in \BT_\iota$ such that  $e_{\phi,\iota,y}V \neq 0$. We will deduce this from a similar result obtained by \cite{JF_exhaust} and \cite{Fi-mod-ell} when proving that the types constructed by Kim and Yu produce types for all Bernsteins blocks. We first recall the relevant result translated to our setting. 
Observe that in order to apply Yu's construction described in
\ref{Yu_construction}, all that is needed is a triple $(\vec
{\mathbf{G}}, \vec \psi, x)$, that we will refer to as a ``truncated Yu datum'', in which $\vec {\mathbf{G}}=({\mathbf{G}}^{0}\subset\hdots\subset
{\mathbf{G}}^{d})$ is a tame twisted Levi sequence, $x\in \BT$ lies in the image of an admissible embedding
$\BT({\mathbf{G}}^{0},F)\injo\BT$, and $\vec\psi=(\psi_{i})_{i=0,\hdots d}$ is a collection
of characters\footnote{In the literature the standard notation for these characters is
  $\phi_{i}$ but in this paper the letter $\phi$ has already been dedicated to parameters.} $\psi_{i}: {\mathbf{G}}^{i}(F)\To{} \CM^{\times}$ such that 

-- the depths of $\psi_0, \hdots, \psi_{d-1}$ form an increasing sequence $r_{-1}:=0<r_{0}<\hdots < r_{d-1}$, 

-- either $\psi_{d}=1$ (we then put $r_{d}=r_{d-1}$) or $\psi_{d}$ has depth $r_{d} > r_{d-1}$.  

\noindent To such a truncated Yu datum, the procedure of \ref{Yu_construction} attaches  a pair
$(\Kp_x ,\hat\psi_{x})$
consisting of an open pro-$p$-subgroup  and a character of this subgroup. This pair
defines in turn an idempotent $e(\vec{\mathbf{G}}, \vec \psi, x)$ in $\HC_{R}(G)$, supported
on $\Kp_x$.  

Analogous to the discussion at the beginning of the proof of \cite[Theorem~4.1]{Fi-mod-ell}, using the observation that the images of characters of pro-$p$ groups factor through $\mu_p^{\infty} \to{} R^\times$ and using the divisibility of $\mu_p^{\infty}$ to extend characters, we may apply the arguments of the proofs of \cite[Theorem~6.1]{JF_exhaust} (whose weaker form, the containment of a truncated datum in the notation of \textit{loc.\ cit.}, suffices here) and \cite[Lemma~7.6]{JF_exhaust} together with \cite[Lemma~7.3]{JF_exhaust} to obtain the following result:
  \emph{For any $RG$-module $V$, there is a normalized, generic
    truncated  Yu datum
    $(\vec{\mathbf{G}},\vec\psi,x)$ such that $e(\vec{\mathbf{G}},\vec\psi,x)V\neq 0$.}

Here,  a (truncated) Yu datum is called \emph{generic} if for all $i<d$, the restriction
$(\psi_{i})|_{G^{i}_{x,r_{i}}}$ is represented by a ${\mathbf{G}}^{i+1}$-generic element (as
in \ref{generic}), and is called \emph{normalized} if  $\psi_{i}$ is trivial on $({\mathbf{G}}_{i})_{\rm
	sc}(F)$ for all $i$. With this result in hand, we are left to prove that for any normalized, generic, truncated Yu
datum there is a wild inertia parameter $\phi$, an embedding $\iota\in I_{\phi}$ and a point $y\in
\BT_{\iota}$ such that $e_{\phi,\iota,y}=e(\vec{\mathbf{G}},\vec\psi,x)$. This follows from the next lemma.
\end{proof}

\begin{lemme} \label{from_Yu_data_to_parameters}
Here we only assume that $p$
 is odd and not a torsion prime of $\hat{\mathbf{G}}$ (weakening a bit our usual assumptions). 
  Let $(\vec{\mathbf{G}},\vec\psi,x)$ be a normalized, generic truncated Yu datum. Then there is a
  pair $(\phi, \iota)$ consisting of a wild inertia parameter $\phi$ of $\mathbf{G}$ and an
  $F$-rational embedding $\iota \in I_{\phi}$ such that $\mathbf{G}^0_\iota=\mathbf{G}^0$,
  $\Kp_{\iota, x}=\Kp_{x}$, and $\check\phi^{+}_{\iota,  x}=\hat \psi_x$.  
\end{lemme}
\begin{proof}
 The character
 $\check\varphi_{0}:=\prod_{k=0}^{d}  (\psi_{k})|_{{\mathbf{G}}^{0}}$ of ${\mathbf{G}}^{0}(F)$ is
 trivial on $({\mathbf{G}}^{0})_{\rm  sc}(F)$ since the datum is normalized.
By Remark \ref{Borel_proc},  $\check\varphi_{0}$ comes
from some element $\hat\varphi_{0}\in Z^{1}(W_{F},Z(\hat{\mathbf{G}}^{0}))$.

Let us choose a tamely ramified  maximal torus $\mathbf{S}$ in ${\mathbf{G}}^{0}$. There is a
canonical $W_{F}$-equivariant embedding 
$Z(\hat{\mathbf{G}}^{0})\injo \hat{\mathbf{S}}$ that allows us to pushforward $\hat\varphi_{0}$
into $Z^{1}(W_{F},\hat{\mathbf{S}})$, giving the Langlands parameter
$\varphi_{0}:=\hat\varphi_{0}\rtimes \id_{W_{F}}:W_{F}\To{}{^{L}\mathbf{S}}$
of the character $(\check\varphi_{0})|_{\mathbf{S}(F)}$.
Consider the
$W_{F}$-stable conjugacy class of embeddings $\hat{\mathbf{S}}\injo \hat{\mathbf{G}}$ that is dual
to $\mathbf{S}_{\ov F} \to{} \mathbf{G}_{\ov F}$. Any 
such embedding $\kappa: \hat{\mathbf{S}}\injo \hat{\mathbf{G}}$ can be extended to a tamely ramified $L$-embedding
$^{L}\kappa:\,{^{L}}\mathbf{S}\injo {^{L}}{\mathbf{G}}$ (see 
\cite[Lemma 5.2.6]{Kaletha}). We put $\varphi:={^{L}\kappa}\circ \varphi_{0} :
W_{F}\To{}{^{L}\mathbf{G}}$ and $\phi:= \varphi|_{P_{F}}$.

By construction, $\phi(P_{F})$ is contained in $\kappa(\hat{\mathbf{S}})$, hence it is abelian.
Since $p$ is not a torsion prime for $\hat{\mathbf{G}}$, the centralizer
$C_{\hat{\mathbf{G}}}(\phi)$ is a Levi subgroup of
$\hat{\mathbf{G}}$ by Lemma \ref{centralizer_Levi} iii). As usual, we write
$\hat{\mathbf{S}}_{\phi}:=C_{\hat{\mathbf{G}}}(\phi)_{\rm ab}$. From $\varphi$, we get an action of
$W_{F}$ on $\hat{\mathbf{S}}_{\phi}$, and the 
embedding $\kappa: \hat{\mathbf{S}}\injo C_{\hat{\mathbf{G}}}(\phi)$ induces a $W_{F}$-equivariant
epimorphism of tori
$\hat{\mathbf{S}}\twoheadrightarrow\hat{\mathbf{S}}_{\phi}$.
The dual of this epimorphism then provides an $F$-rational
embedding $\iota :\, \mathbf{S}_{\phi}\injo \mathbf{S}\injo {\mathbf{G}}$, whose mere existence
implies that $\phi$ is a wild inertia parameter of $\mathbf{G}$, by  Proposition
\ref{dual_embeddings_phi} ii).

What we have done so far is to associate a pair $(\phi,\iota)$ to any normalized truncated Yu datum
$(\vec{\mathbf{G}},\vec\psi,x)$. We now use that our given truncated Yu datum is also generic, 
and we choose a tame extension $E$ that splits
$\mathbf{S}$. Then we claim that for any $r\in \RM$ such that $r_{i-1}<r\leq r_{i}$ we have 
$$ \left\{\alpha\in \Sigma(\mathbf{S},{\mathbf{G}}), 
\check\varphi_{0}(N_{E/F}(\alpha^{\vee}(E^{\times}_{r})))=\{1\}\right\}=\Sigma(\mathbf{S},{\mathbf{G}}^{i}).$$
Indeed, the case $i=0$ is addressed in the proof of Lemma 3.6.9 of \cite{Kaletha} and the
same proof applies to $i<d$. As in the proof of Lemma \ref{rootsystemCr}, it follows that the identification between
$\Sigma(\mathbf{S},{\mathbf{G}})$ and $\Sigma(\hat{\mathbf{S}},\hat{\mathbf{G}})^{\vee}$
identifies $\Sigma(\mathbf{S},{\mathbf{G}}^{i})$ and 
$\Sigma(\hat{\mathbf{S}},C_{\hat{\mathbf{G}}}(\phi(I_{F}^{r})))^{\vee}$ for $r_{i-1}<r\leq
r_{i}$. 
In other words we have
${\mathbf{G}}^{i}= C_{\mathbf{G}}(\iota(\mathbf{S}_{\phi,r}))$ and, in particular
${\mathbf{G}}^{i}={\mathbf{G}}_{\iota}^{i}$. It also follows that $\Kp_{\iota,x}=K^{+}_{x}$ for any
point $x\in \BT_{\iota}$.

Finally, for each $i$, the character
$\check\varphi_{i}:=\prod_{k=i}^{d}  (\psi_{k})|_{{\mathbf{G}}_{i}(F)}$ of ${\mathbf{G}}_{i}(F)$ is
also trivial on $({\mathbf{G}}_{i})_{\rm  sc}(F)$ for all $i$, hence comes from some
some element $\hat\varphi_{i}\in Z^{1}(W_{F},Z(\hat{\mathbf{G}}^{i}))$ and, by construction,
we have
$(\hat\varphi_{0})|_{I_{F}^{r_{i}}} = (\hat\varphi_{i})|_{I_{F}^{r_{i}}}$ in
$H^{1}(I_{F}^{r_{i}},Z(\hat{\mathbf{G}}_{0}))$.
We find ourselves in the setting of \ref{def_characters} and we infer that
$\check\phi^{+}_{\iota,  x}=\hat \psi_x$.  
\end{proof}

\alin{Compatibility with isogenies}
To any isogeny $f:\,{\mathbf{G}}'\To{}{\mathbf{G}}$ is associated a canonical conjugacy class of dual isogenies $\hat f:\,
\hat{\mathbf{G}}\To{}\hat{\mathbf{G}}'$.  Moreover, if $f$ is defined over $F$, then any such
dual isogeny can be  extended to a morphism of
$L$-groups $^{L}{\mathbf{G}}\To{}{^{L}{\mathbf{G}}'}$. We thus get a well defined transfer map
$f^{*}:\,\Phi(P_{F},{\mathbf{G}})\mapsto \Phi(P_{F},{\mathbf{G}}')$. Note that our hypothesis
\eqref{H1} and \eqref{H2} hold for ${\mathbf{G}}'$ since  they are assumed to hold for ${\mathbf{G}}$. 
If we fix a dual isogeny $\hat f$ and a morphism $\hat\phi :P_{F}\To{}\hat{\mathbf{G}}$ such that $\phi = \hat \phi \times \id$ is a wild inertia parameter, 
then we obtain 
an isogeny of Levi subgroups $C_{\hat{\mathbf{G}}}(\phi)=C_{\hat{\mathbf{G}}}(\hat\phi)\To{}C_{\hat{\mathbf{G}}'}(\hat
f\circ\hat\phi)$ with kernel $\ker(\hat f)$, which dually provides a  conjugacy class of  $F$-rational isogenies
${\mathbf{G}}_{f^{*}\phi}\To{} {\mathbf{G}}_{\phi}$ with kernel 
$\ker(f)$ together with an $F$-rational isogeny $\mathbf{S}_{f^{*}\phi}
\To{}\mathbf{S}_{\phi}$. Now 
 any $\iota\in I_{\phi}$ induces an isomorphism $\mathbf{S}_{\phi}\simto
 Z({\mathbf{G}}_{\iota})^{\circ}$, and the discussion in \ref{dual_embeddings} shows that
 this
 isomorphism lifts uniquely to an
 isomorphism $\mathbf{S}_{f^{*}\phi}\simto Z(f^{-1}({\mathbf{G}}_{\iota}))^{\circ}$, thus
 providing an element $f^{*}\iota\in I_{f^{*}\phi}$. In this way we obtain a bijection
 $I_{\phi}=I_{f^{*}\phi}$ that respects $F$-rationality, but the $G'$-conjugacy is a
 priori coarser than the $G$-conjugacy. 

 \begin{pro}
  In this setting, the pull-back functor $f^{*}:\,\Rep_{R}(G)\To{}\Rep_{R}(G')$ takes
  $\Rep_{R}^{\phi,I}(G)$ into $\prod_{I'\subseteq f^{*}I}\Rep_{R}^{f^{*}\phi,I'}(G')$.
 \end{pro}
 \begin{proof}
Let us identify the Bruhat--Tits buildings of ${\mathbf{G}}$ and $\mathbf{G'}$. Then, from the
definitions we see that $\BT_{\iota}=\BT_{f^{*}\iota}$ for any $\iota\in
I_{\phi}$. Moreover, for $x\in\BT_{\iota}$, we have a surjection 
${\vG'}^{+}_{f^{*}\phi,f^{*}\iota,x}\twoheadrightarrow \Kp_{\phi,\iota,x}$ and the construction of the
characters shows that $\check{(f^{*}\phi)}_{f^{*}\iota,x}^{+}$ is the pull back of
$\check\phi^{+}_{\iota,x}$. The claim then follows from the definition of the categories
under consideration. 
 \end{proof}

\newcommand{\phiM}{{\phi_{\scriptscriptstyle M}}}
\newcommand{\iotaM}{\iota_{\scriptscriptstyle M}}
\newcommand{\IM}{I_{\scriptscriptstyle M}}
\def\vM{{K}}

\alin{Compatibility with parabolic induction} \label{compat_parab_induc}
Suppose that $\phi$ comes from an $F$-Levi subgroup ${\mathbf{M}}$ of
${\mathbf{G}}$, in the sense that there is a dual embedding ${^{L}{\mathbf{M}}}\To{}{^{L}{\mathbf{G}}}$
and a factorization 
$\phi: P_{F}\To{\phiM} {^{L}{\mathbf{M}}}\To{}{^{L}{\mathbf{G}}}$, where $\phiM$ is a
wild inertia parameter of ${\mathbf{M}}$ (recall that this means that
 $\phiM$ admits an extension to a 
relevant Langlands parameter $W'_{F}\To{}{^{L}{\mathbf{M}}}$).  

In this context we have
$\hat{\mathbf{M}}_{\phiM}=\hat{\mathbf{M}}\cap\hat{\mathbf{G}}_{\phi}$, whence a
$W_{F}$-equivariant surjection $\hat{\mathbf{S}}_{\phiM}\twoheadrightarrow
\hat{\mathbf{S}}_{\phi}$, which on the dual side induces an injection of $F$-tori
$\mathbf{S}_{\phi}\injo \mathbf{S}_{\phiM}$. It follows that any $F$-rational
Levi-center-embedding $\iotaM:\,\mathbf{S}_{\phiM}\injo{\mathbf{M}}$ in the set
$I_{\phiM}$ induces an embedding $\iota:\,\mathbf{S}_{\phi}\injo{\mathbf{G}}$ in the set
$I_{\phi}$. Obviously, $M$-conjugate embeddings lead to $G$-conjugate embeddings, so that
any choice of an $M$-conjugacy class $\IM\subset I_{\phiM}$ points to a $G$-conjugacy
class $I\subset I_{\phi}$.  We will use the notation $(\phi_{M},I_{M})\mapsto (\phi,I)$ to express
the fact that $\phi$
comes from $\phi_{M}$ and $I$ from $I_{M}$.

\begin{thm}
  Let ${\mathbf{P}} $ be a parabolic
  $F$-subgroup of ${\mathbf{G}}$ with Levi component ${\mathbf{M}}$, and denote by $i_{P}$
  the associated (not normalized) parabolic induction functor, and by $r_{P}$ the corresponding Jacquet functor.
  \begin{enumerate}
  \item If the pair $(\phi,I)$ comes from the pair $(\phi_{M},I_{M})$, then  we have 
$$ i_{P}\left(\Rep^{\phiM,\IM}_{R}(M)\right) \subseteq \Rep^{\phi,I}_{R}(G).$$ 
\item Assuming that $p$ does not divide the order of the absolute Weyl group of ${\mathbf{G}}$, we have
$$r_{P}(\Rep^{\phi,I}_{R}(G))\subseteq \prod_{(\phiM,\IM)\mapsto
      (\phi,I)} \Rep^{\phiM,\IM}_{R}(M).$$
\item Assuming further that $C_{\hat{\mathbf{G}}}(\phi)\subseteq \mathbf{\hat{\mathbf{M}}}$, then $i_{P}$
    induces an equivalence of categories 
$$\Rep_{R}^{\phiM,\IM}(M)\simto\Rep_{R}^{\phi,I}(G)$$
with quasi inverse the composition of $r_{P}$ and
    the projection onto the $(\phiM,\IM)$-factor.
  \end{enumerate}
\end{thm}

\ali \label{setup_par_ind}
We start with some preliminary observations to prepare for proving the theorem. First note that the equality
$\hat{\mathbf{M}}_{\phiM}=\hat{\mathbf{M}}\cap\hat{\mathbf{G}}_{\phi}$
provides two surjective maps
$\hat{\mathbf{S}}_{\phiM} \To{} \hat{\mathbf{S}}_{\phi}$
and $\hat{\mathbf{S}}_{\phiM} \To{} \hat{\mathbf{M}}_{\rm ab}$ whose
product $\hat{\mathbf{S}}_{\phiM} \To{}
\hat{\mathbf{S}}_{\phi}\times\hat{\mathbf{M}}_{\rm ab}$ has finite
kernel (since $Z(\hat{\mathbf{M}}_{\phiM})^{\circ}=Z(\hat{\mathbf{M}})^{\circ}Z(\hat{\mathbf{G}}_{\phi})^{\circ}$).
Dually, this provides two inclusions $\mathbf{S}_{\phi}\subset \mathbf{S}_{\phiM}$ and
$Z({\mathbf{M}})^{\circ}\subset \mathbf{S}_{\phiM}$
such that
$\mathbf{S}_{\phiM}=\mathbf{S}_{\phi}Z({\mathbf{M}})^{\circ}$. In particular, if $\iotaM\in
\IM$ induces $\iota\in I_{\phi}$ through the first embedding, we have
$\iotaM(\mathbf{S}_{\phiM})=\iota(\mathbf{S}_{\phi})Z({\mathbf{M}})^{\circ}$ and therefore 
${\mathbf{M}}_{\iotaM}={\mathbf{M}}\cap{\mathbf{G}}_{\iota}$.

Denote by $\BT_{M}$ the image of any admissible embedding
$\BT({\mathbf{M}},F)\injo\BT({\mathbf{G}},F)$.
Then we have $\BT_{\iotaM}=\BT_{\iota}\cap\BT_{M}$.
Now, given a point $x\in \BT_{\iotaM}$, the
constructions of the previous subsections provide us with two idempotents $e_{\iota,x}\in\HC_{R}(G)$
and $e_{\iotaM,x}\in\HC_{R}(M)$ respectively associated to pairs
$(\Kp_{\iota,x},\check\phi^{+}_{\iota,x})$ and $(\Kp_{\iotaM,x},\check\phi^{+}_{\iotaM,x})$
consisting of an open pro-$p$-group of $G$, resp., $M$, and a character of that group.
In order to compare these two idempotents, it is useful to enlarge the ring $\HC_{R}(G)$ to the ring
$RG$ consisting of all compactly supported distributions on $G$ (not necessarily locally
constant). We will use the analogous notation for any open, compact subgroup of $G$ as introduced in Section \ref{sec:notation}. In particular, we may
regard $e_{\iota,x}$ and $e_{\iotaM,x}$ as idempotents in the ring $RG_{x}$ of all distributions
on the stabilizer $G_{x}$ of $x$.
This ring also contains the averaging idempotents $e_{U_{x}}$,  $e_{U_{x,0+}}$, $e_{\Uu_{\iota,x}}$ and
$e_{\Up_{\iota,x}}$ of the closed
pro-$p$-subgroups $U_{x}:=U\cap G_{x}$,  $U_{x,0+}:=U\cap G_{x,0+}$,
$\Uu_{\iota,x}:=U\cap\Ku_{\iota,x}$ and $\Up_{\iota,x}:=U\cap\Kp_{\iota,x}$ of 
$G_{x}$, where ${\mathbf{U}}$ 
denotes the unipotent radical of ${\mathbf{P}} $. Similarly, we have
idempotents $e_{\bar U_{x}}$, $e_{\bar U_{x,0+}}$, $e_{\bUu_{\iota,x}}$  and $e_{\bUp_{\iota,x}}$,
where $\bar{\mathbf{U}}$
is the unipotent radical of  the opposite parabolic subgroup $\bar{\mathbf{P}} $ of ${\mathbf{P}} $ with
respect to ${\mathbf{M}}$.  

The following lemma contains the main  technical points of the proof of the above theorem. 

\begin{lemme} \label{lemma_ind_parab}
 For any $x\in\BT_{\iotaM}$ and with the foregoing notation, we have:
  \begin{enumerate}
  \item \label{item:Iwahori}
    $e_{\iota,x}=e_{\Up_{\iota,x}}e_{\iotaM,x}
    e_{\bUp_{\iota,x}}=e_{\Up_{\iota,x}}e_{\bUp_{\iota,x}}e_{\iotaM,x}=
e_{\bUp_{\iota,x}}e_{\iotaM,x}e_{\Up_{\iota,x}}$ in $R\Kp_{\iota,x}$.
\item \label{item:non-degenerate1}
  $(e_{\Uu_{\iota,x}}e_{\bUu_{\iota,x}}e_{\Uu_{\iota,x}})e_{\iota,x} =
  |\Uu_{\iota,x}/\Up_{\iota,x}|^{-1}e_{\Uu_{\iota,x}}e_{\iota,x}$ 
\item \label{item:non-degenerate2}
  $e_{\iota,x}\in R\Ku_{\iota,x}\,e_{\Uu_{\iota,x}}e_{\iotaM,x}e_{\bUu_{\iota,x}}R\Ku_{\iota,x}$
\item \label{item:HL-like}
  $e_{U_{x,0+}}e_{\bar U_{x}} e_{\iota_{M},x}\in RG_{x} e_{U_{x}}e_{\bar U_{x}}e_{\iota_{M},x}$  and 
   $e_{\iota_{M},x}e_{U_{x}}e_{\bar U_{x,0+}}\in  e_{\iota_{M},x}e_{U_{x}}e_{\bar U_{x}}
   RG_{x}$.
 \item \label{item:intertwining}
   $e_{\iota,x}e_{U_{x}}e_{\iota,x}= |U_{x}/U_{\iota,x}|^{-1}e_{U_{\iota,x}}e_{\iota,x}$ and
   $e_{\iota,x}e_{U_{x,0+}}e_{\iota,x}=|U_{x,0+}/\Uu_{\iota,x}|^{-1}e_{\Uu_{\iota,x}}e_{\iota,x}$. 
 \item \label{item:centralizer}
   If $C_{\hat{\mathbf{G}}}(\phi)\subseteq \hat{\mathbf{M}}$, then
   $U_{\iota,x}=\Uu_{\iota,x}$ and $\bar U_{\iota,x}=\bUu_{\iota,x}$.
  \end{enumerate}
\end{lemme}
\begin{proof}
  \ref{item:Iwahori}) Since $\iota(\mathbf{S}_{\phi})\subset\iotaM(\mathbf{S}_{\phi_{M}})\subset {\mathbf{M}}$, we may
  apply \ref{Iwahori} on the Iwahori factorization of $(\Kp_{\iota,x},\check\phi^{+}_{\iota,x})$.
  Denoting by $e^{M}_{\iota,x}$ the idempotent of $RG_{x}$ associated to the character
  $(\check\phi^{+}_{\iota,x})|_{M\cap\Kp_{\iota,x}}$, this lemma implies the following
  decomposition in $RG_{x}$ (actually in $R\Kp_{\iota,x}$) :
  $$e_{\iota,x}=e_{\Up_{\iota,x}}e_{\iota,x}^{M}e_{\bUp_{\iota,x}}
  =e_{\Up_{\iota,x}}e_{\bUp_{\iota,x}}e_{\iota,x}^{M}=
  e_{\bUp_{\iota,x}}e_{\iota,x}^{M}e_{\Up_{\iota,x}}.$$
  So it remains to see that $e_{\iotaM,x}=e_{\iota,x}^{M}$. Equivalently, we need to show that
  \begin{center}
    $M\cap\Kp_{\iota,x} = \KpM_{\iotaM,x}$ and
    $(\check\phi^{+}_{\iota,x})|_{M\cap\Kp_{\iota,x}}=\check\phi^{+}_{\iotaM,x}$.
  \end{center}
  
  Let us first check the equality of groups. We have seen in \ref{setup_par_ind} that   $\iotaM(\mathbf{S}_{\phiM})=\iota(\mathbf{S}_{\phi})Z({\mathbf{M}})^{\circ}$ and therefore 
  ${\mathbf{M}}_{\iotaM}={\mathbf{M}}\cap{\mathbf{G}}_{\iota}$. 
Similarly, for any $r>0$ we have
$\iotaM(\mathbf{S}_{\phiM,r})= \iota(\mathbf{S}_{\phi,r})Z({\mathbf{M}})^{\circ}$ and
therefore $C_{\mathbf{M}}(\iotaM(\mathbf{S}_{\phiM,r}))={\mathbf{M}}\cap C_{\mathbf{G}}(\iota(\mathbf{S}_{\phi,r}))$.
It follows in particular that the set of jumps $r'_{-1}=0<r'_{0}<\hdots <r'_{d'-1}$ of the decrasing filtration
$(\mathbf{S}_{\phiM,r})_{r}$ is a subset of the set of jumps $r_{-1}=0<r_{0}<\hdots< r_{d-1}$ of
the filtration $(\mathbf{S}_{\phi,r})_{r}$. For $0\leq j< d'$ write $i_{j}$ for the
unique integer between $0$ and $d-1$ such that $r'_{j}=r_{i_{j}}$, and put $i_{-1}=-1$. 
Then, with the notation of  \ref{Yu_construction}, we have $M\cap G^{i}_{\iota,x,r+}
=M^{j}_{\iotaM,x,r+}$ for all $i=0,\hdots, d$ and $j$ such that $i_{j-1}<i\leq i_{j}$, and
in particular we see that $M\cap G^{i}_{\iota,x,(r_{i-1}/2)+}\subseteq
M^{j}_{\iotaM,x,(r'_{j-1}/2)+}$ with equality if $i=i_{j-1}+1$.
This implies the second equality in:
\begin{eqnarray*}
  M\cap \Kp_{\iota,x} & = & (M\cap G^{0}_{\iota,x,0+})(M\cap G^{1}_{\iota,x,(r_{0}/2)+})\cdots
                                (M\cap G^{d}_{\iota,x,(r_{d-1}/2)+})\\
& = & M^{0}_{\iotaM,0+} M^{1}_{\iotaM,(r'_{0}/2)+}\cdots
M^{d'}_{\iotaM,(r'_{d'-1}/2)+} = \KpM_{\iotaM,x}
\end{eqnarray*}
The first equality follows from the Iwahori decomposition of each $G^{i}_{\iota,x,(r_{i-1}/2)+}$
with respect to the pair $(\bar{\mathbf{P}} ,{\mathbf{P}} )$,  as
in the proof of Lemma \ref{Iwahori}.

Let us now turn to characters.
By definition, $\check\phi^{+}_{\iotaM,x}$ is a
product $\prod_{j=0}^{d'} (\psi_{M,j}^{+})|_{\KpM_{\iotaM,x}}$ with $\psi_{M,j}^{+}$
a certain character of $M_{\iotaM,x,0+}^{j}M_{x,(r'_{j}/2)+}$,
while
$(\check\phi^{+}_{\iota,x})|_{\KpM_{\iotaM,x}}$ is a product $\prod_{i=0}^{d}
(\psi_{i}^{+})|_{\KpM_{\iotaM,x}}$ with  $\psi_{i}^{+}$  a certain character of $G_{\iota,x,0+}^{i}G_{x,(r_{i}/2)+}$.
Note that if $i_{j-1}<i\leq i_{j}$, we have 
$$M\cap G_{\iota,x,0+}^{i}G_{x,(r_{i}/2)+} =
(M\cap G_{\iota,x,0+}^{i})M_{x,(r_{i}/2)+} = M^{j}_{\iotaM,x,0+}M_{x,(r_{i}/2)+} \supseteq
M^{j}_{\iotaM,x,0+}M_{x,(r'_{j}/2)+}.$$
Therefore, it will suffice to prove that $\psi_{M,j}^{+} = \prod_{i=i_{j-1}+1}^{i_{j}}
(\psi_{i}^{+})|_{M^{j}_{\iotaM,x,0+}M_{x,(r'_{j}/2)+}}$.
Recall that $(\psi_{i}^{+})|_{G_{\iota,x,0+}^{i}}$ is the restriction of a
character $\check\varphi_{i}\check\varphi_{i+1}^{-1}$ of $G_{\iota}^{i}$ that depends on
the choice of $\hat\varphi_{i}\in H^{1}(W_{F},Z(\hat{\mathbf{G}}_{\phi,r_{i-1}+}))$
extending $\hat\phi|_{I_{F}^{r_{i-1}+}}$ and the choice of $\hat\varphi_{i+1}\in H^{1}(W_{F},Z(\hat{\mathbf{G}}_{\phi,r_{i}+}))$
extending $\hat\phi|_{I_{F}^{r_{i}+}}$. Similarly, $(\psi_{M,j}^{+})|_{M_{\iotaM,x,0+}^{j}}$ is the restriction of a character
$\check\varphi_{M,j}\check\varphi_{M,j+1}^{-1}$ that depends on
the choice of $\hat\varphi_{M,j}\in H^{1}(W_{F},Z(\hat{\mathbf{M}}_{\phiM,r'_{j-1}+}))$
extending $\hat\phiM|_{I_{F}^{r'_{j-1}+}}$ and the choice of $\hat\varphi_{M,j+1}\in H^{1}(W_{F},Z(\hat{\mathbf{M}}_{\phiM,r'_{j}+}))$
extending $\hat\phiM|_{I_{F}^{r'_{j}+}}$.  By Lemma \ref{lemme_indep}, these choices
eventually do not matter, in the sense that $\prod_{j=0}^{d'} (\psi_{M,j}^{+})|_{\KpM_{\iotaM,x}}$  and $\prod_{i=0}^{d}
(\psi_{i}^{+})|_{\KpM_{\iotaM,x}}$ do not depend on them, so we may choose 
$\hat\varphi_{M,k}$ to be the  composition of
$\hat\varphi_{i_{k-1}+1}$ with the inclusion $Z(\hat{\mathbf{G}}_{\phi,r_{i_{k-1}}+})\subset
Z(\hat{\mathbf{M}}_{\phiM,r'_{k-1}+})$ for all $1 \leq k \leq d'$. In this way, 
 we ensure that the characters 
$\psi_{M,j}^{+}$ and  $\prod_{i=i_{j-1}+1}^{i_{j}} \psi_{i}^{+}$ coincide on
$M^{j}_{\iotaM,x,0+}$. It then remains to see that they also coincide on
$M_{x,(r'_{j}/2)+}$,
by checking that Yu's extension procedures over $M$ and $G$ are compatible. For $G$
and index $i$, this procedure rests on the decomposition
$\gG_{x,(r_{i}/2)+:r_{i}+}=\gG^{i}_{\iota,x,(r_{i}/2)+:r_{i}+}\oplus
\nG^{i}_{\iota,x,(r_{i}/2)+:r_{i}+}$. 
The key point is then that
whenever $i_{j-1}<i\leq i_{j}$,
intersecting this decomposition with  $\mG$ (which amounts to taking the weight-0 part
of the action of the maximal split central torus of $M$) yields back the
corresponding decomposition $\mG_{x,(r_{i}/2)+:r_{i}+}=\mG^{j}_{\iota, x,(r_{i}/2)+:r_{i}+}\oplus 
\nG^{M,j}_{\iota,x,(r_{i}/2)+:r_{i}+}$.

\ref{item:non-degenerate1}) and \ref{item:non-degenerate2})
Thanks to the Heisenberg property of Proposition \ref{Heisenberg}, 
the computation of  \cite[\S 5.28]{finitude} shows\footnote{indeed, we specialize  the notation of
  \emph{loc.\ cit.}  as follows:  
${G}^{\dag}:=\Ku_{\iota,x}$,
${G}^{*}:=\Kp_{\iota,x}$, ${U}^{\dag}:=\Uu_{\iota,x}$, ${U}^{*}:=\Up_{\iota,x}$, and $\theta:=\check\phi_{\iota,y}^{+}$.} 
that for each character $\chi : \Uu_{\iota,x}/\Up_{\iota,x}\To{} R^{\times}$, and denoting by
$e_{\chi}$ the associated idempotent in $R\Uu_{\iota,x}$, we have
$$ |\Uu_{\iota,x}/\Up_{\iota,x}| e_{\chi}e_{\bUu_{\iota,x}}e_{\chi}e_{\iota,x}= e_{\chi}e_{\iota,x}.$$
Taking $\chi=1$, we obtain \ref{item:non-degenerate1}). Summing over all $\chi$'s we get
$$  |\Uu_{\iota,x}/\Up_{\iota,x}| \sum_{\chi : \Uu_{\iota,x}/\Up_{\iota,x}\To{} R^{\times}}
e_{\chi}e_{\bUu_{\iota,x}}e_{\chi}e_{\iota,x}=e_{\iota,x},$$
showing in particular that  $e_{\iota,x}\in R\Ku_{\iota,x} e_{\bUu_{\iota,x}}R\Ku_{\iota,x}$. Similarly,
we have $e_{\iota,x}\in R\Ku_{\iota,x} e_{\Uu_{\iota,x}}R\Ku_{\iota,x}$, hence also
$e_{\iota,x}\in  R\Ku_{\iota,x} e_{\Uu_{\iota,x}}R\Ku_{\iota,x}e_{\bUu_{\iota,x}}R\Ku_{\iota,x}$.
But $\Ku_{\iota,x}$ also has the Iwahori decomposition property with respect to
$({\mathbf{P}} ,\bar{\mathbf{P}} )$, so that
$R\Ku_{\iota,x} e_{\Uu_{\iota,x}}R\Ku_{\iota,x}e_{\bUu_{\iota,x}}R\Ku_{\iota,x}=R\Ku_{\iota,x}
e_{\Uu_{\iota,x}}e_{\bUu_{\iota,x}}R\Ku_{\iota,x}$.
Finally, recall that $e_{\iota,x}$ is central in $R\Ku_{\iota,x}$ and observe that
$e_{\Uu_{\iota,x}}e_{\iota,x} e_{\bUu_{\iota,x}}=e_{\Uu_{\iota,x}}e_{\iotaM,x}e_{\bUu_{\iota,x}}$.

\ref{item:HL-like}) This is Proposition 9.3 of \cite{finitude}.  However, since the proof there leaves many details
to the reader, we supply these details here. That proposition is ultimately a consequence of
Corollary 5.10 of \emph{loc.\ cit.} applied to the following set up:
\begin{itemize}
\item $\underline{G}:=G_{x,0}$  (and therefore $\underline{G}^{\dag}=G_{x,0+}$, see below),
\item  $\underline{G'}:=\Ko_{\iota,x}$ (and therefore  $\underline{G'}^{\dag}=\Ku_{\iota,x}$ and
  $\underline{M'}=\Ku_{\iotaM,x}$),
\item $\varepsilon':=e_{\iotaM,x}$ and $\tilde\varepsilon':=e_{\iota,x}$.
\end{itemize}
In order to apply Corollary 5.10 of \emph{loc.\ cit.}, there are a number of properties to check.
First of all, both $\underline{G}$ and $\underline{G}'$ need to be groups of integral points of a
connected smooth model of ${\mathbf{G}}$ over $\OC_{F}$. Here, by definition $G_{x,0}$ is
the parahoric subgroup of $G_{x}$ and the desired model was constructed by Bruhat--Tits, while the
desired model for $\Ko_{\iota,x}$ was constructed by Yu in \cite[Prop~10.2]{Yu_models}. 
Then, the groups $\underline{G}^{\dag}$, resp.,
$\underline{G'}^{\dag}$ should be the integral points of the dilatation of the unipotent radical of
the special fiber of $\underline{G}$, resp., $\underline{G}'$, or more explicitly, the
pro-$p$-radical of $\underline{G}$, resp., $\underline{G}'$. Hence they are indeed given by
$G_{x,0+}$, resp., $\Ku_{\iota,x}$. Note that the hypothesis
$\underline{G}'\cap\underline{G}^{\dag}\supseteq \underline{G'}^{\dag}$ of \emph{loc.\ cit.} is indeed
satisfied (and even equality holds).
Next, the parabolic subgroups ${\mathbf{P}} $ and $\bar{\mathbf{P}} $
should be $\underline{G'}$-admissible in the sense of \cite[\S 5.1]{finitude}. For this, it is
enough to see that the maximal split central torus $\mathbf{A_{M}}$ of ${\mathbf{M}}$ extends to a subtorus of 
Yu's model of $\Ko_{\iota,x}$. But the latter contains the Bruhat--Tits  model of the parahoric
subgroup $G_{\iota,x,0}^{0}$ of $G_{\iota}$ by \cite[Prop 10.4 iv)]{Yu_models}. Hence any maximal
split torus $\mathbf{S}$ of ${\mathbf{G}}_{\iota}$ 
whose apartment in $\BT_{\iota}$ contains $x$ extends to a split torus of Yu's model. Since $x\in \BT_{\iotaM}
=\BT({\mathbf{M}}_{\iotaM})$ and ${\mathbf{M}}_{\iotaM}={\mathbf{M}}\cap{\mathbf{G}}_{\iota}$ is an $F$-Levi
subgroup of ${\mathbf{G}}_{\iota}$, we may choose such an $\mathbf{S}$ that is contained in
${\mathbf{M}}_{\iotaM}$. But $\mathbf{A_{M}}$ is contained in the center of ${\mathbf{M}}_{\iotaM}$,
hence also in $\mathbf{S}$, and we see that   $\mathbf{A_{M}}$ extends to a split subtorus of Yu's model.
There are three further requirements on the idempotents $\varepsilon'$ and
$\tilde\varepsilon'$. First, $\varepsilon'$ should be ``essentially of depth zero'' in the sense of
\cite[Lemma 5.6]{finitude}. Actually, we have checked this property in the proof of ii)
above. Next, hypothesis i) of \cite[Cor 5.10]{finitude} follows from i) of this lemma, while
hypothesis ii) of \emph{loc.\ cit.} follows from Proposition \ref{prop_intertwining} ii) (which
implies more generally that the centralizer of $e_{\iota,x}$ in $G_{x}$ is $\K_{\iota,x}$).
We may now apply Corollary 5.10 of \cite{finitude}, thus completing the proof of iii).

\ref{item:intertwining})  By Proposition \ref{prop_intertwining} ii) and \eqref{Gaction}, we have
$e_{\iota,x}ue_{\iota,x}=0$ for all $u\in U_{x}\setminus U_{\iota,x}$ while, by Proposition \ref{prop_intertwining} i), we have $e_{\iota,x}ue_{\iota,x}=ue_{\iota,x}$ if $u\in U_{\iota,x}$ . Writing
$e_{U_{x}}=|U_{x}/U_{\iota,x}|^{-1}\sum_{u\in U_{x}/U_{\iota,x}}ue_{U_{\iota,x}}$ we obtain
the first identity. The second one is proved in the same way.

\ref{item:centralizer})
The hypothesis $C_{\hat{\mathbf{G}}}(\phi)\subseteq \hat{\mathbf{M}}$ implies that
$\mathbf{S}_{\phi}=\mathbf{S}_{\phi_{M}}$. Since $\iotaM(\mathbf{S}_{\phi_{M}})\supseteq
Z({\mathbf{M}})^{\circ}$,  we obtain  that 
${\mathbf{G}}_{\iota}\subseteq {\mathbf{M}}$ (and actually ${\mathbf{G}}_{\iota}={\mathbf{M}}_{\iotaM}$).
In particular, we have
${\mathbf{U}}\cap{\mathbf{G}}_{\iota}=\bar {\mathbf{U}}\cap{\mathbf{G}}_{\iota}=\{1\}$ and it follows that
$U\cap \K_{\iota,x}=U\cap \Ku_{\iota,x}$ and  $\bar U\cap \K_{\iota,x}=\bar U\cap
\Ku_{\iota,x}$.
\end{proof}

\begin{lemme} \label{comput_par_ind}
  With the notation of \ref{setup_par_ind}, there is a $G\times \K_{\iotaM,x}$-equivariant isomorphism
  $$ i_{P}\left(\cInd{\KpM_{\iota_{M},x}}{M}\left(\check\phi^{+}_{\iota_{M},x}\right)\right)
  \simeq \cInd{U_{\iota,x}\Kp_{\iota,x}}{G}(\tilde\phi^{+}_{\iota,x})$$
  where $\tilde\phi^{+}_{\iota,x}$ is the unique character of $U_{\iota,x}\Kp_{\iota,x}$
  that extends $\check\phi^{+}_{\iota,x}$ and the trivial character of $U_{\iota,x}$.
\end{lemme}
Here the actions of $K_{\iotaM,x}$ come from the fact this group centralizes both the
character $\check\phi^{+}_{\iota_{M},x}$ of $\Kp_{\iotaM,x}$ and the character
$\tilde\phi^{+}_{\iota, x}$ of $U_{\iota,x}\Kp_{\iota,x}$. Explicitly, and through the identifications
\begin{center}
  $\cInd{\KpM_{\iota_{M},x}}{M}\left(\check\phi^{+}_{\iota_{M},x}\right)=\HC_{R}(M)e_{\iotaM,x}$
  and
  $\cInd{U_{\iota,x}\Kp_{\iota,x}}{G}(\tilde\phi^{+}_{\iota,x})=\HC_{R}(G)e_{U_{\iota,x}}e_{\iota,x}$
\end{center}
these actions are given by multiplication on the right. 
\begin{proof}
  Thanks to (\ref{item:HL-like}) of Lemma \ref{lemma_ind_parab}, 
  \cite[Cor.~3.6(ii)]{finitude} provides us with an $RG\otimes R\K_{\iotaM,x}$-linear
  isomorphism
$$ \cInd{G_{x}}{G}\left(\HC_{R}(G_{x})e_{\bar U_{x,0+}}e_{U_{x}}
  e_{\iota_{M},x}\right)  \simto
i_{P}\left(\cInd{\KpM_{\iota_{M},x}}{M}\left(\check\phi^{+}_{\iota_{M},x}\right)\right)
.$$
 Note that in \textit{loc.\ cit.}, the equivariance is for $RG\otimes_{R}(
 e_{\iotaM,x}RM_{x}e_{\iotaM,x})$, and here we compose the second action with the map
 $R\K_{\iotaM,x}\To{\times e_{\iotaM,x}}e_{\iotaM,x}RM_{x}e_{\iotaM,x}$. 
(Note also that replacing $i_P$ by the normalized induction provides an isomorphic co-domain of the above isomorphism.) 

 It thus remains to find an $RG_{x}\otimes_{R}RK_{\iotaM,x}$-linear isomorphism between 
  $$\HC_{R}(G_{x})e_{\bar U_{x,0+}}e_{U_{x}} e_{\iota_{M},x}=\HC_{R}(G_{x})e_{\bar U_{x,0+}}e_{\iota_{M},x}e_{U_{x}} =\HC_{R}(G_{x})e_{\bar U_{x,0+}}e_{\iota,x}e_{U_{x}},$$ 
  where the first equality follows from $\K_{\iotaM,x} \subset M_x$ normalizing $U_x$, and
  $$\HC_{R}(G_{x})e_{U_{\iota,x}}e_{\iota,x}=\cInd{U_{\iota,x}\Kp_{\iota,x}}{G_{x}}(\tilde\phi^{+}_{\iota,x}).$$
  By Lemma \ref{lemma_ind_parab} (\ref{item:intertwining}),
  the inclusions
  $$
  \HC_{R}(G_{x})e_{\iota,x}e_{\bar U_{x,0+}}e_{\iota,x}
  \subseteq \HC_{R}(G_{x})e_{\bar U_{x,0+}}e_{\iota,x}
  \subseteq \HC_{R}(G_{x})e_{\bUu_{\iota,x}}e_{\iota,x}$$
  are equalities. 
  For the analogous reason, i.e., because $e_{\iota,x}e_{U_{\iota,x}}$ is a scalar multiple of $e_{\iota,x}e_{U_{x}}e_{\iota,x}$, the maps
  $$ \HC_{R}(G_{x})e_{\bUu_{\iota,x}}e_{\iota,x}e_{U_{\iota,x}} \To{\times e_{U_{x}}}
  \HC_{R}(G_{x})e_{\bUu_{\iota,x}}e_{\iota,x}e_{U_{x}} \To{\times e_{\iota,x}}
  \HC_{R}(G_{x})e_{\bUu_{\iota,x}}e_{\iota,x}e_{U_{x}}e_{\iota,x}$$
  are isomorphisms. Similarly, by Lemma \ref{lemma_ind_parab} (\ref{item:non-degenerate1}), the inclusions 
  $$ \HC_{R}(G_{x})e_{\Uu_{\iota,x}}e_{\bUu_{\iota,x}}e_{\Uu_{\iota,x}}e_{\iota,x}
  \subseteq \HC_{R}(G_{x})e_{\bUu_{\iota,x}}e_{\Uu_{\iota,x}}e_{\iota,x}
  \subseteq \HC_{R}(G_{x})e_{\Uu_{\iota,x}}e_{\iota,x}$$ 
  are equalities. 
  From the last equality, and Proposition \ref{prop_intertwining}(i), we obtain using $e_{\Uu_{\iota,x}}e_{U_{\iota,x}}=e_{U_{\iota,x}}$ that
  \begin{eqnarray*}
  \HC_{R}(G_{x})e_{U_{\iota,x}}e_{\iota,x}&=&\HC_{R}(G_{x})e_{\Uu_{\iota,x}}e_{U_{\iota,x}}e_{\iota,x}=\HC_{R}(G_{x})e_{\Uu_{\iota,x}}e_{\iota,x}e_{U_{\iota,x}}=\HC_{R}(G_{x})e_{\bUu_{\iota,x}}e_{\Uu_{\iota,x}}e_{\iota,x}e_{U_{\iota,x}}\\
  &=&\HC_{R}(G_{x})e_{\bUu_{\iota,x}}e_{\iota,x}e_{U_{\iota,x}} .
\end{eqnarray*}
  Combining all the above observations, 
  we also have 
   \[\HC_{R}(G_{x})e_{\bar U_{x,0+}}e_{\iota,x}e_{U_{x}}=\HC_{R}(G_{x})e_{\bUu_{\iota,x}}e_{\iota,x}e_{U_{x}}=\HC_{R}(G_{x})e_{\bUu_{\iota,x}}e_{\iota,x}e_{U_{x}} ,\]
  and we see that the right multiplication by $e_{U_{x}}$
  induces an isomorphism
  $\HC_{R}(G_{x})e_{U_{\iota,x}} e_{\iota,x}\simto\HC_{R}(G_{x})e_{\bar U_{x,0+}}e_{\iota,x}e_{U_{x}}$,
  which is $G_{x}\times\K_{\iotaM,x}$-equivariant because $\K_{\iotaM,x} \subset M_x$ normalizes $U_x$.
\end{proof}

\alin{Proof of Theorem \ref{compat_parab_induc}}
i) In view of the projective generator given in Theorem \ref{thm_Serresubcat},
it is sufficient to prove that for each $\iotaM\in \IM$ and $x\in
  \BT_{\iota}\cap\BT_{M}$ we have 
$$i_{P}\left(\cInd{\KpM_{\iotaM,x}}{M}\left(\check\phi^{+}_{\iotaM,x}\right)\right)
\in \Rep^{\phi,I}_{R}(G).$$
This follows from Lemma \ref{comput_par_ind}.

ii) The hypothesis on $p$ is inherited by Levi subgroups of ${\mathbf{G}}$ so that, by Theorem
\ref{exhaust}, we  have
the two decompositions $\Rep_{R}(G)=\prod\Rep_{R}^{\phi,I}(G)$ and
$\Rep_{R}(M)=\prod\Rep_{R}^{\phiM,\IM}(M)$. Therefore ii) follows from i) by Frobenius reciprocity.

iii)  The equivalence can be checked after adjoining a square root of $p$ to the coefficient field, so we may assume without loss of generality that $R$ is a $\bZ[\mu_{p^\infty},\frac{1}{\sqrt{p}}]$-algebra. Now it suffices to prove the stated claim for the normalized parabolic induction and Jacquet functor as this normalization preserves the categories and equivalence. So for the remainder of the proof we denote by $r_P=r_P^G$ and $i_P=i_P^G$ the normalized Jacquet functor and the normalized parabolic induction, and we write $\tilde r_{P}$ for the composition of $r_{P}$ with the projection on
$\Rep_{R}^{\phiM,\IM}(M)$.  We will first show that $\tilde r_{P}\circ i_{P}$ is
isomorphic to the identity functor on $\Rep_{R}^{\phiM,\IM}(M)$ so that, in particular,
$i_{P}$ is fully faithful on $\Rep_{R}^{\phiM,\IM}(M)$.

To this aim, recall that Frobenius reciprocity is given by a natural transformation
$r_{P}\circ i_{P}\To{}\id$ which is an epimorphism in the category of additive endofunctors
of $\Rep_{R}(M)$ and whose kernel is described by the Mackey formula as follows: there is
a filtration indexed by double cosets $P\dot{w}P$ in $G\setminus P$ whose graded pieces are of
the form $\FC_{\dot w}:=\Ad_{\dot w}\circ i_{P\cap M^{w}}^{M^{w}}\circ   r_{M\cap P^{w}}^M$.  Here, we have chosen
representatives $\dot w$ in the rational normalizer $N_{G}(T)$ of a 
maximally split maximal torus ${\mathbf{T}}$ of ${\mathbf{M}}$ and $w$ is the image of $\dot
w$ in the absolute Weyl group $W_{\mathbf{G}}({\mathbf{T}})$. In this situation, ${\mathbf{M}}\cap
{\mathbf{P}} ^{w}$ is a parabolic 
$F$-subgroup of ${\mathbf{M}}$ with Levi component ${\mathbf{M}}\cap {\mathbf{M}}^{w}$, while
${\mathbf{P}} \cap {\mathbf{M}}^{w}$ is a parabolic 
subgroup of ${\mathbf{M}}^{w}$ with the same Levi component ${\mathbf{M}}\cap {\mathbf{M}}^{w}$. 
It then follows from the parts i) and ii)  that 
$$\FC_{\dot w}\left( \Rep_{R}^{\phiM,\IM}(M)\right) \subseteq
\prod_{(\phi_{w},I_{w})\mapsto (\phiM,\IM)} \Ad_{\dot w}\left(\Rep_{R}^{\phi_{w},I_{w}}(M^{w})\right)$$
where the product is over pairs $(\phi_{w},I_{w})$ relative to
${\mathbf{M}}\cap{\mathbf{M}}^{w}$ that map to $(\phiM,\IM)$, and whose pushforward to
${\mathbf{M}}^{w}$  we still denote by
$(\phi_{w},I_{w})$. 

Let us draw the dual picture. We may assume that $\hat{\mathbf{M}}$ contains a reference maximal torus $\hat{\mathbf{T}}$
in $\hat{\mathbf{G}}$ (part of a $W_{F}$-stable pinning of $\hat{\mathbf{G}}$). We
have a duality between ${\mathbf{T}}$ and $\hat{\mathbf{T}}$ that exchanges roots and
coroots. This induces an isomorphism $w\mapsto w,\, W_{\mathbf{G}}({\mathbf{T}})\simto
W_{\hat{\mathbf{G}}}(\hat{\mathbf{T}})$. 
Let us choose a lift $\hat w$ of $w$ in the normalizer
$N_{\hat{\mathbf{G}}}(\hat{\mathbf{T}})$. Then $\hat{\mathbf{M}}^{\hat w}$ is a Levi
subgroup of $\hat{\mathbf{G}}$ that is dual to ${\mathbf{M}}^{w}$ and $\Ad_{\hat w^{-1}}$ is a dual
isogeny (actually isomorphism) to $\Ad_{\dot w}$. Therefore the last inclusion can be
rewritten as
$$\FC_{\dot w}\left( \Rep_{R}^{\phiM,\IM}(M)\right) \subseteq
\prod_{(\phi_{w},I_{w})\mapsto (\phiM,\IM)} \left(\Rep_{R}^{\Ad_{\hat w^{-1}}(\phi_{w},I_{w})}(M)\right)$$
with the same convention as above.
Now let $\phi_{w}: P_{F}\To{} \hat{\mathbf{M}}\cap\hat{\mathbf{M}}^{\hat w}$ be a parameter
for ${\mathbf{M}}\cap {\mathbf{M}}^{w}$ whose pushforward to $\hat{\mathbf{M}}$ represents
$\phiM$. Assume that  $\Ad_{\hat w^{-1}}(\phi_{w})$ also represents $\phiM$. Then there is
some $\hat m\in\hat{\mathbf{M}}$ such that $\Ad_{\hat m\hat w^{-1}}(\phi_{ w})=\phi_{w}$, i.e., $\hat m\hat w^{-1}\in C_{\hat{\mathbf{G}}}(\phi_{w})$. By our
assumption, this implies that $\hat w\in\hat{\mathbf{M}}$, hence $\dot w\in {\mathbf{M}}$,
which contradicts the fact that $P\dot w P\neq P$. 
This means that  the projection of $\FC_{\dot w}\left(
  \Rep_{R}^{\phiM,\IM}(M)\right)$ on $\Rep^{\phiM,\IM}_{R}(M)$ is zero, and finally we
have proven that the natural transformation $r_{P}\circ i_{P}\To{}\id$ induces an
isomorphism
$\tilde r_{P}\circ i_{P}\simto \id_{\Rep^{\phiM,\IM}(M)}$.

Now, to conclude that $\tilde r_{P}$ and $i_{P}$ are quasi-inverse equivalences of
categories, it suffices to prove that $\tilde r_{P}$ is conservative on
$\Rep^{\phi,I}_{R}(G)$. So let $V$ be a non-zero object of $\Rep^{\phi,I}_{R}(G)$. By definition,
there is a point $x\in\BT_{\iota}$ such that $e_{\iota,x}V\neq 0$. Recall that, under the hypothesis
$C_{\hat{\mathbf{G}}}(\phi)\subseteq\hat{\mathbf{M}}$, we have
${\mathbf{G}}_{\iota}\subseteq {\mathbf{M}}$, so that
$\BT_{\iota}\subseteq \BT_{M}$ and $x$ actually lies in $\BT_{\iotaM}$.
We claim that $e_{U_{x}}e_{\bar U_{x,0+}}e_{\iotaM,x}V\neq 0$. Indeed,
by \eqref{item:Iwahori} and (\ref{item:intertwining}) of Lemma \ref{lemma_ind_parab}, we have 
$e_{\iota,x} e_{U_{x}} e_{\bar U_{x,0+}}e_{\iotaM,x} e_{\iota,x}
=e_{\iota,x} e_{U_{x}}e_{\iotaM,x}  e_{\bar U_{x,0+}}e_{\iota,x}
=e_{\iota,x} e_{U_{x}}e_{\Up_{\iota,x}}e_{\iotaM,x}e_{\bar \Up_{\iota,x}}  e_{\bar U_{x,0+}}e_{\iota,x}
=e_{\iota,x} e_{U_{x}}e_{\iota,x}  e_{\bar U_{x,0+}}e_{\iota,x}
=e_{U_{\iota,x}}e_{\bUu_{\iota,x}}e_{\iota,x}$, which is also equal to
$e_{\Uu_{\iota,x}}e_{\bUu_{\iota,x}}e_{\iota,x}=e_{\Uu_{\iota,x}}e_{\iota,x}e_{\bUu_{\iota,x}}=e_{\Uu_{\iota,x}}e_{\iota_M,x}e_{\bUu_{\iota,x}}$, by (\ref{item:centralizer}) of the same
lemma. 
So we deduce from Lemma \ref{lemma_ind_parab}(\ref{item:non-degenerate2}) that
$e_{\iota,x}\in RK_{\iota,x}e_{U_{x}} e_{\bar U_{x,0+}}e_{\iotaM,x}RK_{\iota,x}$ and
our claim $e_{U_{x}}e_{\bar U_{x,0+}}e_{\iotaM,x}V\neq 0$ follows from $e_{\iota,x}V\neq 0$.
Using now 
part (\ref{item:HL-like}) of  Lemma \ref{lemma_ind_parab},
 we may apply Prop~3.1 of \cite{finitude}, which tells us that the natural map $V\To{}
 r_{P}(V)=V_{U}$ induces an isomorphism
 $e_{U_{x}}e_{\bar U_{x,0+}}e_{\iotaM,x}V\simto e_{\iotaM,x} r_{P}(V)$.
 In particular, $e_{\iotaM,x} r_{P}(V)$ is non-zero, hence the projection of $r_{P}(V)$ on
 $\Rep^{\phiM,I_{M}}(M)$ is non-zero.

 \begin{prop}\label{indep_parab}
   In the setting of Theorem~\ref{compat_parab_induc} iii), the adjoint pair of equivalences
   $\tilde r_{P}:\Rep^{(\phi,I)}_{R}(G) \rightleftarrows \Rep^{(\phi_{M},I_{M})}_{R}(M):i_{P}$
 is independent of the choice of ${\mathbf{P}} $, up to isomorphism. 
 \end{prop}
 \begin{proof}
 	As in the proof of Theorem \ref{compat_parab_induc}iii) we may adjoin a square root of $p$ to our coefficient field and assume that $\tilde r_P$ and $i_P$ are normalized.
   Let $\mathbf{Q}$ be another parabolic $F$-subgroup of ${\mathbf{G}}$ with Levi component
   ${\mathbf{M}}$, yielding another adjoint pair of inverse equivalences $(\tilde r_{Q}, i_{Q})$.
   By the geometric lemma, the composition $r_{P}\circ i_{Q}$ has a filtration indexed by double
   cosets $Q\dot{w}P$ in $G$ with graded pieces of
   the form $\Ad_{\dot w}\circ i_{P\cap M^{w}}\circ   r_{M\cap Q^{w}}$. In particular, the graded
   piece associated to the coset $QP$ is the identity endofunctor of $\Rep_{R}(M)$.
   As in the proof of Theorem \ref{compat_parab_induc}iii), only that graded piece survives
   when restricting the domain
   to $\Rep^{(\phi_{M},I_{M})}_{R}(M)$ and projecting onto this factor. So we obtain an isomorphism
   $\tilde r_{P}\circ i_{Q}\simto \id_{\Rep^{(\phi_{M},I_{M})}_{R}(M)}$. By adjunction, we obtain a
   morphism
   $(i_{Q})|_{\Rep^{(\phi_{M},I_{M})}_{R}(M)}\To{}(i_{P})|_{\Rep^{(\phi_{M},I_{M})}_{R}(M)}$ whose
   composition with $\tilde r_{P}$ is an isomorphism. Since $\tilde r_{P}$ is conservative (as we saw in the proof of Theorem \ref{compat_parab_induc}iii)), that
   morphism is an isomorphism too.
 \end{proof}

%%%%---------------------

\section{Equivalences of categories and reduction to depth zero}
Recall that we assume that $p$ is odd and satisfies hypotheses \eqref{H1} and \eqref{H2}, i.e., $p$ is not a torsion prime for $\mathbf{G}$ nor for $\hat{\mathbf{G}}$.
To a pair $(\phi,I)$ consisting of a  wild
inertia  parameter  $\phi: P_{F}\To{}\hat{\mathbf{G}}$ whose centralizer
$C_{\hat{\mathbf{G}}}(\phi)$ is a Levi subgroup of $\hat{\mathbf{G}}$, and
a $G$-conjugacy class $I$ of rational Levi-center-embeddings $\mathbf{S}_{\phi}\injo{\mathbf{G}}$, we
have associated in Section \ref{sec:category-repphi-i_rg} a Serre subcategory 
$\Rep^{(\phi,I)}_{R}(G)$ of $\Rep_{R}(G)$, where $R$ is any commutative $\Rmin$-algebra.
Picking $\iota\in I$, we can also view $\phi$ as a
 wild inertia parameter for ${\mathbf{G}}_{\iota}$ with the same associated torus
$\mathbf{S}_{\phi}$, for which the corresponding set of 
Levi-center-embeddings is just the singleton $\{\iota\}$, whence a Serre subcategory
$\Rep^{\phi}_{R}(G_{\iota})$. The aim of this section is to construct some ``natural''
equivalences 
\begin{equation*}
  \label{desired_equivalences}
  \Rep^{\phi}_{R}(G_{\iota}) \simto \Rep^{(\phi,I)}_{R}(G).
\end{equation*}
As explained in \ref{subsec:twisted_depth_zero}, $\Rep^{\phi}_{R}(G_{\iota})$ is a twist of the depth-$0$
subcategory $\Rep^{1}_{R}(G_{\iota})$ by any character $\check\varphi:G_{\iota}\To{}R^{\times}$ 
associated to an extension of $\phi$ to a $1$-cocycle
$\varphi:\,W_{F}\To{}Z(\hat{\mathbf{G}}_{\iota})=Z(C_{\hat{\mathbf{G}}}(\phi))$. So, these
equivalences of categories fall in the general paradigm of ``reduction to depth $0$'' and, actually,
what we are going to construct are equivalences of categories
\begin{equation*}
  \label{desired_equivalences_bis}
  \Rep^{1}_{R}(G_{\iota}) \simto \Rep^{(\phi,I)}_{R}(G).
\end{equation*}

Unless specified otherwise, $R$  denotes a commutative $\Rmin$-algebra $R$. However, 
our main results will need $R$ to contain a $4^{th}$ root of unity and a square root of $p$. 

\subsection{Main result} \label{sec:main-result-strategy-1}
We fix a pair $(\phi,I)$ and an embedding $\iota\in I$ as above. 
Recall that $\BT'_\iota$ denotes the image of the extended Bruhat--Tits building of $\mathbf{G}_\iota$ in the reduced Bruhat--Tits building $\BT'$ of $\mathbf{G}$, the group $G_{\iota,x}$ is the stabilizer of a point $x \in \BT'_\iota$ in $G_\iota$ and we define the group $K_{\iota, x}$\index[notation]{Kaiotax@$\K_{\iota,x}$}  to be $G_{\iota,x}\Ku_{\iota,x}$, where $\Ku_{\iota,x}:=\Ku_{\iota,\hat{x}}$\index[notation]{Kdagger@$\Ku_{\iota,x}$} denotes the group defined in \ref{sec:K-def} for $\hat{x}$ any preimage of $x$ in $\BT_\iota$. Similarly we write $\Kp_{\iota,x}:=\Kp_{\iota,\hat{x}}$\index[notation]{Kplus@$\Kp_{\iota,x}$} 
  for $\hat{x}$ any preimage of $x$ in $\BT_\iota$.

\begin{theo}\label{thm-main}
 If $R$ is a $\Rmintwo$-algebra, there exists an equivalence of categories
	\[ \cI: \Rep^{1}_{R}(G_{\iota}) \simto \Rep^{(\phi,I)}_{R}(G) \]
and an explicitly constructed family $\{\cW_{\iota,
  x}\}_{x \in \BT'_\iota}$ of $R\K_{\iota,x}$-modules $\cW_{\iota,x}$
such that, for any $x \in \BT'_\iota$ and any $R$-representation 
$\rho$ of $G_{\iota,x}$ that is trivial on $G_{\iota,x,0+}$, there is an isomorphism
of $RG$-modules 
\[ \cI (\cind_{G_{\iota,x}}^{G_{\iota}}(\rho))\simeq
	\cind_{\K_{\iota,x}}^{G}(\WC_{\iota,x}\otimes_{R}\rho)
	\]
through which $\cI$  induces an isomorphism of $R$-algebras
\[\End_{RG_{\iota}}(\cind_{G_{\iota,x}}^{G_{\iota}}(\rho))
\simto \End_{RG}(\cind_{\K_{\iota,x}}^{G}(\WC_{\iota,x}\otimes_{R}\rho)).\]
\end{theo}
The proof 
 of this theorem will be the content of most of the remainder of this
paper.
We refer to Theorem \ref{thm-main-2} and
Corollary \ref{coro_indep_e} for more details on the  $R\K_{\iota,x}$-modules $\cW_{\iota, x}$,
which require more preparation before being properly defined.
Suffice it to say,  they 
will be constructed via a suitable theory of Heisenberg--Weil
representations over $R$ that we introduce in  \ref{section-HW-coef-prep}.
That description will allow us to deduce an isomorphism between positive-depth and depth-zero Hecke algebras attached to types, see \ref{sec:hecke-types}, in particular Corollary \ref{cor-Hecke-alg-isom}, for details.

\subsection{Localization on the buildings} \label{sec:main-result-strategy-2}

\emph{In the remainder of the paper, the integer $e$\index[notation]{e@$e$} is  always
  supposed to satisfy the conditions of \ref{choice_e}, 
  and be such that barycenters of $1$-facets are $e$-vertices
(this ensures that stabilizers of $e$-facets in $G$ are also fixators of these $e$-facets)}.

\alin{$G$-equivariant coefficient systems on $\BC'$} \label{def_coef_syst}
We define a category $[\BT'_{\bullet/e}/G]$ as follows:
\begin{itemize}
  \item Its set of objects is the set $\BT'_{\bullet/e}$\index[notation]{B'bullete@$\BT'_{\bullet/e}$} of $e$-facets of  the reduced building of
    $G$.
  \item For $\FC,\FC'\in\BT'_{\bullet/e}$, we set $\Hom(\FC,\FC'):=\{g\in G,
    \ov{g\FC}\supseteq\FC'\}$,
    with composition given by multiplication in $G$. 
\end{itemize}

	A \emph{$G$-equivariant coefficient system of $R$-modules} is a functor 
	$\VC:\, [\BT'_{\bullet/e}/G]\To{} R-\Mod$. 	
Concretely, it is given by a collection of $R$-modules
$(\VC_{\FC})_{\FC\in\BT'_{\bullet/e}}$ with face maps
$\VC_{\FC}\To{\beta_{\cV,\FC,\FC'}}\VC_{\FC'}$\index[notation]{betaVFF'@$\beta_{\cV,\cF,cF'}$}
for each pair of 
facets such that $\o\FC\supseteq\FC'$ (corresponding to $g=1 \in \Hom(\cF, \cF')$), and
isomorphisms $g_{\VC,\FC}:\VC_{\FC}\simto 
\VC_{g\FC}$\index[notation]{gVF@$g_{\cV,\cF}$} subject to appropriate transitivity and
compatibility relations. We might also simply write $\beta_{\FC,\FC'}$\index[notation]{betaFF'@$\beta_{\cF,cF'}$} for $\beta_{\cV,\FC,\FC'}$ and $g_{\FC}$\index[notation]{gF@$g_{\cF}$} for $g_{\VC,\FC}$ if $\cV$ is clear from the context. In particular, each $\VC_{\FC}$ carries an $R$-linear
action of the stabilizer $G_{\FC}$ of $\FC$ in $G$. We say that the $G$-equivariant coefficient system is
\emph{smooth} if these actions are smooth. 
Smooth \emph{$G$-equivariant coefficient systems} form an abelian category $\Coef_{R}(\BT'_{\bullet/e}/G)$\index[notation]{CoefR@$\Coef_{R}(\BT'_{\bullet/e}/G)$}
where the morphisms are given by
natural transformations of functors.

\def\cart{{\rm cart}}

\alin{Coefficient systems and representations} \label{reps-G-coeff-sys}
There is a projection functor $[\BT'_{\bullet/e}/G]\To{\pi}[*/G]$ to
 the category  $[*/G]$ with $1$ object $*$ with set of endomorphisms $G$.
Note that a representation $V$ of $G$ on an $R$-module is given by a
functor $[*/G]\To{} R-\Mod$. Composition with $\pi$
then provides an exact functor
$$ \pi^{*}: \Rep_{R}(G)\To{} \Coef_{R}(\BT'_{\bullet/e}/G),$$ which
takes $V$ to the constant coefficient system $\FC\mapsto V$ with
$G$-equivariant structure given by the action of $G$  on $V$. The functor $\pi^{*}$ has a left adjoint
that we denote by $\pi_{!}$. Explicitly we have
$$\application{\pi_{!}:}{\Coef_{R}(\BT'_{\bullet/e}/G)}{\Rep_{R}(G)}{\VC}{\colim_{\BT'_{\bullet/e}}\VC},$$ where
the colimit is taken over the functor $\VC$ restricted to the category
associated with the poset $\BT'_{\bullet/e}$,  and the action of $G$ arises from 
its action on that poset and the
$G$-equivariant structure on $\cV$.

\alin{$(\phi,I)$-coefficient systems}  \label{phi_I_coef_systems}
Recall from Proposition \ref{indep_facet} iii) that for $x\in\BT$, the idempotent
$e_{\phi,I,x}$ only depends on the $e$-facet $\FC\in\BT'$ that contains the image of $x$
in $\BT'$, so we may denote this idempotent by $e_{\phi,I,\FC}$.
We define:\index[notation]{CoefphiIR@$\Coef^{\phi,I}_{R}(\BT'_{\bullet/e}/G)$}
\begin{center}
  \begin{tabular}[c]{lcl}
  $\Coef^{\phi,I}_{R}(\BT'_{\bullet/e}/G) $ & $:=$ &
                                        \begin{minipage}[c]{0.6\linewidth}
                                          the full subcategory of $\Coef_{R}(\BT'_{\bullet/e}/G)$ that consists
                                          of all smooth $G$-equivariant coefficient systems $\VC$
                                          such that, for any  facet $\cF$, we have
                                          $\VC_{\FC} = e_{\phi,I,\FC}\VC_{\FC}$.
                                        \end{minipage}
  \end{tabular}
\end{center}
It follows from the definition of $\pi_{!}$ that
$$\pi_{!}(\Coef^{\phi,I}_{R}(\BT'_{\bullet/e}/G))\subset \Rep^{\phi,I}_{R}(G).$$
By Proposition \ref{idemp_commut}, we have equalities
$e_{\phi,I,\FC}e_{\phi,I,\FC'}=e_{\phi,I,\FC}$
whenever $\o\FC\supseteq\FC'$. Therefore, for any smooth
$G$-equivariant coefficient system $\VC$ on $\BT'_{\bullet/e}$, we
have $\beta_{\VC,\FC,\FC'}(e_{\phi,I,\FC}\VC_{\FC})\subseteq e_{\phi,I,\FC'}\VC_{\FC'}$.
On the other hand, the fact that $e_{\phi,I,g\FC}=g e_{\phi,I,\FC} g^{-1}$ implies that
$g_{\VC,\FC}(e_{\phi,I,\FC}\VC_{\FC})=e_{\phi,I,g\FC}\VC_{g\FC}$. We thus get a functor 
$$\application{(-)^{\phi,I}:}{\Coef_{R}(\BT'_{\bullet/e}/G)}
{\Coef^{\phi,I}_{R}(\BT'_{\bullet/e}/G)}{\VC}{\VC^{\phi,I}
  :\,\FC\mapsto e_{\phi,I,\FC}\VC_{\FC} }$$
which is  right adjoint to the inclusion functor
$\Coef^{\phi,I}_{R}(\BT'_{\bullet/e}/G)\subset \Coef_{R}(\BT'_{\bullet/e}/G)$. 
Composing with $\pi^{*}$, we get a functor
$$ (\pi^{*})^{\phi,I} :\, \Rep_{R}(G)\To{} \Coef^{\phi,I}_{R}(\BT'_{\bullet/e}/G)$$
that is right adjoint to $\pi_{!}$ restricted to $\Coef^{\phi,I}_{R}(\BT'_{\bullet/e}/G)$.
Explicitly, for $V\in\Rep_{R}(G)$, the coefficient system $(\pi^*)^{\phi, I}(V)$ is given by the data
\begin{itemize}
\item $\forall\FC\in\BT'_{\bullet/e},\,(\pi^{*})^{\phi,I}(V)_{\FC}= e_{\phi,I,\FC}V$
\item $\forall \FC,\FC'$ with $\o\FC\supseteq\FC'$, the map $\beta_{\FC,\FC'}$ is the inclusion
  $e_{\phi,I,\FC}V=e_{\phi,I,\FC}e_{\phi,I,\FC'}V\subseteq e_{\phi,I,\FC'}V$
\item $\forall\FC\in \BT'_{\bullet/e},\forall g\in G$, the map
  $g_{\FC}:\, e_{\phi,I,\FC}V\To{g} e_{\phi,I,g\FC}V=ge_{\phi,I,\FC}g^{-1}V$ is given by the action of $g$.
\end{itemize}

We observe that $(\pi^{*})^{\phi,I}$ factors through the following full subcategory: 
\index[notation]{CoefphiIRcart@$\Coef^{\phi,I}_{R}(\BT'_{\bullet/e}/G)^{\cart}$}
\begin{center}
  \begin{tabular}[c]{lcl}
  $\Coef^{\phi,I}_{R}(\BT'_{\bullet/e}/G)^{\cart}$ & $:=$ &
                                        \begin{minipage}[c]{0.6\linewidth}
                                          the full subcategory of $\Coef_{R}(\BT'_{\bullet/e}/G)$ that consists
                                          of all smooth $G$-equivariant coefficient systems $\VC$
                                          such that, for any pair of facets $(\cF, \cF')$
                                          satisfying
                                          $\o\FC\supseteq\FC'$, the map $\beta_{\VC,\FC,\FC'}$
                                          induces an isomorphism
                                          $\VC_{\FC}\simto e_{\phi,I,\FC}\VC_{\FC'}$.
                                        \end{minipage}
  \end{tabular}
\end{center}
We refer to coefficient systems in $\Coef^{\phi,I}_{R}(\BT'_{\bullet/e}/G)^{\cart}$ as \emph{Cartesian $(\phi,I)$-coefficient systems}.

\begin{pro} 
  The above functors  
  induce quasi-inverse equivalences of categories
\ini \begin{equation}\label{thm_equiv_repr_coef} 
  \xymatrix{
    \pi_{!}:\, \Coef^{\phi,I}_{R}(\BT'_{\bullet/e}/G)^{\cart}
    \ar@<0.5ex>[r]^-{\sim} &
    \Rep^{\phi,I}_{R}(G):\, (\pi^{*})^{\phi,I} \ar@<0.5ex>[l]^-{\sim}
  }\end{equation}
\end{pro}
\begin{proof}
  This is Theorem (4.11) of \cite{Lanard} applied to the system of idempotents
  $(e_{\phi,I,x})_{x\in\BT'_{0/e}},$ except that our facet decomposition is the $e$-subdivision
  of  the one used in \emph{loc.\ cit.}, and in particular our vertices are not necessarily vertices
  in the context of \emph{loc.\ cit.}  Beyond the original techniques of \cite{MS1}
  that we
  have already adapted to our context in the proof of Lemma~\ref{lemma_MS_e_subdivided}, the main
  point of the argument of Lanard in \cite{Lanard} is the construction of ``local maps'' in $\S 3$ there.
  This relies in turn on the notion of ``admissible path'' between vertices ($\S 2$ of
  \emph{loc.cit.}), which only depends on the geometry of the facet decomposition of an
  apartment. But our facet decomposition is homothetic to that used in \cite{Lanard}, therefore the
  notion of admissible path works the same, and the argument goes through in our setting.   
\end{proof}

\alin{$G_{\iota}$-equivariant coefficient systems on $\BT'_{\iota}$}
Fix $\iota\in I$. Recall from \ref{choice_e} that, due to our choice of integer $e$, the
image $\BT_{\iota}'$ of $\BT_{\iota}$ in $\BT'$ 
is stable under the $e$-facet decomposition of $\BT'$, in the sense that an $e$-facet of $\BT'$ is
either contained in, or disjoint from, $\BT'_{\iota}$.
We denote by $\BT'_{\iota,\bullet/e}$\index[notation]{B'iotabullete@$\BT'_{\iota, \bullet/e}$}
the set of $e$-facets of $\BT'$ contained in $\BT_{\iota}'$.
\footnote{Note that there is also an intrinsic notion of $e$-facet for
$\BT'_{\iota}$, which is a priori coarser, and that we \emph{do not}
consider here. 
So, by ``$e$-facet'' we always mean an $e$-facet for $\BT'$.} 
Since the action of $G_{\iota}$ on $\BT'_{\iota}$ preserves the $e$-facet
decomposition, we may define a category $[\BT'_{\iota,\bullet/e}/G_{\iota}]$ and
an abelian $R$-linear category $\Coef_{R}(\BT'_{\iota,\bullet/e}/G_{\iota})$ of
smooth $G_{\iota}$-equivariant coefficient systems on $\BT'_{\iota,\bullet/e}$ as in
\ref{def_coef_syst}. We then have a projection functor
$[\BT'_{\iota,\bullet/e}/G_{\iota}]\To{\pi_{\iota}} [*/G_{\iota}]$ that induces a pair of
  adjoint functors 
  $$ (\pi_{\iota,!},\pi_{\iota}^{*}): \Coef_{R}(\BT'_{\iota,\bullet/e}/G_{\iota}) \rightleftarrows
\Rep_{R}(G_{\iota}),$$
as in \ref{reps-G-coeff-sys}. Moreover, since $G_{\iota,x,0+}$ only depends on the $e$-facet $\FC_{e}(x)$ containing
the point $x$, we may also consider the category of ``depth $0$'' coefficients system
defined as
\begin{center}
  \begin{tabular}[c]{lcl}
  $\Coef^{1}_{R}(\BT'_{\iota,\bullet/e}/G_{\iota})$ & $:=$ &
                                        \begin{minipage}[c]{0.6\linewidth}
                                          the full subcategory of
                                          $\Coef_{R}(\BT'_{\iota,\bullet/e}/G_{\iota})$ that
                                          consists 
                                          of all smooth $G_{\iota}$-equivariant coefficient systems $\VC$
                                          such that, for any facet $\cF$, we have 
                                          $\VC_{\FC} = e_{G_{\iota,\FC,0+}}\VC_{\FC}$,
                                        \end{minipage}
  \end{tabular}
\end{center}
where $e_{G_{\iota,\FC,0+}}$ is the idempotent averaging over $G_{\iota,\FC,0+}$ and thereby projecting onto the $G_{\iota,\FC,0+}$-fixed vectors. 
We observe that $\pi_{\iota,!}(\Coef^{1}_{R}(\BT'_{\iota,\bullet/e}/G_{\iota}))\subset \Rep^{1}_{R}(G_{\iota})$.

As in \ref{phi_I_coef_systems}, the equalities
 $e_{\iota,\FC,0+}e_{\iota,\FC',0+}=e_{\iota,\FC,0+}$ whenever $\o\FC\supseteq
\FC'$  and $e_{\iota,g\FC,0+}=ge_{\iota,\FC,0+}g^{-1}$ whenever $g\in G_{\iota}$ allow us to
construct  a right adjoint $\VC\mapsto \VC^{1}$ to the inclusion of
$\Coef^{1}_{R}(\BT'_{\iota,\bullet/e}/G_{\iota})$ into
$\Coef_{R}(\BT'_{\iota,\bullet/e}/G_{\iota})$, by setting
$\VC^{1}_{\FC}:=e_{\iota,\FC,0+}\VC_{\FC}$ for all $\FC\in\BT'_{\iota,\bullet/e}$.
Precomposing with $\pi_{\iota}^{*}$ then provides us with a pair of adjoint functors
  $$ (\pi_{\iota,!},(\pi_{\iota}^{*})^{1}): \Coef^{1}_{R}(\BT'_{\iota,\bullet/e}/G_{\iota}) \rightleftarrows
\Rep^{1}_{R}(G_{\iota}).$$
As in \ref{phi_I_coef_systems}, $(\pi_{\iota}^{*})^{1}$ factors through the following full
 subcategory
\index[notation]{Coef1RBT'@$\Coef^{1}_{R}(\BT'_{\iota,\bullet/e}/G_{\iota})$}
\begin{center}
  \begin{tabular}[c]{lcl}
  $\Coef^{1}_{R}(\BT'_{\iota,\bullet/e}/G_{\iota})^{\cart}$ & $:=$ &
                                        \begin{minipage}[c]{0.6\linewidth}
                                          the full subcategory of
                                          $\Coef_{R}(\BT'_{\iota,\bullet/e}/G_{\iota})$ that
                                          consists 
                                          of all smooth $G_{\iota}$-equivariant coefficient systems $\VC$
                                          such that, for any pair of facets $(\cF, \cF')$ satisfying
                                          $\o\FC\supseteq\FC'$, the map $\beta_{\VC,\FC,\FC'}$
                                          induces an isomorphism
                                          $\VC_{\FC}\simto e_{G_{\iota,\FC,0+}}\VC_{\FC'}$,
                                        \end{minipage}
  \end{tabular}
\end{center}
to whose objects we refer to as \emph{Cartesian depth-$0$} coefficient
systems.

\label{reps-Giota-coeff-sys}

\begin{pro} 
  The two above functors
  induce quasi-inverse equivalences of categories
  \ini
\begin{equation} \label{thm_equiv_repr_coef_iota}
  \xymatrix{
    \pi_{\iota,!}:\, \Coef^{1}_{R}(\BT'_{\iota,\bullet/e}/G_{\iota})^{\cart}
    \ar@<0.5ex>[r]^-{\sim} &
    \Rep^{1}_{R}(G_{\iota}):\, (\pi^{*}_{\iota})^{1} \ar@<0.5ex>[l]^-{\sim}
  }
\end{equation}
\end{pro}
\begin{proof}
This again follows from \cite{Lanard} with the same adaptation as in the proof of Proposition \ref{reps-G-coeff-sys}, namely
one needs to define ``local maps'' using ``admissible paths''. But for two $e$-facets $\FC,\FC'$
contained in $\BT'_{\iota}$, the combinatorial convex hull (for the
$e$-facet decomposition) $\HC_{e}(\FC,\FC')$ is contained in $\BT'_{\iota}$ and, therefore, any admissible
path from $\FC$ to $\FC'$ is contained in $\BT'_{\iota}$.
\end{proof}

In view of the equivalences (\ref{thm_equiv_repr_coef}) and (\ref{thm_equiv_repr_coef_iota}), our
aim will now be to construct a pair of adjoint equivalences of categories:
\ini\begin{equation} 
  \label{equiv_coef}
  \xymatrix{
    \Coef^{\phi,I}_{R}(\BT'_{\bullet/e}/G) ^{\cart}   \ar@<0.5ex>[r]^-{\sim} &
    \Coef^{1}_{R}(\BT'_{\iota,\bullet/e}/G_{\iota})^{\cart} \ar@<0.5ex>[l]^-{\sim} ,
  }
\end{equation}
which we discuss in the next subsection.

\subsection{Heisenberg--Weil coefficient systems} \label{sec:main-result-strategy-3}

In order to construct an equivalence as in \eqref{equiv_coef} we will use an auxiliary coefficient system that we call a \emph{Heisenberg--Weil coefficient system} and introduce below. To define such a coefficient system, we will rely on some results that we will prove in Section \ref{sec:heisenberg-rep}. We will also state some related results in this subsection that we will prove in Sections \ref{sec:heis-weil-rep} and \ref{sec:heis-weil-coeff}. In Section \ref{sec:constr-equiv} we then prove Theorem \ref{thm-main} assuming the results we stated in this subsection, Section \ref{sec:main-result-strategy-3}. Since Sections \ref{sec:heisenberg-rep}, \ref{sec:heis-weil-rep} and \ref{sec:heis-weil-coeff} do not rely on results from Sections \ref{sec:main-result-strategy-1}, \ref{sec:main-result-strategy-2}, \ref{sec:main-result-strategy-3} and \ref{sec:constr-equiv}, there is no circular reasoning. Presenting the results in this order allows readers interested in the results rather than the technical details of the proof to faster reach their goal and at the same time provides readers interested in all the details with a motivation for Sections \ref{sec:heisenberg-rep}, \ref{sec:heis-weil-rep} and \ref{sec:heis-weil-coeff}.

\alin{Heisenberg and Weil representations}\label{section-HW-coef-prep}
We will need the notion of Heisenberg representation, which we have to adapt to our
setting of a general commutative $\Rmin$-algebra $R$. According to Proposition (\ref{indep_facet})
ii), for  $x\in\BT_{\iota}$,     
the groups $\Ku_{\iota,x}$ and
$\Kp_{\iota,x}$ and the character $\check\phi^{+}_{\iota,x}$ 
depend only on the $e$-facet $\cF$ containing $x$. 
We may thus  denote them
by\index[notation]{KdaggeriotaF@$\Ku_{\iota,\FC}$}\index[notation]{KplusiotaF@$\Kp_{\iota,\FC}$}\index[notation]{phihatplusiotaF@$\check\phi^{+}_{\iota,\cF}$}
$\Ku_{\iota,\FC}$, $\Kp_{\iota,\cF}$ and $\check\phi^{+}_{\iota,\FC}$.

\begin{defn}
For  $\FC\in\BT'_{\iota,\bullet/e}$, a Heisenberg representation $\Heis_{\iota,\FC}$ for the triple
$(\Ku_{\iota,\FC},\Kp_{\iota,\cF},\check\phi^{+}_{\iota,\FC})$ is a finitely generated, projective $R\Ku_{\iota,\FC}$-module
that is $\check\phi^{+}_{\iota,\FC}$-isotypic when restricted to $R\Kp_{\iota,\FC}$ and that has $R\Ku_{\iota,\FC}$-endomorphism ring $R$.
\end{defn}

We refer to Lemma \ref{lem_Heisenberg} for a proof of existence, a discussion of
uniqueness, and a proof of the main property of interest to us, which is that such a
representation is a projective generator of the category   
$\HC_{R}(\Ku_{\iota,\FC})e_{\iota,\FC}-\Mod$ of all $\check\phi^{+}_{\iota,\FC}$-isotypic
smooth $R\Ku_{\iota,\FC}$-modules.
In particular, we have a pair of adjoint equivalences of categories:
\ini\begin{equation}
\label{Heisenberg-equivalence-of-cat}  
\Hom_{R\Ku_{\iota,\FC}}(\Heis_{\iota,\FC},-) :\,\,
  \HC_{R}(\Ku_{\iota,\FC})e_{\iota,\FC}-\Mod \rightleftarrows 
  R-\Mod
  \,\, : \Heis_{\iota,\FC}\otimes_{R} -.
\end{equation}

Also, as in the case where the coefficient ring is $\bC$, we will show in Section
\ref{sec:heis-weil-rep} how the theory of Weil 
representations produces  extensions of Heisenberg representations to
the bigger group\index[notation]{KaiotaF@$\K_{\iota, \cF}$} 
 \[\K_{\iota,\FC}:=G_{\iota,\FC}\Ku_{\iota,\FC} ,\]
 where\index[notation]{GiotaF@$G_{\iota,\FC}$} $G_{\iota,\FC}$ is the stabilizer of
 $\FC$ in $G_{\iota}$. We call such representations ``Heisenberg--Weil representations''. 
Any Heisenberg--Weil representation
$\kappa_{\iota,\FC}$ provides an enhancement of the above pair of adjoint equivalences to the following pair of adjoint equivalences:
\ini\begin{equation}
  \label{local_equivalences}  \;\;\;\;\;\;
  \Hom_{R\Ku_{\iota,\FC}}(\kappa_{\iota,\FC},-) :\,\,
  \HC_{R}(\K_{\iota,\FC})e_{\iota,\FC}-\Mod \rightleftarrows 
  R[G_{\iota,\FC}/G_{\iota,\FC,0+}]-\Mod
  \,\, : \kappa_{\iota,\FC}\otimes_{R} -.
\end{equation}
Here, the category on the right hand side is justified by the isomorphism
$G_{\iota,\FC}/G_{\iota,\FC,0+}\simto\K_{\iota,\FC}/\Ku_{\iota,\FC}$ arising from the
inclusion $G_{\iota,\FC}\subseteq\K_{\iota,\FC}$. 

So, a natural strategy to construct equivalences as in (\ref{equiv_coef}) is to try and arrange the
``local equivalences'' (\ref{local_equivalences}) in a sufficiently coherent way so that they
respect face maps and actions.
To this aim, we introduce the notion of a Heisenberg--Weil coefficient
system below. This will be a coefficient system on a category that will serve as an
intermediate between $[\BT'_{\bullet/e}/G]$ and $[\BT'_{\iota,\bullet/e}/G_{\iota}]$.
At this point, it might be helpful to note that when $\o\FC\supseteq
\FC'$, we have the following inclusions: 
$\Kp_{\iota,\FC'}\subseteq\Kp_{\iota,\FC}\subseteq \K_{\iota,\FC}\subseteq\K_{\iota,\FC'}$ (the
last one because $e$ was chosen so that the stabilizers of $e$-facets
coincide with their fixators).
On the other hand, $\Ku_{\iota,\FC}$ and $\Ku_{\iota,\FC'}$ are not contained in one
another, but we have $\Ku_{\iota,\FC}=(\Ku_{\iota,\FC}\cap\Ku_{\iota,\FC'})G_{\iota,\FC,0+}$.
These facts allow us to make the following definition:
\begin{DEf} We denote by $[\BT'_{\iota,\bullet/e}/\K_{\iota}]$ the following subcategory of
  $[\BT'_{\bullet/e}/G]$ :
  \begin{itemize}
  \item Its set of objects is the subset $\BT'_{\iota,\bullet/e}$ of $e$-facets of $\BT'$ contained
    in $\BT'_{\iota}$.
  \item For $\FC,\FC'\in\BT'_{\iota,\bullet/e}$, morphisms are given by
    $$\Hom(\FC,\FC'):=
    \{g\in \K_{\iota,\FC'}\,G_{\iota},\; \ov{g\FC}\supseteq\FC'\}=
    \{g\in \Ku_{\iota,\FC'}\,G_{\iota},\; \ov{g\FC}\supseteq\FC'\},$$ with
    composition given by multiplication in $G$.
  \end{itemize}
\end{DEf}
As usual, a coefficient system $\WC$ on $[\BT'_{\iota,\bullet/e}/K_{\iota}]$ is just a functor to
$R-\Mod$. Note that this is equivalent to giving a $G_{\iota}$-equivariant coefficient
system on $\BT'_{\iota,\bullet/e}$ together with for every $\FC\in\BT'_{\iota,\bullet/e}$
a smooth action of $\Ku_{\iota,\FC}$ on $\WC_{\iota,\FC}$ such that, 
for each pair $\FC,\FC'$ such that $\o\FC\supseteq\FC'$, the face map
$\beta_{\WC,\FC,\FC'}$ is $\Ku_{\iota,\FC}\cap \Ku_{\iota,\FC'}$-equivariant and, 
 for each $g\in G_{\iota}$, the action map $g_{\WC,\FC}$ induces 
a $\Ku_{\iota,\FC}$-equivariant isomorphism $g_{\WC,\FC}: \WC_{\FC}\simto g^{*}\WC_{g\FC}$ 
that  coincides with the action of $g$ on  $\WC_{\FC}$ whenever $g\in G_{\iota,\FC}$.
Note that since $\beta_{\WC,\FC,\FC'}$ is also $G_{\iota, \cF}=(G_{\iota, \cF} \cap G_{\iota, \cF'})$-equivariant, as part of the definition of  a $G_{\iota}$-equivariant coefficient
system on $\BT'_{\iota,\bullet/e}$, it follows that $\beta_{\WC,\FC,\FC'}$ is also $K_{\iota, F}$-equivariant.
\begin{DEf} \label{def_WH_coef_syst}
  A coefficient system   $\WC_{\iota}:\, [\BT'_{\iota,\bullet/e}/\K_{\iota}]\To{} R-\Mod$
  is called a \emph{Heisenberg--Weil} coefficient system if :
  \begin{enumerate}
  \item[1.]  For each $\FC\in\BT'_{\iota,\bullet/e}$, the $R\Ku_{\iota,\FC}$-module
    $\WC_{\iota,\FC}$ is a Heisenberg representation for the triple
    $(\Ku_{\iota,\FC},\Kp_{\iota,\cF},\check\phi^{+}_{\iota,\FC})$.    
  \item[2.]  For each pair of facets  $\FC'\subseteq\o\FC$ in $\BT'_{\iota,\bullet/e}$, the map
    $\beta_{\WC_{\iota},\FC,\FC'}$ is a generator of the $R$-module
    $\Hom_{R(\Ku_{\iota,\FC}\cap\Ku_{\iota,\FC'})}(\WC_{\iota,\FC},\WC_{\iota,\FC'})$.
  \end{enumerate}
\end{DEf}

\begin{prop} \label{rem_WH_coef_syst} Let $\WC_{\iota}$ be a Heisenberg--Weil coefficient system.
Then the $R$-module $\Hom_{R(\Ku_{\iota,\FC}\cap\Ku_{\iota,\FC'})}(\WC_{\iota,\FC},\WC_{\iota,\FC'})$ is free of rank one, and 
  $\beta_{\WC_{\iota},\FC,\FC'}$ induces
  \begin{enumerate}
  \item \label{rem_WH_coef_system-i} a $\Ku_{\iota,\FC'}$-equivariant isomorphism
    $ {\rm ind}_{\Ku_{\iota,\FC}\cap\Ku_{\iota,\FC'}}^{\Ku_{\iota,\FC'}}
    \WC_{\iota,\FC}\simto\WC_{\iota,\FC'}$, and
  \item
    \label{rem_WH_coef_system-ii}
    a $\Ku_{\iota,\FC}\cap\Ku_{\iota,\FC'}$-equivariant isomorphism
    $\WC_{\iota,\FC}\simto e'_{\iota,\FC}\WC_{\iota,\FC'}$, where $e'_{\iota,\FC}$ is the
    idempotent associated to the restriction of $\check\phi_{\iota,\FC}$ to
    $\Kp_{\iota,\FC}\cap\Ku_{\iota,\FC'}$.
  \end{enumerate}
Since $\beta_{\WC_{\iota},\FC,\FC'}$ is also compatible with the action of $G_{\iota}$,
    the isomorphism in (\ref{rem_WH_coef_system-ii}) is actually $\K_{\iota,\FC}$-equivariant and
     $e'_{\iota,\FC}\WC_{\iota,\FC'}=e_{\iota,\FC}\WC_{\iota,\FC'}$.
\end{prop}
\begin{proof}
	  We show in Lemma \ref{intertwining_Heisenberg} \ref{intertwining_Heisenberg-i}) that the $R$-module $\Hom_{R(\Ku_{\iota,\FC}\cap\Ku_{\iota,\FC'})}(\WC_{\iota,\FC},\WC_{\iota,\FC'})$
	is in general invertible, granted Property 1 of Definition \ref{def_WH_coef_syst}. In the presence of a generator, it is thus free of
	rank $1$. We show in Lemma \ref{intertwining_Heisenberg} \ref{intertwining_Heisenberg-iii}) that then a generator like   $\beta_{\WC_{\iota},\FC,\FC'}$ induces the isomorphisms in \ref{rem_WH_coef_system-i}) and \ref{rem_WH_coef_system-ii}). That $\beta_{\WC_{\iota},\FC,\FC'}$  is  $\K_{\iota,\FC}$-equivariant was already observed above, and hence $e'_{\iota,\FC}\WC_{\iota,\FC'}=e_{\iota,\FC}\WC_{\iota,\FC'}$.
\end{proof}

\begin{theo} \label{thm_WH_coef_system-general}
	Let $R$ be a commutative $\Rmintwo$-algebra. Then
	there exists a Heisenberg--Weil coefficient system $\WC_{\iota}$ in $\Coef_{R}(\BT'_{\iota,\bullet/e}/K_{\iota})$.
\end{theo}

We postpone the proof of this theorem to Section
\ref{sec:heis-weil-coeff}, since it requires a long technical preparation done in Section
\ref{sec:heis-weil-rep}. This theorem will be the main ingredient in the construction of the equivalence
(\ref{equiv_coef}).
We will also show the following  uniqueness result in
\ref{uniqueness_WH_coef}.

\begin{prop} \label{uniqueness_WH_coef-system}
  If $\WC_{\iota}$ and $\WC'_{\iota}$ are two Heisenberg--Weil coefficient systems,
  then there is an invertible $R$-module
$L$ and a depth-$0$ character $\theta:G_{\iota}\To{}R^{\times}$ such that
$\WC'_{\iota}\simeq\WC_{\iota}\otimes_{R}L_{\theta}$. 
\end{prop}

\subsection{Construction of the equivalences} \label{sec:constr-equiv}
In this section, we associate to any Heisenberg--Weil coefficient system a pair
of equivalences as in (\ref{equiv_coef}) satisfying the extra properties of Theorem
\ref{thm-main}. More precisely, we prove the more precise version of Theorem, \ref{thm-main}, which is Theorem 
\ref{thm-main-2} below, contingent on the proofs of the results stated in Section \ref{sec:main-result-strategy-3}, which in turn are provided in Sections \ref{sec:heisenberg-rep}, \ref{sec:heis-weil-rep} and \ref{sec:heis-weil-coeff} (which are independent of the present section, so there is no circular reasoning).

\alin{Overview} \label{geom_interp}
To a Heisenberg--Weil coefficient system $\WC_{\iota}$ as in Definition
\ref{def_WH_coef_syst}, we will associate a pair of quasi-inverse equivalences of categories
$$  \xymatrix{
    \IC_{\WC_{\iota}} :\, \Coef^{1}_{R}(\BT'_{\iota,\bullet/e}/G_{\iota})  \ar@<0.5ex>[r]^-{\sim} &
   \Coef^{\phi,I}_{R}(\BT'_{\bullet/e}/G)\,:\RC_{\WC_{\iota}} \ar@<0.5ex>[l]^-{\sim}
 }$$
 that respect cartesian objects. Since the construction of these functors, given in
 \ref{first-direction} and \ref{sec:seconddirection}, might appear
 quite technical at first sight, we try to offer here a geometric intuition. This section
 is logically independent from the sequel and can be skipped.

  Recall that $\WC_{\iota}$ is a  coefficient system on the category
  $[\BC'_{\iota,\bullet/e}/K_{\iota}]$ introduced above Definition \ref{def_WH_coef_syst}. We have functors
  $$ [\BC'_{\iota,\bullet/e}/\o G_{\iota}] \overset{\hbox{\tiny $p$}}{\longleftarrow}
  [\BC'_{\iota,\bullet/e}/K_{\iota}] \To{q}  [\BC'_{\bullet/e}/G].$$
  Here $q$ is the natural inclusion. On the other side, 
  $[\BC'_{\iota,\bullet/e}/\o G_{\iota}]$ denotes the category with set of objects $\BT'_{\iota,\bullet/e}$
  and morphisms given by
  $\Hom(\FC,\FC')=\{g\in G_{\iota}, \o{g\FC}\supseteq\FC'\}/G_{\iota,\FC,0+}$.
  The functor  $p$ is the identity on objects and is given on  morphisms by the (well
  defined) maps
  $$ g=kg_{\iota}\in \Ku_{\iota,\FC'}G_{\iota} \hbox{(such that $g\o\FC\supseteq\FC'$)}
  \mapsto \o{g_\iota}\in G_{\iota}/G_{\iota,\FC,0+}.$$
  Note that $\Coef_{R}([\BC'_{\iota,\bullet/e}/\o G_{\iota}])$ identifies to
  $\Coef_{R}^{1}([\BC'_{\iota,\bullet/e}/G_{\iota}])$ via the pullback $r^{*}$ along the
  projection functor $[\BC'_{\iota,\bullet/e}/G_{\iota}]\To{r}[\BC'_{\iota,\bullet/e}/\o
  G_{\iota}]$.

  Associated with $p$ and $q$ are two pull-back functors (given by precomposition)
 $$ \Coef_{R}(\BC'_{\iota,\bullet/e}/\o G_{\iota}]) \To{p^{*}}
  \Coef_{R}(\BC'_{\iota,\bullet/e}/K_{\iota}) \overset{\hbox{\tiny $q^{*}$}}{\longleftarrow}  \Coef_{R}(\BC'_{\bullet/e}/G).$$
  The functor $p^{*}$  has a right adjoint $p_{*}$ given by
$p_{*}(\WC)_{\FC}:=\{w\in\WC_{\FC}, \forall\FC'\subseteq\o\FC, \beta_{\WC,\FC,\FC'}(w)\in
(\WC_{\FC'})^{\Ku_{\iota,\FC'}}\}$ for $\FC\in\BT'_{\iota,\bullet/e}$. The 
functor $q^{*}$ has a left adjoint $q_{!}$, a description of which is done in \ref{sec:seconddirection}.
Now, on $\Coef_{R}(\BC'_{\iota,\bullet/e}/K_{\iota})$ we have tensor products,
given on facets by $(\WC\otimes_{R}\WC')_{\FC}=\WC_{\FC}\otimes_{R}\WC'_{\FC}$ with
obvious face maps and action maps.  We also have internal homs defined on facets by
$\HC om(\WC,\WC')_{\FC} := \Hom(\WC|_{\o\FC},\WC'|_{\o\FC})$, where the $\Hom$ on the right hand side
denotes the $R$-module of morphisms between coefficient systems $\WC$ and $\WC'$
restricted to $\o\FC$. As in sheaf theory, these constructions satisfy the usual
Hom-tensor adjunctions.

It now becomes  natural to consider the adjoint pair of pull-push functors with kernel $\WC_{\iota}$ : 
  $$ \IC_{\WC_{\iota}}:\, \VC_{\iota} \mapsto  q_{!} (\WC_{\iota}\otimes_{R} p^{*}\VC_{\iota})
  \hbox{ and }
  \RC_{\WC_{\iota}}:\, \VC \mapsto  p_{*} ({\rm \mathcal{H}om}_{R}(\WC_{\iota}, q^{*}\VC)).$$

  In the next three sections we provide more explicit definitions of these functors
  and a proof  that they induce the desired equivalences between our categories of interest. However, we
   leave it to the reader to check that these explicit definitions indeed describe the above functors.

\alin{A functor from $\Coef^{\phi,I}_{R}(\BT'_{\bullet/e}/G)$ to $\Coef^{1}_{R}(\BT'_{\iota,\bullet/e}/G_{\iota})$}
\label{first-direction}
This is the easier direction, since we start from a coefficient system defined on the target building.
As above, we fix a Heisenberg--Weil coefficient system $\WC_{\iota}$ as in Definition \ref{def_WH_coef_syst}, and let
$\VC$ be a coefficient system in $\Coef^{\phi,I}_{R}(\BT'_{\bullet/e}/G)$. We define a
coefficient system $\VC_{\iota}$ on $\BT'_{\iota}$ as follows.

\begin{itemize}
\item $\forall\FC\in\BT'_{\iota,\bullet/e}$, set
  $\VC_{\iota,\FC}:=\Hom_{R\Ku_{\iota,\FC}}(\WC_{\iota,\FC},\VC_{\FC})$.
\item If $\o\FC\supseteq\FC'$, define $\beta_{\VC_{\iota},\FC,\FC'}$ as the unique map that makes the
  following diagram commute:
  $$\xymatrix{
    \VC_{\iota,\FC}=\Hom_{R\Ku_{\iota,\FC}}(\WC_{\iota,\FC},\VC_{\FC})
    \ar[d]_{u\mapsto \beta_{\VC,\FC,\FC'}\circ u} \ar@{..>}[rr]^{\beta_{\VC_{\iota},\FC,\FC'}}
    && \VC_{\iota,\FC'}=\Hom_{R\Ku_{\iota,\FC'}}(\WC_{\iota,\FC'},\VC_{\FC'})
    \ar[d]^{u\mapsto u\circ\beta_{\WC,\FC,\FC'}} \\
    \Hom_{R(\Ku_{\iota,\FC}\cap\Ku_{\iota,\FC'})}(\WC_{\iota,\FC},\VC_{\FC'})
    \ar@{=}[rr] && \Hom_{R(\Ku_{\iota,\FC}\cap\Ku_{\iota,\FC'})}(\WC_{\iota,\FC},\VC_{\FC'})
  } $$
  Existence and uniqueness follow from \ref{rem_WH_coef_system-i}) of Proposition \ref{rem_WH_coef_syst}  and Frobenius reciprocity.
\item If $g\in G_{\iota}$, define $g_{\VC_{\iota},\FC}:\VC_{\iota,\FC}\To{\sim}\VC_{\iota,g\FC}$ as
  $$\xymatrix{ \VC_{\iota,\FC}=\Hom_{R\Ku_{\iota,\FC}}(\WC_{\iota,\FC},\VC_{\FC})
  \ar[rrr]^{u\mapsto g_{\VC,\FC}\circ u\circ (g_{\WC,\FC})^{-1}}
 &&& \Hom_{R\Ku_{\iota,g\FC}}(\WC_{\iota,g\FC},\VC_{g\FC})}=\VC_{\iota,g\FC}$$
\end{itemize}

\begin{lem}
  The above data define an object $\VC_{\iota}$ of
  $\Coef^{1}_{R}(\BT'_{\iota,\bullet/e}/G_{\iota})$, that is Cartesian if $\VC$ is Cartesian.
\end{lem}
\begin{proof} The transitivity of the action maps, i.e., the identities
  $(hg)_{\VC_{\iota},\FC}=h_{\VC_{\iota},g\FC}\circ g_{\VC_{\iota},\FC}$, follows from that of the
  action maps for $\VC$ and $\WC$.
  
  Let us check the transitivity of the face maps. Given $\FC,\FC',\FC''$ such that
  $\o\FC\supseteq\o\FC'\supseteq\o\FC''$, we see that
  $\beta_{\VC_{\iota},\FC',\FC''}\circ\beta_{\VC_{\iota},\FC,\FC'}$ makes the diagram that defines
  $\beta_{\VC_{\iota},\FC,\FC''}$ commute. Hence, by uniqueness, both maps are equal.

  Let us now check compatibility of the face maps with the action. Given $\FC,\FC'$ such that
  $\o\FC\supseteq\FC'$ and $g\in G_{\iota}$, the map
  $g_{\VC_{\iota},\FC'}\circ\beta_{\VC_{\iota},\FC,\FC'}\circ (g_{\VC_{\iota},\FC})^{-1}$ makes the
  following diagram commute :
    $$\xymatrix{
      \VC_{\iota,g\FC}=\Hom_{R\Ku_{\iota,g\FC}}(\WC_{\iota,g\FC},\VC_{g\FC})
      \ar[d]_{u\mapsto  (g_{\VC,\FC})^{-1}\circ u\circ g_{\WC,\FC}}
    \ar[rrr]^{g_{\VC_{\iota},\FC'}\circ\beta_{\VC_{\iota},\FC,\FC'}\circ (g_{\VC_{\iota},\FC})^{-1}} 
    &&& \VC_{\iota,g\FC'}=\Hom_{R\Ku_{\iota,g\FC'}}(\WC_{\iota,g\FC'},\VC_{g\FC'})
    \ar[d]^{u\mapsto (g_{\VC,\FC'})^{-1}\circ u\circ g_{\WC,\FC'}}
    \\
    \VC_{\iota,\FC}=\Hom_{R\Ku_{\iota,\FC}}(\WC_{\iota,\FC},\VC_{\FC})
    \ar[d]_{u\mapsto \beta_{\VC,\FC,\FC'}\circ u} \ar[rrr]^{\beta_{\VC_{\iota},\FC,\FC'}}
    &&& \VC_{\iota,\FC'}=\Hom_{R\Ku_{\iota,\FC'}}(\WC_{\iota,\FC'},\VC_{\FC'})
    \ar[d]^{u\mapsto u\circ\beta_{\WC,\FC,\FC'}} \\    
    \Hom_{R(\Ku_{\iota,\FC}\cap\Ku_{\iota,\FC'})}(\WC_{\iota,\FC},\VC_{\FC'})
    \ar@{=}[rrr]  \ar[d]_{u\mapsto g_{\VC,\FC'}\circ u \circ (g_{\WC,\FC})^{-1}}
    &&& \Hom_{R(\Ku_{\iota,\FC}\cap\Ku_{\iota,\FC'})}(\WC_{\iota,\FC},\VC_{\FC'})
    \ar[d]^{u\mapsto g_{\VC,\FC'}\circ u \circ (g_{\WC,\FC})^{-1}} \\
     \Hom_{R(\Ku_{\iota,g\FC}\cap\Ku_{\iota,g\FC'})}(\WC_{\iota,g\FC},\VC_{g\FC'})
    \ar@{=}[rrr]
    &&& \Hom_{R(\Ku_{\iota,g\FC}\cap\Ku_{\iota,g\FC'})}(\WC_{\iota,g\FC},\VC_{g\FC'})
  } $$
Since $\VC$ is a coefficient system, we have $\beta_{\VC,\FC,\FC'}\circ (g_{\VC,\FC})^{-1}=
(g_{\VC,\FC'})^{-1}\circ\beta_{\VC,g\FC,g\FC'}$, hence the composition of the left vertical maps is
$u\mapsto \beta_{\VC,g\FC,g\FC'}\circ u$. Similarly, $\WC$ being a $G_\iota$-equivariant coefficient system, we have
$g_{\WC,\FC'}\circ\beta_{\WC,\FC,\FC'}=\beta_{\WC,g\FC,g\FC'}\circ g_{\WC,\FC}$, and the composition
of the right vertical maps is $u\mapsto u\circ\beta_{\WC,g\FC,g\FC'}$. It follows that the top map
makes the defining diagram of $\beta_{\VC_{\iota},g\FC,g\FC'}$  commute, hence $\beta_{\VC_{\iota},g\FC,g\FC'}=  g_{\VC_{\iota},\FC'}\circ\beta_{\VC_{\iota},\FC,\FC'}\circ (g_{\VC_{\iota},\FC})^{-1}$ as desired.

At this point, we have proved that $\VC_{\iota}$ is a $G_{\iota}$-equivariant coefficient
system on $\BT'_{\iota,\bullet/e}$. Since $G_{\iota,\FC,0+}$ visibly acts trivially on
$\VC_{\iota,\FC}$, this coefficient system lies in $\Coef^{1}_{R}(\BT'_{\iota,\bullet/e}/G_{\iota})$.

Let us now assume that $\VC$ is Cartesian. It remains to check that $\beta_{\VC_{\iota},\FC,\FC'}$
induces an isomorphism $\VC_{\iota,\FC}\simto e_{G_{\iota,\FC,0+}}\VC_{\iota,\FC'}$ for all
pairs of facets $\FC'\subset\o\FC\subset\BT'_{\iota,\bullet/e}$.
The injectivity of $\beta_{\VC_{\iota},\FC,\FC'}$ follows from
its defining diagram and the injectivity of $\beta_{\VC,\FC,\FC'}$. The fact that the image of
$\beta_{\VC_{\iota},\FC,\FC'}$ is invariant under the subgroup $G_{\iota,\FC,0+}$ of
$G_{\iota,\FC'}$ follows from its compatibility with the action maps and the fact that
$G_{\iota,\FC,0+}\subseteq \Ku_{\iota,\FC}$ acts trivially on $\VC_{\iota,\FC}$.  Now, let $u\in
\VC_{\iota,\FC'}=\Hom_{R\Ku_{\iota,\FC'}}(\WC_{\iota,\FC'},\VC_{\FC'})$. The image
$u(\beta_{\cW,\cF,\cF'}(\WC_{\iota,\FC}))$ of the subspace
$\beta_{\cW,\cF,\cF'}(\WC_{\iota,\FC})=e'_{\iota,\cF}\cW_{\iota,\cF'} \subseteq \cW_{\iota,\cF'}$ is contained in the $(\check\phi_{\iota,\FC})|_{\Kp_{\iota,\FC}\cap\Ku_{\iota,\FC'}}$-isotypic
submodule of $\VC_{\FC'}$. If, additionally, $u$ is $G_{\iota,\FC,0+}$-invariant, then  also
$u(\beta_{\cW,\cF,\cF'}(\WC_{\iota,\FC}))\subseteq (\VC_{\FC'})^{G_{\iota,\FC,0+}}$, because
$G_{\iota,\FC,0+} \subset G_{\iota,\FC} \cap G_{\iota,\FC}$ acts trivially on $\WC_{\iota,\FC}$ and hence also on $\beta_{\cW,\cF,\cF'}(\WC_{\iota,\FC})$. Since
$G_{\iota,\FC,0+}.(\Kp_{\iota,\FC}\cap\Ku_{\iota,\FC'})= \Kp_{\iota,\FC}$ and
$\check\phi_{\iota,\FC}$ is trivial on $G_{\iota,\FC,0+}$, this means that
$u(\beta_{\cW,\cF,\cF'}(\WC_{\iota,\FC}))\subseteq e_{\iota,\FC}\VC_{\iota,\FC'}$.  But,
$\VC\in\Coef^{\phi,I}_{R}(\BT'_{\bullet/e}/G)$, so we know that $\beta_{\VC,\FC,\FC'}$ is an
isomorphism $\VC_{\FC}\simto e_{I,\FC}\VC_{\FC'}$. Applying the idempotent
$e_{\iota,\FC}$ on both sides, we see that $\beta_{\VC,\FC,\FC'}$ induces an
isomorphism $e_{\iota,\FC}\VC_{\FC}\simto e_{\iota,\FC}\VC_{\FC'}$.
Composing $u \circ \beta_{\cW,\cF,\cF'}$ with the inverse of
that isomorphism
yields a morphism $v : \WC_{\iota,\FC}\To{}e_{\iota,\FC}\VC_{\FC}$, which is both 
$\Ku_{\iota,\FC}\cap\Ku_{\iota,\FC'}$-equivariant and $G_{\iota,\FC,0+}$-equivariant.
Since $G_{\iota,\FC,0+}.(\Ku_{\iota,\FC}\cap\Ku_{\iota,\FC'})= \Ku_{\iota,\FC}$, we see that $v$
belongs to $\VC_{\iota,\FC}=\Hom_{R\Ku_{\iota,\FC}}(\WC_{\iota,\FC},\VC_{\FC})$. By construction, we
have $\beta_{V_{\iota},\FC,\FC'}(v)=u$.
\end{proof}

The assignment $\VC\mapsto\VC_{\iota}$ is clearly functorial in $\VC$, providing a functor
$$ \RC_{\WC_{\iota}} :\, \Coef^{\phi,I}_{R}(\BT'_{\bullet/e}/G)\To{} \Coef^{1}_{R}(\BT'_{\iota,\bullet/e}/G_{\iota}),$$ 
that maps $\Coef^{\phi,I}_{R}(\BT'_{\bullet/e}/G)^{\cart}$ into $\Coef^{1}_{R}(\BT'_{\iota,\bullet/e}/G_{\iota})^{\cart}$.

\alin{A functor from $\Coef^{1}_{R}(\BT'_{\iota,\bullet/e}/G_{\iota})$ to $\Coef^{\phi,I}_{R}(\BT'_{\bullet/e}/G)$}\label{sec:seconddirection}
This direction carries additional complications since we start from a coefficient
system only defined on a subset of the target building. So, let
$\VC_{\iota}$ be an object of $\Coef^{1}_{R}(\BT'_{\iota,\bullet/e}/G_{\iota})$.
We define the following objects.

\begin{itemize}
\item For any $\FC\in \BT'_{\bullet/e}$, define
   $$ \VC_{\FC}:= \left( \bigoplus_{\FC_{\iota}\in\BT'_{\iota,\bullet/e}}
     \VC_{\FC,\FC_{\iota}}     \right)_{\!\!G_{\iota}}
\;\hbox{ with }\;
\VC_{\FC,\FC_{\iota}} := \CC^{\infty}_{c}(G_{\FC\FC_{\iota}})\otimes_{R\K_{\iota,\FC_{\iota}}} (\WC_{\iota,\FC_{\iota}}
      \otimes_{R}\VC_{\iota,\FC_{\iota}}).
             $$  
  Here, \index[notation]{GFFiota@$G_{\FC\FC_{\iota}}$}$G_{\FC\FC_{\iota}}:=\{g\in G, g\FC_{\iota}=\FC\}$ is an open subset of  $G$ and
  $\CC^{\infty}_{c}$ denotes smooth, compactly supported $R$-valued functions\footnote{We
    set $\CC^{\infty}_{c}(G_{\FC\FC_{\iota}})=\{0\}$ when $G_{\FC\FC_{\iota}}$ is empty.}. In
  particular, $\CC^{\infty}_{c}(G_{\FC\FC_{\iota}})$ has a right $R\K_{\iota,\FC_{\iota}}$-module
  structure arising from  multiplication
   in $G$, i.e., $g \in G$ acts on a function by right-translation by $g^{-1}$, while $\WC_{\iota,\FC_{\iota}}$ is a left
  $R\K_{\iota,\FC_{\iota}}$-module by definition, and so is
  $\VC_{\iota,\FC_{\iota}}$ through the
  projection map $\K_{\iota,\FC_{\iota}}\twoheadrightarrow G_{\iota,\FC_{\iota}}$. Finally,
  $h\in G_{\iota}$ acts as the sum of all maps $\VC_{\FC,\FC_{\iota}}\To{}\VC_{\FC,h\FC_{\iota}}$,
  $f\otimes w\otimes v\mapsto fh^{-1}\otimes  h_{\WC_{\iota},\FC_{\iota}}(w)\otimes
  h_{\VC_{\iota},\FC_{\iota}}(v)$.
  
\item For $g\in G$, let  $g_{\VC,\FC}:\VC_{\FC}\To{}\VC_{g\FC}$ be induced by the sum over $\cF_{\iota}\in\BT'_{\iota,\bullet/e}$
  of all maps $\VC_{\FC,\FC_{\iota}}\To{}\VC_{g\FC,\FC_{\iota}}$,
  $f\otimes w\otimes v\mapsto gf\otimes w\otimes v$, where $gf(h)=f(g^{-1}h)$ for $h \in G_{(g\cF)\cF_\iota}$.

\item For $\FC,\FC'$ such that $\o\FC\supseteq\FC'$, we let
  $\beta_{\VC,\FC,\FC'}:\VC_{\FC}\To{}\VC_{\FC'}$ be the map induced by the sum  over $\cF_{\iota}, \cF'_{\iota} \in\BT'_{\iota,\bullet/e}$ with $\o\FC_{\iota}\supseteq\FC'_{\iota}$ of all maps
  \[p_{\FC\FC_{\iota},\FC'\FC'_{\iota}}\otimes
  \beta_{\WC_{\iota},\FC_{\iota},\FC'_{\iota}}\otimes\beta_{\VC_{\iota},\FC_{\iota},\FC'_{\iota}}
:\,  \VC_{\FC,\FC_{\iota}}\To{}\VC_{\FC',\FC'_{\iota}} \]
 where
 $p_{\FC\FC_{\iota},\FC'\FC'_{\iota}}:
\CC^{\infty}_{c}(G_{\FC\FC_{\iota}})\To{}\CC^{\infty}_{c}(G_{\FC'\FC'_{\iota}}) $
denotes restriction of functions to $G_{\FC\FC_{\iota}}\cap G_{\FC'\FC'_{\iota}}$ followed
by extension by zero to $G_{\FC'\FC'_{\iota}}$. These maps descend to
  $G_{\iota, \cF_\iota}$-coinvariants since $G_{\iota,\FC_{\iota}} \subseteq \K_{\iota,\FC_{\iota}}\subseteq\K_{\iota,\FC'_{\iota}}$
  (by our choice of $e$), and then the sum
  descends to $G_{\iota}$-coinvariants.
\end{itemize}

\begin{lem}
  The above data define an object $\VC$ in $\Coef^{\phi,I}_{R}(\BT'_{\bullet/e}/G)$, that
  is Cartesian if $\VC_{\iota}$ is cartesian. 
\end{lem}
\begin{proof}
  Both the transitivity of the action maps and their compatibility with face maps are
  straightforward. To prove transitivity for face maps, let $\FC,\FC'$ and $\FC''$ such
  that $\o\FC\supseteq\o\FC'\supseteq\FC''$. Using transitivity of face maps for the
  coefficient systems $\WC_{\iota}$ and $\VC_{\iota}$, we are left to prove that for any
  two facets $\FC_{\iota},\FC''_{\iota} \in\BT'_{\iota,\bullet/e}$ such that
  $\o\FC_{\iota}\supseteq\FC''_{\iota}$, we have
  $$ p_{\FC\FC_{\iota},\FC''\FC''_{\iota}} = \sum_{\FC'_{\iota}, \, \o\FC_{\iota}\supseteq\o\FC'_{\iota}\supseteq\FC''_{\iota}}
  p_{\FC'\FC'_{\iota},\FC''\FC''_{\iota}}\circ p_{\FC\FC_{\iota},\FC'\FC'_{\iota}},$$
  or equivalently, that
  $$ G_{\FC\FC_{\iota}}\cap G_{\FC''\FC''_{\iota}} =
  \bigsqcup_{\FC'_{\iota}, \, \o\FC_{\iota}\supseteq\o\FC'_{\iota}\supseteq\FC''_{\iota}}
  G_{\FC\FC_{\iota}}\cap G_{\FC'\FC'_{\iota}}\cap G_{\FC''\FC''_{\iota}}.$$
  This follows from the observation that  for any $g\in G_{\FC\FC_{\iota}}\cap G_{\FC''\FC''_{\iota}}$,
  the facet $\FC'_{\iota}(g):=g^{-1}\FC'$ satisfies
  $\o\FC_{\iota}\supseteq\o{\FC'_{\iota}(g)}\supseteq\FC''_{\iota}$ and, in particular,
  belongs to $\BT'_{\iota,\bullet/e}$.

    At this point, we know that $\VC$ is a smooth $G$-equivariant coefficient system on
    $\BT'_{\bullet/e}$. To show that it belongs to $\Coef_{R}^{\phi,I}(\BT'_{\bullet/e}/G)$,
    it suffices to prove that $e_{I,\FC}\VC_{\FC,\FC_{\iota}}=\VC_{\FC,\FC_{\iota}}$ for
    all facets $\FC_{\iota}$. We may assume that  $\FC_{\iota}$ is  $G$-conjugate to $\FC$
    since otherwise both sides vanish.
    Observe that
    $\VC_{\FC,\FC_{\iota}} \simfrom
    \CC^{\infty}_{c}(G_{\FC\FC_{\iota}})e_{\iota,\FC_{\iota}}\otimes_{R\K_{\iota,\FC_{\iota}}}
    (\WC_{\iota,\FC_{\iota}} \otimes_{R}\VC_{\iota,\FC_{\iota}})$ since $e_{\iota,\FC_{\iota}}$
    acts as identity on $\WC_{\iota,\FC_{\iota}} \otimes_{R}\VC_{\iota,\FC_{\iota}}$.
    Similarly $e_{I,\FC}\VC_{\FC,\FC_{\iota}} \simfrom
    e_{I,\FC}\CC^{\infty}_{c}(G_{\FC\FC_{\iota}})e_{\iota,\FC_{\iota}}\otimes_{R\K_{\iota,\FC_{\iota}}}
    (\WC_{\iota,\FC_{\iota}} \otimes_{R}\VC_{\iota,\FC_{\iota}})$. But
    $e_{I,\FC}\CC^{\infty}_{c}(G_{\FC\FC_{\iota}})e_{\iota,\FC_{\iota}}=
    \CC^{\infty}_{c}(G_{\FC\FC_{\iota}})e_{I,\FC_{\iota}}e_{\iota,\FC_{\iota}}= 
    \CC^{\infty}_{c}(G_{\FC\FC_{\iota}})e_{\iota,\FC_{\iota}}$
    since $e_{I,\FC_{\iota}}e_{\iota,\FC_{\iota}}=e_{\iota,\FC_{\iota}}$ by Lemma
    \ref{ortho_idemp}.
    
    Let us now assume that $\VC_{\iota}$ is Cartesian. It then remains to show that $\beta_{\VC,\FC,\FC'}$ induces an
  isomorphism $\VC_{\FC}\simto e_{I,\FC}\VC_{\FC'}$ whenever $\o\FC\supseteq\FC'$. To this
  aim, it is useful to note that $\VC_{\FC}$ can also be written as 
     $$ \VC_{\FC}= \left( \bigoplus_{\FC_{\iota}\in\BT'_{\iota,\bullet/e}}
     \VC_{\FC,\FC_{\iota}}^{\dag}     \right)_{G_{\iota}}
\;\hbox{ with }\;
\VC_{\FC,\FC_{\iota}}^{\dag} := \CC^{\infty}_{c}(G_{\FC\FC_{\iota}})\otimes_{R\Ku_{\iota,\FC_{\iota}}} (\WC_{\iota,\FC_{\iota}}
      \otimes_{R}\VC_{\iota,\FC_{\iota}}).
             $$  
Indeed, the $G_{\iota}$-coinvariants on the sum take care of the difference between
$\Ku_{\iota,\FC_{\iota}}$ and $\K_{\iota,\FC_{\iota}}$ coinvariants on each summand.
Since $g^{-1}\FC'\in\BT'_{\iota,\bullet/e}$ for any $g\in G_{\FC\FC_{\iota}}$, we have a
further decomposition
$\VC^{\dag}_{\FC,\FC_{\iota}} =
\bigoplus_{\FC'_{\iota}\in\BT'_{\bullet/e},\o\FC_{\iota}\supseteq\FC'_{\iota}}
\CC^{\infty}_{c}(G_{\FC\FC_{\iota}}\cap G_{\FC'\FC'_{\iota}})
\otimes_{R\Ku_{\iota,\FC_{\iota}}} 
(\WC_{\iota,\FC_{\iota}}\otimes_{R}\VC_{\iota,\FC_{\iota}})$,
which yields the following expression :
\ini\begin{equation}
  \label{VF}
  \VC_{\FC} = \left(\bigoplus_{\cF_{\iota}, \cF_{\iota}' \in \BT'_{\bullet/e}, \, \o\FC_{\iota}\supseteq\FC_{\iota}'}
    \CC^{\infty}_{c}(G_{\FC\FC_{\iota}}\cap G_{\FC'\FC'_{\iota}})
    \otimes_{R\Ku_{\iota,\FC_{\iota}}} 
    (\WC_{\iota,\FC_{\iota}}\otimes_{R}\VC_{\iota,\FC_{\iota}})\right)_{\G_{\iota}}.
\end{equation}
  On the other hand, for $g\in G_{\FC'\FC'_{\iota}}$, the facet $g^{-1}\FC$ need not be
  contained in $\BT'_{\iota,\bullet/e}$, but
  we still  have a decomposition
    $$ \VC^{\dag}_{\FC',\FC'_{\iota}} = \bigoplus_{\GC\in\BT'_{\bullet/e}, \, \o\GC\supseteq\FC'_{\iota}} 
  \CC^{\infty}_{c}(G_{\FC'\FC'_{\iota}}\cap G_{\FC\GC})\otimes_{R\Ku_{\iota,\FC'_{\iota}}}
  (\WC_{\iota,\FC'_{\iota}}\otimes_{R}\VC_{\iota,\FC'_{\iota}}) ,$$
  hence also
  \begin{eqnarray*}
   e_{I,\FC}\,\VC^{\dag}_{\FC',\FC'_{\iota}}
   & =& \bigoplus_{\GC\in\BT'_{\bullet/e}, \, \o\GC\supseteq\FC'_{\iota}} 
  e_{I,\FC}\CC^{\infty}_{c}(G_{\FC'\FC'_{\iota}}\cap G_{\FC\GC})\otimes_{R\Ku_{\iota,\FC'_{\iota}}}
  (\WC_{\iota,\FC'_{\iota}}\otimes_{R}\VC_{\iota,\FC'_{\iota}}) \\
  &=& \bigoplus_{\GC\in\BT'_{\bullet/e}, \, \o\GC\supseteq\FC'_{\iota}} 
  \CC^{\infty}_{c}(G_{\FC'\FC'_{\iota}}\cap G_{\FC\GC})e_{I,\GC}\otimes_{R\Ku_{\iota,\FC'_{\iota}}}
  (\WC_{\iota,\FC'_{\iota}}\otimes_{R}\VC_{\iota,\FC'_{\iota}}).
\end{eqnarray*}
The summand associated to $\GC$ above is non-zero only if $G_{\FC'\FC'_{\iota}}\cap
  G_{\FC\GC}\neq\emptyset$ and $e_{I,\GC}e_{\iota,\FC'_{\iota}}\neq 0$.  By Proposition
  \ref{prop_intertwining} ii), the latter
  condition implies the existence of a $\iota'\in I$ such that $\GC\in \BT'_{\iota'}$ and
  $\iota'\in \Ku_{\iota,\FC'_{\iota}}.\iota$. Since $\Ku_{\iota,\FC'_{\iota}}\subseteq
  G_{\FC_{\iota},0+}$ fixes $\GC$, this implies that $\GC\in \BT'_{\iota}$, 
  so we can
  rewrite the last sum as
  $$ e_{I,\FC}\,\VC^{\dag}_{\FC',\FC'_{\iota}} = \bigoplus_{\FC_{\iota}\in\BT'_{\iota,\bullet/e},\o\FC_{\iota}\supseteq\FC'_{\iota}} 
  \CC^{\infty}_{c}(G_{\FC'\FC'_{\iota}}\cap G_{\FC\FC_{\iota}})e_{I,\FC_{\iota}}\otimes_{R\Ku_{\iota,\FC'_{\iota}}}
  (\WC_{\iota,\FC'_{\iota}}\otimes_{R}\VC_{\iota,\FC'_{\iota}}).$$
  Using \ref{rem_WH_coef_system-i}) of Proposition \ref{rem_WH_coef_syst}, we may rewrite it further as
  \[ 
  e_{I,\FC}\,\VC^{\dag}_{\FC',\FC'_{\iota}} 
  \xleftarrow[\oplus(\mathrm{id} \otimes \beta_{\cW,\cF_\iota,\cF'_\iota} \otimes \mathrm{id})]{\simeq}
   \bigoplus_{\FC_{\iota}\in\BT'_{\bullet/e},\o\FC_{\iota}\supseteq\FC'_{\iota}} 
  \CC^{\infty}_{c}(G_{\FC'\FC'_{\iota}}\cap  G_{\FC\FC_{\iota}})e_{I,\FC_{\iota}}
  \otimes_{R(\Ku_{\iota,\FC'_{\iota}}\cap\Ku_{\iota,\FC_{\iota}})}
  (\WC_{\iota,\FC_{\iota}}\otimes_{R}\VC_{\iota,\FC'_{\iota}}) 
  \]
  Recall the idempotent $e'_{\iota,\FC}$ of Proposition \ref{rem_WH_coef_syst} \ref{rem_WH_coef_system-ii}),
  which is supported on $\Ku_{\iota,\FC'_{\iota}}\cap\Kp_{\iota,\FC_{\iota}} \subseteq \Ku_{\iota,\FC'_{\iota}}\cap\Ku_{\iota,\FC_{\iota}}$ and acts
  trivially on $\WC_{\iota,\FC_{\iota}}\otimes\VC_{\iota,\FC'_{\iota}}$. 
  Since $\Ku_{\iota,\FC'_{\iota}}\cap\Kp_{\iota,\FC_{\iota}}$ contains the group
  $K'_{\iota,\FC_{\iota}}$ introduced in Remark \ref{rem_intertw},
   Remark \ref{rem_ortho_idemp} tells us that
  $e_{I,\FC_{\iota}}e'_{\iota,\FC_{\iota}}= e_{\iota,\FC_{\iota}}$, hence each
  $e_{I,\FC_{\iota}}$ in the last sum can be replaced by $e_{\iota,\FC_{\iota}}$. 
  We thus get the following expression of $e_{I,\FC}\VC_{\FC'}$:
\[
e_{I,\FC}\VC_{\FC'} 
\xleftarrow[(\oplus(\mathrm{id} \otimes \beta_{\cW,\cF_\iota,\cF'_\iota} \otimes \mathrm{id}))_{G_\iota}]{\simeq} 
 \left(\bigoplus_{\cF_{\iota}, \cF'_{\iota} \in \BT'_{\iota,\bullet/e}, \atop \o\FC_{\iota}\supseteq\FC'_{\iota}}
    \CC^{\infty}_{c}(G_{\FC'\FC'_{\iota}}\cap  G_{\FC\FC_{\iota}})e_{\iota,\FC_{\iota}}
    \otimes_{R(\Ku_{\iota,\FC'_{\iota}}\cap\Ku_{\iota,\FC_{\iota}})}
    (\WC_{\iota,\FC_{\iota}}\otimes_{R}\VC_{\iota,\FC'_{\iota}})\right)_{G_{\iota}}.
  \] 
Now, since $\Ku_{\iota,\FC_{\iota}}=(\Ku_{\iota,\FC'_{\iota}}\cap\Ku_{\iota,\FC_{\iota}})G_{\iota,\FC_{\iota},0+}$,
  we also have
\[
e_{I,\FC}\VC_{\FC'} 
\xleftarrow[(\oplus(\mathrm{id} \otimes \beta_{\cW,\cF_\iota,\cF'_\iota} \otimes \mathrm{id}))_{G_\iota}]{\simeq} 
 \left(\bigoplus_{\cF_{\iota}, \cF'_{\iota} \in \BT'_{\iota,\bullet/e}, \,\o\FC_{\iota}\supseteq\FC'_{\iota}}
    \CC^{\infty}_{c}(G_{\FC'\FC'_{\iota}}\cap  G_{\FC\FC_{\iota}})e_{\iota,\FC_{\iota}}
    \otimes_{R\Ku_{\iota,\FC_{\iota}}}
    (\WC_{\iota,\FC_{\iota}}\otimes_{R}\VC_{\iota,\FC'_{\iota}})\right)_{G_{\iota}}.
  \]
Moreover,  since $e_{\iota,\FC_{\iota}}$ acts trivially on $\WC_{\iota,\FC_{\iota}}$ and through
  $e_{G_{\iota,\FC_{\iota},0+}}$ on $\VC_{\iota,\FC'_{\iota}}$, we have
  \[e_{\iota,\FC_{\iota}}(\WC_{\iota,\FC_{\iota}}\otimes_{R}\VC_{\iota,\FC'_{\iota}})
  =  \WC_{\iota,\FC_{\iota}}\otimes_{R} e_{G_{\iota,\FC_{\iota},0+}}\VC_{\iota,\FC'_{\iota}}
 \xleftarrow[\mathrm{id} \otimes \beta_{\VC_{\iota},\FC_{\iota},\FC_{\iota}'}]{\simeq} \WC_{\iota,\FC_{\iota}}\otimes_{R}\VC_{\iota,\FC_{\iota}}.\]
Combining this with (\ref{VF}) we obtain
  \ini\begin{equation}
  \label{eVF'}
  e_{I,\FC}\VC_{\FC'} \simfrom  \left(\bigoplus_{\cF_{\iota}, \cF'_{\iota} \in \BT'_{\iota,\bullet/e}, \, \o\FC_{\iota}\supseteq\FC'_{\iota}}
    \CC^{\infty}_{c}(G_{\FC'\FC'_{\iota}}\cap  G_{\FC\FC_{\iota}})
    \otimes_{R\Ku_{\iota,\FC_{\iota}}}
    (\WC_{\iota,\FC_{\iota}}\otimes_{R}\VC_{\iota,\FC_{\iota}})\right)_{G_{\iota}} = \cV_\cF,
\end{equation}
where the isomorphism agrees by construction with the map $\beta_{\VC,\FC,\FC'}$. This finishes the proof of the lemma.
\end{proof}

Again, the assignment $\VC_{\iota}\mapsto\VC$ is easily seen to be functorial in $\VC_{\iota}$, providing a functor
$$ \IC_{\WC_{\iota}} :\, \Coef^{1}_{R}(\BT'_{\iota,\bullet/e}/G_{\iota})\To{}\Coef^{\phi,I}_{R}(\BT'_{\bullet/e}/G)$$ 
that maps $\Coef^{1}_{R}(\BT'_{\iota,\bullet/e}/G_{\iota})^{\cart}$ into $\Coef^{\phi,I}_{R}(\BT'_{\bullet/e}/G)^{\cart}$.

\begin{theo}\label{main_under_assumption}
  The functors $\RC_{\WC_{\iota}}$ and $\IC_{\WC_{\iota}}$ are quasi-inverse equivalences of categories
\[ \xymatrix{
    \IC_{\WC_{\iota}} :\, \Coef^{1}_{R}(\BT'_{\iota,\bullet/e}/G_{\iota})  \ar@<0.5ex>[r]^-{\sim} &
   \Coef^{\phi,I}_{R}(\BT'_{\bullet/e}/G)\,:\RC_{\WC_{\iota}} \ar@<0.5ex>[l]^-{\sim} }
    \]
  that respect cartesian objects on both sides.
\end{theo}
\begin{proof}
	\inisub
	\begin{subequations}
  Let us simplify our notation $\IC:=\IC_{\WC_{\iota}}$ and $\RC:=\RC_{\WC_{\iota}}$  in this proof. We first construct an isomorphism
  $\id_{\Coef^{1}_{R}(\BT'_{\iota,\bullet/e}/G_{\iota})}\simto \RC\circ\IC$. So let us start with an
  object $\VC_{\iota}\in \Coef^{1}_{R}(\BT'_{\iota,\bullet/e}/G_{\iota})$. For any
  $\FC\in\BT'_{\iota,\bullet/e}$, we have
$$   \RC\circ\IC(\VC_{\iota})_{\FC}=\Hom_{R\Ku_{\iota,\FC}}\left(\cW_{\iota,\FC},
      \left(\bigoplus_{\FC_{\iota}\in\BT'_{\iota,\bullet/e}}\CC^{\infty}_{c}(G_{\FC\FC_{\iota}})
      \otimes_{R\K_{\iota,\FC_{\iota}}}(\WC_{\iota,\FC_{\iota}}\otimes_{R}\VC_{\iota,\FC_{\iota}})\right)_{G_{\iota}}\right).$$
We thus can define an $R$-linear map 
\begin{equation} \label{auxilliarymap1}\VC_{\iota,\FC}\To{} \Hom_{R\Ku_{\iota,\FC}}\left(\cW_{\iota,\FC},\CC^{\infty}_{c}(G_{\FC\FC})
  \otimes_{R\K_{\iota,\FC}}(\WC_{\iota,\FC}\otimes_{R}\VC_{\iota,\FC})\right)
  \To{} \RC\circ\IC(\VC_{\iota})_{\FC},
\end{equation}  
where the first map sends $v\in\VC_{\iota}$ to $(w\mapsto e_{\iota,\FC}\otimes (w\otimes v))$. Let us prove that this map is an isomorphism. We have the following sequence of isomorphisms explained below:  
  \begin{eqnarray*}
    \RC\circ\IC(\VC_{\iota})_{\FC}
 & =&
      \Hom_{R\Ku_{\iota,\FC}}\left(\cW_{\iota,\FC},
      \left(\bigoplus_{\FC_{\iota}\in\BT'_{\iota,\bullet/e}}\CC^{\infty}_{c}(G_{\FC\FC_{\iota}})
      \otimes_{R\K_{\iota,\FC_{\iota}}}(\WC_{\iota,\FC_{\iota}}\otimes_{R}\VC_{\iota,\FC_{\iota}})\right)_{G_{\iota}}\right)\\ 
 & \simfrom&
      \left(
      \bigoplus_{\FC_{\iota}\in\BT'_{\iota,\bullet/e}}
      \Hom_{R\Ku_{\iota,\FC}}\left(\cW_{\iota,\FC},\CC^{\infty}_{c}(G_{\FC\FC_{\iota}})  
      \otimes_{R\K_{\iota,\FC_{\iota}}}
      (\WC_{\iota,\FC_{\iota}}\otimes_{R}\VC_{\iota,\FC_{\iota}})\right)
      \right)_{G_{\iota}} \\
& \simfrom&
      \left(
      \bigoplus_{\FC_{\iota}\in\BT'_{\iota,\bullet/e}}
      \Hom_{R\Ku_{\iota,\FC}}\left(\cW_{\iota,\FC},e_{\iota,\FC}\,\CC^{\infty}_{c}(G_{\FC\FC_{\iota}})e_{\iota,\FC_{\iota}}  
      \otimes_{R\K_{\iota,\FC_{\iota}}}
      (\WC_{\iota,\FC_{\iota}}\otimes_{R}\VC_{\iota,\FC_{\iota}})\right)
     \right)_{G_{\iota}} \\
& =&
      \left(
      \bigoplus_{\FC_{\iota}\in G_{\iota}\cdot\FC}
      \Hom_{R\Ku_{\iota,\FC}}\left(\cW_{\iota,\FC},e_{\iota,\FC}\,\CC^{\infty}_{c}(G_{\FC\FC_{\iota}})e_{\iota,\FC_{\iota}}  
      \otimes_{R\K_{\iota,\FC_{\iota}}}
      (\WC_{\iota,\FC_{\iota}}\otimes_{R}\VC_{\iota,\FC_{\iota}})\right)
     \right)_{G_{\iota}} \\
 &\simfrom&
     \left(
     \Hom_{R\Ku_{\iota,\FC}}\left(\cW_{\iota,\FC},e_{\iota,\FC}\,\CC^{\infty}_{c}(G_{\FC\FC})e_{\iota,\FC}
     \otimes_{R\K_{\iota,\FC}}(\WC_{\iota,\FC}\otimes_{R}\VC_{\iota,\FC})\right)
     \right)_{G_{\iota,\FC}}\\
&\simto&
    \Hom_{R\Ku_{\iota,\FC}}\left(\cW_{\iota,\FC},
    \left(\CC^{\infty}_{c}(\K_{\iota,\FC})
     \otimes_{R\K_{\iota,\FC}}(\WC_{\iota,\FC}\otimes_{R}\VC_{\iota,\FC})
    \right)_{G_{\iota,\FC}}      \right)\\
&=&     \Hom_{R\Ku_{\iota,\FC}}\left(\cW_{\iota,\FC},
    \WC_{\iota,\FC}\otimes_{R}\VC_{\iota,\FC}\right).
  \end{eqnarray*}
The first natural map (between the second and first line in the above sequence of isomorphism) is an isomorphism since
$\Hom_{R\Ku_{\iota,\FC}}(\WC_{\iota,\FC},-)$ commutes with colimits 
because  $\WC_{\iota,\FC}$ is a finitely generated
projective $R\Ku_{\iota,\FC}$-module.  The second map is 
induced by the inclusions
$e_{\iota,\FC}\,\CC^{\infty}_{c}(G_{\FC\FC_{\iota}})e_{\iota,\FC_{\iota}}\subseteq
\CC^{\infty}_{c}(G_{\FC\FC_{\iota}})$. It is an isomorphism because
the central idempotent $e_{\iota,\FC}\in\HC_{R}(\Ku_{\iota,\FC})$ acts trivially on
$\WC_{\iota,\FC}$ and, similarly, the central idempotent
$e_{\iota,\FC_{\iota}}\in\HC_{R}(\K_{\iota,\FC_{\iota}})$ acts trivially on
$(\WC_{\iota,\FC_{\iota}}\otimes\VC_{\iota,\FC_{\iota}})$. The second equality says that a
facet $\FC_{\iota}$ provides a non-zero contribution to the big sum only if it is a
$G_{\iota}$ translate of $\FC$. To see this, observe that, for $g\in G_{\FC\FC_{\iota}}$, 
$e_{\iota,\FC}ge_{\iota,\FC_{\iota}}=
e_{\iota,\FC}e_{g\iota,\FC}g$
is non-zero in $RG$ only if
$g\iota\in \Ku_{\iota,\FC}\iota$, due to Lemma \ref{ortho_idemp}. In this case, $g$
can be written $g=kg_{\iota}$ with $k\in\Ku_{\iota,\FC}$ and $g_{\iota}\in G_{\iota}$, and
it follows that $\FC_{\iota}=g^{-1}\FC=(g_{\iota})^{-1}\FC\in G_{\iota}\cdot\FC$. The
third map is well defined and an isomorphism since $G_{\iota}$ permutes transitively the summands and
$G_{\iota,\FC}$ is the stabilizer of the summand
associated to $\FC_{\iota}=\FC$. To see that the fourth map is an isomorphism, we first use again the projectivity of $W_{\FC,\iota}$ as an $RK_{\iota,\FC}$-module to exchange the $\Hom$ and the coinvariants
(recall that $G_{\iota,\FC}$ only acts on  the target of the $\Hom$). Then, as above, we
observe that, for $g\in G_{\FC}$, the product $e_{\iota,\FC}ge_{\iota,\FC}$ is non-zero if
and only if $g\in \Ku_{\iota,\FC}G_{\iota,\FC}=\K_{\iota,\FC}$. This allows to rewrite
$e_{\iota,\FC}\CC^{\infty}_{c}(G_{\FC})e_{\iota,\FC}=\CC^{\infty}_{c}(K_{\iota,\FC})e_{\iota,\FC}$.
Then the idempotent $e_{\iota,\FC}$ can be moved to the right hand side of the
$\otimes_{R_{\K_{\iota,\FC}}}$, where it acts trivially.

The precomposition of this isomorphism with the map $\VC_{\iota,\FC}\To{}\RC\circ\IC(\VC_{\iota})_{\FC}$ in \eqref{auxilliarymap1}
is just $v\mapsto (w\mapsto w\otimes v)$ and it is an isomorphism of $R$-modules by the pair of adjoint equivalences \eqref{Heisenberg-equivalence-of-cat} because $\WC_{\iota,\FC}$
is a Heisenberg representation (and it is even an isomorphism of $RG_{\iota,\FC}$-modules by \eqref{local_equivalences}).
 It is easy to see that this isomorphism is compatible with
the action maps, and a bit more tedious, but still straightforward, to check it is compatible
with face maps.
Also it is clearly functorial in $\VC_{\iota}$, so
that we have just constructed an isomorphism
  $\id_{\Coef^{1}_{R}(\BT'_{\iota,\bullet/e}/G_{\iota})}\simto \RC\circ\IC$.

  Let us now construct an isomorphism $\IC\circ\RC\simto
  \id_{\Coef^{\phi,I}_{R}(\BT'_{\bullet/e}/G)}$. Let $\VC$ be an object in
  $\Coef^{\phi,I}_{R}(\BT'_{\bullet/e}/G)$ and $\FC$ a facet in $\BT'_{\bullet/e}$. By
  construction, we have
  \[ \IC\circ\RC(\VC)_{\FC} =
  \left(\bigoplus_{\FC_{\iota}\in\BT'_{\iota,\bullet/e}}
    \CC^{\infty}_{c}(G_{\FC\FC_{\iota}})\otimes_{R\K_{\iota,\FC_{\iota}}}
  \left(\WC_{\iota,\FC_{\iota}}\otimes_{R}\Hom_{R\Ku_{\iota,\FC_{\iota}}}(\WC_{\iota,\FC_{\iota}},\VC_{\FC_{\iota}})
  \right)\right)_{G_{\iota}}.
  \]
  For fixed $\FC_{\iota}$, consider the map
  $$\application{}
  {\CC^{\infty}_{c}(G_{\FC\FC_{\iota}})\otimes_{R}
  \left(\WC_{\iota,\FC_{\iota}}\otimes_{R}\Hom_{R\Ku_{\iota,\FC_{\iota}}}(\WC_{\iota,\FC_{\iota}},\VC_{\FC_{\iota}})
  \right)}{\VC_{\FC}}{h\otimes (w\otimes \theta)}{ h_{\VC,\FC_{\iota}}(\theta(w))}.
  $$
  By definition of the actions, it descends to a map
  $$\CC^{\infty}_{c}(G_{\FC\FC_{\iota}})\otimes_{R\K_{\iota,\FC_{\iota}}}
  \left(\WC_{\iota,\FC_{\iota}}\otimes_{R}\Hom_{R\Ku_{\iota,\FC_{\iota}}}(\WC_{\iota,\FC_{\iota}},\VC_{\FC_{\iota}})
  \right)\To{}\VC_{\FC}
  $$
  and the sum of all these maps descends to the $G_{\iota}$-coinvariants,  providing a map
\ini
\begin{equation}
  \label{map}
  \IC\circ\RC(\VC)_{\FC}\To{} \VC_{\FC}.
\end{equation}
  Let us prove that this map is an isomorphism. To this aim, observe first that each
  evaluation map 
  $\WC_{\iota,\FC_{\iota}}\otimes_{R}\Hom_{R\Ku_{\iota,\FC_{\iota}}}(\WC_{\iota,\FC_{\iota}},\VC_{\FC_{\iota}})\To{}\VC_{\FC_{\iota}}$
  induces an isomorphism of $R$-modules
\ini
\begin{equation}
  \label{eval}
    \WC_{\iota,\FC_{\iota}}\otimes_{R}\Hom_{R\Ku_{\iota,\FC_{\iota}}}(\WC_{\iota,\FC_{\iota}},\VC_{\FC_{\iota}})\simto
    e_{\iota,\FC_{\iota}}\VC_{\FC_{\iota}}
\end{equation}
because $\WC_{\iota,\FC_{\iota}}$ is a Heisenberg representation  with central
   character $\check\phi_{\iota,\FC_{\iota}}$. This is even an isomorphism
of $R\K_{\iota,\FC}$-modules by (\ref{local_equivalences}).

  On the other hand, we know from Lemma \ref{ortho_idemp} that
  $K_{\iota,\FC_{\iota}}=\Ku_{\iota,\FC_{\iota}}G_{\iota,\FC_{\iota}}$ is the centralizer of $e_{\iota,\FC_{\iota}}$ in
  $G_{\FC_{\iota}}$, and that the various idempotents  $e_{h\iota,\FC_{\iota}}$ for
  $h\in G_{\FC_{\iota}}/K_{\iota,\FC_{\iota}}$ are pairwise orthogonal. Therefore, the
  action map
  $ \CC^{\infty}_{c}(G_{\FC_{\iota}})\otimes_{R\K_{\iota,\FC_{\iota}}}
  e_{\iota,\FC_{\iota}}\VC_{\FC_{\iota}} \To{}\VC_{\FC_{\iota}}$ induces an isomorphism
  $$ \CC^{\infty}_{c}(G_{\FC_{\iota}})\otimes_{R\K_{\iota,\FC_{\iota}}}
  e_{\iota,\FC_{\iota}}\VC_{\FC_{\iota}} \simto
  \bigoplus_{h\in G_{\FC_{\iota}}/K_{\iota,\FC_{\iota}}}e_{h\iota,\FC_{\iota}}\VC_{\FC_{\iota}}
  =\sum_{h\in G_{\FC_{\iota}}/K_{\iota,\FC_{\iota}}}e_{h\iota,\FC_{\iota}}\VC_{\FC_{\iota}}\subseteq\VC_{\FC_{\iota}}
  .$$
  This in turn induces an isomorphism
\ini\begin{equation}
    \label{act}
    \CC^{\infty}_{c}(G_{\FC\FC_{\iota}})\otimes_{R\K_{\iota,\FC_{\iota}}}
    e_{\iota,\FC_{\iota}}\VC_{\FC_{\iota}} \simto
    \bigoplus_{\iota'\in (G_{\FC\FC_{\iota}}.\iota)/\sim_{\FC}}e_{\iota',\FC}\VC_{\FC}
    = e_{I_{\FC\FC_{\iota}},\FC}\VC_{\FC}, 
\end{equation}
  where the equivalence $\sim_{\FC}$ on the subset $I_{\FC\FC_{\iota}}:=G_{\FC\FC_{\iota}}.\iota$ of $I_{\FC}$ is
  that of \ref{ortho_idemp}, and where we have written
  $e_{I_{\FC\FC_{\iota}},\FC}:=\sum_{\iota'\in (I_{\FC\FC_{\iota}})/\sim_{\FC}}e_{\iota',\FC}$,
  which is an idempotent in $\HC_{R}(G_{\FC})$.
  Combining (\ref{eval}) and (\ref{act}) yields the following isomorphism
 $$\IC\circ\RC(\VC)_{\FC} \simto
   \left(\bigoplus_{\FC_{\iota}\in\BT'_{\iota,\bullet/e}}
    e_{I_{\FC\FC_{\iota}},\FC} \VC_{\FC}
   \right)_{G_{\iota}}.
  $$  
  Since  $G_{\iota}$ acts only on the index set,  we can rewrite the right hand
  side as 
  $\bigoplus_{\FC_{\iota}\in\BT'_{\iota,\bullet/e}/G_{\iota}}e_{I_{\FC\FC_{\iota}},\FC}\VC_{\FC},$
  and the map (\ref{map}) factors as
  $$ \IC\circ\RC(\VC)_{\FC} \simto \bigoplus_{\FC_{\iota}\in\BT'_{\iota,\bullet/e}/G_{\iota}}e_{I_{\FC\FC_{\iota}},\FC}\VC_{\FC}
  \To{\sum} \VC_{\FC}.$$
  So, in order to show that this map is an isomorphism, it suffices to show that the
  idempotents $e_{I_{\FC\FC_{\iota}},\FC}$ are pairwise orthogonal in
  $\HC_{R}(G_{\FC})$ and that their sum is
  $e_{I,\FC}$. Equivalently, we need to show that
  $I_{\FC}/_{\sim_{\FC}}=\bigsqcup_{\FC_{\iota}\in\BT'_{\iota,\bullet/e}/G_{\iota}}
  I_{\FC\FC_{\iota}}/_{\sim_{\FC}}$.
  Obviously, we have $I_{\FC}=\bigsqcup_{\FC_{\iota}\in\BT'_{\iota,\bullet/e}/G_{\iota}}
  I_{\FC\FC_{\iota}}$, so it is a matter of showing that this partition splits the
  equivalence relation $\sim_{\FC}$. In other word, given
   $\FC_{1},\FC_{2}\in\BT'_{\iota,\bullet/e}$ and $g_{i}\in G_{\FC\FC_{i}}$ such that
   $e_{g_{1}\iota,\FC}=e_{g_{2}\iota,\FC}$, we need to show that $\FC_{1}$ and $\FC_{2}$
   are $G_{\iota}$-conjugate. But, as in Lemma \ref{ortho_idemp}, there is $k$ in
   $\Ku_{g_{1}\iota,\FC}$ such that $g_{2}\iota=kg_{1}\iota$, whence an element
   $g_{\iota}\in G_{\iota}$ such that $g_{2}g_{\iota}=kg_{1}$. Since $k$ fixes $\FC$, it
   follows that $\FC_{1}=g_{1}^{-1}\FC= g_{1}^{-1}k^{-1}\FC=
   g_{\iota}^{-1}g_{2}^{-1}\FC=g_{\iota}^{-1}\FC_{2}$. Hence, $\FC_{1}$ and $\FC_{2}$ are
   $G_{\iota}$-conjugate, and we have proved that the natural map~(\ref{map}) is an isomorphism.

   As in the other direction, the above identifications are natural enough to be
   compatible with face maps and action maps on both sides, and we have thus constructed
   an isomorphism $\IC\circ\RC\simto
  \id_{\Coef^{\phi,I}_{R}(\BT'_{\bullet/e}/G)}$, as desired. 
\end{subequations}
\end{proof}

\ali \label{def-functors} We now reap the rewards of the above constructions by setting
\begin{eqnarray*}
&  \IC_{\WC_{\iota}} :\, \Rep^{1}_{R}(G_{\iota})\To{(\pi_{\iota}^{*})^{1}}
  \Coef_{R}^{1}(\BT'_{\iota,\bullet/e}/G_{\iota})^{\cart} \To{\IC_{\WC_{\iota}}}
  \Coef_{R}^{\phi,I}(\BT'_{\bullet/e}/G)^{\cart} \To{\pi_{!}} \Rep_{R}^{\phi,I}(G)  \\
&  \RC_{\WC_{\iota}} :\, \Rep_{R}^{\phi,I}(G) \To{(\pi^{*})^{\phi,I}}
  \Coef_{R}^{\phi,I}(\BT'_{\bullet/e}/G)^{\cart} \To{\RC_{\WC_{\iota}}}
  \Coef_{R}^{1}(\BT'_{\iota,\bullet/e}/G_{\iota})^{\cart} \To{\pi_{\iota,!}} \Rep^{1}_{R}(G_{\iota}) 
\end{eqnarray*}

\begin{theo} \label{thm-main-2}
  Let $\WC_{\iota}$ be a Heisenberg--Weil coefficient system as in Definition \ref{def_WH_coef_syst}.
  The two above functors are quasi-inverse equivalences of categories :
  \[  \xymatrix{
    \IC_{\WC_{\iota}} :\, \Rep^{1}_{R}(G_{\iota})  \ar@<0.5ex>[r]^-{\sim} &
   \Rep^{\phi,I}_{R}(G)\,:\RC_{\WC_{\iota}} \ar@<0.5ex>[l]^-{\sim}.
 }\]
Moreover, for any $e$-vertex $\FC_{0}\in\BT'_{\iota,\bullet/e}$ and any representation 
$\rho$ of $G_{\iota,\FC_{0}}$ that is trivial on $G_{\iota,\FC_{0},0+}$, there is an isomorphism
of $RG$-modules
$$ \IC_{\WC_{\iota}}(\cind_{G_{\iota,\FC_{0}}}^{G_{\iota}}(\rho))\simeq
\cind_{\K_{\iota,\FC_{0}}}^{G}(\WC_{\iota,\FC_{0}}\otimes_{R}\rho)$$
through which $\IC_{\WC_{\iota}}$  induces an isomorphism of $R$-algebras
$$\End_{RG_{\iota}}(\cind_{G_{\iota,\FC_{0}}}^{G_{\iota}}(\rho))
\simto \End_{RG}(\cind_{\K_{\iota,\FC_{0}}}^{G}(\WC_{\iota,\FC_{0}}\otimes_{R}\rho)).$$
\end{theo}
\begin{proof}
The first statement is now a consequence of  Theorem \ref{main_under_assumption},
(\ref{thm_equiv_repr_coef}) and (\ref{thm_equiv_repr_coef_iota}).
To prove the second statement, we first note that the following diagrams are
$1$-commutative (i.e., commutative up to isomorphism of  functors).
  $$\xymatrix{
    \Coef^{\phi,I}_{R}(\BT'_{\bullet/e}/G) \ar@<1ex>[d]^{\RC_{\WC_{\iota}}}
    \ar[r]^-{\pi_{!}} & \Rep^{\phi,I}_{R}(G) \ar@<1ex>[d]^{\RC_{\WC_{\iota}}}
    \\
    \Coef^{1}_{R}(\BT'_{\iota,\bullet/e}/G_{\iota})
    \ar@<1ex>[u]^{\IC_{\WC_{\iota}}} \ar[r]^-{\pi_{\iota,!}}
    & \Rep^{1}_{R}(G_{\iota}) \ar@<1ex>[u]^{\IC_{\WC_{\iota}}}
  }$$
Indeed, this follows by adjunction from the  1-commutativity of the
  following diagram, which holds by construction.
  $$\xymatrix{
    \Coef^{\phi,I}_{R}(\BT'_{\bullet/e}/G) \ar@<1ex>[d]^{\RC_{\WC_{\iota}}}
    & \Rep^{\phi,I}_{R}(G) \ar@<1ex>[d]^{\RC_{\WC_{\iota}}} \ar[l]_-{(\pi^{*})^{\phi,I}} 
    \\
    \Coef^{1}_{R}(\BT'_{\iota,\bullet/e}/G_{\iota})
    \ar@<1ex>[u]^{\IC_{\WC_{\iota}}} 
    & \Rep^{1}_{R}(G_{\iota}) \ar@<1ex>[u]^{\IC_{\WC_{\iota}}} \ar[l]_-{(\pi_{\iota}^{*})^{1}}
  }$$
Now, denote by $\VC_{\iota,\FC_{0},\rho}$
  the coefficient system on $\BT'_{\iota}$ defined by
  \begin{align*}
  &  \VC_{\iota,\FC_{0},\rho}(\FC):=\CC^{\infty}(G_{\iota,\FC\FC_{0}})\otimes_{RG_{\iota,\FC_{0}}}\rho,
    \;\;\forall\FC\in\BT'_{\iota,\bullet/e}  \\
  &  g_{\FC} :\, \VC_{\iota,\FC_{0},\rho}(\FC)\To{g\otimes\id}\VC_{\iota,\FC_{0},\rho}(g\FC),
    \;\;\forall\FC\in\BT'_{\iota,\bullet/e},\forall g\in G \\
    & \beta_{\FC,\FC'}=0,\,\, \forall\FC,\FC'\hbox{ s.t. } \o\FC'\subsetneq\o\FC
  \end{align*}
  Observe that   $\VC_{\iota,\FC_{0},\rho}(\FC)=\{0\}$ unless $\FC$ is
  $G_{\iota}$-conjugate to $\FC_{0}$, hence it follows that
$$ \pi_{\iota,!}(\VC_{\iota,\FC_{0},\rho}) = \cind_{G_{\iota,\FC_{0}}}^{G_{\iota}}(\rho).$$
By definition, $\VC_{\iota,\FC_{0},\rho}$ belongs to
$\Coef^{1}_{R}(\BT'_{\iota,\bullet/e}/G_{\iota})$
and $\IC_{\WC_{\iota}}(\VC_{\iota,\FC_{0},\rho})$ is again a coefficient system with
trivial face maps, that vanishes outside the $G$-orbit of $\FC_{0}$ and  is given by
$$ \IC_{\WC_{\iota}}(\VC_{\iota,\FC_{0},\rho})(\FC)=
\CC^{\infty}(G_{\FC\FC_{0}})\otimes_{R\K_{\iota,\FC_{0}}}(\WC_{\iota,\FC_{0}}\otimes_{R}\rho)$$
for all facets $\FC\in\BT'_{\bullet/e}$.  It follows that
$$\pi_{!}(\IC_{\WC_{\iota}}(\VC_{\iota,\FC_{0},\rho}))=\cind_{\K_{\iota,\FC_{0}}}^{G}(\WC_{\iota,\FC_{0}}\otimes_{R}\rho).$$
By the first $1$-commutative diagram above, we conclude that
$$ \IC_{\WC_{\iota}}(\cind_{G_{\iota,\FC_{0}}}^{G_{\iota}}(\rho))
=\IC_{\WC_{\iota}}(\pi_{\iota,!}(\VC_{\iota,\FC_{0},\rho}) 
\simeq \pi_{!}(\IC_{\WC_{\iota}}(\VC_{\iota,\FC_{0},\rho}))=
\cind_{\K_{\iota,\FC_{0}}}^{G}(\WC_{\iota,\FC_{0}}\otimes_{R}\rho).$$

Finally, the last statement follows from the second one and the fact that
$\IC_{\WC_{\iota}}$ is an equivalence of categories.
\end{proof}

\alin{Dependence on the choices}
Our construction  depends on two choices: the integer $e$ and the
Heisenberg--Weil coefficient system $\WC_{\iota}$.
Let us first fix $e$, and rename
$ \IC_{\WC_{\iota}}^{e} :\, \Rep_{R}^{1}(G_{\iota}) \leftrightarrows
\Rep^{\phi,I}_{R}(G):\,\RC_{\WC_{\iota}}^{e}$
the pair of inverse equivalences of Theorem \ref{thm-main-2}, to emphasize
dependence on $e$. 
The effect of changing $\WC_{\iota}$ is clear from Proposition \ref{uniqueness_WH_coef-system}.  Given another
Heisenberg--Weil coefficient system $\WC'_{\iota}$, and with $L$, $\theta$ as in that
statement, we have
\ini
\begin{equation} \label{indep_HW}
  \IC_{\WC_{\iota}'}^{e}= \IC_{\WC_{\iota}}^{e}\circ (L_{\theta}\otimes_{R} -) \hbox{ and }
  \RC_{\WC_{\iota}'}^{e}= (L_{\theta})^{\otimes -1}\otimes_{R} -)\circ\RC_{\WC_{\iota}}^{e}.
\end{equation}

Now, let us change $e$ to some multiple $e'$. There is a functor
$[\BT'_{\bullet/e'}/G]\To{}[\BT'_{\bullet/e}/G]$ that takes
an $e'$-facet $\GC$ to the unique $e$-facet $\FC_{e}(\GC)$ that contains $\GC$. Composing
with this functor induces a linear exact functor
$\Coef_{R}^{\phi,I}([\BT'_{\bullet/e}/G])\To{\CC_{e,e'}}\Coef_{R}^{\phi,I}([\BT'_{\bullet/e'}/G])$,
which satisfies $\CC_{e,e'}\circ (\pi^{*})^{\phi,I}_{e}=(\pi^{*})^{\phi,I}_{e'}$ (here we use
the notation of (\ref{thm_equiv_repr_coef}) with an extra index emphasizing dependence on
$e$). In particular, $\CC_{e,e'}$ is an equivalence of categories.
Similarly, we have a functor $[\BT'_{\iota,\bullet/e'}/G_{\iota}]\To{}[\BT'_{\iota,\bullet/e}/G_{\iota}]$
that  induces an equivalence of categories
$\Coef_{R}^{1}([\BT'_{\iota,\bullet/e}/G_{\iota}])\To{\CC_{\iota,e,e'}}\Coef_{R}^{1}([\BT'_{\iota,\bullet/e'}/G_{\iota}])$
such that $\CC_{\iota,e,e'}\circ (\pi_{\iota}^{*})^{1}_{e}=(\pi_{\iota}^{*})^{1}_{e'}$. Finally, 
composing with
$[\BT'_{\iota,\bullet/e'}/K_{\iota}]\To{}[\BT'_{\iota,\bullet/e}/K_{\iota}]$, any
Heisenberg--Weil coefficient system $\WC_{\iota}$ for the $e$-facet subdivision induces 
a Heisenberg--Weil coefficient system $\CC_{\iota,e,e'}(\WC_{\iota})=:\WC'_{\iota}$ for the $e'$-facet subdivision.

 \begin{pro} \label{prop_indep_e}
 With the above notation, the
 equivalences of categories $\RC_{\WC_{\iota}}^{e}, \IC_{\WC_{\iota}}^{e}$, resp.
 $\RC_{\WC'_{\iota}}^{e'},\IC_{\WC'_{\iota}}^{e'}$ of Theorem
   \ref{thm-main-2} respectively associated to  $(e,\WC_{\iota})$ and $(e',\WC_{\iota}')$
   satisfy
 $$ \RC_{\WC_{\iota}}^{e}\simeq\RC_{\WC'_{\iota}}^{e'} \hbox{ and }\IC_{\WC_{\iota}}^{e} \simeq \IC_{\WC'_{\iota}}^{e'}.$$   
 \end{pro}
 \begin{proof}
It follows from their construction  that the functors of Theorem \ref{main_under_assumption} satisfy
the commutation relations
$\CC_{\iota,e,e'}\circ\RC_{\WC_{\iota}}=\RC_{\WC'_{\iota}}\circ\CC_{e,e'}$ and
$\CC_{e,e'}\circ\IC_{\WC_{\iota}}= \IC_{\WC'_{\iota}}\circ \CC_{\iota,e,e'}$,
from which the proposition follows.
 \end{proof}

 \begin{coro} \label{coro_indep_e}
   Let $\WC_{\iota}$ be a Heisenberg--Weyl coefficient system and $\IC_{\WC_{\iota}}$ the
   equivalence of categories of Theorem \ref{thm-main-2}. Then
   for any $x \in\BT'_{\iota}$ and any $R$-representation
$\rho$ of $G_{\iota,x}$ that is trivial on $G_{\iota,x,0+}$, there is an isomorphism
of $RG$-modules
\[ \IC_{\WC_{\iota}}(\cind_{G_{\iota,x}}^{G_{\iota}}(\rho))\simeq
\cind_{\K_{\iota,x}}^{G}(\WC_{\iota,x}\otimes_{R}\rho).\]
Here $\WC_{\iota,x}$ denotes the $R\K_{\iota, x}$-module $\cW_{\iota, \cF}$ where $\cF$ is the $e$-facet that contains $x$.
 \end{coro}
 \begin{proof}
 The case where $\FC$ is an $e$-vertex of $\BT'_{\iota,\bullet/e}$ is done in
 Theorem \ref{thm-main-2}. In general, let $y$ be the
 barycenter of $\FC$. Then $y$ becomes a vertex in
 $\BT'_{\iota,\bullet/e'}$ for some multiple $e'$ of $e$. Thanks to the last
 proposition, we deduce from Theorem \ref{thm-main-2}
 that $\IC_{\WC_{\iota}}(\cind_{G_{\iota,y}}^{G_{\iota}}(\rho))\simeq
\cind_{\K_{\iota,y}}^{G}(\WC_{\iota,y}\otimes_{R}\rho),$
and we conclude since $\K_{\iota,x}=\K_{\iota,\FC}=\K_{\iota,y}$,  
$\WC_{\iota,x}=\WC_{\iota,\FC}=\WC_{\iota,y}$ and accordingly for
$G_{\iota,x}$ and $G_{\iota,x,0+}$.
 \end{proof}

\begin{rema}
  We will see in Corollary \ref{cor-choice-of-WH} below that given a
  point $x \in\BT'_{\iota}$ and any twisted
  $R\K_{\iota, x}$-Heisenberg--Weil representation $\kappa_{\iota,x}$
  (as defined in Definition \ref{Def-HeisWeil} below) whose
  restriction to $\Kp_{\iota,x}$ is
  $\check{\phi}^+_{\iota, x}$-isotypic, there exists a
  Heisenberg--Weil coefficient system such that
  $\cW_{\iota,x} \simeq \kappa_{\iota,x}$.
\end{rema}

\alin{Proof of Theorem \ref{thm-main} (admitting the results of Section \ref{sec:main-result-strategy-3})} \label{proof-main}  That theorem, and actually a
  more precise version of it, now follows from Theorem \ref{thm-main-2}
  and Corollary \ref{coro_indep_e}. Of course, it still remains to prove the
  results claimed in Section \ref{sec:main-result-strategy-3} 
 which we have assumed in the above proof. This is the subject of the next subsections.

\subsection{Heisenberg representations}\label{sec:heisenberg-rep}

In this subsection, we adapt the theory of Heisenberg representations to
 coefficients in any commutative $\Rmin$-algebra $R$. 
 As in the previous subsections, we have fixed an embedding $\iota \in I$. 
 According to \ref{Heisenberg}, for any $x\in\BT_{\iota}$,
 the triple $(\Ku_{\iota,x},\Kp_{\iota,x},\check\phi_{\iota,x}^{+})$ is
a ``Heisenberg triple'' in the following sense.

\begin{DEf} \label{def_Heis_triple}
  A triple $(\Ku,\Kp,\check\phi)$ consisting of a pro-$p$-group $\Ku$, an open normal
  subgroup $\Kp\subset\Ku$ and a continuous homomorphism $\check\phi:\,\Kp\To{}\mu_{p^{\infty}}$
  centralized by $\Ku$, is called a \emph{Heisenberg triple}\index[terminology]{Heisenberg triple} if  $\Ku/\Kp$ is an
  $\FM_{p}$-vector space  and the map $(k,k')\mapsto 
  \check\phi(kk'k^{-1}(k')^{-1})$ defines a ($\mu_{p}$-valued) perfect pairing on $\Ku/\Kp$. 
\end{DEf}

Note that the pairing is necessarily symplectic since $p\neq 2$. For any commutative $\Rmin$-algebra
$R$, we still denote by $\check\phi$ the character $\Ku\To{}\mu_{p^{\infty}}\subset
R^{\times}$, and we denote by $e\in\HC_{R}(\Ku)$ the corresponding idempotent, which is
central and supported on $\Kp$. 
The category of smooth $R\Ku$-modules whose
restriction to $\Kp$ is $\check\phi$-isotypic
identifies with the category  $ e\HC_R(\Ku)-\Mod$
of modules over the finite $R$-algebra $e \HC_R(\Ku)$.

\begin{lemme} \label{lem_Heisenberg}
  Let $(\Ku,\Kp,\check\phi)$ be a Heisenberg triple.
  For a commutative $\Rmin$-algebra $R$,  there is a
  smooth $R \Ku$-module $\Heis$ such that 
  \begin{enumerate}
  \item $\Heis$ is finitely generated and projective over $R$.
  \item $\Heis|_{\Kp}$ is $\check\phi$-isotypic.
  \item $\End_{R\Ku}(\Heis)=R$.
  \end{enumerate}
Any such  $\Heis$ has $R$-rank $\sqrt{[\Ku:\Kp]}$ and
is a projective generator of 
    $e  \HC_R(\Ku)-\Mod$.
     In particular,
    we have adjoint  equivalences of categories  
    $$ \Hom_{R\Ku}(\Heis, -):\,e \HC_R(\Ku)-\Mod
    \rightleftarrows R-\Mod :\, \Heis \otimes_R - $$
  Moreover, if $\Heis'$ is
  another $R\Ku$-module satisfying (i), (ii) and (iii), then
  $\Hom_{R\Ku}(\Heis,\Heis')$ is an
  invertible $R$-module (i.e. finitely generated projective of
  rank $1$).
\end{lemme}
\begin{proof}
  We first construct a representation $\Heis$ that satisfies (i), (ii) and
  (iii). Set $V:=\Ku/\Kp$, pick a Lagrangian subspace
  $W\subseteq V$ and denote by $\Ku_{W}\subseteq \Ku$ the preimage of $W$ in $\Ku$. We thus have
  $[\Ku:\Ku_{W}]=[\Ku_{W}:\Kp]=\sqrt{[\Ku:\Kp]}$, and we denote
  this integer by $a$.
  there are exactly $a$ distinct characters $\neweta:\Ku_{W}\To{}
  \mu_{p^{\infty}}\subseteq R^{\times}$ that extend $\check\phi$ and
  they are transitively permuted under the conjugation action of
  $\Ku$. Now pick such a character $\neweta$, denote by $R_{\neweta}$ the
  associated $R\Ku_{W}$-module with underlying $R$-module $R$, and put
  $$ \Heis:= \cind_{\Ku_{W}}^{\Ku}R_{\neweta}.$$
  This certainly satisfies (i) and (ii). Moreover, we have
  $(\cind_{\Ku_{W}}^{\Ku}R_{\neweta})|_{\Ku_{W}}=\bigoplus_{\neweta'}
  R_{\neweta'}$ by the Mackey formula, so (iii) follows from Frobenius
  reciprocity. For the same reason, the $R$-module
  $\Hom_{R\Ku}(\Heis,\cind_{\Ku_{W}}^{\Ku}R_{\neweta'})$  is free of rank $1$ for any character $\neweta':\Ku_{W}\To{}
  \mu_{p^{\infty}}\subseteq R^{\times}$ that extends $\check\phi$. Pick
  a generator $f$ of this $R$-module. Its image contains the $\Ku_{W}$-$\neweta$-eigenspace of
  $\cind_{\Ku_{W}}^{\Ku}R_{\neweta'}$, which generates the latter $R\Ku$-module, hence $f$
  is surjective. Being a morphism of free $R$-modules of the same rank, it is also
  injective, hence an isomorphism. Now observe that each $R_{\neweta'}$ is a projective
  $R\Ku_{W}$-module since $\Ku_{W}$ is a pro-$p$-group and $p\in R^{\times}$. Therefore
  $\bigoplus_{\neweta'}R_{\neweta'}$ is a projective
  generator of $(e \HC_R(\Ku_{W}))-\Mod$, and Frobenius reciprocity implies that $\Heis$ is
  a projective generator of $(e \HC_R(\Ku))-\Mod$. This formally
  implies that the adjoint pair of functors
  $$ \Hom_{R\Ku}(\Heis,-):\, (e\HC_R(\Ku))-\Mod \rightleftarrows
  R-\Mod:\, \Heis\otimes_{R}- $$
  are equivalences of categories. 
  Finally, let $\Heis'$ be another smooth $R\Ku$-module satisfying (i), (ii) and
  (iii). Since it is $R$-projective, it is $R\Ku$-projective. Therefore, the $R$-module 
  $P:=\Hom_{R\Ku}(\Heis,\Heis')$ is finitely generated projective and satisfies 
  $\End_{R}P=R$. So $P$ has rank $1$ and $\Heis'\simeq \Heis\otimes_{R}P$ is also a
  projective generator of $e\HC_{R}(\Ku)-\Mod$.
\end{proof}

\begin{DEf}\label{Def-Heisenberg}
  A representation $\Heis$ as in the previous lemma will be called a \textit{Heisenberg
    representation}\index[terminology]{Heisenberg representation} for the Heisenberg triple  $(\Ku,\Kp,\check\phi)$.
\end{DEf}

\ali \label{section-different-heisenberg}
Note that, when $R$ is a field or when $R=\R$, any invertible
$R$-module is free, so $\Heis$ is \emph{uniquely determined}
by properties (i), (ii) and (iii) up to isomorphism. In general, $\eta$ is only unique up
to isomorphism and twisting by an invertible $R$-module.
In any case,  we always have a canonical
evaluation isomorphism
$$ \Heis\otimes_{R}\Hom_{R\Ku}(\Heis,\Heis')\simto
\Heis'$$ for any two choices of
Heisenberg representations $\Heis$ and $\Heis'$.

\medskip
As already noted before, for any $x\in\BC_{\iota}$, the triple
$(\Ku_{\iota,x},\Kp_{\iota,x},\check\phi_{\iota,x}^{+})$ is a Heisenberg triple, by
\ref{Heisenberg}. We will generally denote by $\Heis_{\iota,x}$ a Heisenberg
representation for this triple.
Note that  this notion  only depends on the e-facet $\FC_{e}(x)$ of
$\BT_{\iota}$ that contains $x$. This allows us to  use the notation $\Heis_{\iota,\FC}$ instead of $\Heis_{\iota,x}$ for the facet
$\FC\in\BT'_{\iota,\bullet/e}$ that contains $x$, whenever we need it.

\begin{lemme} \label{intertwining_Heisenberg}
  Let  $x,x' \in\BT_{\iota}$ and choose Heisenberg
  representations $\Heis_{\iota,x}$ and $\Heis_{\iota,x'}$.
  \begin{enumerate}
  \item \label{intertwining_Heisenberg-i} The $R$-module
    $\H_{x,x'}:=\Hom_{R(\Ku_{\iota,x}\cap \Ku_{\iota,x'})}(\Heis_{\iota,x},\Heis_{\iota,x'})$ is
    invertible.
  \item The evaluation map
    $\Heis_{\iota,x}\otimes_{R}\H_{x,x'}\To{}\Heis_{\iota,x'}$ factors
    as:
    $$\Heis_{\iota,x}\otimes_{R}\H_{x,x'} \twoheadrightarrow
    (e_{\iota,x',x}\,\Heis_{\iota,x})\otimes_{R}\H_{x,x'} \simto
    e_{\iota,x,x'}\Heis_{\iota,x'}\injo \Heis_{\iota,x'},$$
    where $e_{\iota,x,x'}$ is the idempotent associated to the restriction of
      $\check\phi^{+}_{\iota,x}$ to $\Kp_{\iota,x}\cap\Ku_{\iota,x'}$,
      and the first map is the projection obtained by applying this idempotent.
  \item \label{intertwining_Heisenberg-iii}  If $\o{\FC_{e}(x)}\supseteq \FC_{e}(x')$, then the evaluation maps induce:
    \begin{enumerate}
    \item an $R\Ku_{\iota,x'}$-isomorphism
      $\cind_{\Ku_{\iota,x}\cap \Ku_{\iota,x'}}^{\Ku_{\iota,x'}}\Heis_{\iota,x} \otimes_{R}
      \H_{x,x'}\simto \Heis_{\iota,x'}$
    \item an $R(\Ku_{\iota,x}\cap \Ku_{\iota,x'})$-isomorphism
      $\Heis_{\iota,x}\otimes_{R}\H_{x,x'}\simto e_{\iota,x,x'}\Heis_{\iota,x'}$.
    \end{enumerate}
  \end{enumerate}
\end{lemme}
\begin{proof}
	\inisub
	\begin{subequations}
		 i) By the previous lemma, we may assume $x\neq x'$.
  Let $({\mathbf{P}} ,\bar{\mathbf{P}} )$ be the pair of $F$-rational
  opposite parabolic subgroups of ${\mathbf{G}}$ constructed in \ref{opposite_parabolic_sgps}.
  As usual, we denote by ${\mathbf{M}}$
  their common Levi component and by ${\mathbf{U}}$ and $\bar{\mathbf{U}}$
  the respective unipotent radicals.   
  Then both $\Ku_{\iota,x}$ and $\Ku_{\iota,x'}$ have the Iwahori
  decomposition with respect to this pair, and the 
  following equalities and inclusions hold:
\begin{equation}
    \label{eq:inclusions}
    M\cap \Ku_{\iota,x}=M\cap \Ku_{\iota,x'}, \,\,\,\,
    \bar U\cap \Kp_{\iota,x} \supseteq \bar U\cap \Ku_{\iota,x'}, \,\,\,\,
    U\cap\Ku_{\iota,x} \subseteq U\cap\Kp_{\iota,x'}.
  \end{equation}
 By Lemma \ref{Iwahori}, $\bar U\cap \Kp_{\iota,x}$ acts trivially on $\Heis_{\iota,x}$ and
 $U\cap \Kp_{\iota,x'}$ acts trivially on $\Heis_{\iota,x'}$. Hence any
 $\Ku_{\iota,x}\cap\Ku_{\iota,x'}$-equivariant map $\Heis_{\iota,x}\To{}\Heis_{\iota,x'}$ factors
 through $(\Heis_{\iota,x})_{U\cap \Ku_{\iota,x}}$ and lands in $(\Heis_{\iota,x'})^{\bar U\cap \Ku_{\iota,x'}}$.
 In view of the Iwahori decomposition
 \[ \Ku_{\iota,x}\cap\Ku_{\iota,x'}=(U\cap\Ku_{\iota,x}) (M\cap\Ku_{\iota,x}) (\bar U\cap
 \Ku_{\iota,x'}),\]
this means that
\begin{equation}
  \label{Hom_Heis}
  \Hom_{R(\Ku_{\iota,x}\cap \Ku_{\iota,x'})}\left(\Heis_{\iota,x},\Heis_{\iota,x'}\right)
  = \Hom_{R(M\cap \Ku_{\iota,x})}\left((\Heis_{\iota,x})_{U\cap \Ku_{\iota,x}},(\Heis_{\iota,x'})^{\bar U\cap \Ku_{\iota,x'}}\right).
\end{equation}
 By Proposition \ref{Heisenberg}
 ii), the restriction to $(M\cap \Ku_{\iota,x})/(M\cap \Kp_{\iota,x})$ of the bilinear form $\theta$ associated to
 $\check\phi_{\iota,x}^{+}$ is non-degenerate, so that
$(M\cap \Ku_{\iota,x},M\cap \Kp_{\iota,x},(\check\phi_{\iota,x}^{+})|_{M\cap
  \Kp_{\iota,x}})$ is a Heisenberg triple.
 We claim that both
 $(\Heis_{\iota,x})_{U\cap \Ku_{\iota,x}}$ and $(\Heis_{\iota,x'})^{\bar U\cap \Ku_{\iota,x'}}$ are
 such Heisenberg representations for this triple. According to the last statement of Lemma
 \ref{lem_Heisenberg}, this concludes our proof.

 To prove the claim, we may argue with any particular choice of  $\Heis_{\iota,x}$, and in
 particular with a model 
 $\Heis_{\iota,x}=\cind_{\Ku_{W}}^{\Ku_{\iota,x}}R_{\neweta}$ as in  the proof of Lemma
 \ref{lem_Heisenberg}, for a suitable Lagrangian  $W\subseteq \Ku_{\iota,x}/\Kp_{\iota,x}$ and character
 $\nu$.  Namely,  choose first
 $W$ of the form $W_{\bar U}\oplus W_{M}$ where $W_{\bar U}=(\bar U\cap
 \Ku_{\iota,x})/(\bar U\cap \Kp_{\iota,x})$ and $W_{M}$ is a Lagrangian in
 $(M\cap \Ku_{\iota,x})/(M\cap \Kp_{\iota,x})$, and  pick any 
 $\neweta$ that is trivial on $W_{U}$. Then using Mackey's formula we see 
 that  $ (\cind_{\Ku_{W}}^{\Ku_{\iota,x}}R_{\neweta})_{U\cap\Ku_{\iota,x}} =
(\cind_{P\cap\Ku_{W}}^{P\cap\Ku_{\iota,x}}R_{\neweta})_{U\cap\Ku_{\iota,x}}\simto
 \cind_{M\cap\Ku_{W}}^{M\cap\Ku_{\iota,x}}R_{\neweta}$, which is indeed a Heisenberg
 representation for  $(M\cap\Ku_{\iota,x},\check\phi_{\iota,x}^{+})$.  We argue similarly
 for $(\Heis_{\iota,x'})^{\bar U\cap\Ku_{\iota,x'}}$, by computing with a Lagrangian of
 the form $W_{U}\oplus W_{M}$ in $\Ku_{\iota,x'}/\Kp_{\iota,x'}$. 

 \medskip
 
 ii) By the above discussion, the evaluation map
 $\Heis_{\iota,x}\otimes_{R}\H_{x,x'}\To{}\Heis_{\iota,x'}$  factorizes as  
\begin{equation}
  \label{Hom_Heis_eval}
  \Heis_{\iota,x}\otimes_{R}\H_{x,x'} \twoheadrightarrow
  (\Heis_{\iota,x})_{U\cap\Ku_{\iota,x}}\otimes_{R}\H_{x,x'}\simto
  (\Heis_{\iota,x'})^{\bar U\cap\Ku_{\iota,x'}}
  \injo \Heis_{\iota,x'},
\end{equation}
where the map in the middle is an isomorphism of
$R(M\cap\Ku_{\iota,x})$-modules and is also an evaluation map
through the identification (\ref{Hom_Heis}).

On the other hand, by (\ref{eq:inclusions})
we have 
$ \Kp_{\iota,x}\cap\Ku_{\iota,x'}=(\Kp_{\iota,x}\cap \Kp_{\iota,x'})(\bar U\cap \Ku_{\iota,x'})$.
Since $\check\phi_{\iota,x}^{+}$ is trivial on $\bar U\cap\Kp_{\iota,x}$ and coincides
with $\check\phi_{\iota,x'}^{+}$ on $\Kp_{\iota,x}\cap \Kp_{\iota,x'}$, we have in
$\HC_{R}(\Ku_{\iota,x'})$ the equality
$e_{\iota,x,x'}e_{\iota,x'}=e_{\bar U\cap\Ku_{\iota,x'}}e_{\iota,x'}$, and therefore $(\Heis_{\iota,x'})^{\bar U\cap\Ku_{\iota,x'}}=e_{\iota,x,x'}\Heis_{\iota,x'}$.
Exchanging the roles of $x$ and $x'$, we also have
$e_{\iota,x',x}e_{\iota,x}=e_{U\cap\Ku_{\iota,x}}e_{\iota,x}$
in $\HC_{R}(\Ku_{\iota,x})$. Therefore, the projection
$\Heis_{\iota,x}\twoheadrightarrow (\Heis_{\iota,x})_{U\cap\Ku_{\iota,x}}$ factors as
$\Heis_{\iota,x}\twoheadrightarrow e_{\iota,x',x}\,\Heis_{\iota,x}\simto
(\Heis_{\iota,x})_{U\cap\Ku_{\iota,x}}$,
where the first map is given by the action of $e_{\iota,x',x}$. This completes the proof
of ii).

iii)
Under the assumption that  $\FC_{e}(x')\subseteq\o{\FC_{e}(x)}$,
 the segment $[x,x']$ has a non-empty open intersection with the facet
 $\FC_{e}(x)$. By definition of the $e$-facet decomposition, it
 follows that for each $i=0,\hdots, d$, we have
 $U\cap G^{i}_{\iota,x,(r_{i-1}/2)+}=U\cap G^{i}_{\iota,x,(r_{i-1}/2)}$ and
 $\bar U\cap G^{i}_{\iota,x,(r_{i-1}/2)+}=\bar U\cap G^{i}_{\iota,x,(r_{i-1}/2)}$,
 from which we deduce that
 $ U\cap \Ku_{\iota,x}=U\cap \Kp_{\iota,x} \hbox{ and }
 \bar U\cap \Ku_{\iota,x}=\bar U\cap \Kp_{\iota,x}.$
 In particular, we have 
$\Ku_{\iota,x}=(M\cap\Ku_{\iota,x})\Kp_{\iota,x}$ and
 $(\Heis_{\iota,x})^{U\cap\Ku_{\iota,x}}=\Heis_{\iota,x}=(\Heis_{\iota,x})_{U\cap\Ku_{\iota,x}}$,
   so that iii)(b) follows from ii).

   Now, pick a Lagrangian subspace $W_{M}$ in
   $(M\cap \Ku_{\iota,x})/(M\cap\Kp_{\iota,x})=\Ku_{\iota,x}/\Kp_{\iota,x}$, denote by
   $\Ku_{\iota,x,W_{M}}$ its preimage in $\Ku_{\iota,x}$, and pick a character $\nu$ of
   $\Ku_{\iota,x,W_{M}}$ that extends $\check\phi^{+}_{\iota,x}$. Then we know that
   $\cind_{\Ku_{\iota,x,W_{M}}}^{\Ku_{\iota,x}}\nu$ is a Heisenberg representation for
   $\check\phi^{+}_{\iota,x}$, and we observe that 
   $$(\cind_{\Ku_{\iota,x,W_{M}}}^{\Ku_{\iota,x}}\nu)|_{\Ku_{\iota,x}\cap\Ku_{\iota,x'}}=
   \cind_{\Ku_{\iota,x,W_{M}}\cap\Ku_{\iota,x'}}^{\Ku_{\iota,x}\cap \Ku_{\iota,x'}}\nu$$
   because $\Ku_{\iota,x}=\Kp_{\iota,x}(\Ku_{\iota,x}\cap \Ku_{\iota,x'})$
   and $\Ku_{\iota,x,W_{M}}=\Kp_{\iota,x}(\Ku_{\iota,x,W_{M}}\cap \Ku_{\iota,x'})$. Now, the point
   is that the group
   $\Ku_{\iota,x,W_{M}}\cap\Ku_{\iota,x'}$ is the preimage of the Lagrangian subspace
   $W_{M}\oplus (\bar U\cap\Ku_{\iota,x'})/(\bar U\cap\Kp_{\iota,x'})$ of
   $\Ku_{\iota,x'}/\Kp_{\iota,x'}$, while the restriction of $\nu$ to $\Kp_{\iota,x'}$
   coincides with $\check\phi^{+}_{\iota,x'}$ (by Lemma \ref{Iwahori}). It follows that
   $$ \cind_{\Ku_{\iota,x}\cap\Ku_{\iota,x'}}^{\Ku_{\iota,x'}}
   (\cind_{\Ku_{\iota,x,W_{M}}}^{\Ku_{\iota,x}}\nu)
   =\cind_{\Ku_{\iota,x,W_{M}}\cap\Ku_{\iota,x'}}^{\Ku_{\iota,x'}}\nu$$
   is a Heisenberg representation for $\check\phi^{+}_{\iota,x'}$. This proves 
   iii)(a)  when
   $\Heis_{\iota,x}=\cind_{\Ku_{\iota,x,W_{M}}}^{\Ku_{\iota,x}}\nu$ and
   $\Heis_{\iota,x'}=\cind_{\Ku_{\iota,x,W_{M}}\cap\Ku_{\iota,x'}}^{\Ku_{\iota,x'}}\nu$, and
   the general case
   follows from the last statement of Lemma \ref{lem_Heisenberg}.
   
\end{subequations}
\end{proof}

\begin{lemme} \label{composition_Heisenberg}
  Suppose given three points  $x,x', x'' \in\BT_{\iota}$ and Heisenberg representations $\Heis_{\iota,x}$, $\Heis_{\iota,x'}$, and
  $\Heis_{\iota, x''}$ of $\Ku_{\iota, x}$, $\Ku_{\iota, x'}$, and $\Ku_{\iota, x''}$, respectively. Assume that
$$ (\Ku_{\iota,x}\cap\Ku_{\iota,x'}\cap\Ku_{\iota,x''})(\Kp_{\iota,x}\cap\Kp_{\iota,x''})=\Ku_{\iota,x}\cap\Ku_{\iota,x''}.$$
Then the inclusion
$\Hom_{R(\Ku_{\iota,x}\cap \Ku_{\iota,x''})}(\Heis_{\iota,x},\Heis_{\iota,x''})
\subseteq \Hom_{R(\Ku_{\iota,x}\cap\Ku_{\iota,x'}\cap \Ku_{\iota,x''})}(\Heis_{\iota,x},\Heis_{\iota,x''})$
is an equality, and the resulting
composition map 
  $$\Hom_{R(\Ku_{\iota,x}\cap \Ku_{\iota,x'})}(\Heis_{\iota,x},\Heis_{\iota,x'}) \otimes_{R}
  \Hom_{R(\Ku_{\iota,x'}\cap \Ku_{\iota,x''})}(\Heis_{\iota,x'},\Heis_{\iota,x''})\To{}
  \Hom_{R(\Ku_{\iota,x}\cap \Ku_{\iota,x''})}(\Heis_{\iota,x},\Heis_{\iota,x''})$$
 is an isomorphism. 
\end{lemme}
\begin{proof}
  By definition, $(\Heis_{\iota,x})|_{\Kp_{\iota,x}\cap\Kp_{\iota,x''}}$ is
  $\check\phi^{+}_{\iota,x}$-isotypic, while $(\Heis_{\iota,x''})|_{\Kp_{\iota,x}\cap\Kp_{\iota,x''}}$ is
  $\check\phi^{+}_{\iota,x''}$-isotypic. By lemma \ref{intertw_converse}, both characters
  $\check\phi^{+}_{\iota,x}$ and $\check\phi^{+}_{\iota,x''}$ agree on
  $\Kp_{\iota,x}\cap\Kp_{\iota,x''}$. Therefore, any $R$-linear map
  $\Heis_{\iota,x}\To{}\Heis_{\iota,x''}$ is actually $\Kp_{\iota,x}\cap\Kp_{\iota,x''}$-equivariant,
  and the claimed equality of the lemma follows immediately from its hypothesis.

  In particular, the composition map of the lemma is well defined. Let us first prove that
  it is non-zero. To this aim, according to the factorization in point ii) of Lemma
  \ref{intertwining_Heisenberg}, it is enough to prove that the map
  $e_{\iota,x,x'}\,\Heis_{\iota,x'}\xrightarrow{e_{\iota,x'',x'}*-}
  e_{\iota,x'',x'}\Heis_{\iota,x'}$ is non-zero. Since $\Heis_{\iota,x'}$ is a projective
  generator of  $\HC_{R}(\Ku_{\iota,x'})e_{\iota,x'}-\Mod$ with endomorphism ring $R$, the evaluation map 
  \begin{center}
    $\Heis_{\iota,x'}\otimes_{R} H \simto \HC_{R}(\Ku_{\iota,x'})e_{\iota,x'}$, with
    $H:=\Hom_{R\Ku_{\iota,x'}}(\Heis_{\iota,x'},\HC_{R}(\Ku_{\iota,x'})e_{\iota,x'})$
  \end{center}
  is  an  isomorphism of $R\Ku_{\iota,x'}$-modules. Therefore, it is enough to show that
  the product $e_{\iota,x'',x'}e_{\iota,x,x'}$ is non-zero in
  $\HC_{R}(\Ku_{\iota,x'})e_{\iota,x'}$. As in (\ref{eq:intertw_in_terms_of_idemp}), this
  is equivalent to $\check\phi^{+}_{\iota,x}$ and $\check\phi^{+}_{\iota,x''}$ agreeing on
  $(\Kp_{\iota,x}\cap\Kp_{\iota,x'})\cap(\Kp_{\iota,x'}\cap\Kp_{\iota,x''})$. But this
  follows from Lemma \ref{intertw_converse}.

  We have just proved that the composition map of the lemma is non-zero, for any commutative
  $\Rmin$-algebra $R$. Since it is compatible with change of scalars, this means it
  remains non-zero after any such change of scalars, in particular after reducing modulo
  any maximal ideal. Now, this composition map is an $R$-linear map between two invertible
  $R$-modules. Since it is non-zero modulo any maximal ideal, it is an isomorphism.
\end{proof}

\begin{coro} \label{coro_comp_Heisenberg}
Let $C$ be a chamber in $\BT'_{\iota}$ (i.e., facet of maximal dimension for the usual
Bruhat--Tits polysimplicial structure on $\BT'_{\iota}$), and let $x,x',x''$ be three points in the
closure $\bar C$ of $C$, such that  $x''$ is in $C$ and
$\FC_{e}(x')\subseteq\o{\FC_{e}(x)}$.
Then the composition map 
  $$\Hom_{R(\Ku_{\iota,x}\cap \Ku_{\iota,x'})}(\Heis_{\iota,x},\Heis_{\iota,x'}) \otimes_{R}
  \Hom_{R(\Ku_{\iota,x'}\cap \Ku_{\iota,x''})}(\Heis_{\iota,x'},\Heis_{\iota,x''})\To{}
  \Hom_{R(\Ku_{\iota,x}\cap \Ku_{\iota,x''})}(\Heis_{\iota,x},\Heis_{\iota,x''})$$
 is well defined and is an isomorphism. 
\end{coro}
\begin{proof} Since $\FC_{e}(x')\subseteq\o{\FC_{e}(x)}$, we have for
  all $i>0$, the inclusion
  $G^{i}_{\iota,x,r_{i-1}/2}\subseteq G^{i}_{\iota,x',r_{i-1}/2}$,
  from which we get the equality
  $\Ku_{\iota,x}=(\Ku_{\iota,x}\cap\Ku_{\iota,x'})G_{\iota,x,0+}$. In particular, an
  element $g\in\Ku_{\iota,x}\cap\Ku_{\iota,x''}$ can be written $g=kg_{\iota}$ with
  $k\in \Ku_{\iota,x}\cap\Ku_{\iota,x'}$ and $g_{\iota}\in G_{\iota,x,0+}$.
  On the other hand,
  since $x\in\bar C$ and $x''\in C$, we have $G_{\iota,x,0+}\subseteq
  G_{\iota,C,0+}=G_{\iota, x'',0+}$, hence also
  $G_{\iota,x,0+}\subseteq \Kp_{\iota,x}\cap\Kp_{\iota,x''}\subseteq \Ku_{\iota,x''}$.
  It follows that $g_{\iota}\in \Ku_{\iota,x''}$, hence also $k\in \Ku_{\iota,x''}$. So we
  have proved
  $\Ku_{\iota,x}\cap\Ku_{\iota,x''}=(\Ku_{\iota,x}\cap\Ku_{\iota,x'}\cap\Ku_{\iota,x''})G_{\iota,x,0+}$,
  which implies
  $\Ku_{\iota,x}\cap\Ku_{\iota,x''}=(\Ku_{\iota,x}\cap\Ku_{\iota,x'}\cap\Ku_{\iota,x''})(\Kp_{\iota,x}\cap\Kp_{\iota,x''}).$
  It remains to apply the previous lemma.
\end{proof}

\subsection{A family of Heisenberg--Weil representations}  \label{sec:heis-weil-rep}
 \def\Noy{N}
\def\V{V}
\def\C{C}
\def\PSp{P^{\Sp}}

The aim of this section is to construct a family $(\kappa_{\iota,x})_{x\in\BT'_{\iota}}$
of smooth $R\K_{\iota,x}$-modules
such that :
\begin{enumerate}
\item for all $x\in\BT'_{\iota}$, the restriction $(\kappa_{\iota,x})|_{\Ku_{\iota,x}}$ is
  a Heisenberg representation for $\check\phi^{+}_{\iota,x}$,
\item for any two $x,x'\in\BT'_{\iota}$, the inclusion
  $\Hom_{R(\K_{\iota,x}\cap\K_{\iota,x'})}(\kappa_{\iota,x},\kappa_{\iota,x'})
  \subseteq \Hom_{R(\Ku_{\iota,x}\cap\Ku_{\iota,x'})}(\kappa_{\iota,x},\kappa_{\iota,x'})$
  is an equality.
\end{enumerate}
Working over the base ring $\o\ZM[\frac 1p]$, it would not be
difficult to get a non-constructive proof of  existence of $R\K_{\iota,x}$-modules satisfying i). However, the
lack of uniqueness of such extensions makes it difficult to choose them such that ii) is also
satisfied. Following \cite[\S 11]{Yu_tamescusp} and \cite{Gerardin}, the theory of Heisenberg--Weil
representations produces explicit extensions that are even already defined over the base ring 
$\Rmintwo$, although they still do not satisfy ii). To ensure property ii), we will twist
these extensions by the quadratic characters introduced in \cite{FKS} and \cite{AFMO2}.

\alin{Generalized Heisenberg--Weil representation over $\Rmintwo$-algebras -- an explicit model}
\label{Section_generalized_weil-explicit}
\inisub
\begin{subequations}
  Let $C$ be a cyclic group of order $p^N$
  for some positive integer $N$ and denote by $\mu_p$ its subgroup of order $p$. Let $V$ be an $\bF_p$-vector space with a non-degenerate symplectic form $\theta$ valued in $\mu_p \subseteq C$. We define the group $C \wt{\boxtimes} V$ to have underlying set $C \times V$ and group law given by 
$$(a_1, v_1)\cdot(a_2,v_2)=(a_1\cdot a_2 \cdot \sqrt{\theta(v_1,v_2)}, v_1+v_2). $$ Note that the special case of $N=1$ yields the usual Heisenberg $\bF_p$-group $\mu_{p} \wt{\boxtimes} V$.

If $\nu:\,C\To{}\mu_{p^{\infty}}$ is an injective homomorphism, the triple
$(C \wt{\boxtimes} V,C,\nu)$ is a Heisenberg triple in the sense of Definition \ref{def_Heis_triple}.
Pick a Lagrangian subspace $W \subseteq V$ and, for any commutative $\Rmin$-algebra $R$,  set 
\[\Heis_{\neweta, W, R}:= \ind{C \times W}{C \wt{\boxtimes} V}{\neweta \times 1},\]
where $\nu$ is seen as a character $C\To{}\mu_{p^{\infty}}\subset R^{\times}$. This is a
Heisenberg representation in the sense of Definition \ref{Def-Heisenberg}.
We drop the subscript $R$ and simply write $\Heis_{\neweta, W}$ instead of
$\Heis_{\neweta, W, R}$ if $R$ is clear from the context.
We let $\Sp(V)$ act on $C \wt{\boxtimes} V$ by acting on the second coordinate, i.e., for $g \in \Sp(V)$ and $(a, v) \in C \wt{\boxtimes} V$, we have $g(a,v)=(a,gv)$. 
Note that this action preserves the character $\neweta$ of the center $C=C \times \{0\} \subseteq C \wt{\boxtimes} V$ and it preserves the subgroup $\mu_{p} \wt{\boxtimes} V$.

If $R= \bC$, the restriction of $\Heis_{\neweta, W,\bC}$ to $\mu_{p} \wt{\boxtimes} V$ is a
Heisenberg representation in the classical sense, with central character $\neweta|_{\mu_p}$. 
According to \cite[Theorem~2.4]{Gerardin}, there is a representation
$\Weil_{\CM}$ of $\Sp(V)$ on the vector space $V_{\Heis_{\neweta, W,\bC}}$
underlying $\Heis_{\neweta, W,\bC}$, such that the product map
$$\Weil_{\bC}\ltimes \Heis_{\nu,W,\bC} :\, \Sp(V) \ltimes
(\mu_{p}\wt{\boxtimes} V) \To{} \Aut_{\CM}(V_{\Heis_{\neweta,W,\bC}})$$ is a
homomorphism. Since $\End_{\CM(\mu_{p}\wt{\boxtimes}
  V)}(\Heis_{\neweta,W,\bC})=\CM$, this property makes $\omega_{\CM}$
unique up to a twist by a character, and therefore unique, since
$\Sp(V)=[\Sp(V),\Sp(V)]$, unless $p=3$ and $\dim_{\FM_{p}} V=2$, in
which case, we follow the choice made in \cite[Theorem~2.4]{Gerardin}.

\begin{lem}
The submodule  $V_{\Heis_{\neweta, W,\Rmintwo}}$ of $V_{\Heis_{\neweta, W,\bC}}$ is stable
under $\omega_{\bC}(\Sp(V))$.
\end{lem}
\begin{proof}
  Pick  a Lagrangian subspace  $W^{-} \subset V$ that satisfies $V = W \oplus W^{-}$.
  Restriction of functions from $C\wt{\boxtimes}V$ to $\{1\}\times W^{-}$ provides an
  isomorphism from the $\bC$-vector space $V_{\Heis_{\neweta, W,\bC}}$ to the
  space $\CM[W^{-}]$ of $\bC$-valued functions on $W^{-}$, which also sends the
  $\Rmintwo$-submodule $V_{\Heis_{\neweta, W,\Rmintwo}}$ onto the submodule
  $\Rmintwo[W^{-}]$ of $\Rmintwo$-valued functions on $W^{-}$.
  The proof of \cite[Theorem~2.4]{Gerardin} contains an explicit description of
  the representation $\Weil_{\bC}$ transported on $\CM[W^{-}]$ via the above isomorphism.
  The group $\Sp(V)$ is generated
  by the parabolic subgroup $P$ that stabilizes $W \subset V$ and an element
  $s \in \Sp(V)$ that exchanges $W$ and $W^{-}$.  Gérardin (\cite[(2.7) and
  (2.8)]{Gerardin}) shows that the action of $P$ is given by
  \begin{align*}
    (\Weil_{\bC}(p)(F))(y)&=\text{sgn}(\det{}_{\bF_p}(p|_{W}))F(p^{-1}(y))   & \text{ if } &p \in P \text{ stabilizes } W^{-} \\
    (\Weil_{\bC}(p)(F))(y)&=\chi'(p)F(y)& \text{ if } &p \in P \text{ fixes } W \text{ pointwise, }  
  \end{align*}
  for $F \in \bC[W^{-}]$ and some explicit quadratic character $\chi'$.  Moreover, by
  \cite[(2.18)]{Gerardin}, we have for $F \in \bC[W^{-}]$ and $y \in W^{-}$
  \begin{equation}
    \label{explicit-Weil-3}
    (\Weil_{\bC}(s)(F))(y) \in \mu_4 \cdot p^{-\dim W /2}\sum_{z \in W^{-}} F(z) \cdot \neweta \circ \theta(z, (s|_{W})^{-1}y) dz.
  \end{equation}
Since $\Sp(V)$ is generated by $P$ and $s$, these formulas show that the submodule
$\Rmintwo[W^{-}] \subset \bC[W^{-}]$ is preserved 
by the action of the full $\Sp(V)$. Going back through the above isomorphism, we deduce
that $V_{\Heis_{\neweta,W, \Rmintwo}}$ is stable under $\omega_{\bC}(\Sp(V))$.
\end{proof}
\end{subequations}

Let us denote by
$\omega_{\nu,\Rmintwo}: \Sp(V) \To{} \Aut_{\Rmintwo}V_{\Heis_{\neweta, W,\Rmintwo}}$
the representation provided by the above lemma.
By definition of $\omega_{\bC}$, we have
$$\Heis_{\neweta,W,\Rmintwo}({^{g}k})= \omega_{\nu,\Rmintwo}(g)\Heis_{\neweta,W,\Rmintwo}(k)\omega_{\nu,\Rmintwo}(g)^{-1}$$
for all $g\in \Sp(V)$ and $k\in \mu_{p}\wt{\boxtimes}V$. Since $C=C \times \{0\}$ is
centralized by $\Sp(V)$ and acts via a character on $V_{\Heis_{\neweta, W,\Rmintwo}}$, the
same equality holds for $g\in \Sp(V)$ and $k\in C\wt{\boxtimes}V$.
We thus get a representation
$\omega_{\nu,\Rmintwo}\ltimes \Heis_{\neweta,W,\Rmintwo}$  of the full
group $\Sp(V)\ltimes (C\wt{\boxtimes}V)$ on $V_{\Heis_{\neweta, W,\Rmintwo}}$.

\begin{defn}\label{Def-Weilrep}
  Let $R$ be a commutative $\Rmintwo$-algebra. We denote by $\Weil_{\neweta, W,R} \ltimes
  \Heis_{\neweta, W,R}$ (or simply $\Weil_{\neweta, W} \ltimes \Heis_{\neweta, W}$ if the
  context is clear), the $R(\Sp(V) \ltimes (C \wt{\boxtimes} V))$-module obtained by base
  change from the above $\Rmintwo(\Sp(V) \ltimes (C \wt{\boxtimes} V))$-module.
  We call 
  it the \textit{$(\eta, W)$-Heisenberg--Weil representation} and we call its restriction $\Weil_{\neweta, W}$ to $\Sp(V)$ the \textit{$(\eta, W)$-Weil representation}.
\end{defn}
Note that the restriction of the $(\eta, W)$-Heisenberg--Weil representation $\Weil_{\neweta, W} \ltimes
\Heis_{\neweta, W}$ to $C \wt{\boxtimes} V$ is the
Heisenberg representation $\Heis_{\neweta, W}$ we started with.

From now on $R$ denotes a commutative $\Rmintwo$-algebra.
 
\begin{rak}
	If $p \equiv 1 \text{ mod } 4$, then by \cite[(2.17) and (2.18)]{Gerardin} the $\mu_4$ in \eqref{explicit-Weil-3} can be replaced by $\mu_2 =\{\pm1 \}$ and we could work with the ring $\Rminthree$ instead of $\Rmintwo$ throughout.
\end{rak}

\begin{lemme}\label{Lemma-restriction-of-Weilrep}
	\inisub
	\begin{subequations}
	Let $V^+$ be a subspace of $W$, and write $V^+ \oplus V^0$ for the orthogonal complement of $V^+$ in $V$. Then the restriction of $\theta$ to $V^0$ is non-degenerate and we may form $C \wt{\boxtimes} V^0$. Let $P$ be the subgroup of $\Sp(V)$ that stabilizes $V^+$.
	Then the restriction of the $(\eta, W)$-Heisenberg--Weil representation to $P \ltimes (C\wt{\boxtimes} V)$ satisfies:
	\begin{equation}\label{eqn-restriction-of-Weilrep}
	 (\Weil_{\neweta, W} \ltimes \Heis_{\neweta, W})|_{P \ltimes (C\wt{\boxtimes} V)} 
	\simeq
	\Ind{P \ltimes ((C\wt{\boxtimes}V^0)\times V^+)}{P \ltimes (C \wt{\boxtimes} V)} (\Weil_{\neweta, W\cap V^0} \ltimes \Heis_{\neweta, W \cap V^0}) \otimes (\chi^{V^+} \ltimes 1),
	\end{equation}
	where $\Weil_{\neweta, W\cap V^0} \ltimes \Heis_{\neweta, W \cap V^0}$ is the $(\eta, W \cap V^0)$-Heisenberg--Weil representation of  $\Sp(V^0) \ltimes (C\wt{\boxtimes} V^0)$ on which $P \ltimes ((C\wt{\boxtimes}V^0) \times V^+)$ acts via the projection to $\Sp(V^0) \ltimes ((C\wt{\boxtimes}V^0) \times \{0\})$, and $\chi^{V^+}$ denotes the character $P \ra \{\pm 1\} \subset R^\times$ given by $p \mapsto \mathrm{sgn}(\det_{\bF_p}(p|_{V^+}))$. 
\end{subequations}
\end{lemme}
\begin{proof}
	Note that $W=V^+ \oplus (W \cap V^0)$ and that $W \cap V^0$ is a Lagrangian of $V^0$. Hence we have by definition and transitivity of induction
	\[\Heis_{\neweta, W}= \ind{C \times W}{C \wt{\boxtimes} V}{\neweta \times 1}
	 \simeq \ind{(C \wt{\boxtimes} V^0) \times V^+}{C \wt{\boxtimes} V}{\ind{C \times (W \cap V^0)}{C \wt{\boxtimes} V^0}{\neweta \times 1} \times 1} 
	 = \ind{(C \wt{\boxtimes} V^0) \times V^+}{C \wt{\boxtimes} V}{\Heis_{\neweta, W \cap V^0}}  \]
	This means the two sides of \eqref{eqn-restriction-of-Weilrep} are isomorphic as
        representations of $C \wt{\boxtimes} V$, and we use this isomorphism to identify
        the underlying $R$-modules on both side. Since both sides are representations of
        $P \ltimes (C \wt{\boxtimes} V)$, the action of $p \in P$ on each side provides an
        isomorphism between $\Heis_{\neweta, W}$ and $\Heis_{\neweta, pW}$. Since
        $\End_{R(C\wt{\boxtimes}V)}(\Heis_{\neweta,W})= R$, we therefore conclude that the
        action of $P$ on both sides of \eqref{eqn-restriction-of-Weilrep} agrees up to
        twist by a character valued in the roots of unity contained in $R^\times$. It is
        enough to determine this character for $R=\Rmintwo$, where we can 
        base change to $\bC$ and apply \cite[Theorem~2.4.(b)]{Gerardin}\footnote{The statement of
          \cite[Theorem~2.4.(b)]{Gerardin} omits the factor $(\chi^{V^+} \ltimes 1)$,
          which is a typo. See also \cite[Lemma~3.2]{Fi-mod-ell}.}
        to see that the desired character is $\chi^{V^+}$ as claimed.
        \end{proof}

\begin{DEflem}\label{Def-Lemma-Weil-rep}
Let $\Heis$ be a Heisenberg representation for the Heisenberg triple
 $(C\wt{\boxtimes} V,C,\nu)$.
	Then the evaluation isomorphism $\Heis_{\neweta, W} \otimes_R \Hom_{R(C\wt{\boxtimes} V)}(\Heis_{\neweta, W}, \Heis) \xra{\sim} \Heis$ turns the representation $(\Weil_{\neweta, W} \ltimes \Heis_{\neweta, W}) \otimes 1$ on the $R$-module $V_{\Heis_{\neweta, W}} \otimes_R \Hom_{R(C\wt{\boxtimes} V)}(\Heis_{\neweta, W}, \Heis)$ into a representation $\Weil \ltimes \Heis$ of $\Sp(V) \ltimes (C\wt{\boxtimes} V)$ on $V_{\Heis}$ that is independent of the choice of Lagrangian $W \subset V$. We call the representation $\Weil \ltimes
	\Heis$ the \textit{Heisenberg--Weil representation associated to $\Heis$} and $\Weil$ the \textit{Weil representations associated to the Heisenberg representation $\Heis$}.
\end{DEflem}
\begin{proof}
	Let $W$ and $W'$ be two Lagrangian subspaces of $V$. Then there exists $g \in
        \Sp(V)$ such that $gW=W'$. 
        The action of $g$ on functions on $C\wt{\boxtimes}V$ provides an $R$-linear isomorphism
        $I_{g}:\,V_{\Heis_{W,\nu}}\simto V_{\Heis_{W',\nu}}$ that intertwines
        $\Heis_{W,\nu}$ with $(\Heis_{W',\nu})^{g}:=\Heis_{W',\nu}\circ {\rm act}_{g}$.
        Therefore, composing with $\Weil_{\neweta, W'}(g)^{-1}$, we get 
        an isomorphism $\Heis_{\neweta, W} \xra{\sim} \Heis_{\neweta, W'}$, which is
        a generator of the invertible $R$-module
        $\Hom_{R(C\wt{\boxtimes} V)}(\Heis_{\neweta, W}, \Heis_{\neweta, W'} )$.
        From the uniqueness of the  Weil representation or from the explicit formulas in the proof
        of Lemma  \ref{Section_generalized_weil-explicit}, we see that this isomorphism
        has to be $\Sp(V)$-equivariant, giving
        an isomorphism $\Weil_{\neweta, W} \ltimes \Heis_{\neweta, W} \xra{\sim}
        \Weil_{\neweta, W'} \ltimes \Heis_{\neweta, W'}$. 
	It follows that the evaluation isomorphism $\Heis_{\neweta, W} \otimes_R \Hom_{R(C\wt{\boxtimes} V)}(\Heis_{\neweta, W}, \Heis_{\neweta, W'} ) \xra{\sim} \Heis_{\neweta, W'}$ sends $(\Weil_{\neweta, W} \ltimes \Heis_{\neweta, W}) \otimes 1$ to $\Weil_{\neweta, W'} \ltimes \Heis_{\neweta, W'}$, and hence the transport of $(\Weil_{\neweta, W} \ltimes \Heis_{\neweta, W}) \otimes 1$ via 
	\[ \Heis_{\neweta, W} \otimes_R \Hom_{R(C\wt{\boxtimes} V)}(\Heis_{\neweta, W}, \Heis) \xleftarrow{\sim} \Heis_{\neweta, W} \otimes_R \Hom_{R(C\wt{\boxtimes} V)}(\Heis_{\neweta, W}, \Heis_{\neweta, W'} ) \otimes_R \Hom_{R(C\wt{\boxtimes} V)}(\Heis_{\neweta, W'}, \Heis) \xra{\simeq} \Heis\]
	agrees with the transport of 
	$(\Weil_{\neweta, W'} \ltimes \Heis_{\neweta, W'}) \otimes 1$ via 
	\( \Heis_{\neweta, W'} \otimes_R \Hom_{R(C\wt{\boxtimes} V)}(\Heis_{\neweta, W'}, \Heis) \xra{\simeq} \Heis .\)
\end{proof}

\begin{lemme}\label{lemme-WeilHeis-hom-abstract}
	Let $\Heis$ and $\Heis'$ be two Heisenberg representations for $(C \wt{\boxtimes}
        V,C,\nu)$ with associated Weil representations $\Weil$ and $\Weil'$. Then the inclusion
	\[ \Hom_{R(\Sp(V) \ltimes (C \wt{\boxtimes} V))}(\Weil \ltimes \Heis,\Weil' \ltimes \Heis')\subseteq\Hom_{R(C \wt{\boxtimes} V)}(\Heis,\Heis')\]
        is an equality.
      \end{lemme}
\begin{proof}
  By Definition/Lemma \ref{Def-Lemma-Weil-rep},  evaluation induces an isomorphism
  $$(\Weil \ltimes \Heis) \otimes_R \Hom_{R(C \wt{\boxtimes} V)}(\Heis,\Heis')
  \simto \Weil' \ltimes \Heis'.$$ 
  Since $\Hom_{R(C \wt{\boxtimes} V)}(\Heis,\Heis')$ is a flat $R$-module, this induces in turn  an isomorphism
  $$\Hom_{R(\Sp(V) \ltimes (C \wt{\boxtimes} V))}(\Weil \ltimes \Heis,\Weil' \ltimes \Heis')
  \simto
  \End_{R(\Sp(V) \ltimes (C \wt{\boxtimes} V))}(\Weil \ltimes \Heis)
  \otimes_{R}\Hom_{R(C \wt{\boxtimes} V)}(\Heis,\Heis').$$
  This reduces the lemma to the case where $\Heis=\Heis'$, where it follows from       
$$ R =\Hom_{R(C \wt{\boxtimes} V)}(\Heis,\Heis) \supseteq \Hom_{R(\Sp(V) \ltimes (C \wt{\boxtimes} V))}(\Weil \ltimes \Heis,\Weil \ltimes \Heis) \supseteq R.$$        
\end{proof}

\alin{From $p$-adic groups to the abstract Heisenberg--Weil setting}
We take up our usual setup associated to a point $x$ in $\BC_{\iota}$. 
We denote by  \index[notation]{Cx@$\C_x$}$\C_{x}$ the image $\check\phi_{\iota,x}^{+}(\Kp_{\iota,x})$ of the character
$\check\phi_{\iota,x}^{+}$ in $\mu_{p^{\infty}}$.
Recall from \ref{Heisenberg} that the map $(a,b) \mapsto
[a,b]:=\check\phi_{\iota,x}^{+}(aba^{-1}b^{-1})$ yields a non-degenerate symplectic form $V_x \times V_x \ra \C_x$ on
\index[notation]{Vx@$\V_x$}$\V_{x}:=\Ku_{\iota, x}/\Kp_{\iota, x}$ that factors through the subgroup $\mu_p \subseteq \C_x$ of order $p$. By construction,
the conjugation action of $G_{\iota,x}$ on $\Ku_{\iota, x}$ induces an action on
$V_{x}$  that preserves this non-degenerate symplectic form. We denote by
$\sigma:\,G_{\iota,x}\To{}\Sp(V_{x})$ the corresponding morphism.

\begin{lem} \label{lemma-special-morphism} 
  There exists a homomorphism
  $f:  \Ku_{\iota, x}  \ra  \C_{x}\widetilde{\boxtimes}\V_{x}$
  such that
  \begin{enumerate}
  \item The following diagram is commutative
    $$\xymatrix{
      \Kp_{\iota, x} \ar@{->>}[d]_{\check\phi_{\iota,x}^{+}} \ar@{^{(}->}[r] & \Ku_{\iota,x} \ar@{->>}[r] \ar[d]^{f} & \V_{x} \ar@{=}[d]\\
      \C_{x} \ar@{^{(}->}[r] & \C_{x}\widetilde{\boxtimes}\V_{x} \ar@{->>}[r] & \V_{x}
      }$$
    \item The product map
      $\sigma\ltimes f : G_{\iota,x}\ltimes\Ku_{\iota, x}  \ra  \Sp(\V_{x})\ltimes(\C_{x}\widetilde{\boxtimes}\V_{x})$
      is a homomorphism.
 \end{enumerate}
 Moreover, such a $f$ is unique up to post-composition by conjugation by an element in
 $\C_{x}\widetilde{\boxtimes}\V_{x}$. 
 More precisely, if $f'$ is another homomorphism satisfying i) and ii), then there is
 $w\in(\V_{x})^{\sigma(G_{\iota,x})}$ such that
 $\forall k\in \Kp_{\iota,x},\, f'(k)=(c(k).[w, v(k)],v(k))$, 
 where $f(k)=(c(k),v(k))$.
  
\end{lem}
\begin{proof}
  Recall the groups $J^{i}_{\iota,x}$ and $J^{i+}_{\iota,x}$ introduced in \ref{auxiliary_groups}
  and the equality $\Ku_{\iota,x}=G_{\iota,x,0+} \prod_{i=1}^d J^{i}_{\iota,x}$. Set
  $$
  J^{\mathrm{der}}:=\prod_{i=1}^d \left(J^{i}_{\iota,x}\cap \mathbf{G}^{i}_{\iota,\rm der}(F)\right)
  \hbox{ and }
  J^{\mathrm{der}}_{+}:=\prod_{i=1}^d \left(J^{i+}_{\iota,x}\cap \mathbf{G}^{i}_{\iota,\rm  der}(F)\right)$$

    By \cite[Corollary~7.2]{JF_exhaust} we have  $\Ku_{\iota,x}=G_{\iota,x,0+} J^{\rm der}.$
   The desired morphism $f$ has to
   coincide with  $\check\phi_{i}^{+}$ on $G_{\iota,x,0+}$
   since  $ G_{\iota,x,0+} \subset \Kp_{\iota, x}$.
  We will follow Yu to define $f$ on   $J^{\mathrm{der}}$.
  Note that    $J^{\mathrm{der}}/J^{\mathrm{der}}_+ \simto \Ku_{\iota,x}/\Kp_{\iota,x} =  \V_x $ and
  $\check\phi_{\iota,x}^{+}(J^{\rm der})=\mu_{p}\subset C_{x}$.
  Therefore, setting $N:=\ker(\check\phi_{\iota,x}^{+}|_{J^{\rm der}})$, the group $J/N$ is a
  Heisenberg $p$-group in the sense of \cite[\S 10]{Yu_tamescusp}.
  We can now apply the proof of \cite[Proposition~11.4]{Yu_tamescusp} to
        our setting by replacing Yu's $J$ by our $J^{\mathrm{der}}$, his $J_+$ by our
        $J^{\mathrm{der}}_+$, his $N$ by ours, his $G'$ by our
        $G_\iota$, his $y$ by our $x$,
        his $(G', G)(F)_{y,(r+,s+)}$ by our $\prod_{i=1}^d (G^{i-1}_\iota \cap
        (G^i_\iota)^{\mathrm{der}},(G^i_\iota)^{\mathrm{der}})_{x,r_{i-1}+,r_{i-1}/2+}$,
        and his $G_a(F)_{y,s}$ by our $U_{\alpha}\cap G(F)_{x,r_i/2}$,
        where $i$ is the lowest index 
        such that the root $\alpha$
        occurs in $\mathfrak{g}_{\iota}^{i}$.      
        His argument goes through and provides us with a group homomorphism $f^{\mathrm{Yu}}$ that fits in a 
        commutative  diagram
    $$\xymatrix{
      J^{\rm der}_{+} \ar@{->>}[d]_{\check\phi_{\iota,x}^{+}} \ar@{^{(}->}[r]
      & J^{\rm der}  \ar@{->>}[r] \ar[d]^{f^{\rm Yu}} & \V_{x} \ar@{=}[d]\\
      \mu_{p} \ar@{^{(}->}[r] & \mu_{p}\widetilde{\boxtimes}\V_{x} \ar@{->>}[r] & \V_{x}
      }$$
      and such that
      $\sigma\ltimes f^{\rm Yu}:\,G_{\iota, x} \ltimes J^{\mathrm{der}}\To{}
      \Sp(V_{x})\ltimes (\mu_{p}\wt{\boxtimes}V_{x})$
      is a homomorphism.
      Since $f^{\rm Yu}$ coincides with $\check\phi_{\iota,x}^{+}$ on $G_{\iota,x,0+}\cap
      J^{\rm der}\subset J^{\rm der}_{+}$, we can define a morphism
      $f : \Ku_{\iota,x}\To{} C_{x}\wt{\boxtimes}V_{x}$ by setting
      $f(k):=\check\phi_{\iota,x}^{+}(g)f^{\rm Yu}(j)$ for any $k=gj$ with $g\in
      G_{\iota,x,0+}$ and $j\in J^{\rm der}$. The commutativity of the diagram in i) of
      the lemma follows from that of the diagram above. Property ii) follows from the same
      property for $f^{\rm Yu}$ and the fact that $G_{\iota,x}$ centralizes $\check\phi_{\iota,x}^{+}$.

      Now, let $f'$ be another morphism satisfying i) and ii). There is an
      automorphism $\alpha$ of the extension
      $C_{x}\injo C_{x}\wt{\boxtimes}V_{x}\twoheadrightarrow V_{x}$ such that $f'=\alpha\circ
      f$. Any such automorphism has the form $(c,v)\mapsto (c\Lambda(v),v)$ for some
      morphism $\Lambda : V_{x}\To{} C_{x}$. Since $\V_x$ is an $\bF_p$-vector space, $\Lambda$
      factors through $\mu_p$ and is a linear form. Hence there exists $w \in \V_x$ such
      that $\Lambda(v)=[w, v]$ for all $v \in \V_x$.  It follows that,
      writing $f(k)=(c(k),v(k))$ and $f'(k)=(c'(k),v'(k))$  for
      $k\in\Ku_{\iota,x}$, we have $v'(k)=v(k)$ and $c'(k)=c(k)[w,v(k)]$. 
      Now, property ii) for $f$, resp. $f'$,  means that
      $f(gkg^{-1})=(c(k),\sigma(g).v(k))$, resp. $f'(gkg^{-1})=(c'(k),\sigma(g).v'(k))$, from
      which we deduce that $c'(k)=c(k)[w,\sigma(g).v(k)]$, for all $g\in\G_{\iota,x}$ and $k\in\Ku_{\iota,x}$.
      It follows that $[\sigma(g)^{-1}w,v(k)]=[w,\sigma(g).v(k)]=[w,v(k)]$ for all
      $k\in\Ku_{\iota,x}$. Since $k\mapsto v(k)$ is surjective onto $V_{x}$, we deduce
      that $w\in (V_{x})^{\sigma(G_{x})}$.
\end{proof}

\alin{A family of non-twisted Heisenberg--Weil representations}
Let $\Heis_{\iota,x}$ be a Heisenberg representation for the triple
$(\Ku_{\iota,x},\Kp_{\iota,x},\check\phi_{\iota,x}^{+})$. We will
extend this representation to a
representation of $\K_{\iota, x}$ using the theory of Heisenberg--Weil representations.

Pick a morphism $f: \Ku_{\iota, x}  \ra \C_{x}\widetilde{\boxtimes}\V_{x}$ as in Lemma \ref{lemma-special-morphism}.
Note that for any Heisenberg representation $\Heis$ of
$\C_{x}\widetilde{\boxtimes}\V_{x}$, the pullback representation $\Heis\circ f$ is
a Heisenberg representation of $\Ku_{\iota,x}$. It then follows from the classification of
Heisenberg representations that all Heisenberg representations of $\Ku_{\iota,x}$ arise by
pullback from $f$. In particular, there is a unique Heisenberg
representation $\Heis_{f}$ of $\C_{x}\widetilde{\boxtimes}\V_{x}$ on the module $V_{\Heis_{\iota,x}}$
such that $\Heis_{\iota,x}=\Heis_{f}\circ f$.
Denote by $\Weil_{f}$ the Weil representation $\Sp(V_{x})\To{}\Aut_{R}(V_{\Heis_{\iota,x}})$ associated to
$\Heis_{f}$ by Definition/Lemma \ref{Def-Lemma-Weil-rep}.

\begin{lem}
  The representation $\Weil_{\iota,x}:=\Weil_{f}\circ \sigma :
  G_{\iota,x}\To{}\Aut_{R}(V_{\Heis_{\iota,x}})$ is independent of $f$.
\end{lem}
\begin{proof}
  Let $f'$ be another morphism as in Lemma \ref{lemma-special-morphism} and let $w$ be an
  element of $V_{x}$ as in the last sentence of that lemma.  Then
  $\Heis_{f'}=\Heis_{f}\circ {\rm Inn}_{(1,w)}$ where
  ${\rm Inn}_{(1,w)}$ denotes inner conjugation by $(1,w)$ in $C_{x}\wt{\boxtimes}V_{x}$.
  It follows that the $R$-linear automorphism $\Heis_{f}(1,w)$ of $V_{\Heis_{\iota,x}}$ induces
  an isomorphism of representations $\Heis_{f}\simto\Heis_{f'}$.
  By Lemma \ref{lemme-WeilHeis-hom-abstract}, it also induces an isomorphism $\Weil_{f}\simto\Weil_{f'}$.
  Hence for any $u\in\Sp(V_{x})$, we have
  $$\Weil_{f'}(u)=\Heis_{f}(1,w)\circ \Weil_{f}(u)\circ \Heis_{f}(1,w)^{-1}
  =  \Weil_{f}(u) \circ \Heis_{f}(1,u^{-1}.w)\circ \Heis_{f}(1,w)^{-1}.$$
 Since $w$ is fixed under $\sigma(G_{\iota,x})$, we get $\omega_{f'}\circ
 \sigma=\omega_{f}\circ \sigma$, as desired.
\end{proof}

  We call $\omega_{\iota,x}$ the Weil representation associated to $\Heis_{\iota,x}$.
  By construction, the product map
$\Weil_{\iota,x}\ltimes\Heis_{\iota,x}:\, G_{\iota,x}\ltimes \Ku_{\iota,x}\To{} \Aut_{R}(V_{\Heis_{\iota,x}})$
is a group homomorphism. However, this representation does not descend along the
multiplication map $G_{\iota,x}\ltimes\Ku_{\iota,x}\To{}\K_{\iota,x}$. To be able to
descend along this map, we need to twist $\omega_{\iota,x}$ by a character of
$G_{\iota,x}$ that extends $\check\phi_{\iota,x}^{+}$. We will use the restriction of the
character $\check\varphi_{0} : G_{\iota}\To{}\mu_{p^{\infty}}\subset R^{\times}$ chosen in
\ref{def_characters}. This choice is independent of $x$ and $\Heis_{\iota,x}$.

\begin{defn}\label{Def-HeisWeil-non-twisted}
  We define the representation $\HeisWeil^{\rm nt}_{\iota, x}$ of $\K_{\iota, x}$ to be the unique
  $R$-representation such that the composition of the morphism 
	\begin{eqnarray*}
		G_{\iota,x} \ltimes \Ku_{\iota, x}  & \twoheadrightarrow & \K_{\iota, x} \\
		(k_1, k_2) & \mapsto & k_1k_2
	\end{eqnarray*}
	with $\HeisWeil^{\rm nt}_{\iota, x}$ yields the representation
        $(\check\varphi_{0}|_{G_{\iota,x}} \Weil_{\iota, x})  \ltimes
        \Heis_{\iota,x}$.  We call $\HeisWeil^{\rm nt}_{\iota, x}$ the \emph{non-twisted $R\K_{\iota,x}$-Heisenberg--Weil representation associated with $\Heis_{\iota,x}$}. 
\end{defn}

\begin{cor}\label{cor-WeilHeis-hom-non-twisted}
	Let $\Heis_{\iota,x}$ and $\Heis'_{\iota,x}$ be two Heisenberg representations of
        $\Ku_{\iota, x}$ to which we associate the representations $\HeisWeil^{\rm
          nt}_{\iota, x}$ and ${\HeisWeil'}^{\rm nt}_{\iota, x} $ of $\K_{\iota, x}$ via
        Definition \ref{Def-HeisWeil-non-twisted}. Then
	\[ \Hom_{R\K_{\iota, x}}(\HeisWeil_{\iota, x}^{\rm nt} ,{\HeisWeil'}^{\rm nt}_{\iota, x})
          =\Hom_{R\Ku_{\iota, x}}(\Heis_{\iota, x},\Heis'_{\iota, x})\]
\end{cor}
\begin{proof}
	This follows from Lemma \ref{lemme-WeilHeis-hom-abstract} since $\Ku_{\iota, x}$ surjects onto $\C_x \wt{\boxtimes} \V_x$.
\end{proof}

\alin{The quadratic twist}\label{section-quadratic-twist}
We write  $\eps_x^{\vec \bfG_\iota}$, or also simply $\eps_x$, for the quadratic character of $\K_{\iota,x}$ defined in \cite[Notation~3.6.4]{AFMO2} arising from \cite{FKS} attached to the sequence $\vec \bfG_\iota$ and the generic characters $\psi_i$, i.e, 
\[ \eps_x = \prod_{i=1}^d \eps_x^{G^i_\iota/G^{i-1}_\iota}, \]
where $\eps_x^{G^i_{\iota}/G^{i-1}_{\iota}}$ is trivial on $\Ku_{\iota,x}$, and on
$\K_{\iota,x}/\Ku_{\iota,x} \simeq G_{\iota,x}/G_{\iota,x,0+}$ it is given by the
restriction of the quadratic character $\eps_x^{G^i_{\iota}/G^{i-1}_{\iota}}$ on
$G^i_{\iota,x}$ defined in \cite[Lemma~4.1.2]{FKS} for $r=r_{i-1}$. 
The only property about the quadratic character $\eps_x$ that we use is that it satisfies the following.
\begin{lem}[{\cite{FKS}}]
	Let $x, x' \in \BT'_{\iota}$. Then
\inisub
\begin{subequations}
	\begin{equation}\label{eqn-eps-delta}
	\left(\eps_x \right)|_{\K_{\iota,x} \cap \K_{\iota,x'}} \cdot \delta_{x'}^x=
        \left(\eps_{x'}\right)|_{\K_{\iota,x} \cap \K_{\iota, x'}} \cdot \delta_x^{x'},
\end{equation}
\end{subequations}
where $\delta_{x'}^x(g)=\mathrm{sgn}(\det_{\bF_p}(u_g))$ with $u_g$ denoting the action induced via conjugation by $g \in \K_{\iota,x}\cap\K_{\iota,x'}$ on the $\bF_p$-vector space 
\[((\Ku_{\iota,x}\cap \Kp_{\iota, x'})\Kp_{\iota,x})/\Kp_{\iota,x}.\] 
\end{lem}
\begin{proof}
	This follows from applying \cite[Lemma~4.1.2]{FKS} to each
        $\eps_x^{G^i_\iota/G^{i-1}_\iota}$ and taking the product.
\end{proof}

\begin{DEf}\label{Def-HeisWeil}
	Given a Heisenberg representation $\Heis_{\iota,x}$ of $\Ku_{\iota, x}$,
	we define  the \emph{twisted $R\K_{\iota,x}$-Heisenberg--Weil representation associated with $\Heis_{\iota,x}$} 
        as $$\kappa_{\iota,x}:=\varepsilon_{x} \cdot \kappa^{\rm nt}_{\iota,x}.$$
	More generally, an $R\K_{\iota,x}$-module is called a \emph{twisted
          $R\K_{\iota,x}$-Heisenberg--Weil representation}\index[terminology]{twisted
          $R\K_{\iota,x}$-Heisenberg--Weil representation} if it is the twisted
        $R\K_{\iota,x}$-Heisenberg--Weil representation associated to some Heisenberg
        representation of $\Ku_{\iota, x}$.
\end{DEf}

\begin{rema}
  If $R$ is an algebraically closed field, the representation $\HeisWeil_{\iota,x}$
  agrees with the following representation constructed in \cite{AFMO2}: if
  $r_{d-1} = r_{d}$, resp. $r_{d-1} < r_{d}$, set
 	\begin{align*} \!
        &  \textnormal{HW}_x=
          \bigl( 
		\left(
		\bfG_\iota=\bfG^0 \subsetneq \bfG^1 \subsetneq \ldots \subsetneq \bfG^d
		\right),
		(r_0, \ldots , r_{d-1}), x, \bfG^{0}(F)_{x}, (\psi_0, \ldots , \psi_{d-1})
		\bigr), \, \hbox{ resp.} \\
         & \textnormal{HW}_x=
          \bigl(
		\left(
		\bfG_\iota=\bfG^0 \subsetneq \bfG^1 \subsetneq \ldots \subsetneq \bfG^d \subseteq \bfG^{d+1} := \bfG^{d}
		\right),
		(r_0, \ldots , r_{d}), x, \bfG^{0}(F)_{x}, (\psi_0, \ldots , \psi_{d})
		\bigr) 
	\end{align*}
    for $x \in \BT_\iota$.
	Then $\textnormal{HW}_x$ is a Heisenberg--Weil datum as defined in
        \cite[Definition~3.6.1]{AFMO2} by Lemma \ref{generic}. Note that by our construction we have fixed
        embeddings $\BT_\iota=\BT(\bfG^0, F) \subseteq \BT(\bfG^1, F)\subseteq \hdots \subseteq \BT(\bfG,F)$ so we do not record them as part of the Heisenberg--Weil datum. 
	Then it follows from \cite[Corollary~2.5]{Gerardin} and \cite[Lemma~3.6.8]{AFMO2} that $(\HeisWeil_x, V_{\HeisWeil_x})$ is isomorphic to the representation of $K_{\iota,x}$ obtained from $\textnormal{HW}_x$ via the twisted Heisenberg--Weil construction as in \cite[Notation~3.6.3]{AFMO2}. 	
\end{rema}

\begin{coro}\label{cor-WeilHeis-hom}
	Let $\Heis_{\iota,x}$ and $\Heis'_{\iota,x}$ be two Heisenberg representations of $\Ku_{\iota, x}$ with associated twisted $R\K_{\iota,x}$-Heisenberg--Weil representations $\HeisWeil_{\iota,x}$ and $\HeisWeil'_{\iota,x}$. Then
	\[\Hom_{R\K_{\iota, x}}(\HeisWeil_{\iota, x} ,\HeisWeil'_{\iota, x})=\Hom_{R\Ku_{\iota, x}}(\Heis_{\iota, x},\Heis'_{\iota, x})\]
\end{coro}
\begin{proof}
	This follows from Corollary \ref{cor-WeilHeis-hom-non-twisted} and Definition \ref{Def-HeisWeil}.
\end{proof}

\begin{prop}\label{Hom_HW_H}
	Let $x, x' \in \BT'_{\iota}$ and let $\Heis_{\iota,x}$ and $\Heis_{\iota,x'}$ be two Heisenberg representations  with associated twisted $R\K_{\iota,x}$-Heisenberg--Weil representations $\HeisWeil_{\iota,x}$ and $\HeisWeil_{\iota,x'}$. Then 
	\[ \Hom_{R(\K_{\iota,x}\cap\K_{\iota,x'})}(\HeisWeil_{\iota,x},\HeisWeil_{\iota,x'})
	= \Hom_{R(\Ku_{\iota,x}\cap\Ku_{\iota,x'})}(\Heis_{\iota,x},\Heis_{\iota,x'}) . \]
\end{prop}
\begin{proof}
	\inisub
	\begin{subequations}
          If $x = x'$, the claimed equality holds since
          both sides are equal to $R$. So let us assume $x\neq x'$, and
         let $(\mathbf{P},\bar{\mathbf{P}})$ be the pair of opposite parabolic subgroups of $\mathbf{G}$ with unipotent radicals $\mathbf{U}$ and $\bar{\mathbf{U}}$, respectively, that are attached to the segment $(x,x')$ as in the proof of Lemma \ref{intertw_converse} and Lemma \ref{intertwining_Heisenberg}, i.e., we have
\begin{equation*}
M\cap \K_{\iota,x}=M\cap \K_{\iota,x'}, \quad
\bar U\cap \Kp_{\iota,x} \supseteq \bar U\cap \K_{\iota,x'}, \quad
U\cap\K_{\iota,x} \subseteq U\cap\Kp_{\iota,x'} ,
\end{equation*}	
\[  \K_{\iota,x}\cap\K_{\iota,x'}=(U\cap\K_{\iota,x}) (M\cap\K_{\iota,x}) (\bar U\cap
\K_{\iota,x'}) \]
and
\begin{equation} 
\Ku_{\iota,x}\cap\Ku_{\iota,x'}=(U\cap\Ku_{\iota,x}) (M\cap\Ku_{\iota,x}) (\bar U\cap
\Ku_{\iota,x'}) 
\end{equation}
Hence the conjugation action of $\K_{\iota,x}\cap\K_{\iota,x'}$ preserves the subspace \[\V_U:=(U \cap \Ku_{\iota, x})/(U \cap \Kp_{\iota,x})=((\Ku_{\iota, x}\cap \Kp_{\iota, x'})\Kp_{\iota, x})/\Kp_{\iota, x}\]
 of $\Ku_{\iota, x}/\Kp_{\iota, x}=\V_x$.
Therefore the image of $\K_{\iota,x}\cap\K_{\iota,x'}$ in $\Sp(\V_x)$ is contained in the parabolic subgroup $\PSp$ that is the stabilizer of $\V_U$ in $\Sp(\V_x)$. 
We write $\V_M:=(M \cap \Ku_{\iota, x})/(M \cap \Kp_{\iota,x})$.
Then by Lemma \ref{Lemma-restriction-of-Weilrep} the restriction of the Heisenberg--Weil representation $\Weil \ltimes \Heis_{\iota,x}$ associated to $\Heis_{\iota,x}$ to $\PSp \ltimes (\C_{x}\wt{\boxtimes}\V_x)$ satisfies: 
\[ (\Weil \ltimes \Heis_{\iota,x})|_{\PSp \ltimes (\C_{x}\wt{\boxtimes}\V_x)} 
\simeq
\Ind{\PSp \ltimes ((\C_{x}\wt{\boxtimes}\V_M)\times \V_U)}{\PSp \ltimes (\C_{x}\wt{\boxtimes}\V_x)} (\Weil^M \ltimes \Heis^M_{\iota,x}) \otimes (\chi^U \ltimes 1),
 \]
 where $\Weil^M \ltimes \Heis^M_{\iota,x}$ denotes a Heisenberg--Weil representation of  $\Sp(\V_M) \ltimes (\C_{x}\wt{\boxtimes}\V_M)$ with central character $\check\phi_{\iota,x}^{+}$ on which $\PSp \ltimes (\V_U \times (\C_{x}\wt{\boxtimes}\V_M))$ acts via the projection to $\Sp(\V_M) \ltimes (\{0\} \times (\C_{x}\wt{\boxtimes}\V_M))$, and $\chi^U$ denotes the character $\PSp \ra \{\pm 1\}$ given by $p \mapsto \text{sgn}(\det_{\bF_p}(p|_{\V_U}))$. Note that the composition of $\K_{\iota,x}\cap\K_{\iota,x'} \ra \PSp$ with $\chi^U$ is the character $\delta_{x'}^{x}$ defined in \ref{section-quadratic-twist}.
 Thus  we obtain
\[ (\HeisWeil_{\iota, x})|_{\K_{\iota,x}\cap\K_{\iota,x'}} 
\simeq
\left(\Ind{(U\cap\K_{\iota,x}) (M\cap\K_{\iota,x}) (\bar U\cap
	\Kp_{\iota,x})}{(\K_{\iota,x}\cap\K_{\iota,x'})\Ku_{\iota, x}}    \eps_x^{\vec{\bfG}_\iota}  \eps_x^{\vec{\bfM}_\iota} \HeisWeil^M_{\iota, x} \otimes \delta_{x'}^{x}\right)|_{\K_{\iota,x}\cap\K_{\iota,x'}}  ,
\]
where $\HeisWeil^M_{\iota, x}$ denotes the representation of $M \cap \K_{\iota,x}$ constructed analogously to $\HeisWeil_{\iota, x}$ in Definition \ref{Def-HeisWeil} replacing $\vec{\bfG}_{\iota}$ by $\vec{\bfM}_{\iota}$ and $\Heis_{\iota,x}$ by $\Heis^M_{\iota,x}$ in the construction, and $\eps_x^{\vec{\bfM}_\iota} \HeisWeil^M_{\iota, x}$ is viewed as a representation of $(U\cap\K_{\iota,x}) (M\cap\K_{\iota,x}) (\bar U\cap \Kp_{\iota,x})$ by letting $(U\cap\K_{\iota,x})$ and $(\bar U\cap
\Kp_{\iota,x})$ act trivially.
From this, we deduce that 
\begin{equation}\label{eqn-HeisWeil1} 
	(\HeisWeil_{\iota, x}|_{\K_{\iota,x}\cap\K_{\iota,x'}})_{U\cap\Ku_{\iota,x}} 
\simeq
\left(\eps_x^{\vec{\bfG}_\iota}  \eps_x^{\vec{\bfM}_\iota} \HeisWeil^M_{\iota, x} \otimes \delta_{x'}^{x}\right)|_{\K_{\iota,x}\cap\K_{\iota,x'}}  ,
\end{equation}
and that the surjection $V_{\HeisWeil_{\iota, x}}  \twoheadrightarrow (V_{\HeisWeil_{\iota, x}})_{U\cap\Ku_{\iota,x}} $ is $\K_{\iota,x}\cap\K_{\iota,x'}$-equivariant.
Similarly, we obtain
\begin{equation}\label{eqn-HeisWeil2} 
(\HeisWeil_{\iota, x'}|_{\K_{\iota,x}\cap\K_{\iota,x'}})^{\bar U\cap\Ku_{\iota,x'}} 
\simeq
\left(\eps_{x'}^{\vec{\bfG}_\iota}  \eps_{x'}^{\vec{\bfM}_\iota} \HeisWeil^M_{\iota, x'} \otimes \delta_{x}^{x'}\right)|_{\K_{\iota,x}\cap\K_{\iota,x'}}  ,
\end{equation}
and in particular that the subspace $V_{\HeisWeil_{\iota, x'}}^{\bar U\cap\Ku_{\iota,x'}}$ of $V_{\HeisWeil_{\iota, x'}}$ is $\K_{\iota,x}\cap\K_{\iota,x'}$-stable.
We can thus consider the following commutative diagram of inclusions
$$
\xymatrix{
\Hom_{R(\K_{\iota,x} \cap \K_{\iota,x'})}\left((\HeisWeil_{\iota,x})_{U\cap\Ku_{\iota,x}},(\HeisWeil_{\iota,x'})^{\bar U\cap\Ku_{\iota,x'}}\right)
\ar@{}[d]|-*[@]{\subset} \ar@{}[r]|-*[@]{\subset} & 
\Hom_{R(\K_{\iota,x} \cap \K_{\iota,x'})}\left(\HeisWeil_{\iota,x},\HeisWeil_{\iota,x'}\right)
\ar@{}[d]|-*[@]{\subset}
\\
\Hom_{R(\Ku_{\iota,x} \cap \Ku_{\iota,x'})}\left((\Heis_{\iota,x})_{U\cap\Ku_{\iota,x}},(\Heis_{\iota,x'})^{\bar U\cap\Ku_{\iota,x'}}\right)
\ar@{}[r]|-*[@]{\subset} &
\Hom_{R(\Ku_{\iota,x}\cap \Ku_{\iota,x'})}\left(\Heis_{\iota,x},\Heis_{\iota,x'}\right)
}$$
where both horizontal inclusions are obtained by precomposition with projection and
composition with inclusion. By \eqref{Hom_Heis}, the bottom inclusion is an equality.
On the other hand, since the image of $x$ and $x'$ in the reduced building of $M$ agree,
we have $\eps_x^{\vec{\bfM}_\iota}=\eps_{x'}^{\vec{\bfM}_\iota}$ and  
\[ \Hom_{R(M \cap \K_{\iota, x})}(\HeisWeil^M_{\iota, x} ,\HeisWeil^M_{\iota, x'})=\Hom_{R(M \cap \Ku_{\iota, x})}(\Heis^M_{\iota, x},\Heis^M_{\iota, x'}) \]
by Corollary \ref{cor-WeilHeis-hom}. Combining this equality with \eqref{eqn-HeisWeil1},
\eqref{eqn-HeisWeil2} and \eqref{eqn-eps-delta}, we obtain that the left vertical
inclusion in the above diagram is an equality.
It follows that the two remaining inclusions are equalities too.
\end{subequations}
\end{proof}

\subsection{Construction of a Heisenberg--Weil coefficient system}  \label{sec:heis-weil-coeff}

In this section we finally prove Theorem \ref{thm_WH_coef_system-general}, so we let $R$ be a commutative
$\Rmintwo$-algebra.

We will first prove the result under the following assumption, which we assume from now on until the end of \ref{sec:HW-underassumption-3}.
\begin{assumption}\label{assumption}
	We have  $Z(C_{\hat{\mathbf{G}}}(\phi))^{\varphi(W_{F}),\circ}=Z(\hat{\mathbf{G}})^{W_{F},\circ}$. Equivalently,
	the maximal split central torus of ${\mathbf{G}}_{\iota}$ coincides
	with that of ${\mathbf{G}}$,
	i.e., $Z(\mathbf{G}_{\iota})/Z(\mathbf{G})$ is anisotropic.
\end{assumption}

The role of
this assumption is to ensure that the subset $\BT'_{\iota}$ of $\BT'$ is the \emph{reduced}
building of $G_{\iota}$. In particular, it is equipped with the canonical  Bruhat--Tits
polysimplicial structure. Note that each Bruhat--Tits polysimplex is  the union of the $e$-facets it
contains. 
However, beware that our $e$-facets were defined with respect to $G$, and they may not be
$e$-facets with respect to $G_{\iota}$.

\begin{no}
	For every chamber $\cC$ (i.e., polysimplex of maximal dimension) of
        $\BT'_{\iota}$, we denote by $\cF_{\cC}$ the unique $e$-facet
        that contains the barycenter of $\cC$. By our choice of $e$,
        this is actually an $e$-vertex, whose stabilizer $G_{\iota,\FC_{\CC}}$ coincides
        with the stabilizer $G_{\iota,\CC}$ of $\CC$. 
\end{no}

\alin{Basic choices} \label{basic-choices} We choose and fix the following data :
\begin{enumerate}
\item a chamber $\cC_0$ of $\BT'_{\iota}$, whose closure we denote by $\ov\CC_{0}$.
\item  a set $\sS$ of representatives of $G_{\iota,\CC_{0}}$-orbits of $e$-facets of $\BT'$ lying in
  $\ov\CC_{0}$. Note that this set must contain $\FC_{\CC_{0}}$ and 
 is also a set of representatives of   
  $G_\iota$-orbits of $e$-facets lying in $\BT'_\iota$.

\item  for every $\cF_{0} \in \sS$, a twisted $R\K_{\iota, \cF_{0}}$-Heisenberg--Weil
  representation $(\HeisWeil_{\iota,\cF_{0}},\WC_{\iota,\FC_{0}}^{0})$ such that the invertible $R$-module
  $\Hom_{R(\K_{\iota,\cF_{0}}\cap\K_{\iota,\FC_{\CC_{0}}})}
  (\HeisWeil_{\iota, \cF_{0}},\HeisWeil_{\iota, \cF_{\CC_{0}}})=:H_{\FC_{0},\FC_{\CC_{0}}}$
 is free. This can be achieved as follows: Start with an arbitrary choice
  $(\kappa'_{\iota,\FC_{0}})_{\FC_{0}\in\sS}$ and set
  $\kappa_{\iota,\FC_{0}}:=  \kappa'_{\iota,\FC_{0}}\otimes_{R}
  (H_{\FC_{0},\FC_{\CC_{0}}})^{*}$ where
  $(H_{\FC_{0},\FC_{\CC_{0}}})^{*}:=\Hom_{R}(H_{\FC_{0},\FC_{\CC_{0}}},R)$
  denotes the inverse $R$-module.

\item   a generator $\alpha_{\cF_{0},\cF_{\cC_0}}$ of the $R$-module
    $\Hom_{R(\K_{\iota,\cF_{0}}\cap\K_{\iota,\FC_{\CC_{0}}})}
  (\HeisWeil_{\iota, \cF_{0}},\HeisWeil_{\iota, \cF_{\CC_{0}}})$.
\end{enumerate}

\alin{Construction of a Heisenberg--Weil coefficient system under Assumption \ref{assumption}, step 1: modules and actions}
For any $\cF \in \BT'_{\iota,\bullet/e}$, we let $\cF_0 \in \sS$ be the unique element in
the $G_\iota$-orbit of $\cF$. As before, we denote by $G_{\iota,\cF \cF_0}:=\{g\in G_\iota, g\cF_0=\cF\}$
the transporter of $\FC_{0}$ to $\FC$, an open subset of $G_{\iota}$, and we set
$$ \K_{\iota,\FC\FC_{0}}:=
G_{\iota,\FC\FC_{0}}\K_{\iota,\FC_{0}}=\K_{\iota,\FC}G_{\iota,\FC\FC_{0}},$$
which is an open subset of $G$. Note that, since $G_{\iota,\FC}$ is contained in
$\K_{\iota,\FC}$, the set $\K_{\iota,\FC\FC_{0}}$ is a right 
$\K_{\iota,\FC_{0}}$-coset, and a left  $\K_{\iota,\FC}$-coset. Now we set
\[ \cW_{\iota, \cF}:=\cC_c^\infty(\K_{\iota, \cF \cF_{0}})\otimes_{R\K_{\iota, \cF_0}} \WC^{0}_{\iota,\cF_0}.\]
For any facet $\FC'\in\BT'_{\iota,\bullet/e}$ in the $G_{\iota}$-orbit of $\FC$, the set
$\K_{\iota,\FC'\FC}$ is the hom set $\Hom(\FC,\FC')$ in $[\BT'_{\iota,\bullet/e}/K_{\iota}]$.
Therefore, the maps
$$\application{\WC_{\iota,\FC,\FC'}:\,}{\K_{\iota,\FC'\FC}}{\Hom_{R}(\WC_{\iota,\FC},\WC_{\iota,\FC'})}{g}{\lambda_{g}\otimes\id}, $$ 
where $\lambda_{g} : \cC_c^\infty(\K_{\iota, \cF \cF_{0}})\To{}\cC_c^\infty(\K_{\iota, \cF'  \cF_{0}})$ 
is induced by left translation by
  $g^{-1}$,
 define a coefficient system on the full
subcategory of $[\BT'_{\iota,\bullet/e}/K_{\iota}]$ given by the orbit of $\FC_{0}$.
In particular, $\WC_{\iota,\FC,\FC}$ defines a representation of $\K_{\iota,\FC}$ on $\WC_{\iota,\FC}$ 
and, by construction, its restriction to $\Ku_{\iota,\FC}$ is a
Heisenberg representation for $\check\phi^{+}_{\iota,\FC}$.
To maintain consistency of our notation with that of $G_{\iota}$-equivariant coefficient
systems, we also write
$$ g_{\WC,\FC} := \WC_{\iota,\FC,g\FC}(g),\;  \hbox{ for any } g\in G_{\iota}$$

\alin{Construction of a Heisenberg--Weil coefficient system under Assumption \ref{assumption}, step 2: face maps}
With $\FC$ and $\FC_{0}$ as in the previous paragraph, pick also a chamber $\CC$ such
that $\FC\subseteq \ov\CC$. We can then find $g\in G_{\iota}$ such that $\CC=g\CC_{0}$ and
$\FC=g\FC_{0}$. Such a $g$ is unique modulo right multiplication by $G_{\iota,\FC_{0}}\cap
G_{\iota,\CC_{0}}$, and we have
$\K_{\iota, \cF \cF_{0}}\cap\K_{\iota, \FC_{\CC},\cF_{\CC_{0}}}
=g(\K_{\iota,\cF_{0}}\cap\K_{\iota,\cF_{\CC_{0}}})
=(\K_{\iota,\cF}\cap\K_{\iota,\cF_{\CC}})g.$
We thus can define a generator $\alpha_{\FC,\CC}$ of 
$\Hom_{R(\K_{\iota,\cF}\cap\K_{\iota,\FC_{\CC}})}(\WC_{\iota, \cF},\WC_{\iota, \cF_{\CC}})$
by requiring the commutativity of the following diagram of $R(\K_{\iota,\cF}\cap\K_{\iota,\FC_{\CC}})$-modules (where ${can}$ denotes
canonical inclusions):
$$
\xymatrix{
\WC_{\iota,\FC}=\cC_c^\infty(\K_{\iota, \cF \cF_{0}})\otimes_{R\K_{\iota, \cF_0}} \WC^{0}_{\iota,\cF_0}
\ar@{-->}[rr]^{\alpha_{\FC,\CC}}
&&
\WC_{\iota,\FC_{\CC}}=\cC_c^\infty(\K_{\iota, \cF_{\CC} \cF_{\CC_{0}}})\otimes_{R\K_{\iota, \cF_{\CC_{0}}}}\WC^{0}_{\iota,\cF_{\CC_{0}}}
\\ \\
\cC_c^\infty(\K_{\iota, \cF \cF_{0}}\cap\K_{\iota, \FC_{\CC}\cF_{\CC_{0}}})
\otimes_{R(\K_{\iota, \cF_0}\cap\K_{\iota, \cF_{\CC_{0}}})} \WC^{0}_{\iota,\cF_0}
\ar[uu]^{\simeq}_{{\rm can}\otimes\id}
\ar[uurr]_{\,{\rm can}\,\otimes\,\alpha_{\FC_{0},\CC_{0}}} &&
}.
$$
By construction, these maps satisfy the following compatibility with the action maps :
$$\forall g\in G_{\iota},\; g_{\WC,\FC_{\CC}}\circ\alpha_{\FC,\CC}=\alpha_{g\FC,g\CC}\circ g_{\WC,\FC}.$$
Now, let $\FC'\subseteq\o\FC$ be another $e$-facet in the closure of $\FC$. It is also
contained in $\ov\CC$, so we have a map $\alpha_{\FC',\CC}$ as above. By Corollary 
\ref{coro_comp_Heisenberg}, there is a unique
$R(\Ku_{\iota,\FC}\cap\Ku_{\iota,\FC'})$-linear map
$$ \beta_{\FC,\FC',\CC}:\, \WC_{\iota,\FC}\To{}\WC_{\iota,\FC'}$$
such that $\alpha_{\FC,\CC}=\alpha_{\FC',\CC}\circ \beta_{\FC,\FC',\CC}$.
It is a generator of the $R$-module
$\Hom_{R(\Ku_{\iota,\FC}\cap\Ku_{\iota,\FC'})}(\WC_{\iota,\FC},\WC_{\iota,\FC'})$
and,  by
Proposition \ref{Hom_HW_H}, it actually belongs to
$\Hom_{R(\K_{\iota,\FC}\cap\K_{\iota,\FC'})}(\WC_{\iota,\FC},\WC_{\iota,\FC'})$.
By uniqueness, we have
$\beta_{\FC',\FC'',\CC}\circ\beta_{\FC,\FC',\CC}=\beta_{\FC,\FC'',\CC}$ and,
by construction again, we have the following compatibility with the action maps :
$$\forall g\in G_{\iota},\; g_{\WC,\FC'}\circ\beta_{\FC,\FC',\CC}=\beta_{g\FC,g\FC',g\CC}\circ g_{\WC,\FC}.$$

\begin{lem}
  $\beta_{\FC,\FC',\CC}$ is independent of $\CC$. We will denote it simply by $\beta_{\FC,\FC'}$
\end{lem}
\begin{proof}
Recall that $G_{\iota,\FC}$ acts transitively on the set of chambers that contain
$\FC$. Moreover, by our choice of $e$, we have $G_{\iota,\FC}\subseteq G_{\iota,\FC'}$ and
$\K_{\iota,\FC}\subseteq \K_{\iota,\FC'}$, hence also $G_{\iota,\FC}\subseteq \K_{\iota,\FC}\cap\K_{\iota,\FC'}$.
Now, when $g\in G_{\iota,\FC}$, the last displayed property reads
$g_{\WC,\FC'}\circ\beta_{\FC,\FC',\CC}=\beta_{\FC,\FC',g\CC}\circ g_{\WC,\FC}$.
On the other hand, the $\K_{\iota,\FC}\cap\K_{\iota,\FC'}$-equivariance of
$\beta_{\FC,\FC',\CC}$ means that 
$g_{\WC,\FC'}\circ\beta_{\FC,\FC',\CC}=\beta_{\FC,\FC',\CC}\circ g_{\WC,\FC}$.
It follows that $\beta_{\FC,\FC',\CC}=\beta_{\FC,\FC',g\CC}$.
\end{proof}

\alin{Construction of a Heisenberg--Weil coefficient system under Assumption \ref{assumption}, step 3: morphisms}\label{sec:HW-underassumption-3}
Let now $\FC$ and $\FC'$ be arbitrary and
recall that $\Hom(\FC,\FC')=\{g\in K_{\iota,\FC'}G_{\iota},\, \o{g\FC}\supseteq \FC'\}$.
$$ \hbox{For } g\in\Hom(\FC,\FC'),\,\hbox{ set }\,\WC_{\iota,\FC,\FC'}(g):= \WC_{\iota,g^{-1}\FC',\FC'}(g)\circ \beta_{\FC,g^{-1}\FC'}.$$
This is consistent with our previous definition when $g\FC=\FC'$ since $\beta_{\FC,\FC}=\id$.

\begin{lem}
  We have
  $ \WC_{\iota,\FC',\FC''}(h)\circ\WC_{\iota,\FC,\FC'}(g)=\WC_{\iota,\FC,\FC''}(hg)$
  for all $g\in\Hom(\FC,\FC')$ and  $h\in\Hom(\FC',\FC'')$.
\end{lem}
\begin{proof}
  Choose decompositions  $g=kg_{\iota}$ and $h=lh_{\iota}$ with $k\in K_{\iota,\FC'}$,
  $l\in K_{\iota,\FC''}$ and $g_{\iota},h_{\iota}\in G_{\iota}$. We have
  \begin{eqnarray*}
  &&  \WC_{\iota,\FC,\FC'}(g) = \WC_{\iota,\FC',\FC'}(k)\circ
    (g_{\iota})_{\WC,g_{\iota}^{-1}\FC'}\circ\beta_{\FC,g_{\iota}^{-1}\FC'} \\
  &&\WC_{\iota,\FC',\FC''}(h) = \WC_{\iota,\FC'',\FC''}(l)\circ
    (h_{\iota})_{\WC,h_{\iota}^{-1}\FC''}\circ\beta_{\FC',h_{\iota}^{-1}\FC''}
  \end{eqnarray*}
Since $\beta_{\FC',h_{\iota}^{-1}\FC''}$ is $\K_{\iota,\FC'}\cap
K_{\iota,h_{\iota}^{-1}\FC''}$-equivariant and $\K_{\iota,\FC'}\subset K_{\iota,h_{\iota}^{-1}\FC''}$,
we have
$$\beta_{\FC',h_{\iota}^{-1}\FC''}\circ \WC_{\iota,\FC',\FC'}(k)
=\WC_{\iota,h_{\iota}^{-1}\FC'',h_{\iota}^{-1}\FC''}(k)\circ  \beta_{\FC',h_{\iota}^{-1}\FC''}.$$ 
On the other hand, we have already seen that
$$\beta_{\FC',h_{\iota}^{-1}\FC''}\circ (g_{\iota})_{\WC,g_{\iota}^{-1}\FC'}
=(g_{\iota})_{\WC,(h_{\iota}g_{\iota})^{-1}\FC''}\circ  \beta_{g_{\iota}^{-1}\FC',(h_{\iota}g_{\iota})^{-1}\FC''}.$$ 
From the properties of $\WC_{\iota}$ on $G_{\iota}$-orbits, we have
$$\WC_{\iota,\FC'',\FC''}(l)\circ(h_{\iota})_{\WC,h_{\iota}^{-1}\FC''}\circ
\WC_{\iota,h_{\iota}^{-1}\FC'',h_{\iota}^{-1}\FC''}(k)\circ(g_{\iota})_{\WC,(h_{\iota}g_{\iota})^{-1}\FC''}
=\WC_{\iota,(hg)^{-1}\FC'',\FC''}(hg).
$$
Finally, the transitivity of face maps implies
$$ \beta_{g_{\iota}^{-1}\FC',(h_{\iota}g_{\iota})^{-1}\FC''}\beta_{\FC,g_{\iota}^{-1}\FC'}
=\beta_{\FC,(h_{\iota}g_{\iota})^{-1}\FC''}=\beta_{\FC,(hg)^{-1}\FC''}.$$
The four last displayed equalities imply the lemma.
\end{proof}

\begin{pro} \label{thm_WH_coef_system-withassumption}
	Let $R$ be a commutative $\Rmintwo$-algebra, and assume Assumption \ref{assumption}. Then
	there exists a Heisenberg--Weil coefficient system $\WC_{\iota}$ in $\Coef_{R}(\BT'_{\iota,\bullet/e}/K_{\iota})$.
\end{pro}

\begin{proof}
	By Lemma \ref{sec:HW-underassumption-3} the maps $\FC\mapsto \WC_{\iota,\FC}$ and  $g\in\Hom(\FC,\FC')\mapsto
	\WC_{\iota,\FC,\FC'}(g)$ define a functor $[\BT'_{\iota,\bullet/e}/K_{\iota}]\To{}R-\Mod$.
	We have already seen that the action of $\Ku_{\iota,\FC}$ on $\WC_{\iota,\FC}$ through
	$\WC_{\iota,\FC,\FC}$ is a Heisenberg representation, and each face map
	$\beta_{\FC,\FC'}=\WC_{\iota,\FC,\FC'}(1)$ is a generator of
	$\Hom_{R(\Ku_{\iota,\FC}\cap\Ku_{\iota,\FC'})}(\WC_{\iota,\FC},\WC_{\iota,\FC'})$ by construction. 
\end{proof}

We now stop imposing Assumption \ref{assumption}.

\alin{Minimal factorization through a rational Levi subgroup}
Consider the centralizer $\MC:=C_{^{L}{\mathbf{G}}}(Z(C_{\hat{\mathbf{G}}}(\phi))^{\varphi(W_{F}),\circ})$ in
$^{L}{\mathbf{G}}$ of the maximal $W_{F}$-invariant central torus in $C_{\hat{\mathbf{G}}}(\phi)$. This
is a Levi subgroup of $^{L}{\mathbf{G}}$ in the sense of Borel \cite[\S 3.4]{BorelCorvallis},
so, after conjugating $\phi$, we may assume 
$\MC$ is of the form $^{L}{\mathbf{M}}$ for some (non-twisted) $F$-rational Levi subgroup ${\mathbf{M}}$ of
${\mathbf{G}}$. 
By construction, both $\phi$ and $\varphi$ factor through $\MC$, providing
us with a wild inertia parameter  $\phiM\in \Phi(P_{F},\mathbf{M})$ (since any relevant
extension $\varphi'$ for $G$ factors through $\MC$ and thus provides a relevant extension
for $M$). Note that, by definition,
$\hat{\mathbf{M}}=\MC^{\circ}$ is the centralizer in $\hat{\mathbf{G}}$ of
$(\hat{\mathbf{S}}_{\phi})^{W_{F},\circ}$.
On the other side, pick a $\iota : \mathbf{S}_{\phi}\injo\mathbf{G}$ in $I$ and consider the centralizer
$C_{\mathbf{G}}(\iota(\mathbf{S}_{\phi}^{\rm split}))$ of the image
of the maximal split subtorus of $\mathbf{S}_{\phi}$. This is a rational  Levi subgroup of
$\mathbf{G}$ in the $G$-conjugacy class of $\mathbf{M}$. So, after conjugating $\iota$, we
may assume it factors through $\mathbf{M}$, so that
$\mathbf{M}=C_{\mathbf{G}}(\iota(\mathbf{S}_{\phi}^{\rm split}))$. By construction of
$\MC$, we have 
$C_{\hat{\mathbf{G}}}(\phi)\subseteq \hat{\mathbf{M}}$, hence
$\mathbf{S}_{\phi}=\mathbf{S}_{\phiM}$, and $\iota$ determines an embedding $\iota_{M}$ in
$I_{\phi_{M}}$, for which we have ${\mathbf{M}}_{\iotaM}={\mathbf{G}}_{\iota}$.

By construction, Assumption \ref{assumption} is satisfied for $(\hat{\mathbf{M}}, \phi_M,I_M)$. 
So, Proposition \ref{thm_WH_coef_system-withassumption} applies and provides us with a
Heisenberg--Weil coefficient system related to these data. We need to compare the buildings at stake.

Let us take up the notations $\BT_{M}$ and $\BT_{\iota_{M}}$ of \ref{setup_par_ind}. These are
subsets of $\BT$, and we denote by $\BT'_{M}$ and $\BT'_{\iota_{M}}$ their images in $\BT'$. In our
setup here, we actually have $\BT'_{\iota}=\BT'_{\iota_{M}}$.  These subsets are stable under the
$e$-facet decomposition of $\BT'$, and we reserve the notation $\BT'_{M,\bullet/e}$, resp. $\BT'_{\iota_{M},\bullet/e}$,
for the corresponding posets. On the other hand, the reduced Bruhat--Tits building
$\BT'(\mathbf{M},F)$ of $\mathbf{M}$ also comes with its own $e$-facet decomposition, and its subset
$\BT'(\mathbf{M},F)_{\iota_{M}}$ is stable under this decomposition.
We denote by $\BT'(\mathbf{M},F)_{\bullet/e}$, resp., $\BT'(\mathbf{M},F)_{\iota_{M},\bullet/e}$,
the associated posets.
There is a canonical projection $\BT'_{M}\To{}\BT'(\mathbf{M},F)$ that takes any $e$-facet of
$\BT'_{M}$ into an $e$-facet of $\BT'(\mathbf{M},F)$. This projection thus induces a functor
$$\pi:\,[\BT'_{\iota,\bullet/e}/(M\cap G_{\iota})]=[\BT'_{\iota_{M},\bullet/e}/M_{\iota_{M}}]\To{}
[\BT'(\mathbf{M},F)_{\iota_{M},\bullet/e}/M_{\iota_{M}}].$$
Denote by $\K_{\iota_{M},\pi(\FC)}$,  $\Ku_{\iota_{M},\pi(\FC)}$, etc., the
objects analogous to $K_{\iota, \cF}$, $\Ku_{\iota, \cF}$, 
etc., constructed for $(\mathbf M, \phi_M, \iota_M \in I_M)$ and the facet $\pi(\FC)$ in place of $(\mathbf G,
\phi, \iota \in I)$ and the facet $\FC$.
Then we have an inclusion $M\cap \K_{\iota,\FC}\subset \K_{\iota_{M},\pi(\FC)}$ and an equality $M\cap \Ku_{\iota,\FC}=\Ku_{\iota_{M},\pi(\FC)}$. So our functor
$\pi$ can be upgraded to
$$\pi:\,[\BT'_{\iota,\bullet/e}/(M\cap K_{\iota})]\To{}
[\BT'(\mathbf{M},F)_{\iota_{M},\bullet/e}/K_{\iota_{M}}].$$
Now, the Heisenberg--Weil coefficient system $\cW_{\iota_M}$ provided by
Proposition \ref{thm_WH_coef_system-withassumption}, is 
a functor from
$[\BT'(\mathbf{M},F)_{\iota_{M},\bullet/e}/K_{\iota_{M}}]$ to $R-\Mod$, so we may and will compose it
with $\pi$ to get a coefficient system on $[\BT'_{\iota,\bullet/e}/(M\cap K_{\iota})]$ that we
still denote by $\WC_{\iota_{M}}$. In the next paragraphs, we start from $\WC_{\iota_{M}}$ to construct
a Heisenberg--Weil coefficient system $\WC_{\iota}$ on $[\BT'_{\iota,\bullet/e}/K_{\iota}]$.

We set $K^{?}_{\iota_{M},\FC}:=M\cap K^{?}_{\iota,\FC}$ for $?=\dag,+,\emptyset$. We have  $\Ku_{\iota_{M},\FC}=\Ku_{\iota_{M},\pi(\FC)}$,
$\Kp_{\iota_{M},\FC}=\Kp_{\iota_{M},\pi(\FC)}$, and
$K_{\iota_{M},\pi(\FC)}^{\circ}\subset\K_{\iota_{M},\FC}\subset \K_{\iota_{M},\pi(\FC)}$.

\alin{Construction of a Heisenberg--Weil coefficient system 1: modules}
We choose a parabolic subgroup $\mathbf P$ of $\mathbf G$ whose Levi subgroup is $\mathbf
M$ and we denote its unipotent radical by $\mathbf U$ and follow the notation from
\ref{compat_parab_induc} and \ref{setup_par_ind}, in particular $\bar{\mathbf P}$ denotes
the opposite parabolic subgroup of $\mathbf P$ with respect to $\mathbf M$ with unipotent
radical $\bar{\mathbf U}$, and $\Uu_{\iota,\cF}=\Ku_{\iota,\cF} \cap U$,
$\Up_{\iota,\cF}=\Kp_{\iota,\cF} \cap U$, $\bUu_{\iota,\cF}=\Ku_{\iota,\cF} \cap \bar U$,
and $\bUp_{\iota,\cF}=\Kp_{\iota,\cF} \cap \bar U$.
In the setting here, we have
$\K_{\iota,\cF}=\Uu_{\iota,\cF}\K_{\iota_M,\cF}\bUu_{\iota,\cF}$. Indeed, this follows from
$\Ku_{\iota,\FC}=\Uu_{\iota,\cF}\Ku_{\iota_M,\cF}\bUu_{\iota,\cF}$ and $G_{\iota,\FC}=M_{\iota_{M},\FC}$.
Recall from Lemma \ref{lemma_ind_parab} \ref{item:Iwahori}) that $\check\phi_{\iota,\cF}^+$ is the character on $\Kp_{\iota, F}=\Up_{\iota,\cF}\Kp_{\iota_M,\cF}\bUp_{\iota,\cF}$ whose restriction to $\Up_{\iota,\cF}\bUp_{\iota,\cF}$ is trivial and whose restriction to $\Kp_{\iota_M, \cF}$ is $\check\phi_{\iota_M,\cF}^+$.

For every $\cF \in \BT'_{\iota,\bullet/e}$, we extend $\cW_{\iota_M, \cF}$ to a representation of $\Uu_{\iota,\cF}\K_{\iota_M,\cF}\bUp_{\iota,\cF}$ by requiring $\Uu_{\iota,\cF}$ and $\bUp_{\iota,\cF}$ to act trivially. Then we set 
\[ \cW_{\iota, \cF}:= \cInd{\Uu_{\iota,\cF}\K_{\iota_M,\cF}\bUp_{\iota,\cF}}{\K_{\iota,\cF}}{\cW_{\iota_M, \cF}}, \]
which comes with the two $\Uu_{\iota,\cF}\K_{\iota_M,\cF}\bUp_{\iota,\cF}$-equivariant maps corresponding to the identity via Frobenius reciprocity:
\[ i_{\FC}: \cW_{\iota_M, \cF} \hookrightarrow \cW_{\iota, \cF} \quad \text{and} \quad pr_{\FC}: \cW_{\iota, \cF} \twoheadrightarrow \cW_{\iota_M, \cF} .\]
Note that $pr_{\cF} \circ i_{\cF}$ is the identity and $i_{\cF} \circ pr_{\cF}$ is the projection $e_{\Uu_{\iota, \cF}}$ on the submodule on which $\Uu_{\iota,\cF }$ acts trivially.

\begin{lem} \label{Lemma-Wiota-Heisenberg}
	$\cW_{\iota,\cF}|_{\Ku_{\iota,\FC}}$ is a Heisenberg representation for the triple
	$(\Ku_{\iota_{M},\FC},\Kp_{\iota_{M},\FC},\check\phi^{+}_{\iota_{M},\FC})$
\end{lem}
\begin{proof}
	We have 
	$\WC_{\iota,\FC}|_{\Ku_{\iota,\cF}}=
	\cInd{\Uu_{\iota,\cF}\Ku_{\iota_M,\cF}\bUp_{\iota,\cF}}{\Ku_{\iota,\cF}}
	{\cW_{\iota_M,\cF}|_{\Ku_{\iota_{M},\FC}}}$ with 
	$\cW_{\iota_M,\cF}|_{\Ku_{\iota_{M},\FC}}$  a Heisenberg representation for the triple
	$(\Ku_{\iota_{M},\FC},\Kp_{\iota_{M},\FC},\check\phi^{+}_{\iota_{M},\FC})$.
	Since the restriction of $\check\phi_{\iota,\cF}^+$  to $\Up_{\iota,\cF}$ and
	$\bUp_{\iota,\cF}$ is trivial and to $\Kp_{\iota_M, \cF}$ is $\check\phi_{\iota_M,\cF}^+$,
	we can prove that $\WC_{\iota,\FC}|_{\Ku_{\iota,x}}$  is a Heisenberg representation for the triple
	$(\Ku_{\iota,\FC},\Kp_{\iota,\FC},\check\phi^{+}_{\iota,\FC})$ 
	as in Lemma \ref{intertwining_Heisenberg} iii)(a). Namely, up to twisting by an invertible
	$R$-module, we may assume that
	$\cW_{\iota_M, \cF}|_{\Ku_{\iota_{M},\FC}}$ is induced from some Lagrangian $W\subset\Ku_{\iotaM,x}/\Kp_{\iotaM,x}$, in
	which case, $\cW_{\iota, \cF}|_{\Ku_{\iota,\FC}}$ is induced from the Lagrangian
	$\Uu_{\iota,x}/\Up_{\iota,x}\oplus W \subset \Ku_{\iota,x}/\Kp_{\iota,x}$.
\end{proof}

\alin{Construction of a Heisenberg--Weil coefficient system 2: action and face maps}

Let $\cF, \cF' \in \BT'_{\iota,\bullet/e}$ and $g \in G_{\iota, \cF, \cF'} \subset M$,
then $g\Uu_{\iota,\cF}g^{-1}=\Uu_{\iota,g\cF}$, $g\K_{\iota_M,\cF}g^{-1}=\K_{\iota_M,g\cF}$,  $g\bUp_{\iota,\cF}g^{-1}=\bUp_{\iota,g\cF}$ and $g\K_{\iota,\cF}g^{-1}=\K_{\iota,g\cF}$. Hence $g_{\cW_{\iota_M},\cF}$ induces a unique $\K_{\iota, \cF}$-equivariant isomorphism $g_{\cW_{\iota}, \cF}: \cW_{\iota,\cF} \simto g^*\cW_{\iota, g\cF}$, which agrees with the action of $g$ on $ \cW_{\iota,\cF}$ if $g \in G_{\iota, \cF}$.

\begin{lem}\label{Lemma-beta-without-assumption}
	For each pair of e-facets $\cF, \cF' \in \BT'_{\iota,\bullet/e}$ with $\o\cF\supseteq \cF'$, 
	there exists a unique generator $\beta_{\cW_\iota,\cF,\cF'} \in
	\Hom_{R(\Ku_{\iota, \cF} \cap \Ku_{\iota,\cF'})}(\cW_{\iota, \cF},\cW_{\iota, \cF'})$ such that $pr_{\cF'} \circ \beta_{\cW_\iota, \cF, \cF'} \circ i_\cF  = \beta_{\cW_{\iota_M}, \cF, \cF'}$.
\end{lem}

\begin{proof}
	\inisub
	\begin{subequations}
Let $\cF_0$ and $\cF_0'$ be two $e$-facets whose closures contain $\cF$ and $\cF'$, respectively, and such that
\[\Uu_{\iota,\cF_0}=\Uu_{\iota,\cF} \, , \,\, \bUu_{\iota,\cF_0}=\bUp_{\iota,\cF}=\bUp_{\iota,\cF_0} \, , \, \,
\Uu_{\iota,\cF_0'}=\Uu_{\iota,\cF'} \, , \,\, \bUu_{\iota,\cF_0'}=\bUp_{\iota,\cF'}=\bUp_{\iota,\cF_0'} \, ,
\]
and
\[\K_{\iota_M, \cF_0}=\K_{\iota_M, \cF} \, , \,\, \Ku_{\iota_M, \cF_0}=\Ku_{\iota_M, \cF} \, , \,\, \K_{\iota_M, \cF_0'}=\K_{\iota_M, \cF'} \, , \, \,\Ku_{\iota_M, \cF_0'}=\Ku_{\iota_M, \cF'}. 	\] 
This can be achieved by letting $\cF_0$ be the $e$-facet containing $x+ \eps \lambda$
for $x$ a point
in $\cF$, $\eps>0$
sufficiently small and $\lambda$ an appropriate cocharacter of 
the center of $\mathbf M$ (recall 
that the maximal split torus of the center of $\mathbf M$ is contained in $\mathbf M_\iota$).
The $e$-facet $\cF_0'$ can be obtained analogously.           
Then we have
\begin{eqnarray}
&&  \Ku_{\iota,\cF_0}=\Uu_{\iota,\cF_{0}}\Ku_{\iota_M, \cF_{0}}\bUp_{\iota, \cF_{0}}=\Uu_{\iota,\cF}\Ku_{\iota_M, \cF}\bUp_{\iota, \cF}
   \hbox{ and } \cW_{\iota, \cF_0}= \cW_{\iota_M, \cF_0} \simeq \cW_{\iota_M, \cF} \label{choice_F0}
  \\
&& \Ku_{\iota,\cF'_0}=\Uu_{\iota,\cF'_{0}}\Ku_{\iota_M, \cF'_{0}}\bUp_{\iota,
  \cF'_{0}}=\Uu_{\iota,\cF'}\Ku_{\iota_M, \cF'}\bUp_{\iota, \cF'} 
  \hbox{ and } \cW_{\iota, \cF'_0}= \cW_{\iota_M, \cF'_0} \simeq \cW_{\iota_M, \cF'}  \label{choice_F'0}
\end{eqnarray}

Applying Lemma \ref{composition_Heisenberg} to points $(x,x',x'') \in \cF_0 \times \cF \times \cF'$ and to points $(x,x',x'') \in \cF_0 \times \cF' \times \cF'_0$ yields an isomorphism arising from the composition:
		\begin{eqnarray}
		&&	\Hom_{R(\Ku_{\iota, \cF_0} \cap \Ku_{\iota, \cF})}(\cW_{\iota, \cF_0}, \cW_{\iota,\cF}) \otimes 
			\Hom_{R(\Ku_{\iota, \cF} \cap \Ku_{\iota, \cF'})}(\cW_{\iota, \cF},
                   \cW_{\iota,\cF'})\otimes \label{proof-HW-general-eqn2}
			\\ 
		&&\otimes	\Hom_{R(\Ku_{\iota, \cF'} \cap \Ku_{\iota, \cF'_0})}(\cW_{\iota, \cF'}, \cW_{\iota,\cF'_0})
			\simto
			\Hom_{R(\Ku_{\iota, \cF_0} \cap \Ku_{\iota, \cF'_0})}(\cW_{\iota, \cF_0}, \cW_{\iota,\cF'_0})	\notag
		\end{eqnarray}
		and likewise
		\begin{eqnarray}
			& & \qquad \Hom_{R(\Ku_{\iota_M, \cF_0} \cap \Ku_{\iota_M, \cF})}(\cW_{\iota_M, \cF_0}, \cW_{\iota_M,\cF}) \otimes 
			\Hom_{R(\Ku_{\iota_M, \cF} \cap \Ku_{\iota_M, \cF'})}(\cW_{\iota_M, \cF}, \cW_{\iota_M,\cF'}) \otimes\notag \\
			& & \qquad  \otimes
			\Hom_{R(\Ku_{\iota_M, \cF'} \cap \Ku_{\iota_M, \cF'_0})}(\cW_{\iota_M, \cF'}, \cW_{\iota_M,\cF'_0})		
				\simto
			\Hom_{R(\Ku_{\iota_M, \cF_0} \cap \Ku_{\iota_M, \cF'_0})}(\cW_{\iota_M, \cF_0}, \cW_{\iota_M,\cF'_0}).	
			\label{proof-HW-general-eqn3}
		\end{eqnarray}	
Since $\cW_{\iota_M}$ is a Heisenberg--Weil coefficient system, all the Hom-space on the left hand side of \eqref{proof-HW-general-eqn3}  are free $R$-modules of rank 1, and hence so is $	\Hom_{R(\Ku_{\iota_M, \cF_0} \cap \Ku_{\iota_M, \cF'_0})}(\cW_{\iota_M, \cF_0}, \cW_{\iota_M,\cF'_0})$. We will show that the same applies to all the Hom-spaces in \eqref{proof-HW-general-eqn2}.
		
First, we have 
$\Hom_{R(\Ku_{\iota, \cF_0} \cap \Ku_{\iota, \cF'_0})}(\cW_{\iota, \cF_0}, \cW_{\iota,\cF'_0})= 	\Hom_{R(\Ku_{\iota_M, \cF_0} \cap \Ku_{\iota_M, \cF'_0})}(\cW_{\iota_M, \cF_0}, \cW_{\iota_M,\cF'_0}),$
		which we just saw is a free $R$-module of rank 1.
		
 Next, recalling that
 $(\cW_{\iota,
   \cF})|_{\Ku_{\iota,\FC}}=\cInd{\Uu_{\iota,\cF}\Ku_{\iota_M,\cF}\bUp_{\iota,\cF}}{\Ku_{\iota,\cF}}{\cW_{\iota_M,
     \cF}}$ we obtain from (\ref{choice_F0}) that
 \begin{eqnarray*}
   \Hom_{R(\Ku_{\iota, \cF_0} \cap \Ku_{\iota, \cF})}(\cW_{\iota, \cF_0}, \cW_{\iota,\cF})
   &\simeq& 
   \Hom_{R(\Uu_{\iota,\cF}\Ku_{\iota_M, \cF}\bUp_{\iota, \cF})}(\cW_{\iota_M,\cF}, \cW_{\iota,\cF}) 
   \\ &\simeq&  \Hom_{R\Ku_{\iota,\cF}}(\cW_{\iota, \cF},\cW_{\iota, \cF}), 
 \end{eqnarray*}
which is a free $R$-module of rank 1 by 
Lemmas	\ref{Lemma-Wiota-Heisenberg} and \ref{lem_Heisenberg}.
Note that a generator of
		$	\Hom_{R(\Ku_{\iota, \cF_0} \cap \Ku_{\iota, \cF})}(\cW_{\iota,
			\cF_0}, \cW_{\iota,\cF})$ is given by $i_F \circ
		\beta_{\cW_{\iota_M},\cF_0, \cF}$, where we 
		view the isomorphism\footnote{Note that if we choose $\cF_0$ such that $\cF$ and $\cF_0$ map to the same facet in $\BT'_{\iota_M, \bullet/e}$, then $\beta_{\cW_{\iota_M},\cF_0, \cF}$ is just the identity.}
		$ \beta_{\cW_{\iota_M},\cF_0, \cF} \in \Hom_{R(\Ku_{\iota_M, \cF_0} \cap \Ku_{\iota_M, \cF})}(\cW_{\iota_M, \cF_0}, \cW_{\iota_M,\cF}) = \Hom_{R(\Ku_{\iota_M, \cF})}(\cW_{\iota_M, \cF_0}, \cW_{\iota_M,\cF}) $
		also as an isomorphism in $\Hom_{R(\Uu_{\iota,\cF}\Ku_{\iota_M, \cF}\bUp_{\iota,\cF})}(\cW_{\iota_M, \cF_0}, \cW_{\iota_M,\cF})$.
		Similarly, we have 
                \begin{eqnarray*}
                  \Hom_{R(\Ku_{\iota, \cF'} \cap \Ku_{\iota, \cF'_0})}(\cW_{\iota, \cF'}, \cW_{\iota,\cF'_0})
                  &\simeq& 	 
                  \Hom_{R(\Uu_{\iota,\cF'}\Ku_{\iota_M, \cF'}\bUp_{\iota, \cF'})}(\cW_{\iota,\cF'}, \cW_{\iota_M,\cF'},) 
           \\       &\simeq&  \Hom_{R(\K_{\iota,\cF'})}(\cW_{\iota, \cF'},\cW_{\iota, \cF'}), 
                \end{eqnarray*}
		which is a free $R$-module of rank 1, and  a generator of $\Hom_{R(\Ku_{\iota, \cF'}
                  \cap \Ku_{\iota, \cF'_0})}(\cW_{\iota, \cF'}, \cW_{\iota,\cF'_0})$ is given by $\beta_{\cW_{\iota_M},\cF', \cF_0'} \circ pr_{\cF'}$.
		
		So we have shown that all the Hom-spaces in \eqref{proof-HW-general-eqn2} other than
                the invertible module
                $\Hom_{R(\Ku_{\iota, \cF} \cap \Ku_{\iota, \cF'})}(\cW_{\iota, \cF},
                \cW_{\iota,\cF'})$ are free $R$-modules of rank 1. It follows that
                the latter 
                is also free of rank 1 and,
		 combining the isomorphisms \eqref{proof-HW-general-eqn2} and
                 \eqref{proof-HW-general-eqn3},
                 that there exists a unique generator $\beta_{\cW_\iota, \cF, \cF'}$ of $ \Hom_{R(\Ku_{\iota, \cF} \cap \Ku_{\iota,\cF'})}(\cW_{\iota, \cF},\cW_{\iota, \cF'})$ such that 
		\[\beta_{\cW_{\iota_M}, \cF', \cF'_0} \circ pr_{\cF'} \circ \beta_{\cW_\iota, \cF, \cF'} \circ i_\cF \circ  \beta_{\cW_{\iota_M}, \cF_0, \cF} =\beta_{\cW_{\iota_M}, \cF', \cF'_0} \circ \beta_{\cW_{\iota_M}, \cF, \cF'} \circ \beta_{\cW_{\iota_M}, \cF_0, \cF} .\]
		Since $\beta_{\cW_{\iota_M}, \cF', \cF'_0}$ and $\beta_{\cW_{\iota_M}, \cF_0, \cF}$ are isomorphisms, this implies that there exists a unique generator $\beta_{\cW_\iota, \cF, \cF'}$ of $ \Hom_{R(\Ku_{\iota, \cF} \cap \Ku_{\iota,\cF'})}(\cW_{\iota, \cF},\cW_{\iota, \cF'})$ such that 
		$ pr_{\cF'} \circ \beta_{\cW_\iota, \cF, \cF'} \circ i_\cF  = \beta_{\cW_{\iota_M}, \cF, \cF'}  .$		
	\end{subequations}
\end{proof}

For each pair of e-facets $\cF, \cF' \in \BT'_{\iota,\bullet/e}$ with $\o\cF\supseteq \cF'$, set $\beta_{\cW_{\iota},\cF, \cF'}$ to be as in Lemma \ref{Lemma-beta-without-assumption}.

\alin{Construction of a Heisenberg--Weil coefficient system 3: morphisms} \label{sec:HW-withoutassumption-3}
We first observe that the face maps have the following desired properties that will allow us to use them to define a desired Heisenberg--Weil coefficient system.
\begin{lem}
	\begin{enumerate} Let $\cF, \cF' \in \BT'_{\iota, \bullet/e}$ with  $\overline \cF\supseteq \cF'$ . 
		\item \label{facemap-action-properties-i}
		If $\cF'' \in \BT'_{\iota, \bullet/e}$ with $\overline \cF' \supseteq \cF''$, then we have 
		$ \beta_{\cW_\iota, \cF', \cF''} \circ \beta_{\cW_\iota, \cF, \cF'} = \beta_{\cW_\iota, \cF, \cF''}$.
		\item \label{facemap-action-properties-ii} If $g \in G_\iota$, then 
		$ g_{\cW_\iota, \cF'} \circ \beta_{\cW_\iota, \cF, \cF'} =  \beta_{\cW_\iota, g\cF, g\cF'} \circ g_{\cW_\iota, \cF} $.
		\item \label{facemap-action-properties-iii} The map $\beta_{\cW_\iota, \cF, \cF'}$ is $\K_{\iota,\cF} \cap \K_{\iota, \cF'}$-equivariant.
	\end{enumerate}
\end{lem}
\begin{proof}
	\ref{facemap-action-properties-i}) By Lemma \ref{Lemma-beta-without-assumption} and since $\cW_{\iota_M}$ is a Heisenberg--Weil coefficient system, it suffices to show that 
	$  pr_{\cF''} \circ \beta_{\cW_\iota, \cF', \cF''} \circ i_{\cF'}  \circ pr_{\cF'} \circ \beta_{\cW_\iota, \cF, \cF'} \circ i_\cF  =  pr_{\cF''} \circ \beta_{\cW_\iota, \cF', \cF''} \circ \beta_{\cW_\iota, \cF, \cF'}  \circ i_\cF  .$
	Using that $i_{\cF'} \circ pr_{\cF'}=e_{\Uu_{\iota, \cF'}}$,  $pr_{\cF''}=pr_{\cF''} \circ i_{\cF''} \circ pr_{\cF''}=pr_{\cF''} \circ e_{\Uu_{\iota, \cF''}}$,  $\Uu_{\cF'} \subseteq \Uu_{\cF''}$, and that $ \beta_{\cW_\iota, \cF', \cF''}$ is $\Uu_{\cF'} \cap \Uu_{\cF''}= \Uu_{\cF'}$-equivariant, we have 
	\begin{eqnarray*}
		pr_{\cF''} \circ \beta_{\cW_\iota, \cF', \cF''} \circ i_{\cF'}  \circ pr_{\cF'} \circ \beta_{\cW_\iota, \cF, \cF'} \circ i_\cF
		&=& pr_{\cF''} \circ e_{\Uu_{\iota,\cF''}} \circ \beta_{\cW_\iota, \cF', \cF''} \circ e_{\Uu_{\iota,\cF'}} \circ \beta_{\cW_\iota, \cF, \cF'} \circ i_\cF  \\
		&=& pr_{\cF''} \circ e_{\Uu_{\iota,\cF''}} e_{\Uu_{\iota,\cF'}} \circ \beta_{\cW_\iota, \cF', \cF''} \circ \beta_{\cW_\iota, \cF, \cF'} \circ i_\cF  \\
		&=& pr_{\cF''} \circ e_{\Uu_{\iota,\cF''}} \circ \beta_{\cW_\iota, \cF', \cF''} \circ \beta_{\cW_\iota, \cF, \cF'} \circ i_\cF  \\
		&=& pr_{\cF''} \circ \beta_{\cW_\iota, \cF', \cF''}\circ \beta_{\cW_\iota, \cF, \cF'} \circ i_\cF  .
	\end{eqnarray*}
	
	\ref{facemap-action-properties-ii}) We have 
	\begin{eqnarray*}
		pr_{\cF'} \circ g_{\cW_\iota, \cF'} \circ \beta_{\cW_\iota, \cF, \cF'} \circ g_{\cW_\iota, \cF}^{-1} \circ i_{g\cF} 
		& = & 	  g_{\cW_{\iota_M}, \cF'} \circ pr_{\cF'} \circ \beta_{\cW_\iota, \cF, \cF'} \circ i_{g\cF} \circ g_{\cW_{\iota_M}, \cF}^{-1} \\
		&= & g_{\cW_{\iota_M}, \cF'} \beta_{\cW_{\iota_M}, \cF, \cF'}  \circ g_{\cW_{\iota_M}, \cF}^{-1} 
		= 	\beta_{\cW_{\iota_M}, g\cF, g\cF'}  ,
	\end{eqnarray*}
	from which the claim follows by the definition of $ \beta_{\cW_\iota, g\cF, g\cF'} $.
	
	\ref{facemap-action-properties-iii})
	Since $\K_{\iota,\cF} \cap \K_{\iota, \cF'}=(\K_{\iota_M,\cF} \cap \K_{\iota_M, \cF'})(\Ku_{\iota,\cF} \cap \Ku_{\iota, \cF'})$, it suffices to show that $\beta_{\cW_\iota, \cF, \cF'} $ is $\K_{\iota_M,\cF} \cap \K_{\iota_M, \cF'}$-equivariant. Let $k \in \K_{\iota_M,\cF} \cap \K_{\iota_M, \cF'}$, then
	\begin{eqnarray*}
		pr_{\cF'} \circ k \beta_{\cW_\iota, \cF, \cF'} k^{-1} \circ i_{\cF} 
		= 	  k \circ pr_{\cF'} \circ \beta_{\cW_\iota, \cF, \cF'} \circ i_{\cF} \circ k^{-1}
		=  k \beta_{\cW_{\iota_M}, \cF, \cF'} k^{-1}= 	\beta_{\cW_{\iota_M}, \cF, \cF'}  ,
	\end{eqnarray*}
	from which the claim follows by the definition of $ \beta_{\cW_\iota, \cF, \cF'} $.	
\end{proof}

Analogous to \ref{sec:HW-underassumption-3} we can complete the definition of the functor $[\BT'_{\iota, \bullet/e}/\K_\iota] \To{}R-\Mod$ that sends $\cF$ to $\cW_{\iota, \cF}$ as follows. 
For  $\cF, \cF' \in \BT'_{\iota, \bullet/e}$, the functor sends an element $g\in\Hom(\FC,\FC')=\{g \in \K_{\iota, \cF' }G_\iota, \overline{g\cF} \supseteq \cF' \}$ to \[\WC_{\iota,\FC,\FC'}(g):= g_{\cW, g^{-1}\cF'}\circ \beta_{\FC,g^{-1}\FC'}.\]

\begin{cor}\label{cor:HW-withoutassumption-3}
	We have
	$ \WC_{\iota,\FC',\FC''}(h)\circ\WC_{\iota,\FC,\FC'}(g)=\WC_{\iota,\FC,\FC''}(hg)$
	for all $g\in\Hom(\FC,\FC')$ and  $h\in\Hom(\FC',\FC'')$.
\end{cor}
\begin{proof}
	The proof is the same as the proof of Lemma \ref{sec:HW-underassumption-3} using Lemma \ref{sec:HW-withoutassumption-3}.
\end{proof}

\begin{proof}[Proof of Theorem \ref{thm_WH_coef_system-general}]
	 By Corollary \ref{cor:HW-withoutassumption-3} the maps $\FC\mapsto \WC_{\iota,\FC}$ and  $g\in\Hom(\FC,\FC')\mapsto
	\WC_{\iota,\FC,\FC'}(g)$ define a functor $[\BT'_{\iota,\bullet/e}/K_{\iota}]\To{}R-\Mod$. By Lemma \ref{Lemma-Wiota-Heisenberg} the action of $\Ku_{\iota,\FC}$ on $\WC_{\iota,\FC}$ through
	$\WC_{\iota,\FC,\FC}$ is a Heisenberg representation, and by construction, i.e., Lemma \ref{Lemma-beta-without-assumption}, each face map
	$\beta_{\FC,\FC'}=\WC_{\iota,\FC,\FC'}(1)$ is a generator of
	$\Hom_{R(\Ku_{\iota,\FC}\cap\Ku_{\iota,\FC'})}(\WC_{\iota,\FC},\WC_{\iota,\FC'})$.
\end{proof}

The following proposition is not needed for the construction of the equivalence itself, but is a
result of independent interest and allows us to deduce Corollary \ref{cor-choice-of-WH} and
Corollary \ref{cor-Hecke-alg-isom}.
\begin{prop} \label{prop-W-is-twisted-HW}
	$\cW_{\iota, \cF}$ is a twisted $R\K_{\iota, \cF}$-Heisenberg--Weil
	representation whose restriction to $\Kp_{\iota,\cF}$ is  $\check\phi_{\iota,\cF}^+$-isotypic.
\end{prop}
\begin{proof}
	Let $\cW'_{\iota, \cF}$ denote the
	twisted $R\K_{\iota, \cF}$-Heisenberg--Weil 
	representation attached to the Heisenberg representation $\cW_{\iota, \cF}|_{\Ku_{\iota,\cF}}$, we need to show that 	$\cW_{\iota, \cF} \simeq 	\cW'_{\iota, \cF}$.
	Let $\cC$  be an e-facet of maximal dimension whose closure contains $\cF$ and such
	that
     $\Uu_{\iota,\cF}\K_{\iota_M,\cF}\bUp_{\iota,\cF}=\Uu_{\iota,\cC}\K_{\iota_M,\cC}\bUp_{\iota,\cC}$. Then
	by the above observation we have $\cW_{\iota_M, \cF}\simeq \cW_{\iota, \cC}$ viewed as
	representations of $\Uu_{\iota,\cF}\K_{\iota_M,\cF}\bUp_{\iota,\cF}=\K_{\iota, \cC}$ via
	appropriate inflation when needed. Combined with the restriction of $i_{\cF}$ to
	$\Ku_{\iota, \cF}$, we obtain a $\Ku_{\iota, \cF} \cap \Ku_{\iota, \cC}$-equivariant
	homomorphism $\cW_{\iota_M, \cF}\simeq \cW_{\iota, \cC} \ra \cW'_{\iota, \cF}$, which by
	Proposition \ref{Hom_HW_H} is also $\K_{\iota, \cF} \cap \K_{\iota,
		\cC}$-equivariant. Since $\K_{\iota, \cF} \cap \K_{\iota, \cC}=\K_{\iota,
		\cC}=\Uu_{\iota,\cF}\K_{\iota_M,\cF}\bUp_{\iota,\cF}$, Frobenius reciprocity provides us
	with a non-trivial $\K_{\iota, \cF}$-morphism $\cW_{\iota, \cF} \ra \cW'_{\iota,
		\cF}$, which is an isomorphism of $R(\Ku_{\iota, \cF})$-modules by Lemma \ref{intertwining_Heisenberg}.\ref{intertwining_Heisenberg-iii}), hence an isomorphism of $R(\K_{\iota, \cF})$-modules.
\end{proof}

\begin{coro}\label{cor-choice-of-WH}
	Given a point $x \in\BT'_{\iota}$ and any twisted $R\K_{\iota, x}$-Heisenberg--Weil representation $\kappa_{\iota,x}$ whose restriction to $\Kp_{\iota,x}$ is $\check{\phi}^+_{\iota, x}$-isotypic, there exists a Heisenberg--Weil coefficient system $\cW_\iota$ on $\BT'_{\iota, \bullet/e}$ such that $\kappa_{\iota,x} \simeq \cW_{\iota,\cF}$, where $\cF$ denotes the $e$-facet that contains $x$.	
\end{coro}
\begin{proof}
	Since by Definition \ref{Def-HeisWeil} a twisted $R\K_{\iota, x}$-Heisenberg--Weil representation is determined by the underlying Heisenberg representation of $\Ku_{\iota,x}$, it suffices by Proposition \ref{prop-W-is-twisted-HW} to prove that any Heisenberg representation of $\Ku_{\iota,x}$ whose restriction to $\Kp_{\iota,x}$ is $\check{\phi}^+_{\iota, x}$-isotypic can occur as the restriction of $\cW_{\iota,\cF}$ for some  Heisenberg--Weil coefficient system $\cW_\iota$. 
	Let $\cW'_\iota$ be some Heisenberg--Weil coefficient system
        as constructed above where we choose $\cC_0$ and $\sS$ in
        \ref{basic-choices} so that $\sS$ contains the image of $\cF$
        in $\BT'_{\iota_M, \bullet/e}$.
        Then, by \ref{section-different-heisenberg}, we have $\kappa_{\iota,x}|_{\Ku_{\iota,x}} \simeq \cW'_{\iota,\cF}|_{\Ku_{\iota,x}} \otimes_R R'$ where $R'$ denotes the invertible $R$-module $\Hom_{\Ku_{\iota,x}}(\cW'_{\iota,\cF}|_{\Ku_{\iota,x}}, \kappa_{\iota,x}|_{\Ku_{\iota,x}})$. Replacing all the twisted Heisenberg--Weil representations in the basic choices in \ref{basic-choices} by their $R$-tensor product with $R'$, we obtain from the above construction a new Heisenberg--Weil coefficient system $\cW_\iota$ with the property that $\kappa_{\iota,x}|_{\Ku_{\iota,x}} \simeq \cW'_{\iota,\cF}|_{\Ku_{\iota,x}} \otimes_R R' \simeq \cW_{\iota, \cF}|_{\Ku_{\iota,x}}$.
\end{proof}

\alin{Uniqueness -- proof of Proposition \ref{uniqueness_WH_coef-system}} \label{uniqueness_WH_coef}
Let $\WC_{\iota}$ and $\WC'_{\iota}$ be two Heisenberg--Weil
coefficient systems as in Definition \ref{def_WH_coef_syst}.
For any $e$-facet $\FC\in\BT'_{\iota,\bullet/e}$, the $R$-modules $\WC_{\iota,\FC}$ and
$\WC_{\iota,\FC'}$ are both equipped with an action of $\Ku_{\iota,\FC}$ that turns them into
Heisenberg representations for $(\Ku_{\iota,\FC},\Kp_{\iota,\FC},\check\phi_{\iota,\FC}^{+})$.
By Lemma \ref{lem_Heisenberg}, the $R$-module
$H_{\FC}:=\Hom_{R\Ku_{\iota,\FC}}(\WC_{\iota,\FC},\WC'_{\iota,\FC})$ is invertible. We equip the collection $(H_{\FC})_{\FC\in\BT'_{\iota,\bullet/e}}$ with the following structure of
a smooth $G_{\iota}$-equivariant coefficient system whose transition maps are all isomorphisms.
\begin{itemize}
\item If $\FC'$ is an $e$-facet contained in the closure $\o\FC$, then from part \ref{rem_WH_coef_system-i}) of Proposition \ref{rem_WH_coef_syst}, we get an $R$-linear map
  $H_{\FC}\To{\gamma_{\FC,\FC'}}H_{\FC'}$, $\alpha_{\FC}\mapsto \alpha_{\FC'}$,
  characterized by the property that
  $\alpha_{\FC'}\circ\beta_{\WC_{\iota},\FC,\FC'}=\beta_{\WC'_{\iota,\FC,\FC'}}\circ
  \alpha_{\FC}$. This characterization implies the transitivity
  $\gamma_{\FC'',\FC'}\circ\gamma_{\FC',\FC}=\gamma_{\FC'',\FC}$ for $\cF'' \subset \overline{\cF'}$, due to the same
  transitivity properties for the maps
  $\beta_{\WC_{\iota},-}$ and $\beta_{\WC'_{\iota},-}$. 

  Conversely, from point \ref{rem_WH_coef_system-ii}) of Proposition \ref{rem_WH_coef_syst}, we get an
  $R$-linear map $H_{\FC'}\To{}H_{\FC}$, $\alpha_{\FC'}\mapsto\alpha_{\FC}$ characterized
  by the same equality as above. In particular this map is an inverse of $\gamma_{\FC,\FC'}$. 

\item 
  If $g\in G_{\iota}$, we get an isomorphism of $R$-modules $g_{\FC}:H_{\FC}\simto H_{g\FC}$
  by sending $\alpha_{\FC}$ to 
  $g_{\WC'_{\iota},\FC}\circ \alpha_{\FC}\circ (g_{\WC_{\iota},\FC})^{-1}$. The equalities
  $h_{g\FC}\circ g_{\FC}=(hg)_{\FC}$ for any $h,g$ and $\FC$ follow from the same type of
  equalities for $\WC_{\iota}$ and $\WC'_{\iota}$. Compatibility with the face maps
  $\gamma_{\FC,\FC'}$ again follows from the same compatibilities for $\WC_{\iota}$ and
  $\WC'_{\iota}$.
  Finally, since $\alpha_{\FC}$ is $\Ku_{\iota,\FC}$-equivariant,  we have $g_{\FC}=\id$ whenever
  $g\in G_{\iota,\FC,0+}$. 
\end{itemize}
It follows that $(H_{\FC},g_{\FC},\gamma_{\FC,\FC'})_{\FC,\FC'}$ defines a depth-$0$
object in $\Coef_{R}([\BT'_{\iota,\bullet/e}/G_{\iota}])$, that we can also see as an
object in $\Coef_{R}([\BT'_{\iota,\bullet/e}/\bar G_{\iota}])$ and then in
$\Coef_{R}([\BT'_{\iota,\bullet/e}/K_{\iota}])$, as in \ref{geom_interp}.
By construction, the evaluation maps
$\WC_{\iota,\FC}\otimes_{R}H_{\FC}\simto\WC'_{\iota,\FC}$ induce an
isomorphism $\WC_{\iota}\otimes_{R}H\simto\WC'_{\iota}$ in $\Coef_{R}([\BT'_{\iota,\bullet/e}/K_{\iota}])$.

Let us now set $L:=\colim_{\BT'_{\iota,\bullet/e}}H$. Since all face maps of $H$ are
isomorphisms, $L$ is an invertible $R$-module and all canonical maps $H_{\FC}\To{{\rm can}_{\FC} }L$ are
isomorphisms. The $G_{\iota}$-equivariant structure induces a smooth action of $G_{\iota}$
on $L$, given by a character $\theta:G_{\iota}\To{} R^{\times}$, which has depth $0$ since
it is trivial on each subgroup $G_{\iota,\FC,0+}$.  Let us denote by
$L_{\theta}$ the $R$-module $L$ with its action of $G_{\iota}$ through $\theta$. We see that the
collection of maps $({\rm can}_{\FC})_{\FC\in\BT'_{\iota,\bullet/e}}$ provides an
isomorphism from $H$ to the constant $G_{\iota}$-equivariant system associated to $L_{\theta}$. In
other word, we have constructed an isomorphism
$ \WC_{\iota}\otimes_{R}L_{\theta}\simto \WC'_{\iota}$.
\qed

\section{Sample applications}

\subsection{Reduction to depth zero for Hecke algebras of types}\label{sec:hecke-types}

\alin{The setup}
Let $R=\cC$ be an algebraically closed field of characteristic different from $p$, e.g., the complex numbers.
Let $((\overrightarrow{\mathbf{G}}, \mathbf{M}^0), \overrightarrow{r}, x, (K_{M^0}, \rho_{M^0}), \overrightarrow{\psi})$ be a $\mathbf{G}$-datum as in \cite[Definition~4.1.1]{AFMO2} that follows \cite[7.2]{Kim_Yu}. Contrary to \cite{AFMO2} and \cite{Kim_Yu} we do not record the embeddings of the extended Bruhat--Tits buildings of the twisted Levi subgroups $\mathbf{M}^i \subseteq \mathbf{G}^i \subseteq \mathbf{G}$, where $\mathbf{M}^i$ is the centralizer in $\mathbf{G}^i$ of the maximal split torus $\mathbf{A}_{\mathbf{M}^0}$ in the center of $\mathbf{M}^0$, as part of the datum because here we fix such embeddings and then choose $x \in \BT(\mathbf{M}^0) \subseteq \BT(\mathbf{G})$ so that all the conditions of a $\mathbf{G}$-datum are satisfied. From such a $\mathbf G$-datum, the construction in \cite{AFMO2} that follows the construction of types by Kim and Yu (\cite{Kim_Yu}) but includes a twist by the quadratic character of \cite{FKS}, provides us with a compact, open subgroup $K_x \subset G$ and a representation $\rho_x$ thereof. Moreover, 
$((\mathbf{G}^0, \mathbf{M}^0), x, (K_{M^0}, \rho_{M^0}))$ is a depth-zero datum to which there is attached a compact, open subgroup $K^0_x \supseteq K_{M^0}$ of $G^0$ with representation $\rho_x^0$ whose restriction to $K_{M^0}$ is $\rho_{M^0}$, see \cite[5.1]{AFMO1} or \cite[4.1]{AFMO2}. Note that $K^0_x$ is contained in $G^0_x$ and contains the parahoric subgroup $G^0_{x,0}$.
We write $\cH(G, K_x, \rho_{x}):=\End_G(\cind_{K_x}^{G}\rho_x)$ and $\cH(G, K_x^0, \rho_x^0):=\End_G(\cind_{K_x^0}^{G^0}\rho_x^0)$ for the Hecke algebras attached to $(K_x, \rho_x)$ and  $(K_x^0, \rho_x^0)$.
If $\cC=\bC$, then $(K_x^0, \rho_x^0)$ is a depth-zero type for $G^0$
and $(K_x, \rho_x)$ is a type for $G$. If, in addition, $K_{M^0}$ is
chosen to equal $M^0_x$, then these types describe single Bernstein
blocks. In that case, the  Bernstein blocks are equivalent to the
category of right unital complex $\cH(G, K_x, \rho_{x})$-modules and
$\cH(G, K_x^0, \rho_x^0)$-modules, respectively. If $p$ does not
divide the order of the Weyl group of $\mathbf G$, then every
Bernstein block admits a type of this form (\cite{JF_exhaust}).
\begin{coro}\label{cor-Hecke-alg-isom}
	Suppose that $p$ is odd and not a torsion prime of  $\hat{\mathbf G}$, nor of $\mathbf{G}$. 
	We have an isomorphism of $\cC$-algebras $\cH(G, K_x,\rho_{x}) \simeq \cH(G^0, K_x^0, \rho_{x}^0)$.
\end{coro}
\begin{proof}
	We denote by $\bar x$ the image of $x$ in $\BT'_\iota$.
	The type $(K_x, \rho_x)$ only depends on the product $\psi:=\prod \psi_i|_{G^0}$ and not on the single characters $\psi_i$. Applying the construction of the characters from directly before \cite[Lemma~7.3]{JF_exhaust} (there the characters are called $\phi_i$, here we call them $\psi_i'$) to an irreducible subrepresentation of the restriction $\rho_x|_{K^0_x\Kp_{\iota, \bar x}}$ of $\rho_x$ to $K^0_x\Kp_{\iota, \bar x} \subseteq K_x$ (using that the roots of unity are divisible rather than Pontryagin duality to extend characters) and using \cite[Lemma~7.3]{JF_exhaust}, we obtain $\mathbf{G}^{i+1}$-generic characters $\psi_i'$ of $G^i$ that are trivial on the derived subgroup of $G^i$ such that $\prod \psi_i|_{G^0_{x,0+}}=\prod \psi_i'|_{G^0_{x,0+}}$. Setting $\psi_i''=\psi_i'$ for $i>0$ and $\psi_0''=\psi_0' \delta$ for an appropriate character $\delta$ of $G^0$ that is trivial on $G^0_{x,0+}$, we obtain that $\prod \psi_i|_{G^0}=\prod \psi_i''|_{G^0}$, and hence the types constructed from $((\overrightarrow{\mathbf{G}}, \mathbf{M}^0), \overrightarrow{r}, x, (K_{M^0}, \rho_{M^0}), \overrightarrow{\psi})$ and $((\overrightarrow{\mathbf{G}}, \mathbf{M}^0), \overrightarrow{r}, x, (K_{M^0}, \rho_{M^0}), \overrightarrow{\psi''})$ agree. Moreover, also the pairs $(\Kp_{x}, \hat \psi_x)$ attached to the three truncated Yu data $(\overrightarrow{\mathbf{G}}, \overrightarrow{\psi}, x)$,  $(\overrightarrow{\mathbf{G}}, \overrightarrow{\psi'}, x)$ and $(\overrightarrow{\mathbf{G}}, \overrightarrow{\psi''}, x)$ by Yu as in \ref{exhaust} agree. 
        According to Lemma \ref{from_Yu_data_to_parameters}, there exists a pair $(\phi, \iota)$ consisting of a wild inertia parameter $\phi$ and $\iota \in I \subseteq I_{\phi}$ such that $\mathbf{G}^0_\iota=\mathbf{G}^0$,  $\Kp_{\iota, \bar x}=\Kp_{x}$, and $\check\phi^{+}_{\iota, \bar x}=\hat \psi_x$.
	Now 
	the group $K^0_x$ is contained in $G_{\iota,\bar x}$ and by the construction of the types $K_x=K^0_x\Ku_{\iota, \bar x}$ and $\rho_x = \rho_x^0 \otimes \kappa_{\iota, x}|_{K_x}$
	for a twisted $\cC\K_{\iota,\bar x}$-Heisenberg--Weil representation whose restriction to $\Kp_{\iota, x}=\Kp_{\iota, \bar x}$ is $\hat\psi_x=\check\phi^{+}_{\iota, \bar x}$-isotypic.
	Let $\rho:= \cind_{K^0_x}^{G_{\iota,\bar x}}\rho_x^0$, and let $\cW_\iota$ be a Heisenberg--Weil coefficient system on $\BT'_{\iota, \bullet/e}$ such that $\kappa_{\iota, x}=\cW_{\iota, \cF}$, where $\cF$ denotes the $e$-facet that contains $\bar x$, which exists by Corollary \ref{cor-choice-of-WH}. By Corollary \ref{coro_indep_e}, under the equivalence $\cI_{\cW_\iota}$ of categories between $\Rep_{R}^1(G_\iota)$ and $\Rep_R^{\phi, I}(G)$, the representation $\cind_{K_x^0}^{G^0}\rho_x^0 = \cind_{G_{\iota,\bar x}}^{G_\iota}\rho$ gets send to a representation isomorphic to $\cind_{K_{\iota, \bar x}}^{G}(\cW_{\iota, \cF} \otimes_\cC \rho)=\cind_{K_{\iota, \bar x}}^{G}(\kappa_{\iota,x} \otimes_\cC \cind_{K_x}^{K_{\iota,\bar x}}\rho_x^0)=\cind_{K_x}^{G}(\kappa_{\iota,x} \otimes \rho_x^0)\simeq \cind_{K_x}^{G}\rho_x$, where we used that $K_{\iota, \bar x}=G_{\iota,\bar x}\Ku_{\iota, \bar x}$ and $K_x=K^0_x\Ku_{\iota, \bar x}$. Hence $\cH(G, K_x^0, \rho_x^0)=\End_G(\cind_{K_x^0}^{G^0}\rho_x^0)\simeq \End_G(\cind_{K_x}^{G}\rho_x)=\cH(G, K_x, \rho_{x})$.
\end{proof}

This result was previously proven in \cite[Theorem~4.4.1]{AFMO2} (under the assumption that $K_{M^0}$ is normalized by $N_{G^0}(M^0)(F)_{[x]_{M^0}}$ in the notation of \textit{loc.\ cit.}), where a more explicit isomorphism of Hecke algebras has been obtained.

\subsection{Projective generators, and connection with a result of Chinello}

\alin{Projective generators} We return tro the general setting that $R$ may be any commutative $\Rmintwo$-algebra, and we fix a pair $(\phi,I)$ and an embedding $\iota \in I$ as above.
Let $\CC_{0}$ be a chamber of $\BT'(\mathbf{G}_{\iota},F)$ (for the usual Bruhat--Tits
polysimplicial structure), and let  $\bar S$ 
be a set  of representatives of $G_{\iota}$-orbits of vertices of $\BT'(\mathbf{G}_{\iota},F)$ that
is contained in the closure of $\CC_{0}$. 
For each $\bar x \in \bar S$, pick an $e$-vertex $x$ in $\BT'_{\iota,\bullet/e}$ above $\bar x$, and
denote by $S$ the subset of $\BT'_{\iota,\bullet/e}$ formed by these lifts. 
The depth-$0$ category $\Rep_{R}^{1}(G_{\iota})$ is generated by the direct sum
$\bigoplus_{x\in S}\cind_{G_{\iota,x,0+}}^{G_{\iota}}(R)$. Hence, according to Theorem \ref{thm-main}, the $RG$-module
$$\IC_{\WC_{\iota}}\left(\bigoplus_{x\in S}\cind_{G_{\iota,x,0+}}^{G_{\iota}}(R)\right)
\simeq \bigoplus_{x\in S} \cind_{\Ku_{\iota,x}}^{G}(\WC_{\iota,x}|_{\Ku_{\iota,x}})$$
is a projective generator of $\Rep^{\phi,I}_{R}(G)$. Recall from our explicit construction of $\{ \WC_{\iota,x}\}$ that the restriction
$\WC_{\iota,x}|_{\Ku_{\iota,x}}$ is just a Heisenberg representation for
$(\Ku_{\iota,x},\Kp_{\iota,x},\check\phi_{\iota,x}^{+})$, whose definition is quite simple
and does not need the more subtle theory of Heisenberg--Weil representations.  However, the
isomorphism of $R$-algebras
$$\End_{RG_{\iota}}\left(\bigoplus_{x\in S}\cind_{G_{\iota,x,0+}}^{G_{\iota}}(R)\right)
\simto \End_{RG}\left(\bigoplus_{x\in S} \cind_{\Ku_{\iota,x}}^{G}(\WC_{\iota,x}|_{\Ku_{\iota,x}})\right)$$
induced by $\IC_{\WC_{\iota}}$ a priori uses the full force of Theorem \ref{thm_WH_coef_system-general}.

\alin{Example: $\mathbf{G}=\GL_{n}$} \label{example_gln}
Suppose $\mathbf{G}=\GL_{n}$. 
Following \cite[2.4.6]{Datfuncto}, 
the centralizer $C_{\hat{\mathbf{G}}}(\phi)$ together with its canonical outer action of $W_{F}$ is of the form
$\prod_{i=1}^{r}{\rm Ind}_{W_{F_{i}}}^{W_{F}}\GL_{e_{i}}$ for some  tamely ramified extensions $F_{i}/F$
of degree $d_{i}$ satisfying $\sum_{i=1}^{r}e_{i}d_{i}=n$. In other words, we have
$\mathbf{G}_{\phi}=\prod_{i=1}^{r} {\rm Res}_{F_{i}|F}\GL_{e_{i}}$.  Accordingly, we have
$\mathbf{S}_{\phi}=\prod_{i=1}^{r} {\rm Res}_{F_{i}|F}\GL_{1}$. Moreover, there is a unique
$\mathbf{G}(F)$-conjugacy class $I$ of $F$-rational Levi-center-embeddings $\mathbf{S}_{\phi}\injo \mathbf{G}$,
so we may suppress $I$ from the discussion, and all $\mathbf{G}_{\iota}$, $\iota\in I$, are
isomorphic to $\mathbf{G}_{\phi}$. 

In this setting, we know that the depth-$0$ category $\Rep^{1}_{R}(G_{\iota})$ is pro-generated by
$\cind_{G_{\iota,x,0+}}^{G_{\iota}}(R)$ for any point $x$ in $\BT'_{\iota}$ that maps to a vertex of
$\BT'(\mathbf{G}_{\iota},F)$. 
It follows that $\Rep^{\phi}_{R}(G)$ is pro-generated by
$\cind_{\Ku_{\iota,x}}^{G}(\Heis_{\iota,x})$ where
 $\Heis_{\iota,x}$ is any Heisenberg representation for $(\Ku_{\iota,x},\Kp_{\iota,x},\check\phi_{\iota,x}^{+})$.
As above, our equivalence $\IC_{\WC_{\iota}}$  induces an isomorphism of intertwining $R$-algebras
$$\End_{RG_{\iota}}(\cind_{G_{\iota,x,0+}}^{G_{\iota}}(R))
\simto \End_{RG}(\cind_{\Ku_{\iota,x}}^{G}(\Heis_{\iota,\FC_{0}}))$$
and, actually, $\IC_{\WC_{\iota}}$ coincides with the equivalence $\Rep^{1}_{R}(G_{\iota})\simto \Rep^{\phi}_{R}(G)$
induced by that isomorphism  between the  endomorphism algebras of the respective
pro-generators. In this setup, such an isomorphism of algebras was already obtained by 
Chinello in \cite{Chinello}, via explicit computation of generators and relations.

\subsection{Reduction to depth zero and the local Langlands correspondence}
Our reduction-to-depth-$0$ process on the representation theoretic
side of the Langlands program is parallel to a similar process on the Galois side that was
described, e.g., in \cite[\S 3]{DHKM1}. It is thus reasonable to try and leverage both processes to
expand what is known about the local
Langlands correspondence and its recent enhancements in depth $0$ to higher depth.

\alin{Example: about the categorification of the  local Langlands
  correspondence}\label{subsec-CLLC}
Let us assume $\mathbf{G}$ quasi-split. Starting from $\phi\in\Phi(P_{F},\mathbf{G})$
whose centralizer is a Levi subgroup, we might want to fill in the upper arrow in the following
diagram :

\ini\begin{equation}
  \label{diagram_cllc}
  \xymatrix{
    D\Rep^{\phi}_{R}(G) \ar@{^{(}-->}[r] &
    {\rm Ind.Coh}\left(Z^{1}(W_{F}^{0},\hat{\mathbf{G}})/\hat{\mathbf{G}}\right)_{[\phi],R} \\
    \prod_{I\in \IC_{\phi}} D\Rep^{1}_{R}(G_{\phi,I})\ar@{^{(}->}[r]\ar[u]^{\simeq}_{\IC_{\check\varphi_{0}}} &
    {\rm Ind.Coh}\left(Z^{1}(W_{F}^{0}/P_{F},\hat{\mathbf{G}}_{\phi})/\hat{\mathbf{G}}_{\phi}\right)_{R}
    \ar[u]_{\simeq}^{\IC^{\rm spec}_{\varphi_{0}}}
  }
\end{equation}

  In this diagram, we have departed a bit from  our general notation by considering
  $\hat{\mathbf{G}}$ as a group scheme over $\o\ZM[\frac 1p]$.
  We know from \cite[Thm 3.1]{DHKM1} that $\phi$ can be conjugate so as to factor through
  $\hat{\mathbf{G}}(\o\ZM[\frac 1p])$, so in particular $\hat{\mathbf{G}}_{\phi}=C_{\hat{\mathbf{G}}}(\phi)$
  is also a group scheme over $\o\ZM[\frac 1p]$. 
  Let us explain what the solid arrows of the diagram above are about.
  \begin{itemize}
  \item The bottom horizontal arrow denotes the fully faithful embedding that  would follow from the
    ``tame'' categorical local Langlands 
    correspondence introduced in \cite[\S 1.2.2]{Zhu} applied to the quasi-split group
    $\mathbf{G}_{\phi}$. In particular, \emph{a choice  of $I_{0}\in\IC_{\phi}$ with
      $\mathbf{G}_{\phi,I_{0}}\simeq \mathbf{G}_{\phi}$ has to be
      made to get such an embedding.} Once this choice is made,
    from the map $\IC_{\phi}\To{} H^{1}(F,\mathbf{G}_{\phi,I_{0}})$ of Proposition
    \ref{dual_embeddings_phi} iii), we get  a map $\IC_{\phi}\injo B(\mathbf{G}_{\phi})_{{\rm
        basic}}$, $I\mapsto b(I)$. We can then identify $D\Rep_{R}(G_{\phi,I})$ with the
    subcategory of sheaves on ${\rm Isoc}_{G_{\phi}}$ supported on the stratum $b=b(I)$, and
    the claimed embedding then follows from Zhu's statement.
    Note that, at the time of writing these lines, Zhu's statement has been proved only  for
    coefficients  $R=\o\QM_{\ell}$ and assuming that the group is 
    \emph{unramified}, see \cite[Thm 1.6]{Zhu}.
    In contrast, our $\mathbf{G}_{\phi}$ may be tamely ramified. However, in the case
    $\mathbf{G}=\GL_{n}$, the latter issue can be bypassed since each $\mathbf{G}_{\phi}$ is a product
    of restriction of scalars of $\GL_{m}$'s.

  \item The left vertical arrow is provided by this paper. Again, \emph{several choices are possible}.
    Assuming that ${\rm Pic}(R)=\{0\}$, e.g., $R=\o\ZM[\frac 1p]$ or $R=\o\FM_{\ell}$ for $\ell\neq p$, 
    Propositions \ref{uniqueness_WH_coef-system} and \ref{prop_indep_e} imply that the set of
    equivalences $\Rep_{R}(G_{\phi,I})\simto\Rep_{R}^{\phi,I}(G)$ constructed in this paper is a
    torsor under the group of characters $G_{\phi,I}\To{}R^{\times}$ that are trivial on $(G_{\phi,I})_{x,0+}$ for all $x \in \BT(\mathbf{G}_{\phi, I},F)$.
     Actually, among the
    choices made to construct these equivalences (including, for example, the choice of a  $\iota\in
    I$), the only significant one is the choice of a character $\check\varphi_{0}: G_{\iota}\To{}
    R^{\times}$ that extends each
    $(\check\phi_{\iota,x}^{+})_{|G_{\iota,x,0+}}$, made in the definition  of
    Heisenberg--Weil representations \ref{Def-HeisWeil-non-twisted}. In diagram (\ref{diagram_cllc}),
     we have denoted by  $\IC_{\check\varphi_{0}}$ the product of the equivalences associated
    with the Heisenberg--Weil coefficient systems stemming from this choice.

  \item The right vertical arrow is essentially provided by \cite{DHKM1}. By Lemma
    \ref{splittings_Levi}, there exists an extension $\varphi_{0}:W_{F}\To{}{^{L}\mathbf{G}}(\CM)$ of
    $\phi$ such that $\varphi_{0}(W_{F})$ normalizes a pinning $\varepsilon_{\phi}$ of
    $C_{\hat{\mathbf{G}}}(\phi)$. It is not clear if we can choose $\varphi_{0}$ such that 
    $\varphi_{0}(W_{F})\subset{^{L}\mathbf{G}}(\o\ZM[\frac 1{p}])$ but the argument of \cite[Thm.~4.6]{DatIHES}
    shows that we can find one such that $\varphi_{0}(W_{F})\subset{^{L}\mathbf{G}}(\o\ZM[\frac 1{pN}])$,
    whenever the center $Z(\hat{\mathbf{G}})$ is smooth over $\o\ZM[\frac 1{pN}]$.    
    For example, one can set $N=|\pi_{1}(\mathbf{G}_{\rm der})|$.
    \emph{Once $\varphi_{0}$ is chosen},
    the map $(c,w)\mapsto c\varphi_{0}(w)$ provides an $L$-homomorphism
 $$^{L}\varphi_{0}:\,{^{L}\mathbf{G}_{\phi}}=C_{\hat{\mathbf{G}}}(\phi)\rtimes_{\Ad_{\varphi_{0}}}W_{F}\To{}{^{L}\mathbf{G}}.$$
 By construction, the induced map on $1$-cocycles 
    $Z^{1}(W^{0}_{F},\hat{\mathbf{G}}_{\phi})_{R}\simto Z^{1}(W_{F}^{0},\hat{\mathbf{G}})_{R}$
    restricts to  an isomorphism
    $Z^{1}(W^{0}_{F}/P_{F},\hat{\mathbf{G}}_{\phi})_{R}\simto
    Z^{1}(W_{F}^{0},\hat{\mathbf{G}})_{\phi,R}$ for any $\o\ZM[\frac 1{pN}]$-algebra $R$, 
    where the
    right hand side is the moduli space of extensions of $\phi$ to $W_{F}^{0}$. It induces in turn an
    isomorphism of $R$-stacks
    $$ \left(Z^{1}(W^{0}_{F}/P_{F},\hat{\mathbf{G}}_{\phi})/\hat{\mathbf{G}}_{\phi}\right)_{R}
\simto \left(Z^{1}(W_{F}^{0},\hat{\mathbf{G}})/\hat{\mathbf{G}}\right)_{[\phi],R},$$
where the right hand side denotes the summand associated with $\phi$. The symbol $\IC^{\rm
  spec}_{\varphi_{0}}$ of the above diagram is the associated equivalence on ${\rm Ind.Coh}$.   
  \end{itemize}

\emph{How to make these choices ?}  From the above discussion, each one of the solid arrows of the
above diagram depends on choices. In order to get the ``correct'' dashed arrow, one should carefully
make these choices. Here are our expectations, that we hope to verify in future work.
\begin{itemize}
\item The choice of $I_{0}\in \IC_{\phi}$ will depend on a choice of a Whittaker datum $(U,\psi)$   in
  $\mathbf{G}$. Denoting by $\Gamma_{U,\psi}$ the associated Gelfand--Graev $\R$-representation of $G$,
  the set $I_{0}$ should be the unique $I$ in $\IC_{\phi}$ such that the $(\phi,I)$-component
  $(\Gamma_{U,\psi})^{\phi,I}$ of  $\Gamma_{U,\psi}$ is non-zero. Moreover,
  $\IC_{\check\varphi_{0}}^{-1}((\Gamma_{U,\psi})^{\phi,I_{0}})$ should be the depth-$0$ summand of a
  Gelfand--Graev representation of $\mathbf{G}_{\phi,I_{0}}$, that should be used to normalize the
  tame categorical local Langlands correspondence for $\mathbf{G}_{\phi,I_{0}}$ (the lower map in
  our diagram above).
\item There is a priori no best choice for $\varphi_{0}$ nor $\check\varphi_{0}$, but we expect
  that over $R=\o\ZM[\frac 1{pN}]$, 
  for any choice of 
  $\varphi_{0}$, there should be a unique $\check\varphi_{0}$ such that the equivalence
  $\IC_{\check\varphi_{0}}$ is compatible with Langlands functoriality along $^{L}\varphi_{0}$ for irreducible complex 
  representations. 
Before explaining what we mean, here is a natural way of producing a map $\varphi_{0}\mapsto\check\varphi_{0}$.
Suppose given a \emph{tamely ramified }$L$-embedding
${^{L}\mathbf{G}_{\phi}}\injo {^{L}\mathbf{G}}$, i.e., an embedding 
that extends both the inclusion
$C_{\hat{\mathbf{G}}}(\phi)\subset \hat{\mathbf{G}}$ and the identity on $\{1\}\times P_{F}$. 
Writing  such an embedding  as $(c,w)\mapsto c.\psi(w)$ for some
 $\psi:\,W_{F}\To{}{^{L}\mathbf{G}}$ whose image normalizes the pinning $\varepsilon_{\phi}$ of
 $C_{\hat{\mathbf{G}}}(\phi)$, we can then write $\varphi_{0}=\hat\varphi_{0}.\psi$ for some
 $\hat\varphi_{0}\in Z^{1}(W_{F},Z(C_{\hat{\mathbf{G}}}(\phi)))$ that extends $\phi$, as in (the
 proof of) Lemma \ref{splittings_tr}. Applying Borel's procedure to $\hat\varphi_{0}$ then provides
 a $\check\varphi_{0}$. 
 We expect that a good choice of $\psi$, associated to suitable $\chi$-data as in \cite[\S 6.1]{Kaletha_double},
 will provide the ``correct'' matching  between $\varphi_{0}$ and $\check\varphi_{0}$, in the sense
 that $\IC_{\check\varphi_{0}}$ induces the functorial lifting along $^{L}\varphi_{0}$ of the  $L$-packets
 associated to supercuspidal Langlands parameters in \cite{Kaletha_Lpackets}. 
 In any case, once $\psi$ is chosen (or, equivalently, the $L$-embedding
 ${^{L}\mathbf{G}_{\phi}}\injo {^{L}\mathbf{G}}$),  
 the dashed arrow that makes  Diagram (\ref{diagram_cllc}) commutative is independent of the
 choice of a matching pair $(\varphi_{0},\check\varphi_{0})$. 
\end{itemize}

\alin{The example of $\GL_{n}$} \label{CLLC_GLn} In the case  $\mathbf{G}=\GL_{n}$, several points in the above 
discussion simplify drastically. For one, the sets $\IC_{\phi}$ are always singletons, so that there
is no need to choose a base point there. Moreover, when $R=\o\QM_{\ell}$, the tame categorical Langlands
correspondence for $\mathbf{G}_{\phi}$ is available from Zhu's work, because $\mathbf{G}_{\phi}$ is a product of
restrictions of scalars of $\GL_{n_{i}}$'s.  Therefore, the bottom map
of Diagram (\ref{diagram_cllc}) is available for all $\phi\in\Phi(P_{F},\GL_{n})$ (regardless of
$\phi$ having abelian image, actually).

Moreover, since 
$\hat{\mathbf{G}}_{\phi}=C_{\hat{\mathbf{G}}}(\phi)$ is a Levi subgroup of $\GL_{n}$, 
 the projection map
$N_{\hat{\mathbf{G}}}(C_{\hat{\mathbf{G}}}(\phi))\To{}N_{\hat{\mathbf{G}}}(C_{\hat{\mathbf{G}}}(\phi))/C_{\hat{\mathbf{G}}}(\phi)$
has natural sections given by permutation matrices.
Using again that $\mathbf{G}_{\phi}$ is a product of restrictions of scalars of $\GL_{n_{i}}$,
this provides a ``natural'' way of choosing an embedding
$^{L}\mathbf{G}_{\phi}\To{}{^{L}\mathbf{G}}$ and a $\psi$ as in the last bullet point of Section \ref{subsec-CLLC}. 
So, associated to this choice of $\psi$, we get an embedding
$D\Rep^{\phi}_{\o\QM_{\ell}}(\GL_{n}(F))\injo {\rm Ind.Coh}
\left(\Hom(W_{F}^{0},\GL_{n})/\GL_{n}\right)_{[\phi],\o\QM_{\ell}}$
that makes diagram (\ref{diagram_cllc}) commutative.
However, this embedding generally  needs to be ``rectified''. In the 
work of Bushnell and Henniart, this rectification is given by 
precomposing the left vertical arrow of (\ref{diagram_cllc}) by twisting with a certain tamely
ramified character of $G_{\phi}$, called a ``rectifier''.
Alternatively, one could also modify the chosen embedding $^{L}\mathbf{G}_{\phi}\To{}{^{L}\mathbf{G}}$
using appropriate $\chi$-data. The dictionary between both approaches is explained
in \cite{Tam} and \cite{Oi-Tokimoto}, in a slightly different setting, but we expect that their arguments can be adapted
to our setting.
\begin{cor}
  Assuming  that $p>n$,  there is an embedding
$$ D\Rep_{\o\QM_{\ell}}(\GL_{n}(F))\injo {\rm Ind.Coh}
\left(\Hom(W_{F}^{0},\GL_{n})/\GL_{n}\right)_{\o\QM_{\ell}}.$$
\end{cor}
\begin{proof}
 Since $p>n$, the centralizer of any $\phi\in\Phi(P_{F},\GL_{n})$ is a
 Levi subgroup, so, after making choices as explained above,  it only remains to  sum over all
  $\phi\in\Phi(P_{F},\GL_{n})$.
\end{proof}
Such an embedding has already been constructed in \cite{BZCHN}. It would be interesting to compare both constructions.

\medskip

\emph{Over $\o\FM_{\ell}$}. As of writing this section, the tame CLLC is not available over
$\o\FM_{\ell}$. However, Zhu's paper contains a \emph{unipotent}
version, at least when $\ell>n$, that we introduced in the
end of the introduction. In Corollary 1.1.4 of \cite{Datequiv}, a ``reduction-to-unipotent'' result for depth-$0$ blocks of
$\Rep_{\o\FM_{\ell}}(\GL_{n}(F))$ is proven, with a pattern very similar to what we are doing
in this paper, except that we start from tamely ramified ``inertia parameters''. We refer to Section
1.1 of \cite{Datequiv} and also Paragraph 1.2.4 of \cite{Datfuncto} for the details.
Combining that result with our reduction-to-depth-$0$ result, we
obtain the following
\begin{cor}
  Assuming that $p>n$ and $\ell>n$, there is an embedding 
$$ D\Rep_{\o\FM_{\ell}}(\GL_{n}(F))\injo {\rm Ind.Coh}
\left(\Hom(W_{F}^{0},\GL_{n})/\GL_{n}\right)_{\o\FM_{\ell}}.$$
\end{cor}
\begin{proof}
  Same as the last corollary, using Corollary 1.1.4 of \cite{Datequiv}
  on top of our equivalences.
\end{proof}
Again, some choices are involved, and more work is needed to exhibit the ``correct'' embedding.

\clearpage
\printindex[notation]
\printindex[terminology]
\clearpage

\bibliography{artbiblio}

\end{document}